\documentclass[11 pt]{amsart}

\usepackage[utf8]{inputenc}
\usepackage{tikz, tikz-3dplot}
\usepackage{amsmath, amsthm, amssymb,graphics,esint }
\usepackage{amsmath,amsthm,amssymb,amscd,color, xcolor,mathtools,url,tikz}
\usepackage{bbm}
\usepackage{pgfplots}

\newcommand{\tmem}{\underline}
\newcommand{\comment}[1]{\vskip.3cm
\fbox{%
\parbox{0.93\linewidth}{\footnotesize #1}}
\vskip.3cm}

\newcommand{\nin}{\not\in}
\newcommand{\nobracket}{}
\newcommand{\tmop}[1]{\ensuremath{\operatorname{#1}}}
\newtheorem{definition}{Definition}[section]

\newtheorem{thm}{Theorem}[section]

 \newtheorem{lem}{Lemma}[section]
 \newtheorem{prop}{Proposition}[section]

 \theoremstyle{definition}
 \newtheorem{df}[thm]{Definition}
 \theoremstyle{remark}
 \newtheorem{rem}[thm]{Remark}
 
 \numberwithin{equation}{section}

\newcommand{\dd}{\mathop{}\!\mathrm{d}}

\title{Estimates for the Gross-Pitaevskii equation linearized around a vortex}

\author{Charles Collot}
\address{CY Cergy Paris University. Laboratoire AGM, 2 avenue Adolphe Chauvin
95302 Cergy-Pontoise}
\email{charles.collot@cyu.fr}

\author{Pierre Germain} 
\address{Department of Mathematics, Huxley building, South Kensington campus, Imperial College London, London SW7 2AZ, United Kingdom}
\email{pgermain@ic.ac.uk}

\author{Eliot Pacherie}
\address{CNRS and CY Cergy Paris University. Laboratoire AGM, 2 avenue Adolphe Chauvin
95302 Cergy-Pontoise}
\email{eliot.pacherie@cyu.fr}

\begin{document}

\begin{abstract}
We consider the linearized  two-dimensional Gross-Pitaevskii equation around a vortex of degree one, with data in the same equivariance class. Various estimates are proved for the solution; in particular, conditions for optimal decay in $L^\infty$ and boundedness in $L^2$ are identified. The analysis relies on a full description of the spectral resolution of the linearized operator through the associated distorted Fourier transform.
\end{abstract}

\maketitle

\tableofcontents

\section{Introduction}

\subsection{The Gross-Pitaevskii equation and its stationary vortex}
We consider the Gross-Pitaevskii equation
\begin{equation}
\label{GP} \tag{GP}
i \partial_t \psi + \Delta \psi + (1 - |\psi|^2) \psi = 0
\end{equation}
set in space dimension 2. This equation appears in the study of Bose-Einstein condensates, superfluidity and superconductivity (see for instance \cite{AbidHueMetNoreBra},\cite{Ginz-Pit_58a},\cite{KL},\cite{Neu_90},\cite{Pismen}). Adopting radial coordinates $(r,\theta)$ such that $(x_1 , x_2) = r (\cos \theta, \sin \theta)$, we focus on solutions whose expansion in Fourier modes only contain the first harmonic in the angular coordinate. In other words, we work under the ansatz
$$
\psi(t,x) = w(t,r) e^{i\theta}
$$
which is propagated by the flow of the equation.
Since $\Delta(f(r) e^{i\theta}) = \left( \partial_r^2 f + \frac{1}{r} \partial_r f - \frac {1}{r^2} f\right) e^{i\theta}$, the equation becomes
$$
i \partial_t w + \partial_r^2 w + \frac{1}{r} \partial_r w - \frac{1}{r^2} w + (1-|w|^2) w = 0.
$$

There exists a unique stationary solution $\rho(r)$ which satisfies
$$
\partial_r^2 \rho + \frac{1}{r} \partial_r \rho - \frac{1}{r^2} \rho + (1-\rho^2) \rho = 0,
$$
is smooth, bounded and nonnegative with limit $1$ as $r\to \infty$, see Lemma \ref{rhostuff}. The function $\rho(r) e^{i \theta}$ is the only stationnary solution of the Gross-Pitaevskii equation with degree 1 at $+\infty$, see \cite{Mironescu1996LesML}. Our aim is to investigate its asymptotic stability, first at the linear level.

\subsection{The linearized operator}
\label{subsectionlinearized}
Linearizing around $\rho$ gives the equation
\begin{equation}
\begin{split}
\label{equationL}
i \partial_t w + \mathcal{L} w = 0 \qquad \mbox{with} \qquad
\mathcal{L} w = \partial_r^2 w + \frac{1}{r} \partial_r w - \frac{1}{r^2} w + w - 2 \rho^2 w - \rho^2 \overline{w}.
\end{split}
\end{equation}

In order to make this operator complex-linear, we can change the unknown function to
$$
\begin{pmatrix} u \\ v \end{pmatrix} = \begin{pmatrix} w \\ \overline{w} \end{pmatrix}.
$$
The equation \eqref{equationL} becomes equivalent to
\begin{equation}
\label{evolutionH}
i \partial_t \begin{pmatrix} u \\ v \end{pmatrix} = \mathcal{H} \begin{pmatrix} u \\ v \end{pmatrix} , \qquad 
\mathcal{H} = \mathcal{H}_0 + V,
\end{equation}
where the matrix operators $\mathcal{H}_0$ and $V$ are given by
\begin{align*}
& \mathcal{H}_0 = \begin{pmatrix} -\partial_r^2 - \frac{1}{r} \partial_r + \frac{1}{r^2} + 1 & 1 \\ -1 & \partial_r^2 + \frac{1}{r} \partial_r - \frac{1}{r^2} - 1 \end{pmatrix}, \\
& V = (\rho^2 - 1)\begin{pmatrix} 2 & 1 \\ -1 & -2 \end{pmatrix}.
\end{align*}

This is a one-dimensional differential operator, but it combines many difficulties
\begin{itemize}
\item It is a matrix rather than a scalar operator.
\item The constant-coefficient part of the operator is not diagonal;  it actually gives the dispersion relation $\tau = |\xi|\sqrt{2 + \xi^2}$ which interpolates between the wave and Schr\"odinger cases.
\item The potential part of the operator has the critical (scale-invariant) decay $\frac{1}{r^2}$.
\item There is a resonance in $L^\infty$ ("s-wave" in the terminology of two-dimensional Schr\"odinger operators) at zero energy. 
\item Besides the aforementioned resonance (kernel of $\mathcal{H}$), the kernel of $\mathcal{H}^2$ contains a further resonant state. Thus, geometric and algebraic multiplicities differ (in a generalized sense, we are dealing with resonances instead of eigenfunctions).
\end{itemize}

Our aim in the present article is to obtain a detailed description of the spectral resolution of the operator $\mathcal{H}$ through the associated Fourier transform. As a first consequence, we obtain dispersive estimates; this will also be the key to nonlinear stability, which will be the object of a forthcoming paper.

\subsection{Schr\"odinger operators, separation of variables and decay}

In order to prove decay estimates for Schr\"odinger operators (scalar and matrix) in dimension one, a classical approach is to express the spectral projectors (known in this context as the distorted Fourier transform) through the generalized eigenfunctions, which are the solutions of ODEs \cite{BuslaevPerelman,GoldbergSchlag,KriegerSchlag}. This approach is appealing since it gives explicit formulas, but it can be heavy-handed. It generalizes well to higher dimensions provided one can separate variables by exploiting the symmetries of the problem. For this, we refer in particular to \cite{PalaciosPusateri} which deals with the linear stability of vortices in the Ginzburg-Landau equation, which is the relativistic equivalent of the Gross-Pitaevskii equation studied here, see also \cite{SchlagSofferStaubach1,SchlagSofferStaubach2}. All these references address the presence of resonances and critically decaying potentials, which are also a feature of the problem studied in the present paper.

An alternative route to decay estimates proceeds by expanding the resolvent, which applies without symmetry assumptions but gives a less precise approach, see in particular the reviews \cite{Schlag1,Schlag2} and \cite{JourneSofferSogge,RodnianskiSchlag,GoldbergSchlag}.

Because of the various difficulties which are present in the problem at hand, it is not clear whether the general approach can succeed. Therefore, we resorted to separation of variables and the construction  of the distorted Fourier transform to prove decay. We furthermore focused on a single angular harmonic, which is natural since the corresponding symmetry of the solution is preserved by the nonlinear evolution.

We have a further motivation for the construction of the distorted Fourier transform, namely the application to the nonlinear problem in order to prove asymptotic stability of vortices in the Gross-Pitaevskii equation. This will be discussed in the next subsection.

\subsection{Application to the nonlinear problem} In carrying out a very precise analysis of the linearized problem, our ultimate goal is to prove asymptotic stability for the nonlinear problem.

The orbital stability of the vortex $V_1$ has been proven in a metric space in \cite{GPS}. In dimension 1, asymptotic stability of the black soliton, which plays a similar role to the vortex in dimension 2, has been shown in \cite{GraSme}. Higher degrees vortices exist, but they are expected to be unstable although radially stable, see \cite{OvSi}. Finally, we mention \cite{Gallo} and \cite{Gerard_Cauchy1} regarding the nonlinear Cauchy problem.

As far as asymptotic stability in dimension 2 goes, the only result seems to be the exclusion of exponentially growing modes for the linearized problem \cite{delP_Fel_Kow}, \cite{WeinsteinXin}. For the related question of the asymptotic stability of $1$ in the Gross-Pitaevskii equation, a precise analysis of nonlinear interactions led to important progress  \cite{GuoHaniNakanishi,GustafsonNakanishiTsai1,GustafsonNakanishiTsai2,GustafsonNakanishiTsai3}.

Our hope is that the distorted Fourier transform developed in the present article will lead to a proof of asymptotic stability for the nonlinear problem. The use of the distorted Fourier tranform in this context was pioneered in \cite{KriegerSchlag}, see \cite{KoSchlag} for an introduction to this tool. In combination with nonlinear resonances, it turned out to be very effective for one-dimensional problems \cite{ChenLuhrmann,CollotGermain,Germain,GermainPusateri,KairzhanPusateri,LiLuhrmann}.

\subsection{Main results}

\subsubsection{Decay for the group $e^{it\mathcal{H}}$}  Before stating the theorem, we define $\mathcal{S}_1$ to be the set of radial functions $\phi$ such that $\phi(r) e^{i\theta}$ belongs to the Schwartz class $\mathcal{S}$. Then, for $s\in \mathbb R$ we define the weighted norm $\| u\|_{L^{2,s}}=\| \langle r \rangle^s u\|_{L^2(r\dd r)}$ and denote by $L^{2,s}$ the associated Banach space.

\begin{thm}[Decay estimates for localized initial data] 

\label{thmdecayschwartz}
Let $\epsilon>0$. The following hold true

\smallskip
\begin{itemize}
\item[(i)] \emph{Large time decay estimates.} For $\phi \in L^{2,1+\epsilon}$, we have for $t \geqslant2$,
$$
\| e^{it \mathcal{H}} \phi \|_{L^\infty} \lesssim \frac{1}{t^{2/3}} \ \| \phi\|_{L^{2,1+\epsilon}} \qquad \mbox{and} \qquad \| e^{it \mathcal{H}} \phi \|_{L^2} \lesssim \sqrt{\ln t } \ \| \phi\|_{L^{2,1+\epsilon}}.
$$

\smallskip

\item[(ii)] \emph{Improved large time estimates.} If furthermore $ \phi\in L^{2,2+\epsilon}$ and $\langle \phi , (\rho,\rho)^\top \rangle_{L^2(r \dd r)} =0$, then
$$
\| e^{it \mathcal{H}} \phi \|_{L^\infty} \lesssim \frac{1}{t}\| \phi\|_{L^{2,2+\epsilon}} \qquad \mbox{and} \qquad \| e^{it \mathcal{H}} \phi \|_{L^2} \lesssim \| \phi\|_{L^{2,1+\epsilon}}.
$$

\smallskip

\item[(iii)] \emph{General large time estimates.} The above estimates are consequence of the more general estimates for $\phi\in L^{2,2+\epsilon}$ and $t \geqslant2$:
$$
\| e^{it \mathcal{H}} \phi \|_{L^\infty} \lesssim \frac{\| \phi\|_{L^{2,2+\epsilon}}}{t}+\frac{|\langle \phi , (\rho,\rho)^\top \rangle_{L^2(r \dd r)}|}{t^{2/3}}
$$
and for $\phi\in L^{2,1+\epsilon}$
$$
\| e^{it \mathcal{H}} \phi \|_{L^2} \lesssim \| \phi\|_{L^{2,1+\epsilon}}+|\langle \phi , (\rho,\rho)^\top \rangle_{L^2(r \dd r)}| \sqrt{\ln t } .
$$

\smallskip

\item[(iv)] \emph{Short time dispersive estimate.} If $\phi \in L^{2,1+\epsilon}$, then for $0<t< 2$ we have $\| e^{it \mathcal{H}} \phi \|_{L^\infty}\lesssim t^{-1}\| \phi\|_{L^{2,1+\epsilon}}$.

\end{itemize}
\end{thm}

\begin{rem}

\begin{itemize}

\item We believe that the estimate in the endpoint case $\epsilon=0$ might not be true, since the two dimensional embeddings $L^{2,1+\epsilon}\subset L^1$ and $H^{1+\epsilon}\subset L^\infty$ are only true if $\epsilon>0$. We stated decay estimates in weighted $L^2$ spaces, but our proof can be adapted to provide estimates on other function spaces.

\item Decay estimates in $L^p$ spaces for $2<p<\infty$ obviously follow by interpolation.

\item We believe that these estimates are sharp, since they coincide with the estimates for the linearization over the flat background $1$.

\item In particular, the growth of the $L^2$ norm is sharp, in that if $\langle \phi,(\rho,\rho)^\top\rangle_{L^2(r\dd r)}\neq 0$, then there holds $\| e^{it \mathcal{H}} \phi \|_{L^2} \gtrsim \sqrt{1+\ln \langle t\rangle}$ (as follows directly from the proof of Theorem \ref{thm-L2bound}). This shows that the vortex is not linearly stable in $H^1$ (as is the constant solution, see Annex \ref{section flat} or \cite{Gerard_Cauchy2}). On the other hand, it is nonlinearly orbitally stable in the energy space \cite{GPS}. While there is no contradiction since the energy space does not embed in $H^1$, this indicates that the nonlinear stability is sensitive to the topology in which it is measured.

\item The direction $(\rho,\rho)$ is related to a resonance of algebraic multiplicity $2$, see Lemma \ref{ResonanceLemma}.
\end{itemize}
\end{rem}

This theorem is stated without any reference to the spectral decomposition of $\mathcal{H}$, but its proof relies on the construction of the spectral projectors, also known as the distorted Fourier transform, which will be presented below.
Based on the scattering results below, we could also have developed more standard $L^p - L^{p'}$ dispersive estimates, but we refrained from doing so since they entail many more technicalities, and are furthermore less useful when it comes to the nonlinear problem.

\subsubsection{Scattering for the operator $\mathcal{H}$}
We obtain a full description of the generalized eigenfunctions of the operator $\mathcal{H}$. Its spectral properties are the following.

\begin{thm}[Spectrum of the linearized operator]
  \label{bug}

The following properties hold true.
\begin{itemize}
\item \emph{Absence of eigenmodes}. There are no $\lambda \in \mathbb C$ and $\psi \in L^2(r\dd r)\backslash \{0\}$ such that $\mathcal H \psi=\lambda \psi$. 
\item \emph{Essential spectrum}. The essential spectrum of $\mathcal H$ is $\mathbb R$.
\item \emph{Existence and uniqueness of generalized eigenfunction}. For any $\lambda \in \mathbb{R}$, there is a unique (up to normalization) bounded solution $\psi$ to $\mathcal H\psi=\lambda \psi$. Namely, the set $\mathcal{E}_{\lambda} = \ker(\mathcal{H}- \lambda)$ in $(\mathcal{C}^{\infty} (\mathbb{R}^{+ \ast})$ satisfies $\dim \mathcal({E}_\lambda \cap L^\infty) = 1$.
\end{itemize}
\end{thm}

The absence of unstable eigenmodes ($\mathfrak{Im} \lambda \neq 0$) as well as the description of the essential spectrum were obtained in \cite{WeinsteinXin}. So the novelty here is the absence of embedded eigenmodes ($\mathfrak{Im} \lambda =0$) and the existence and uniqueness of generalized eigenfunctions. We would also like to mention here the recent \cite{DPJM}, which focused on the inversion of $\mathcal{H}$ in weighted spaces.

This theorem will be proved in Sections \ref{sectiongeneralities} and \ref{sectionscattering}.
While $\lambda$ corresponds to the time frequency in the equation, it will often be convenient to use the parameter $\xi$, which gives the space frequency at $+ \infty$ of generalized eigenfunctions and is defined as
\[ \xi = (\operatorname{sign} \lambda) \sqrt {\langle \lambda \rangle - 1} \quad \Longleftrightarrow \quad \lambda = \xi \sqrt{\xi^2 +2} \]

By the previous theorem, we can define $\psi(\lambda)$ (or, abusing notations, $\psi(\xi)$) by
$$
\mathcal{E}_\lambda \cap L^\infty = \operatorname{Span}_{\mathbb{C}} (\psi(\lambda)).
$$
The function $\psi$ is uniquely determined by normalizing its behaviour as $r\to \infty$. The following theorem gives a precise description of these generalized eigenfunctions (we refer to Section \ref{sectionnotations} for the various notations used below).

\begin{thm}[Description of generalized eigenfunctions] \label{toohigh} The map $\xi \mapsto \psi (\xi, .) \in \mathcal{C}^\infty_{\tmop{loc}} ([0,\infty), \mathbb{R})$ is smooth in the sense that \[ \xi \rightarrow \psi (\xi, .) \in C^1 (\mathbb{R}, C_{\tmop{loc}}^{\infty}
   ([0,\infty), \mathbb{R})) \cap \mathcal{C}^\infty (\mathbb{R}^{\ast},
   C_{\tmop{loc}}^{\infty} ([0,\infty), \mathbb{R})). \]
Furthermore, the following hold true (below, $\chi$ stands for a compactly supported and smooth cutoff function equal to one in a neighborhood of zero).
\begin{itemize}
\item[(i)](Decomposition into singular and regular parts at low  frequency) If $0 \leqslant \xi \leqslant 2$,
\begin{equation} \label{decomposition-psi-flat}
\psi (\xi, r) = \psi_{\flat}^S (\xi, r) + \psi_{\flat}^R (\xi,
     r), 
\end{equation}
where
\begin{equation} \label{id:psi-flat-S} \psi_{\flat}^S (\xi, r) = \left( \sqrt{\frac{\pi}{2}} b_{\flat} (\xi)
   ((\rho (r) - 1) \chi(\xi r) + J_0 (\xi r)) + c_{\flat} (\xi)
   \frac{\sin \left( \xi r - \frac{\pi}{4} \right)}{\sqrt{\xi r}} (1 -
   \chi(\xi r)) \right) e (\xi)
\end{equation}
with $b_{\flat}^2 + c_{\flat}^2 = 1$ and, for any $j \in \mathbb{N}$,
\begin{equation} \label{bd:estimates-b-c-flat}
| \partial_{\xi}^j (b_{\flat} - 1) | + | \partial_{\xi}^j c_{\flat} |
   \lesssim_j \xi^{2 - j} \ln^2 (\xi) 
\end{equation}
and
\begin{equation} \label{id:psi-flat-R}
\psi_{\flat}^R (\xi, r) = m_{\flat}^{\tmop{loc}} (\xi, r)  \chi (\xi r)+ m_{\flat, 1}^R (\xi, r) \cos (\xi r) + m_{\flat, 2}^R (\xi r) \sin
   (\xi r),
\end{equation}
where for any $j, k \in \mathbb{N}$,
\begin{align} \label{bd:m-flat-loc}
&\left| \partial_{\xi}^j \partial_r^k \left( \frac{m_{\flat}^{\tmop{loc}}
   (\xi, r)}{\xi} \right) \right| \lesssim_{j, k} (1 + r)^{j - k-2}, \\
\label{bd:psi-flat-R}& | \partial_{\xi}^j \partial_r^k (m_{\flat, 1}^R (\xi, r)) | + |
   \partial_{\xi}^j \partial_r^k (m_{\flat, 2}^R (\xi, r)) | \lesssim_{j, k}
   \frac{\xi^{2 - j} \ln^2 (\xi)}{\langle \xi r \rangle^{3 / 2} (1 + r)^k} . 
\end{align}

\item[(ii)](Decomposition into singular and regular parts at high frequency) If $\xi \geqslant \frac{1}{2}$,
\begin{equation} \label{decomposition-psi-sharp}
\psi (\xi, r) = \psi_{\sharp}^S (\xi, r) + \psi_{\sharp}^R (\xi,
     r), 
\end{equation}
with
\begin{align} \label{id:psi-sharp-S}
& \psi_{\sharp}^S (\xi, r) = \left( a_{\sharp}(\xi) J_1 (\xi r) \chi (r) +
     \frac{b_{\sharp} (\xi) \cos (\xi r) + c_{\sharp} (\xi) \sin (\xi
     r)}{\sqrt{\xi r}} (1 - \chi (r)) \right) e (\xi) \\
\label{id:psi-sharp-R} & \psi_{\sharp}^R (\xi, r) = m_{\sharp, 1}^R (\xi, r) \cos (\xi r) +
     m_{\sharp, 2}^R (\xi, r) \sin (\xi r) ,
\end{align}
where the coefficients and regular parts satisfy $b_{\sharp}^2  + c_{\sharp}^2  = 1$ and
\begin{align}
\label{bd:estimates-b-c-sharp}&  |\partial_{\xi}^k (a_{\sharp} - 1) | + | \partial_{\xi}^k (b_{\sharp}
     (\xi) +\frac{1}{\sqrt{2}} | + | \partial_{\xi}^k (c_{\sharp} (\xi) -
     \frac{1}{\sqrt{2}}) | \lesssim_k \frac{1}{\xi^{1 + k}} \quad \mbox{if $k \in \mathbb{N}$} \\
\label{bd:psi-sharp-R} & \partial_r^k \partial_{\xi}^j m_{\sharp, 1}^R (\xi, r) | + |
     \partial_r^k \partial_{\xi}^j m_{\sharp, 2}^R (\xi, r) | \lesssim_{j, k}
     \frac{1}{\xi \langle r \rangle \langle \xi r \rangle^{1/2}} \xi^{- j}
     \frac{\xi^k}{\langle \xi r \rangle^k}  \quad \mbox{if $j,k \in \mathbb{N}$}
\end{align}
\item[(iii)](Parity symmetry) For all $\xi,r$ there holds
$$
\psi (\xi, r) = - \sigma_1 \psi (- \xi, r)
$$
and for any $\alpha \in \{ \flat, \sharp \}, \beta \in \{ S, R \}$,
\[ \psi_{\alpha} (\xi, r) = - \sigma_1 \psi_{\alpha} (- \xi, r),
\qquad
     \psi_{\alpha}^{\beta} (\xi, r) = - \sigma_1 \psi_{\alpha}^{\beta} (- \xi,
     r) . \]
\end{itemize}
\end{thm}

The proof of this theorem encompasses sections \ref{policia} to \ref{sectionzerofreq}.

\begin{rem}

The terms $\Psi^S_\flat$ and $\Psi^S_\sharp$ describe the eigenfunction $\psi$ to leading order for $\lambda \to 0$ and $\lambda \to \infty$ respectively, as well as for $r\to \infty$. In particular, $\Psi^S_\alpha$ is of order $r^{-1/2}$ as $r\to \infty$ for $\alpha \in \{\flat, \sharp\}$, while $|\Psi^R_\alpha|\lesssim \langle r\rangle^{-3/2}$.

Furthermore, we are able to describe more precisely the remainder $\psi^R_\alpha$. It can be written as a product of a slowly varying and decaying amplitude times oscillations, see \eqref{id:psi-flat-R} and \eqref{id:psi-sharp-R}. This gives a WKB-type formula for $\psi$.

In the low frequency limit $\lambda \to 0$, the term $m_{\flat}^{\tmop{loc}}$ in \eqref{id:psi-flat-R} describes the $O(\lambda)$ correction, which is explicit (see Lemma \ref{f0f}). We believe that the next term has a logarithmic factor, explaining the estimate \eqref{bd:psi-flat-R}. This singularity would be an interesting difference with respect to the flat case.
\end{rem}

\subsubsection{Distorted Fourier transform}
With the help of the generalized eigenfunctions $\psi$ defined above, we can define a distorted Fourier transform, which will diagonalize the operator $\mathcal{H}$.
Define the distorted Fourier transform, mapping functions from $\mathbb{R}_+$ to $\mathbb{C}^2$ to functions from $\mathbb{R}$ to $\mathbb{C}$:
$$
\widetilde{\mathcal{F}}(\phi)(\xi) = \int_0^\infty \psi(\xi,r) \cdot \sigma_3 \phi (r) r \dd r, \quad \xi \in \mathbb{R}.
$$
where $\sigma_3$ is the Pauli matrix defined in (\ref{S-v-eq4}),
and its inverse
$$
\widetilde{\mathcal{F}}^{-1}(\zeta) (r)= \frac 1 \pi \int_{-\infty}^\infty \zeta(\xi) \psi(\xi,r) \lambda'(\xi) \operatorname{sign} (\xi) \dd \xi, \quad r \in \mathbb{R}_+.
$$

\begin{thm}
The distorted Fourier transform and its inverse satisfy
$$
\widetilde{\mathcal{F}} \widetilde{\mathcal{F}}^{-1} = \operatorname{Id}, \qquad \widetilde{\mathcal{F}}^{-1} \widetilde{\mathcal{F}} = \operatorname{Id}.
$$
Furthermore, the group can be expressed as
\begin{equation} \label{formula-group}
e^{it\mathcal{H}} = \widetilde{\mathcal{F}}^{-1} e^{it \lambda} \widetilde{\mathcal{F}}.
\end{equation}
All these equalities hold true on spaces of smooth, rapidly decaying functions.
\end{thm}

The above theorem corresponds to Theorem \ref{prop:invertibility-pointwise-properties-fourier} which includes a more precise functional framework. Combined with the description of generalized eigenfunctions in Theorem \ref{toohigh}, this is the starting point of the proof of Theorem \ref{thmdecayschwartz}.

\subsection{Organization of the paper} 

\textit{Section \ref{sectionnotations}} recapitulates some notations which are used throughout the text.

\smallskip

\textit{Section \ref{sectiongeneralities}} describes basic properties of the operator $\mathcal{H}$: the absence of eigenvalues, the coincidence of the essential spectrum with the real axis and resonances at zero energy.

\smallskip

\textit{Section \ref{sectionspectral}} is dedicated to the derivation of explicit formulas for the spectral resolution of $\mathcal{H}$, or in other words the distorted Fourier transform $\widetilde{\mathcal{F}}$ and its (putative) inverse $\widetilde{\mathcal{F}}^{-1}$ from which spectral projectors can be deduced. This also gives a formula for the evolution group $e^{it\mathcal{H}}$. Part of this derivation is rigorous, but part of it is formal. To prove it in full rigor, our strategy will be to start from the formulas we obtained, and check 'by hand' that they have the desired properties. This will be achieved in Section \ref{sectionbijectivity}.

\smallskip

\textit{Section \ref{sectiondistorted}} focuses on the distorted Fourier transform and its inverse whose definitions were found in the previous section. The symmetries of these transformations are laid out, and their boundedness properties between natural functional spaces (continuous and differentiable functions, weighted $L^p$ spaces, Sobolev spaces) are established. Interesting differences with the standard Fourier transform appear.

\smallskip 

\textit{Section \ref{sectionbijectivity}} provides a full proof of the formula for $e^{it \mathcal{H}}$ found above. The main step in this proof is to show that $\mathcal{F}$ and $\mathcal{F}^{-1}$ are indeed inverses of each other, which is checked using harmonic analysis tools.

\smallskip

\textit{Section \ref{sectionestimates}}
relies on the formula for $e^{it\mathcal{H}}$ to prove various estimates in $L^\infty$ and $L^2$ on the solution. In particular, we identify orthogonality conditions which give optimal decay in $L^\infty$ and boundedness in $L^2$. Physically, the solution is bigger on the light cone $r = \sqrt 2 t$, where it can be as large as $t^{-2/3}$; away from the light cone, it decays like $t^{-1}$.

\smallskip 

\textit{Section \ref{sectionjost}} is the first step in the investigation of the ODE $\mathcal{H} f = \lambda f$ giving generalized eigenfunctions of $\mathcal{H}$. The space of solutions has dimension 4 (second order ODE on a 2-vector). We show that bounded solutions close to zero and infinity form subspaces of dimensions 2 and 3 respectively. We classify the behavior of bounded functions close to $\infty$: either exponentially decaying, or oscillatory. The functions with such asymptotic behavior are called the Jost solutions.

\smallskip 

\textit{Section \ref{sectionscattering}} shows that the bounded solutions close to zero can be matched to the bounded solutions close to $\infty$ in a unique way; we call $\psi(\xi,r)$ (up to normalization) the unique nonzero bounded solution of $\mathcal{H} f = \lambda f$. Furthermore, we show that $\psi(\xi,r)$ is never in $L^2$, thus excluding embedded spectrum. This is achieved by constructing a solution of the ODE whose behavior at $0$ and $\infty$ is precisely described.

\smallskip 

\textit{Section \ref{sectionhighfreq}} turns to the description of the eigenfunctions $\psi(\xi,r)$ in the (semiclassical) limit $\xi \to \infty$. As expected, they converge to the solutions in the flat case (zero potential), and precise estimates are established.

\smallskip

\textit{Section \ref{sectionzerofreq}} is dedicated to the study of the eigenfunctions in the limit $\xi \to 0$. Because of the resonance at energy zero, it turns out to be a difficult problem - but is is crucial from the dynamic viewpoint, since the zero energy is responsible for slow decay. As $\xi \to 0$, eigenfunctions look increasingly like the resonance close to $r=0$, but they oscillate like Bessel functions as $r \to \infty$. Connecting the two behaviors is very delicate, the matching being responsible for the error term of size $\xi^2 \ln^2(\xi)$.

\smallskip

\textit{Appendix \ref{sectionspace}} is of a technical nature, but it is used throughout sections \ref{sectionjost} to \ref{sectionzerofreq}. Estimates on the inverses of some first oder differential operators $D_a = \partial_r + a(\lambda ,r)$ are proved, in spaces involving regularity and decay. This is used in earlier sections to analyze the ODE $\mathcal{H} f = \lambda f$, since it can be factored as a composition of $D_a$ operators. However, the right factorization depends on the regime considered.

\smallskip

\textit{Appendix \ref{truebesselclassics}} gathers results on Bessel and modified Bessel functions, of real and complex orders; they appear naturally in our analysis.

\smallskip

\textit{Appendix \ref{section flat}} draws a comparison with the "flat" case, namely the Gross-Pitaevskii equation linearized around the steady solution $1$. In this case, the analysis is immediate since the standard Fourier transform applies, and many common features with the linearization around the vortex emerge. 

\subsection*{Acknowledgements} PG was supported by the Simons Foundation through the Wave Turbulence Collaboration, and by a Wolfson fellowship of the Royal Society.

This result is part of the ERC starting grant project FloWAS that has received funding from the European Research Council (ERC) under the Horizon Europe research and innovation program (Grant agreement No. 101117820). C. Collot was supported by the CY Initiative of Excellence Grant ”Investissements d’Avenir” ANR-16-IDEX-0008 via Labex MME-DII, the grant ”Chaire Professeur Junior” ANR-22-CPJ2-0018-01 of the French National Research Agence and the grant BOURGEONS ANR-23-CE40-0014-01 of the French National Research Agence.

\medskip

The authors are most grateful to Wilhelm Schlag, who pointed out two gaps in an earlier version of the present article; both have been addressed through Lemma \ref{lemmaessential} and Section \ref{sectionbijectivity}. The most serious problem is to derive the formula \eqref{formula-group} for the group $e^{it \mathcal{H}}$. The strategy followed here is the following: first, a formula is established in Section \ref{sectionspectral} by a non-fully rigorous limit obtained from an application of Stone's formula; and second, this formula is checked directly (i.e. without justifying the previous limit) in sections \ref{sectiondistorted} and \ref{sectionbijectivity}. We believe that this novel approach can be of interest for other equations. A complementary approach is proposed in L\"uhrmann-Schlag-Shahshahani \cite{LSS}, which was released at the same time as the first version of the present paper and where Stone's formula is fully analyzed.

\section{Notations}

\label{sectionnotations}

\noindent
\underline{Inequalities} We shall rely on the following notations
\begin{itemize}
\item For two quantities $A$ and $B$, we write $A \lesssim B$ if there exists a constant $C$ such that $A \leq CB$. If the constant $C$ is allowed to depend on parameters $a_1,\dots,a_n$, the notation $\lesssim$ is replaced by $\lesssim_{a_1,\dots,a_n}$.
\item We write $A \sim B$ if $A \lesssim B$ and $B \lesssim A$.
\item We write $A \ll B$ if the implicit constant can be taken sufficiently (depending on the context) small.
\item Finally, for functions $f$ and $g$ of a variable $r$, we write 
$$
f(r) \approx g(r) \; \mbox{as $r \to \infty$} \qquad \mbox{if} \qquad \frac{f(r)}{g(r)} \to 1 \; \mbox{as $r \to \infty$},
$$
with obvious modifications if $r \to 0$ etc...
\end{itemize}

\medskip

\noindent
\underline{Linear operations}
The Pauli matrices are
	\begin{equation} \label{S-v-eq4}
	\begin{split}
	\sigma_1 = 
	\left( \begin{array}{ccc}
	0 \quad 1 \\
	1 \quad 0
	\end{array} \right) , \
	\sigma_2 = 
	\left( \begin{array}{ccc}
	0 \ \ -i \\
	i \qquad 0
	\end{array} \right) ,  \
	\sigma_3 = 
	\left( \begin{array}{ccc}
	1 \qquad 0 \\
	0 \ \ -1
	\end{array} \right) . 
	\end{split}
	\end{equation}
The basis vectors in dimension $2$ are
$$
e_1 = \begin{pmatrix} 1 \\ 0 \end{pmatrix}, \qquad e_2 = \begin{pmatrix} 0 \\ 1 \end{pmatrix}.
$$
Given a matrix $M$, its transpose matrix is denoted $M^\top$, and its conjugate (Hermitian) transpose matrix $M^*$. Similarly for an operator.

\medskip
\noindent
\underline{Miscellaneous} We denote
$$
\langle x \rangle = \sqrt{1 + x^2}
$$

\medskip
\noindent
\underline{Function spaces} We denote
$$
\mathcal C_1^\infty(\mathbb R^+)=\{f \in \mathcal \mathcal{C}^\infty(\mathbb R^+), \ f^{(2k)}(0)=0 \mbox{ for all }k\in \mathbb N\}
$$
and
$$
\mathcal S_1=\{f \in \mathcal \mathcal{C}^\infty_1(\mathbb R^+), \ \frac{|f^{(k)}(r)|}{r^l}\to 0 \mbox{ as }r\to \infty   \mbox{ for all }k,l\in \mathbb N\}.
$$
Note that a function $x\mapsto f(r)e^{i\theta}$ is $\mathcal \mathcal{C}^\infty$ (resp. in the Schwartz class) on $\mathbb R^2$ if and only if $f\in \mathcal \mathcal{C}^\infty_1$ (resp. $f\in \mathcal S_1$). 

The weighted $L^2$ space $L^{2,\sigma}$ with respect to the measure $d\mu$ is given by the norm
$$
\| f(x) \|_{L^{2,\sigma}(d\mu)} = \| \langle x \rangle^\sigma f(x) \|_{L^2(d\mu)}.
$$

\medskip
\noindent
\underline{Cut-off function}. The function $\chi$ denotes a smooth nonnegative function with $\chi(x)=1$ for $x\leq 1$ and $\chi(x)=0$ for $x\geq 2$. For $R>0$ we define $\chi_R(x)=\chi(x/R)$.

\medskip

\noindent
\underline{Fourier analysis}. The standard Bessel functions of the first and second kind are denoted by $J_\nu$ and $Y_\nu$ respectively.

\medskip

\noindent
\underline{Functions of the space and time frequencies.} Throughout the text, $\lambda \in \mathbb{R}$ will denote the time frequency and $\xi \in \mathbb{R}$ the space frequency of generalized eigenfunctions (in the limit $r \to \infty$). They are related by
$$
\xi(\lambda) = \operatorname{sign} \lambda \sqrt{\langle \lambda \rangle - 1}, \qquad \lambda(\xi) = \xi \sqrt{2 + \xi^2}.
$$
Why $\xi$ encodes space oscillations, the following parameter gives the rate of exponential decay or growth in the ODE for generalized eigefunctions
$$
\kappa = \sqrt{1 + \langle  \lambda \rangle}.
$$

The unit vector $e$ is the direction in $\mathbb{R}^2$ of generalized eigenfunctions in the limit $r \to \infty$. It is given by

\begin{align*}
& e(\lambda) = \frac{1}{\sqrt{2 \left( 1 + \lambda^2 + \lambda
   \langle \lambda \rangle \right)}} \left(\begin{array}{c}
     \lambda + \langle \lambda \rangle\\ - 1
   \end{array}\right) \\
& \mbox{or} \quad e(\xi) = \frac{1}{\sqrt{2(1 + \xi^2)(1 + \xi^2 + \xi \sqrt{2 + \xi^2})}} \begin{pmatrix} 1+ \xi^2 + \xi \sqrt{2 + \xi^2} \\ -1 \end{pmatrix}.
\end{align*}
Notice that
$$
e(\lambda) = -\sigma_1 e(-\lambda).
$$

All the results in this paper are written for the vortex of degree $+1$, but they still holds for any vortex of degree $n \in \mathbb{Z}^\ast$ (which are defined in \cite{HerveHerve}).

\section{Generalities and spectrum}

\label{sectiongeneralities}

First of all, the following lemma recapitulates the properties of $\rho$ which will be needed in our analysis.

\begin{lem}[\cite{HerveHerve}]
  \label{rhostuff}
There exists a unique function $\rho$ solving the equation
  \[ \rho'' + \frac{1}{r} \rho' - \frac{\rho}{r^2} + (1 - \rho^2) \rho = 0. \]
such that $\rho(0)=0$ and $\rho(+\infty)=1$.

Furthermore, $\rho' (r) > 0$ for all $r \geqslant 0$ and
  \[ \rho (r) = a r + O_{r \rightarrow 0} (r^3) \]
  for some $a > 0$. Finally,
  \[ 1 - \rho (r) = \frac{1}{2 r^2} + O_{r \rightarrow + \infty} \left(
     \frac{1}{r^4} \right) \]
  and more precisely, for any $k \in \mathbb{N}$ and all $r \geqslant 0$, we
  have
  \[ | \partial_r^k (\rho^2 (r) - 1) | \lesssim_k \frac{1}{(1 + r)^{2 + k}} .
  \]
\end{lem}

We consider $\mathcal{H}$ as an operator on $L^2(r \dd r)$, with domain $H^2(\mathbb{R}^2)$ restricted to functions whose Fourier decomposition in the angular variable only involves the first mode. In other words,
$$
\mathcal{D}(\mathcal{H}) = \{ f(r) \in L^2(r \dd r), \; f(r)e^{i\theta} \in H^2 (\mathbb{R}^2) \}.
$$
This definition of the domain of $\mathcal{H}$ is clean and easy to manipulate, but it involves extending $f$ into $f(r) e^{i\theta}$. An equivalent formulation is as follows
$$
\mathcal{D}(\mathcal{H}) = \{ f(r) \in L^2(r \dd r), \; f \in H^2_{\operatorname{loc}}(0,\infty), \; f''+r^{-1}f'-r^{-2}f \in L^2(r \dd r) \}.
$$
Notice that elements of $\mathcal{D}(\mathcal{H})$ are continuous on $[0,\infty)$ and vanish at the origin.

\begin{lem}[Symmetries]
The operator $\mathcal{H}$ enjoys the symmetries
\begin{equation}
\label{symmetryH}
\begin{split}
& \sigma_1 \mathcal{H} \sigma_1 = - \mathcal{H} \\
& \sigma_3 \mathcal{H} \sigma_3 = \mathcal{H}^*
\end{split}
\end{equation}
(where the adjoint operator $\mathcal{H}^*$ is considered with respect to the Hilbert space $L^2(r \dd r)$).
\end{lem}

\begin{lem}[Closedness]
The operator $\mathcal{H}$ with the domain $\mathcal{D}(\mathcal{H})$ is closed as an operator on $L^2(r \dd r)$.
\end{lem}

\begin{proof} This is obvious if one writes $\mathcal{H}$ as an operator on functions of $L^2(\mathbb{R}^2)$ whose expansion in Fourier modes in the angular variable only contains the first harmonic. \end{proof}

It is useful at this point to introduce the change of coordinates.
$$
\begin{pmatrix} \varphi \\ \psi \end{pmatrix} = \frac{1}{2} \begin{pmatrix} 1 & 1 \\ -i & i \end{pmatrix} \begin{pmatrix} u \\ v \end{pmatrix} \qquad \Longleftrightarrow \qquad \begin{pmatrix} u \\ v \end{pmatrix} = \begin{pmatrix} 1 & i \\ 1 & -i \end{pmatrix} \begin{pmatrix} \varphi \\ \psi \end{pmatrix}
$$
(this corresponds to choosing the coordinates $(\mathfrak{Re} w, \mathfrak{Im} w)$ instead of $(w , \overline{w})$ when complexifying the linearized operator $\mathcal{L}$). The equation satisfied by $(\alpha,\beta)$ becomes
$$
\begin{cases}
& \partial_t \varphi = - L_0 \psi \\ 
& \partial_t \psi = (L_0 - 2\rho^2) \varphi, 
\end{cases}
\quad \mbox{with} \quad
L_0 = \partial_r^2 + \frac{1}{r} \partial_r - \frac{1}{r^2} + 1 - \rho^2.
$$

The two following lemmas appeared in \cite{WeinsteinXin}, we include them here for the sake of completeness. 

\begin{lem}[Unstable eigenvalues] \label{lemmadiscrete}
The operator $\mathcal{H}$ does not have eigenvalues $\lambda$ with $\mathfrak{Im} \lambda \neq 0$. 
\end{lem}

\begin{proof}
Consider $(\varphi,\psi) \in L^2$ such that
$$
\begin{cases} L_0 \psi = i \lambda \varphi \\ (L_0 - 2 \rho^2) \varphi = - i \lambda \psi \end{cases}.
$$
We now take the (Hermitian) inner product of the first line with $(L_0-2\rho^2) \varphi$ and of the second with $L_0 \psi$. We then take real parts and subtract the two identities to obtain that
$$
\mathfrak{Re} (i \lambda) \left[ \int L_0 \psi \, \overline{\psi} \, r \dd r + \int (L_0 - 2\rho^2) \varphi \, \overline{\varphi} \, r \dd r \right] = 0.
$$
Since $L_0(\rho) =0$, and $\rho$ does not vanish on $(0,\infty)$, we conclude by Sturm-Liouville theory that $-L_0$ is positive; this is also the case for $-L_0 + 2\rho^2$ which is $\geq \mathcal{M}_1$. Therefore, $\varphi = \psi = 0$, which was the desired result.
\end{proof}

\begin{lem}[Spectrum] \label{lemmaessential}
The spectrum of the operator $\mathcal{H}$ equals $\mathbb{R}$.
\end{lem}

\begin{proof}

We will show that the spectrum of $\mathcal H$ equals $\mathbb R$. By definition, this will imply that the essential spectrum of $\mathcal H$ equals $\mathbb R$.

The starting point is the diagonalization of $\mathcal{H}_0$:
$$
\mathcal{H}_0 = \begin{pmatrix} -\Delta + 1 & 1 \\ -1 & \Delta - 1 \end{pmatrix} = \frac{1}{d(|D|)} P(|D|) \Lambda(|D|) P(|D|),
$$
where $m(|D|)$ stands for the Fourier multiplier with symbol $m(|\xi|)$ (possibly a matrix) and
\begin{align*}
& P(\eta) = \begin{pmatrix} 1 + \eta^2 + \eta \sqrt{2 + \eta^2} & 1 \\ -1 & -1 - \eta^2 - \eta \sqrt{2 + \eta^2}\end{pmatrix}\\
& d(\eta) = | \det P(\eta)| = 2 \eta  \sqrt{2 + \eta^2} ( 1 + \eta^2 + \eta \sqrt{2+\eta^2}) \\
& \Lambda(\eta) = \begin{pmatrix} \lambda(\eta) & 0 \\ 0 & - \lambda(\eta) \end{pmatrix}
\end{align*}
(where the independent variable $\eta$ ranges over $\mathbb{R}_+$).
Furthermore, the matrix $P(\eta)$ is such that $P(\eta)^2 = d(\eta) \operatorname{Id}$. We easily deduce from this formula that any $\lambda_0$ with $\mathfrak{Im} \lambda_0\neq 0$ belongs to the resolvent set of $\mathcal H_0$ and that $(\mathcal H_0-\lambda_0)^{-1}$ is continuous from $L^2$ to $H^2$.

Indeed, we see that for $\lambda_0$ with $\Im \lambda_0 \neq 0$, 
$$
(\mathcal H_0-\lambda_0)^{-1}=\frac{1}{d(|D|)} P(|D|) (\Lambda(|D|)-\lambda_0)^{-1} P(|D|).
$$
Above, we have $(\Lambda^{-1}(\eta)-\lambda_0)^{-1}=-\lambda_0^{-1}+O(\lambda)$ as $\eta\to 0$, which implies $|(\mathcal H_0-\lambda_0)^{-1}(\eta)|\lesssim 1$ for $0<\eta\lesssim 1$ using $P^2=d(P)$ and $\lambda(\eta)\sim \eta$. Combining with the bound $|(\mathcal H_0-\lambda_0)^{-1}(\eta)|\lesssim \eta^{-2}$ for $\eta\gtrsim 1$, we see that $|(\mathcal H_0-\lambda_0)^{-1}(\eta)|\lesssim_{\lambda_0} \langle \eta\rangle^{-2}$. This indeed proves that $(\mathcal H_0-\lambda_0)^{-1}$ is continuous from $L^2$ to $H^2$.

Now, we first prove that $\mathbb R$ is contained in the spectrum of $\mathcal H$. Let $\lambda_0>0$, and define then $\eta_0>0$ by $\lambda_0=\lambda(\eta_0)$. For $\epsilon>0$, define $\hat u_\epsilon$ to be the inverse Fourier transform of
$$
\widehat{u_\epsilon}(\eta)=\epsilon^{-1/2}\chi\left(\frac{\eta-\eta_0}{\epsilon}\right)P^{-1}(|D|)e_1
$$
and note that $\|u_\epsilon\|_{L^2}\approx_{\lambda_0} 1$ for all $\epsilon$ small enough by Plancherel. The function $v_\epsilon=(\mathcal H_0-\lambda_0)u_\epsilon$ then satisfies $\widehat{v_\epsilon}(\eta)=(\lambda(\eta)-\lambda_0)\widehat{u_\epsilon}(\eta)$ so that $\| v_\epsilon\|_{L^2}\to 0$ as $\epsilon>0$. Combining the fact that $V(r)\to 0$ as $r\to \infty$ with Heisenberg's uncertainty principle, we have $\|Vu_\epsilon \|_{L^2}\to 0$ as $\epsilon\to 0$. Hence we found functions $u_\epsilon\in H^2$ with $\| u_\epsilon\|_{L^2}\approx 1$ and $(\mathcal H_0+V-\lambda_0)u_\epsilon\to 0$ in $L^2$, so that $\lambda_0$ cannot belong to the resolvent set of $\mathcal H+V$.

This shows that $(0,\infty)$ is contained in the spectrum of $\mathcal H$, and an analogous proof shows $(-\infty,0)$ is also contained in it. Hence $\mathbb R$ is contained in the spectrum of $\mathcal H$.

We then prove that the spectrum of $\mathcal H$ is contained in $\mathbb R$, by picking some $\lambda_0$ with $\mathfrak{Im} \lambda_0\neq 0$ and showing it belongs to the resolvent set of $\mathcal H$. As $\mathcal H=(\mathcal H_0-\lambda_0)(\operatorname{Id} +(\mathcal H_0-\lambda_0)^{-1}V)$, it suffices to show that $\operatorname{Id}+(\mathcal H_0-\lambda_0)^{-1}V$ has a continuous inverse on $L^2$. Since $(\mathcal H_0-\lambda_0)^{-1}$ is continuous from $L^2$ to $H^2$, and since $V(r)\to 0$ as $r\to \infty$, we have that $(\mathcal H_0-\lambda_0)^{-1}V$ is a compact operator, thanks to the compactness of the embedding of $H^1$ into $L^2$ on compact sets of $\mathbb R^2$. By Fredholm's alternative, $\operatorname{Id}+(\mathcal H_0-\lambda_0)^{-1}V$ has a continuous inverse if and only if it is injective. Let us show its injectivity by contradiction, assuming that $u+(\mathcal H_0-\lambda_0)^{-1}Vu=0$ for some $u\in L^2$. We would have $u=-(\mathcal H_0-\lambda_0)^{-1}Vu$ so that $u\in H^2$, and then $\mathcal H_0u-\lambda_0 u+Vu=0$, which is impossible by Lemma \ref{lemmadiscrete}. Hence $1+(\mathcal H_0-\lambda_0)^{-1}V$ has a continuous inverse as desired.
\end{proof}

\begin{lem}[Symmetry Induced Resonance]\label{ResonanceLemma}

The operator $\mathcal H$ has a resonance of algebraic multiplicity two at the origin: denoting
\begin{equation}
\label{defXi0Xi1}
\Xi_0 = \begin{pmatrix} \rho \\ -\rho \end{pmatrix} \qquad \mbox{and} \qquad \Xi_1 =  \begin{pmatrix} r\partial_r \rho +\rho \\r\partial_r \rho +\rho \end{pmatrix},
\end{equation}
there holds
$$
\mathcal H \Xi_0 =0 \qquad \mbox{and} \qquad \mathcal H \Xi_1 =2 \Xi_0.
$$
Moreover, the following orthogonality conditions are propagated by the flow of \eqref{evolutionH}: if initially
$$
 \langle \mathcal{V}(0), (\rho,\rho)^\top\rangle =0,
$$
then
$$
 \langle \mathcal{V} (t), (\rho,\rho)^\top\rangle =0, \qquad \mbox{ for all }t \geq0,
$$
and if initially
$$
 \langle \mathcal{V} (0), (r\partial_r \rho+\rho,-r\partial_r\rho-\rho)^\top\rangle = \langle \mathcal{V} (0), (\rho,\rho)^\top\rangle =0,
$$
then 
$$
\langle \mathcal{V} (t), (r\partial_r \rho+\rho, - r\partial_r\rho-\rho)^\top\rangle= \langle \mathcal{V} (t), (\rho,\rho)^\top\rangle =0, \qquad \mbox{ for all }t \geq0 .
$$

\end{lem}

\begin{proof}

We have by the scaling and phase invariances of the equation, that for all $\lambda>0$ and $\gamma\in \mathbb R$,
$$
w(t,x)=\frac{e^{i(\theta +(1-\frac{1}{\lambda^2})t+\gamma)}}{\lambda}\rho\left(\frac{|x|}{\lambda}\right)
$$
is a solution of the Gross-Pitaevskii equation. Differentiating with respect to $\lambda$ and $\gamma$ at $\lambda =1$ gives the two desired solutions in the generalized kernel of $\mathcal H$.

For the second assertion, it follows from the fact that the adjoint $\mathcal H^*$ is given by \eqref{symmetryH}, so that there holds the relations
$$
\mathcal H^* \begin{pmatrix} \rho \\ \rho \end{pmatrix}=0 \qquad \mbox{and} \qquad \mathcal H^* \begin{pmatrix} r\partial_r \rho +\rho \\ -r\partial_r \rho -\rho \end{pmatrix}=-2 \begin{pmatrix} \rho \\ \rho \end{pmatrix}.
$$
\end{proof}

\section{Spectral resolution of $\mathcal{H}$}

\label{sectionspectral}

\subsection{A formula for the evolution group} \label{subsec:evolution-group}

Since $\mathcal{H}$ is not self-adjoint, we cannot apply immediately Stone's theorem which gives the spectral measure; but we can prove directly the following version of Stone's formula.

\begin{lem} \label{lem:semi-group:intermediate-result} There holds if $\phi_1,\phi_2 \in \mathcal{S}$
$$
 \langle e^{it\mathcal{H}} \phi_1 , \phi_2 \rangle_{L^2(2\pi r \dd r)} = \lim_{\substack{R \to \infty \\ b \searrow 0}} \frac{1}{2\pi i} \int_{-R}^R e^{it\lambda} \langle \left[  (\mathcal{H} - (\lambda+b i))^{-1} -  (\mathcal{H} - (\lambda - b i))^{-1}  \right] \phi_1,\phi_2 \rangle_{L^2(2\pi r \dd r)} \dd \lambda
$$
\end{lem}

Given the above formula, it is natural to compute the resolvent jump across the real axis. It is given by the following lemma, where we identify an operator on the half-line with its kernel $K(r,s)$ through the formula
$$
\int K(r,s) f(s) \dd s.
$$

\begin{lem} \label{propjump}
If $\lambda \in \mathbb{R}$, there holds as $\epsilon \to 0$
$$
(\mathcal{H}- (\lambda + i\epsilon))^{-1} - (\mathcal{H}- (\lambda - i\epsilon))^{-1}
\longrightarrow 2 i s \operatorname{sign}(\lambda)  \psi(r,\lambda) \psi(s,\lambda)^\top \sigma_3,
$$
where we identify operators with their kernel and the limit is understood as pointwise convergence of the kernel.
\end{lem}

\underline{Proceeding formally}, the combination of these two lemmas gives the formula
$$
\langle e^{it \mathcal{H}} \phi_1 , \phi_2 \rangle_{L^2(2\pi r \dd r)}
= 2 \int_{-\infty}^\infty e^{it\lambda}  \int_0^\infty \psi(s,\lambda) \cdot \sigma_3 \phi_1(s) \, s \dd s \int_0^\infty \psi(r,\lambda) \cdot \overline{\phi_2(r)} r \dd r \operatorname{sign} (\lambda) \dd \lambda
$$
or equivalently
\begin{equation}
\label{hibou}
e^{it \mathcal{H}} \phi = \frac{1}{\pi } \int_{-\infty}^\infty  e^{it\lambda} \psi(r,\lambda) \int_0^\infty \psi(s,\lambda) \cdot \sigma_3 \phi(s) s \dd s \operatorname{sign} (\lambda) \dd \lambda.
\end{equation}

\underline{Making this step rigorous} seems very challenging: indeed, the limit in Lemma \ref{propjump} is not uniform in $\lambda$, and it is not obvious that both limiting processes are compatible.

For this reason, we do not try to prove the formulas above by a combination of more precise versions of Lemma \ref{lem:semi-group:intermediate-result} and Lemma \ref{propjump}. Rather, we will start from these formulas and prove them directly. This will be achieved in Section \ref{sectionbijectivity}.

\subsection{Limiting formula for the group: proof of Lemma \ref{lem:semi-group:intermediate-result}}

\label{ecureuil1}

We mostly follow Krieger and Schlag \cite{KriegerSchlag}, but we provide the details since they differ significantly. For now, we consider the operator defined on functions on $\mathbb{R}^2$ without any angular symmetry; therefore, $L^p$ spaces are considered with respect to the Euclidean measure.

\medskip
\noindent
\underline{Defining the evolution group via the Hille-Yosida theorem.} As a first step, we apply the Hille-Yosida theorem to makes sense of the evolution group - we are not trying to be precise, and the boundss will be weak, but they will be improved afterwards. Recall that we can use Neumann series to write
\begin{equation*}
(\mathcal{H} + i (a + \lambda))^{-1} 
= \sum_{k=0}^\infty \left[ (\mathcal{H}_0 + i(a +\lambda))^{-1}  V \right]^k (\mathcal{H}_0 + i(a +\lambda))^{-1}
\end{equation*}
if $a+\lambda$ is sufficiently big. By self-adjointness, $(\mathcal{H}_0 + i(a +\lambda))^{-1}$ can be bounded in the operator norm by $\frac{1}{a+\lambda}$. Therefore,
$$
\|(\mathcal{H} + i (a + \lambda))^{-1} \|_{L^2 \to L^2} \lesssim \frac{1}{a+\lambda}\sum_{k=0}^\infty \left(\frac{\|V\|_{\infty}}{a+\lambda} \right)^k= \frac{1}{a+\lambda} \frac{1}{1-\frac{\|V\|_{\infty}}{a+\lambda}} = \frac{1}{a+\lambda -\|V\|_{\infty} } \leq \frac{1}{\lambda}
$$
if $a=\|V\|_\infty$ and $\lambda>0$. In other words, we proved that $i \mathcal{H} - a$ satisfies the resolvent inequality
$$
\|(i \mathcal{H} - a - \lambda)^{-1} \|_{L^2 \to L^2} \leq \frac{1}{\lambda}.
$$
This means that the group $e^{it\mathcal{H}}$ is well-defined and satisfies the bound
$$
\| e^{it\mathcal{H}} \|_{L^2 \to L^2} \leq e^{at} \qquad \mbox{if $t>0$}.
$$

\medskip
\noindent
\underline{The limiting absorption principle.} We claim first that, if $\frac 12 < \sigma <1$,
\begin{equation}
    \label{LAPH0}
\sup_{\lambda>1,\,\epsilon \neq 0} |\lambda|^{1/2} \| (\mathcal{H}_0- \lambda +i\epsilon)^{-1} \|_{L^{2,\sigma} \to L^{2,-\sigma}} < \infty.
\end{equation}

In order to prove this bound, it is convenient to view $\mathcal{H}_0$ as a Fourier multiplier (forgetting the restriction to the first Fourier mode in the angular variable) and to write its symbol  
$$
\mathcal{H}_0 = h(D) \qquad \mbox{with} \qquad 
h(\xi) = \begin{pmatrix} \xi^2 + 1 & 1 \\ -1 & -\xi^2 -1 \end{pmatrix}.
$$
Therefore, $(\mathcal{H}_0 - \lambda)^{-1}$ has symbol
$$
\frac{1}{\lambda^2 + 1 - (\xi^2 + 1)^2} 
\begin{pmatrix} -\xi^2 - 1 - \lambda & - 1 \\ 1 & \xi^2 +1 - \lambda \end{pmatrix},
$$
which can also be written (thinking of the case $\lambda>0$)
$$
\begin{pmatrix}
- \frac{1}{\langle \xi \rangle^2 - \langle \lambda \rangle}
+ \frac{\langle \lambda \rangle - \lambda}{(\langle \xi \rangle^2 - \langle \lambda \rangle)(\langle \xi \rangle^2 + \langle \lambda \rangle)} & - \frac{1}{(\langle \xi \rangle^2 - \langle \lambda \rangle)(\langle \xi \rangle^2 + \langle \lambda \rangle)} \\
\frac{1}{(\langle \xi \rangle^2 - \langle \lambda \rangle)(\langle \xi \rangle^2 + \langle \lambda \rangle)} &
\frac{1}{\langle \xi \rangle^2 + \langle \lambda \rangle}
+ \frac{\langle \lambda \rangle - \lambda}{(\langle \xi \rangle^2 - \langle \lambda \rangle)(\langle \xi \rangle^2 + \langle \lambda \rangle)}
\end{pmatrix}.
$$
The limiting absorption principle \eqref{LAPH0} follows then, by observing that the operator with symbol $\frac{1}{\langle \xi \rangle^2 + \langle \lambda \rangle}$ is bounded on $L^{2,\sigma}$, by classical properties of the Bessel potential, while the operator with symbol $\frac{1}{\langle \xi \rangle^2 - \langle \lambda \rangle}$ satisfies \eqref{LAPH0} by the classical limiting absorption principle for scalar Schr\"odinger operators  \cite{RodnianskiTao}.

We now claim that, for $\lambda_0$ sufficiently big and $\frac 12 < \sigma <1$,
\begin{equation}
    \label{LAP}
\sup_{\lambda>\lambda_0,\,\epsilon \neq 0} |\lambda|^{1/2} \| (\mathcal{H}- \lambda +i\epsilon)^{-1} \|_{L^{2,\sigma} \to L^{2,-\sigma}} < \infty.
\end{equation}
This follows once again from the Neumann series expansion 
$$
(\mathcal{H} - \lambda + i \epsilon ))^{-1} 
= \sum_{k=0}^\infty \left[ (\mathcal{H}_0 - \lambda + i \epsilon )^{-1}  V \right]^k (\mathcal{H}_0 - \lambda + i \epsilon )^{-1},
$$
since $\mathcal{H}_0$ satisfies \eqref{LAP} while $V$ which decays like $\frac{1}{r^2}$, maps $L^{2,-\sigma}$ to $L^{2,\sigma}$.

\medskip
\noindent
\underline{Inverting the Laplace transform} It follows from the application of the Hille-Yosida theorem above that we can write the resolvent of $\mathcal{H}$ as
\begin{equation}
\label{Laplace}
(i\mathcal{H} - z)^{-1} = - \int_0^\infty e^{-tz} e^{it \mathcal{H}}  \dd t \qquad \mbox{if $\mathfrak{Re} z > a$}.
\end{equation}
Regarding the above right-hand side as the Laplace transform (in $z$) of a function (of $t$), the classical inversion formula formally gives
$$
e^{it \mathcal{H}} = - \frac{1}{2\pi i} \int_{b-i \infty}^{b +i\infty} e^{tz} (i\mathcal{H} - z)^{-1} \, \dd z \qquad \mbox{if $b>a$ and $t>0$.}
$$

We claim that this formula holds in a weak sense: if $\phi_1,\phi_2 \in \mathcal{D}(\mathcal{H})$, $b>a$, and $t \neq 0$,
$$
\mathbf{1}_{t>0}  \langle e^{it\mathcal{H}} \phi_1 , \phi_2 \rangle = - \lim_{R \to \infty} \frac{1}{2\pi i} \int_{b - iR}^{b+iR} e^{tz} \langle (i\mathcal{H} - z)^{-1} \phi_1 , \phi_2 \rangle \dd z.
$$
Postponing the proof of this claim for a moment, we take the complex conjugate of the above to get a corresponding formula for $t<0$, and we conclude that
\begin{equation}
\label{coccinelle}
\begin{split}
 \langle e^{it\mathcal{H}} \phi_1 , \phi_2 \rangle & =  - \lim_{R \to \infty} \frac{1}{2\pi i} \int_{b - iR}^{b+iR} e^{tz} \langle (i\mathcal{H} - z)^{-1} \phi_1 , \phi_2 \rangle \dd z +  \lim_{R \to \infty} \frac{1}{2\pi i} \int_{-b - iR}^{-b+iR} e^{tz} \langle (i\mathcal{H} - z)^{-1} \phi_1 , \phi_2 \rangle \dd z \\
& = \lim_{R \to \infty} \frac{1}{2\pi i} \int_{-R}^R e^{it\lambda} \langle \left[ e^{-bt} (\mathcal{H} - (\lambda+ib))^{-1} -  e^{bt} (\mathcal{H} - (\lambda - ib))^{-1}  \right] \phi_1,\phi_2 \rangle \dd \lambda
\end{split}
\end{equation}
To prove the claim, we use the formula \eqref{Laplace} and exchange the order of integration to obtain that
\begin{align*}
& - \frac{1}{2\pi i} \int_{b - iR}^{b+iR} e^{tz} \langle (i\mathcal{H} - z)^{-1} \phi_1 , \phi_2 \rangle \dd z =  \frac{1}{2\pi i} \int_{b - iR}^{b+iR} e^{tz} \int_0^\infty e^{-sz} \langle e^{is \mathcal{H}} \phi_1, \phi_2 \rangle \dd s \dd z \\
 & \qquad  \qquad \qquad = \frac{1}{\pi} \int_0^\infty e^{(t-s) b} \frac{\sin((t-s)R)}{t-s} \langle e^{is \mathcal{H}} \phi_1, \phi_2 \rangle \dd s \overset{R \to \infty}{\longrightarrow} \mathbf{1}_{t>0}  \langle e^{it\mathcal{H}} \phi_1 , \phi_2 \rangle
\end{align*}
if $t \neq 0$, and with limit $\frac{1}{2} \langle \phi , \psi \rangle$ if $t=0$. Here, the limit $R\to \infty$ is justified since $\langle e^{it\mathcal{H}} \phi , \psi \rangle$ is continuous (indeed, it is $\mathcal{C}^1$) and standard properties of convolution by the Dirichlet kernel $\frac{1}{\pi} \frac{\sin (Rs)}{s}$.

\medskip

\noindent \underline{Shifting the contour} 
We consider now equation \eqref{coccinelle} and write it as
$$
 \langle e^{it\mathcal{H}} \phi_1 , \phi_2 \rangle = \lim_{R \to \infty} \frac{1}{2\pi i} \int_{\Gamma_{R,b}} e^{itz} \langle (\mathcal{H} - z)^{-1} \phi_1,\phi_2 \rangle \dd z,
$$
where $\Gamma_{R,b}$ is the contour made up of the union of the segments $i b + [-R,R]$ and $-ib + [-R,R]$, oriented from left to right and right to left respectively. By analyticity of the integrand, we can push the contour to the union of the segments $ [-R,R] + 0 i$ and $[-R,R] - 0i$. This produces error terms corresponding to the integrals over the segments $-R + i [-b,b]$ and $R + i [-b,b]$, but by the limiting absorption principle proved above, the integral over these segments vanishes as $R \to \infty$. This leaves us with
$$
 \langle e^{it\mathcal{H}} \phi_1 , \phi_2 \rangle = \lim_{\substack{R \to \infty \\ b \searrow 0}} \frac{1}{2\pi i} \int_{-R}^R e^{it\lambda} \langle \left[  (\mathcal{H} - (\lambda+b i))^{-1} -  (\mathcal{H} - (\lambda - b i))^{-1}  \right] \phi_1,\phi_2 \rangle \dd \lambda
$$
which is the desired result. 

\subsection{The jump of the resolvent: proof of Lemma \ref{propjump}}

\label{ecureuil2}

Here we switch to the new unknown function $\underline{\psi} = \sqrt{r} \psi$. This amounts to conjugating $\mathcal{H}$ into
$$
\mathcal{H}' = r^{\frac 12} \mathcal{H}r^{-\frac 12}.
$$
It is a bit nicer than $\mathcal{H}$ since 
\begin{equation} \label{jump-resolvent:id:mathcalH'}
\mathcal{H}' = -\sigma_3 \partial_r^2 + U(r), \qquad \mbox{where} \;\;U(r)=V(r)-\frac{3}{4r^2},
\end{equation}
hence the Wronskian of generalized eigenfunctions is independent of $r$. The results of the previous Subsection \ref{subsec:evolution-group} all adapt to $\mathcal H'$, considered as an operator on $L^2(\mathbb R^+,\dd r)$ with domain
$$
\mathcal D(\mathcal H')=\left\{u\in L^2(\mathbb R^+,\dd r), \ (\partial_r^2-\frac{3}{4r^2})u\in L^2(\mathbb R^+,\dd r)\right\}.
$$

Let us first define the vector Wronskian, for $F$ and $G$ two smooth functions from $\mathbb{R}_+$ to $\mathbb{R}^2$
$$
W(F,G) = F' \cdot G - G' \cdot F
$$
(where $\cdot$ stands for the \textit{real} scalar product).

If $F_1,F_2,G_1,G_2$ are vectors in $\mathbb{R}^2$ depending smoothly on $r \in \mathbb{R}_+$, then for the $2\times 2$ matrices 
\begin{equation}
\label{FG}
F = (F_1 | F_2) \qquad \mbox{and} \qquad G = (G_1|G_2)
\end{equation}
the matrix Wronskian is given by
$$
\mathcal{W}(F,G) = F^{' \top} G - F^{\top} G' = 
\begin{pmatrix} W(F_1,G_1) & W(F_1,G_2) \\ W(F_2,G_1) & W(F_2,G_2) \end{pmatrix}.
$$

\begin{lem}
\label{lemwronskian}
\begin{itemize}
    \item[(i)]If $f$ and $g$ satisfy $\mathcal{H}' f = zf$ and $\mathcal{H}' g = zg$, for some $z \in \mathbb{C}$, then $W(f,g)$ is independent of $r$.
    \item[(ii)] With the notation \eqref{FG}, if $F_i$ and $G_i$ satisfy $\mathcal{H}' F_i = F_i$ and $\mathcal{H}' G_i = zG_i$, for some $z \in \mathbb{C}$, then $\mathcal{W}(F,G)$ is independent of $r$.
\end{itemize}
\end{lem}

\begin{proof} The second assertion is an immediate consequence of the first one. To prove the first one, observe that in the identity \eqref{jump-resolvent:id:mathcalH'} one has $\sigma_3 U(r) \sigma_3 = U^\top$. Therefore
\begin{align*}
\frac{d}{dr} W(f,g) & = f'' \cdot g - g'' \cdot f
= (-\sigma_3 \mathcal{H}' + \sigma_3 U(r)) f \cdot g - (-\sigma_3 \mathcal{H}' + \sigma_3 U(r)) g \cdot f \\
& = (-\sigma_3 z + \sigma_3 U(r)) f \cdot g - (-\sigma_3 z + \sigma_3 U(r)) g \cdot f = 0.
\end{align*}

\end{proof}

\begin{lem} \label{lemresolventz} If $z = \lambda + i \epsilon$ with $\lambda,\epsilon>0$, let $F_1,F_2,G_1,G_2$ be linearly independent solutions of $(\mathcal{H}' - z \operatorname{Id}) f =0$ such that $F_1,F_2 \in L^2(1,\infty)$ and $G_1,G_2 \in L^2(0,1)$. 

Denoting $F=(F_1|F_2)$ and $G = (G_1|G_2)$, assume that
$$
\mathcal{W}(F,F) = \mathcal{W}(G,G) = 0
$$
and denote
$$
D = \mathcal{W}(F,G).
$$
Then $(\mathcal{H}'- z \operatorname{Id})^{-1}$ has kernel
$$
K_z(r,s) =
\begin{cases}
& - G(r) D^{-1} F^\top(s) \sigma_3 \qquad \mbox{if $r<s$} \\
& - F(r) D^{-1 \top} G^\top(s) \sigma_3 \qquad \mbox{if $r>s$}. 
\end{cases}
$$

\end{lem}

\begin{proof} 
\underline{The jump condition: proof of Lemma \ref{propjump}} By lemmas \ref{lemmadiscrete} and \ref{lemmaessential}, $z$ is in the resolvent set of $\mathcal{H}'$. As a consequence, the resolvent exists and we can look for its integral kernel $K_z(r,s)$ under the form
$$
K_z(r,s) = \begin{cases}
G(r) B(s)^\top & \mbox{if $r<s$}\\
F(r) A(s)^\top & \mbox{if $r>s$}
\end{cases}
$$
(here, $A,B,F,G$ are $2 \times 2$ matrices).
One checks that $K_z$ is the kernel of the inverse of $\mathcal{H}'- z$ if
$$
\begin{cases}
K_z(r,r-) = K_z(r,r+) \\
\sigma_3 \partial_r K_z(r,r-) - \sigma_3 \partial_r K_z(r,r+) = \operatorname{Id}.
\end{cases}
$$
Under our ansatz, this becomes the matrix equation
$$
\begin{pmatrix}
G & F \\
G'& F'
\end{pmatrix}
\begin{pmatrix}
B^\top \\ -A^\top
\end{pmatrix}
= 
\begin{pmatrix}
0 \\ \sigma_3
\end{pmatrix}.
$$

\bigskip

\noindent \underline{Inverting the $4\times 4$ matrix.} We now claim that
\begin{equation}
\label{inverseFG}
\begin{pmatrix} G & F \\ G' & F' \end{pmatrix}^{-1} = 
\begin{pmatrix} D^{-1} F^{' \top} & - D^{-1} F^\top \\ - D^{-1 \top} G^{' \top } & D^{-1 \top} G^\top \end{pmatrix}.
\end{equation}
Together with the matrix equation for $A$ and $B$, this implies that
$$
\begin{cases}
A = - \sigma_3 G D^{-1 } \\
B = -\sigma_3 F D^{-1 \, \top},
\end{cases}
$$
thus leading to the desired formula for the kernel $K$.

Finally, \eqref{inverseFG} follows from the facts that $\mathcal{W}(F,F) = \mathcal{W}(G,G) =0$, $\mathcal{W}(F,G) = - \mathcal{W}(G,F)^\top$ and from the general identity
$$
\begin{pmatrix} \mathcal W(F,G) & \mathcal W(F,F) \\ \mathcal W(G,G) & \mathcal W(G,F) \end{pmatrix} = \begin{pmatrix} 0 & \operatorname{Id} \\  \operatorname{Id} & 0 \end{pmatrix} \begin{pmatrix} G^\top & G^{' \top} \\ F^\top  & F^{' \top} \end{pmatrix} \begin{pmatrix} 0 & - \operatorname{Id}\\  \operatorname{Id} & 0 \end{pmatrix} \begin{pmatrix} G & F \\ G' & F' \end{pmatrix}
$$
which can be checked by a straightforward computation.
\end{proof}

We can now proceed with the proof of Lemma \ref{propjump}.

\begin{proof}[Proof of Lemma \ref{propjump}] We first note that upon conjugation by $r^{1/2}$, the statement of the Lemma becomes
$$
(\mathcal{H}'- (\lambda + i\epsilon))^{-1} - (\mathcal{H}'- (\lambda - i\epsilon))^{-1}
\longrightarrow K_\lambda(r,s) = 2 i \operatorname{sign}(\lambda) \underline \psi(r,\lambda) \underline \psi(s,\lambda)^\top \sigma_3.
$$

\medskip

\underline{The functions $F$ and $G$ and their Wronskians for $\lambda>0$}. We define the functions $F_i(r,\lambda), G_i(r,\lambda)$ as follows:
\begin{itemize}
\item $G_1(r,\lambda) = |\xi|^{1/2} \underline \psi(r,\lambda)$ as given by Theorem \ref{toohigh}. Thus, $G_1$ is globally bounded, $G_1(r,\lambda) \sim a(\lambda) r^{\frac 32}$ as $r \to 0$, and finally $G_1(r,\lambda) \sim \cos(\xi r + \varphi(\lambda)) e(\lambda)$ as $r \to \infty$.
\item $F_1(r,\lambda)$ is such that $F_1(r,\lambda) + \overline{F_1(r,\lambda)} = G_1(r,\lambda)$ and $F_1(r,\lambda) \sim \frac 12 e^{i \xi (r + \varphi(\lambda))} e(\lambda)$ as $r \to \infty$. This determines $F_1$ completely.
\item $F_2(r,\lambda) = (\varphi_2,\psi_2)(r,\lambda)$ as given by Lemma \ref{factexp} - thus $F_2$ is real and exponentially decaying.
\item Finally, we choose $G_2$ to be a regular solution at $r=0$, real-valued and linearly independent from $G_1$ - this is possible by Lemma \ref{zerosol}.
\end{itemize}

We now compute the Wronskians of these functions. Still denoting $F = (F_1 | F_2)$ and $G= (G_1 | G_2)$, we find that
$$
\mathcal{W}(F,F) = 0 \qquad \mbox{and} \qquad \mathcal{W}(G,G) = 0
$$
by inspecting these expressions as $r \to \infty$ and $r \to 0$ respectively and using (i) in Lemma \ref{lemwronskian}. 
Turning to $\mathcal{W}(F,G)$, considering the limit $r\to \infty$ gives 
$$
W(F_1,G_1) = i \frac{\xi}{2} \qquad \mbox{and} \qquad \mathcal{W}(F_2,G_1) = \mathcal{W}(F_1,G_1) = 0.
$$ 
Considering instead the limit $r \to 0$ gives $W(G_2,F_1 + \overline{F_1}) = W(F_2,F_1) = 0$. This implies that $W(G_2,F_1) \in i \mathbb{R}$. From the above expression for $W(F_1,G_1)$, we see that we can modify $G_2$ by adding to it $cG_1$, with $c \in \mathbb{R}$, to ensure that
$$
W(F_1,G_2) = 0, 
$$
which we assume from now on. Finally, the reality of $F_2$ and $G_2$ implies that
$$
W(F_2,G_2) = d(\lambda) \in \mathbb{R}.
$$

Overall, we have now
\begin{equation}
\label{formulaDlambda}
D = \mathcal{W}(F,G) = \begin{pmatrix} i \frac{\xi}{2} & 0 \\ 0 & d \end{pmatrix}. 
\end{equation}

\bigskip

\noindent \underline{Extension to the complex plane} The functions $F_i(\lambda)$, $G_i(\lambda)$ can be naturally extended to complex values of the spectral parameter: 
$$
z = \lambda + i \epsilon, \qquad \lambda>0, \quad 0 \leq \epsilon \ll 1.
$$ 

\begin{itemize}
\item $G_1(r,z)$ is chosen to be the solution of $(\mathcal{H}' - z \operatorname{Id}) f =0$ such that $G_1(r,z) \sim a(\lambda) r^{\frac{3}{2}}$ as $r \to 0$, where $a(\lambda)$, defined earlier, is such that $G_1(r,\lambda) \sim a(\lambda) r$ as $r\to 0$.  The existence and uniqueness is ensured by a straightforward adaptation of Lemma \ref{zerosol}.
\item Similarly, $G_2(r,z)$ is the solution of $(\mathcal{H}' - z \operatorname{Id}) f =0$ such that $G_2(r,z)$ and $G_2(r,\lambda)$ agree to leading order as $r \to 0$.
\item If $0 < \epsilon \ll 1$, note that $\xi(z) = \sqrt{1+\sqrt{1+z^2}}$ is such that $\mathfrak{Im} \xi(z) > 0$. Indeed, a short computation reveals that $\xi(\lambda+i\epsilon) = \xi+ \frac{i \epsilon \lambda}{\langle \lambda \rangle \xi}+O(\epsilon^2)$, which has a positive imaginary part for $\epsilon \to 0^+$. By adapting Lemma \ref{diary} to the case $\epsilon>0$, we can find a solution $F_1(r,z)$ such that $F_1(r,z) \to F_1(r,\lambda)$ pointwise as $\epsilon \to 0$.
\item Finally, an extension of Lemma \ref{factexp} to the case $\epsilon >0$ gives the desired solution $F_2(r,z)$.
\end{itemize}
Overall, we found extensions of $F_i(\lambda),G_i(\lambda)$ to the complex plane such that $\mathcal{W}(F,F) = \mathcal{W}(G,G) = 0$ and
$$
\begin{cases}
F_i(r,z) \to F_i(r,\lambda) \\
G_i(r,z) \to G_i(r,\lambda) 
\end{cases}
\qquad \mbox{pointwise as $\epsilon \to 0$, for $i=1,2$}.
$$
The derivatives $F_i',G_i'$ converge pointwise as well, so that
$$
D(z) = \mathcal{W}(F(z),G(z)) \to D(\lambda) \qquad \mbox{as $\epsilon \to 0$.}
$$

Therefore, by Lemma \ref{lemresolventz},
\begin{equation}
\label{formulaKlambda}
\mbox{as $\epsilon \to 0$,} \qquad K_z(r,s) \to K_{\lambda+i0}(r,s) = 
\begin{cases}
& - G(r,\lambda) D^{-1}(\lambda) F^\top(s,\lambda) \sigma_3 \qquad \mbox{if $r<s$} \\
& - F(r,\lambda) D^{-1}(\lambda) G^\top(s,\lambda) \sigma_3 \qquad \mbox{if $r>s$}. 
\end{cases}
\end{equation}

\bigskip
\noindent \underline{Resolvent jump for $\lambda>0$} By \eqref{symmetryH}, we see that 
$$
(\mathcal{H}' - \overline{z} \operatorname{Id})^{-1} = \sigma_3 (\mathcal{H}' - z)^{-1 \, *} \sigma_3,
$$
which implies that
$$
K_{\overline{z}}(r,s) = \sigma_3 K_z(s,r)^* \sigma_3,
$$
(where $K_z(s,r)^*$ is the Hermitian conjugate of the matrix $K_z(s,r)$)
and thus
$$
K_{\overline{z}}(r,s) \to \sigma_3 K_{\lambda+i0}(s,r)^* \sigma_3 \qquad \mbox{as $\epsilon \to 0$}.
$$
This gives the kernel of the resolvent jump
$$
K_{\lambda}(r,s) = K_{\lambda+i0}(r,s) - \sigma_3 K_{\lambda+i0}(s,r)^* \sigma_3.
$$
We now use the formula \eqref{formulaKlambda} giving $K_{\lambda+i0}$ and the formula \eqref{formulaDlambda} giving $D(\lambda)$ to compute, for $r>s$,
\begin{align*}
K_{\lambda}(r,s) & = - F(r) D^{-1} G^\top(s) \sigma_3 + \overline{F(r)} \overline{D^{-1} G^\top(s)} \sigma_3 \\
& = (i2 \xi^{-1} F_1(r) | - d^{-1} F_2(r)) (G_1(s) | G_2(s))^\top \sigma_3 + ( i2 \xi^{-1} \overline{F_1(r)} | d^{-1} F_2(r)) {(G_1(s) | G_2(s))^\top} \sigma_3 \\
& = i 2 \xi^{-1} ( F_1(r) + \overline{F_1(r)}) G_1^\top(s) \sigma_3 + d^{-1} F_2(r) (- G_2^\top(s) + G_2^\top(s))\sigma_3 \\
& = 2 i \xi^{-1} G_1 (r) G_1^\top(s) \sigma_3 = 2 i  \underline{\psi}(r,\lambda) \underline{\psi}^\top(s,\lambda) \sigma_3.
\end{align*}
The case $r<s$ can be reduced to the case $r>s$ as follows
\begin{align*}
K_{\lambda}(r,s) & = - G(r) D^{-1} F^\top(s) \sigma_3 + \overline{G(r) D^{-1} F^\top(s)} \sigma_3 = \sigma_3 K_\lambda(s,r)^\top \sigma_3 \\
& = 2 i \xi^{-1} G_1 (r) G_1^\top(s) \sigma_3 =  2 i  \underline{\psi}(r,\lambda) \underline{\psi}^\top(s,\lambda) \sigma_3.
\end{align*}

\bigskip

\noindent \underline{The case $\lambda<0$.} 
Recall that we have for $\lambda<0$, $\psi(r,\lambda) = - \sigma_1 \psi(r,-\lambda)$.
By the symmetry \eqref{symmetryH}, we have
$$
\sigma_1 (\mathcal{H} - z)^{-1} \sigma_1 = -(\mathcal{H} + z)^{-1}.
$$
Using successively this identity, the formula for $K_\lambda(r,s)$ and the fact that $\sigma_1$ and $\sigma_3$ anticommute, we get for $\lambda>0$
\begin{align*}
\left( \mathcal{H} - (-\lambda + 0i) \right)^{-1} - \left( \mathcal{H} - (-\lambda - 0i) \right)^{-1}
& = \sigma_1 \left( \mathcal{H} - (\lambda + 0i) \right)^{-1} \sigma_1
-  \sigma_1 \left( \mathcal{H} - (\lambda - 0i) \right)^{-1} \sigma_1 \\
& =  2i \sigma_1 \underline{\psi}(r,\lambda) \underline{\psi}^\top (s,\lambda) \sigma_3 \sigma_1 \\
& = - 2i \underline{\psi}(r,-\lambda) \underline{\psi}^\top (s,-\lambda) \sigma_3,
\end{align*}
which gives the desired formula.
\end{proof}

\section{Boundedness of the distorted Fourier transform}

\label{sectiondistorted}

\subsection{Definition}
In the following, we will rather use the space frequency than the time frequency as a spectral parameter. We therefore change the integration variable to $\xi$ in the formula for the group in \eqref{hibou}, which gives
\begin{equation}
\label{formulagroupxi}
e^{it \mathcal{H}} \phi  = \frac{1}{\pi } \int_{-\infty}^\infty e^{it \lambda} \psi(r,\xi) \int_0^\infty \psi(s,\xi) \cdot \sigma_3 \phi(s)s\dd s \, \lambda'(\xi) \operatorname{sign}(\xi) \dd \xi ,
\end{equation}
where it is understood that $\lambda = \lambda(\xi) = \xi \sqrt{ \xi^2 + 2}$ and $\lambda'(\xi) = \frac{2(\xi^2 + 1)}{\sqrt{\xi^2 + 2}}$. This formula suggests the following formula for the distorted Fourier transform, which is nothing but the spectral projector associated to $\mathcal{H}$.

\begin{df}
If $\xi \in \mathbb{R}$, the distorted Fourier transform is defined as
$$
\widetilde{\phi}(\xi) = \widetilde{\mathcal{F}}(\phi)(\xi) = \int_0^\infty \psi(\xi,r) \cdot \sigma_3 \phi(r) r \dd r = \frac{1}{2\pi} \langle \phi \,,\, \sigma_3 \psi(\xi,r) \rangle_{L^2(2\pi r\dd r)}
$$
and the inverse distorted Fourier transform as
$$
\widetilde{\mathcal{F}}^{-1}(\zeta)(r) = \frac 1 \pi \int_{-\infty}^\infty \zeta (\xi) \psi(\xi,r) \lambda'(\xi) \operatorname{sign} \xi \dd\xi.
$$
\end{df}

\underline{Note that} the notations $\widetilde{\mathcal{F}}$ and $\widetilde{\mathcal{F}}^{-1}$ should be taken with a grain of salt for the time being. Indeed, we do not know yet that these operators are inverses of each other: this will be the main thrust of Section \ref{sectionbijectivity}.

We start with an elementary lemma which collects some immediate properties of the distorted Fourier transform, showing in particular it is well defined as a mapping from $L^1(r \dd r)$ to $\mathcal{C}^0$.

\begin{lem}[Very first properties of the distorted Fourier transform] 
\label{lemmaveryfirst} The distorted Fourier transform of $\phi$ is well-defined provided $\phi \in L^1(r \dd r)$. Furthermore, if $\phi \in \mathcal{S}_1$,
\begin{itemize}
\item[(i)] $\widetilde{\phi}$ is continuous.
\item[(ii)] $\widetilde{\phi}$ is smooth for $\xi \neq 0$.
\item[(iii)] For any $\xi$, $\displaystyle |\widetilde{\phi}(\xi)| \lesssim \frac{1}{\langle \xi \rangle^3}.$
\end{itemize}
\end{lem}

\begin{proof}
The first two assertions follow readily from Theorem \ref{toohigh}. As for the third, it suffices to consider the case $|\xi| > 1$. Then we rely on the decomposition of Theorem \ref{toohigh} to write
\begin{equation}
\label{herisson}
\widetilde{\phi}(\xi) =  \int \psi^S_\sharp(\xi,r)\cdot \sigma_3 \phi(r) r \dd r + \int \psi^R_\sharp(\xi,r)\cdot \sigma_3 \phi(r) r \dd r.
\end{equation}
By Theorem \ref{toohigh}, the first term can be written
$$
\int \psi^S_\sharp(\xi,r)\cdot \sigma_3 \phi(r) r \dd r = \int \left( a_{\sharp}(\xi) J_1 (\xi r) \chi (r) +
     \frac{b_{\sharp} (\xi) \cos (\xi r) + c_{\sharp} (\xi) \sin (\xi
     r)}{\sqrt{\xi r}} (1 - \chi (r)) \right) e(\xi)\cdot \sigma_3 \phi(r) r \dd r .
$$
The terms on the right-hand side can be viewed as the (standard) two-dimensional Fourier transform of Schwartz functions of the type $e^{i\theta} \phi(r)$; therefore, it is $\lesssim |\xi|^{-N}$ for any $N$.

Turning to the second term in the right-hand side of \eqref{herisson}, we learn from Theorem \ref{toohigh} that $\psi^R_\sharp(\xi,r)$ can be written as a sum of terms of the type $m(\xi,r) e^{\pm i \xi r}$ with
$| \partial_r^j m(\xi,r) | \lesssim \frac{1}{\xi \langle r \rangle \langle \xi r \rangle} \frac{\xi^j}{\langle \xi r \rangle^j}$. Armed with this estimate, we can integrate by parts twice in the second term on the right-hand side of \eqref{herisson} and use that the boundary terms vanish since $\phi(0)=0$ to obtain that
$$
\left| \int \psi^R_\sharp(\xi,r)\cdot \sigma_3 \phi(r) r \dd r \right| \lesssim \frac{1}{\xi^2} \left[ \int \frac{1}{\xi \langle r \rangle \langle \xi r \rangle} \frac{\xi^2}{\langle \xi r \rangle^2} r \dd r + \int \frac{1}{\xi \langle r \rangle \langle \xi r \rangle}  r \dd r \right] \lesssim \frac{1}{\xi^3}.
$$
\end{proof}

\subsection{Even-odd symmetries}
We define even and odd vector and scalar functions as
\begin{align*}
& \phi:(0,\infty)\to \mathbb C^2 \mbox{ is even  if }\sigma_1 \phi=\phi ,\\
& \phi:(0,\infty)\to \mathbb C^2 \mbox{ is odd  if }\sigma_1 \phi=-\phi ,\\
& \zeta: \mathbb R\to \mathbb C \mbox{ is even  if } \zeta(-\xi)=\zeta(\xi) ,\\
& \zeta :\mathbb R \to \mathbb C \mbox{ is odd  if } \zeta(-\xi) =-\zeta(\xi) .
\end{align*}
The even and odd parts of general functions from $(0,\infty)$ to $\mathbb C^2$ and from $\mathbb R$ to $\mathbb C$ are then denoted by
\begin{align*}
& \phi_e(r)=\frac12(\phi(r)+\sigma_1 \phi(r)), \qquad  \phi_o(r)=\frac12(\phi(r)-\sigma_1 \phi(r)), \\
& \zeta_e(\xi)=\frac12(\zeta(\xi)+\zeta(-\xi)), \qquad  \zeta_o(\xi)=\frac12(\zeta(\xi)-\zeta(-\xi)).
\end{align*}

\begin{df}
We define the even and odd parts of $\psi$ to be in the vectorial sense
$$
\psi_e(\xi,r)=\frac12(\psi(\xi,r)+\sigma_1 \psi(\xi,r)) \quad \mbox{and} \quad \psi_o(\xi,r)=\frac12(\psi(\xi,r)-\sigma_1 \psi(\xi,r)).
$$
In this formula, even and odd functions are defined in the sense of vectors, i.e. with respect to the symmetry $\sigma_1$; considering even and odd functions in the sense of functions of $\xi$ (with respect to $\xi \mapsto -\xi$) would have given a different result since $\psi(-\xi,r))=-\sigma_1 \psi(\xi,r)$.
\end{df}

\begin{lem}[Symmetries of the Fourier transform] \label{lem:fourier-symmetries}
We have the following formulas for the even and odd parts of the direct and inverse Fourier transforms,
\begin{align}
\label{even-part-direct-fourier} & (\widetilde{\mathcal F}(\phi))_e(\xi)= \int_0^\infty \psi_o (\xi,r) \cdot\sigma_3 \phi_e(r)r\dd r,\\
\label{odd-part-direct-fourier}& (\widetilde{\mathcal F}(\phi))_o(\xi)= \int_0^\infty \psi_e(\xi,r) \cdot \sigma_3 \phi_o (r)r\dd r,
\end{align}
as well as
\begin{align}
\label{even-part-inverse-fourier}& (\widetilde{\mathcal F}^{-1}(\zeta))_e(r)= \frac 1 \pi \int_{-\infty}^\infty \zeta_e(\xi) \psi_e (\xi,r)\lambda'(\xi) \operatorname{sign}( \xi) \dd \xi,\\
\label{odd-part-inverse-fourier} & (\widetilde{\mathcal F}^{-1}(\zeta))_o(r)= \frac 1 \pi \int_{-\infty}^\infty \zeta_o(\xi) \psi_o (\xi,r)\lambda'(\xi) \operatorname{sign}( \xi) \dd \xi.
\end{align}
The applications $\widetilde{\mathcal{F}}$ and $\widetilde{\mathcal{F}}^{-1}$ preserve even and odd symmetries: they map even functions to even functions and odd functions to odd functions.
\end{lem}

\begin{proof}
Using the relations $\sigma_1\sigma_3=-\sigma_3\sigma_1$, $\psi(-\xi)=-\sigma_1\psi(\xi)$ and $\sigma_1^\top=\sigma_1$ we obtain the identities $\widetilde{\mathcal F}(\phi)(-\xi)=\widetilde{\mathcal F}(\sigma_1 \phi)(\xi)$ and $\widetilde{\mathcal F}^{-1}(\xi\mapsto \zeta(-\xi))=\sigma_1 \widetilde{\mathcal F}^{-1}\zeta$. Furthermore, the functions $\xi \mapsto \psi_o(\xi,r)$ and $\xi \mapsto \psi_e (\xi,r)$ are odd and even respectively. Finally, $\lambda'(-\xi)=\lambda'(\xi)$. The assertions of the lemma follow immediately from combining these identities.
\end{proof}

\subsection{Pointwise bounds}

\begin{prop}[Continuity and differentiability of the distorted Fourier transform]
  \label{pointwiseestimates}There holds
\begin{itemize}
\item[(i)] If $\phi \in L^1 (r d r)$, then $\widetilde{\phi} \in C^0 (\mathbb{R})$ with
  \[ \widetilde{\phi} (0) = \sqrt{\frac{\pi}{4}} \langle \phi, \sigma_3 \Xi_0
     \rangle_{L^2 (r d r)} \]
  (where $\Xi_0$ is defined in (\ref{defXi0Xi1})).
  
\item[(ii)] If $\phi \in L^1 (\langle r \rangle r d r)$, then $\widetilde{\phi} \in
  C^1 (\mathbb{R})$ with
  \[ \widetilde{\phi}' (0) = \sqrt{\frac{\pi}{4}} \langle \phi, \sigma_3 \Xi_1
     \rangle_{L^2 (r d r)} \]
  (where $\Xi_1$ is defined in (\ref{defXi0Xi1})).
  
\item[(iii)] If $\phi \in L^1 (\langle r \rangle^2 r d r)$, then $\widetilde{\phi} \in
  C^2 (\mathbb{R}^{\ast})$ with for any $\xi, \xi' > \xi_0 > 0$,
  \[ | \partial_{\xi} \widetilde{\phi} (\xi) - \partial_{\xi} \widetilde{\phi} (\xi')
     | \lesssim_{\xi_0} | \xi - \xi' | . \]
  
\item[(iv)] If $\phi \in L^1 (\langle r \rangle^2 r d r)$, then for any
  $\varepsilon > 0$ we have
  \[ | \partial_{\xi} \widetilde{\phi} (\xi) - \partial_{\xi} \widetilde{\phi} (0) |
     \lesssim_{\varepsilon} \xi^{1 - \varepsilon} . \]
\end{itemize}
\end{prop}

\begin{proof}
The statements on the regularity of $\widetilde{\phi}$ away from $\xi = 0$ follow readily from Theorem \ref{toohigh}. We now focus only on the limit $\xi
  \rightarrow 0$.
  
  First, by Theorem \ref{toohigh} and Lemma \ref{aghanim}, we infer the
  following decomposition on $\psi (\xi, r)$ for low frequencies. We have
  \[ \psi (\xi, r) = \sqrt{\frac{\pi}{2}} \left( (\rho (r) + J_0 (r \xi) - 1)
     e (\xi) + \frac{\xi}{\sqrt{2}} r \rho' (r) e (\xi)^{\top} \right) +
     \tmop{Err} (\xi, r) \]
  where
  \[ \| \tmop{Err} (\xi, r) \|_{L^{\infty} ([0, \xi^{- 1 / 2}])} + \left\|
     \frac{1}{\langle r \rangle} \partial_{\xi} \tmop{Err} (\xi, r)
     \right\|_{L^{\infty} ([0, \xi^{- 1 / 2}])} = o_{\xi \rightarrow 0} (1) \]
  and
  \[ \| \tmop{Err} (\xi, r) \|_{L^{\infty} ([\xi^{- 1 / 2}, + \infty [)} +
     \left\| \frac{1}{\langle r \rangle} \partial_{\xi} \tmop{Err} (\xi, r)
     \right\|_{L^{\infty} ([\xi^{- 1 / 2}, + \infty [)} \lesssim 1. \]
  This is a slight improvement on Theorem \ref{toohigh}, using the development
  up to terms of size $O (\xi)$ in $F_0$ of Lemma \ref{aghanim} rather than
  simply the term of size $O (1)$ which is $(\rho (r) + J_0 (r \xi) - 1) e
  (\xi)$. The choice of cutting the estimates at $\xi^{- 1 / 2}$ is arbitrary
  and still works for $\xi^{- \alpha}$ given any $\alpha \in] 0, 1 [$. We
  recall that
  \[ e (\xi) = \frac{1}{\sqrt{2 (1 + \xi^2) \left( 1 + \xi^2 + \xi \sqrt{2 +
     \xi^2} \right)}} \left(\begin{array}{c}
       1 + \xi^2 + \xi \sqrt{2 + \xi^2}\\
       - 1
     \end{array}\right) \]
  and that
  \[ \widetilde{\phi} (\xi) = \int_0^{+ \infty} \psi (\xi, r) . \sigma_3 \phi (r)
     r d r. \]
  Since $\left\| (\rho (r) + J_0 (r \xi) - 1) e (\xi) + \frac{\xi}{\sqrt{2}} r
  \rho' (r) e (\xi)^{\top} \right\|_{L^{\infty} ([\xi^{- 1 / 2}, + \infty])}
  \lesssim 1$, we deduce that if $\phi \in L^1 (r d r)$, then
  \[ \left| \int_{\xi^{- 1 / 2}}^{+ \infty} \psi (\xi, r) . \sigma_3 \phi (r)
     r d r \right| \lesssim \| \phi \|_{L^1 ([\xi^{- 1 / 2}, + \infty])}
     \rightarrow 0 \]
  when $\xi \rightarrow 0$, and since $e (0) = \frac{1}{\sqrt{2}}
  \left(\begin{array}{c}
    1\\
    - 1
  \end{array}\right)$, we also have
  \[ \left\| \psi (\xi, r) - \sqrt{\frac{\pi}{4}} \rho \left(\begin{array}{c}
       1\\
       - 1
     \end{array}\right) \right\|_{L^{\infty} ([0, \xi^{- 1 / 2}])} = o_{\xi
     \rightarrow 0} (1) . \]
  We deduce that
  \[ \widetilde{\phi} (\xi) = \int_0^{+ \infty} \sqrt{\frac{\pi}{4}} \rho
     \left(\begin{array}{c}
       1\\
       - 1
     \end{array}\right) . \sigma_3 \phi (r) r d r + o_{\xi \rightarrow 0} (1)
  \]
  hence with $\Xi_1 = \left(\begin{array}{c}
    \rho\\
    - \rho
  \end{array}\right)$, we have
  \[ \widetilde{\phi} (0) = \sqrt{\frac{\pi}{4}} \langle \phi, \sigma_3 \Xi_1
     \rangle_{L^2 (r d r)} . \]
  This completes the proof of (i). By Theorem \ref{toohigh}, we have for low
  frequencies that $| \partial_{\xi} \psi (\xi, r) | \lesssim \langle r
  \rangle$, which allows to differentiate under the integral sum to obtain
  \[ \partial_{\xi} \widetilde{\phi} (\xi) = \int_0^{+ \infty} \partial_{\xi} \psi
     (\xi, r) . \sigma_3 \phi (r) r d r. \]
  We compute that
  \begin{eqnarray*}
    \partial_{\xi} \psi (\xi, r) & = & \sqrt{\frac{\pi}{2}} \left(
    \frac{1}{\sqrt{2}} r \rho' (r) e (\xi)^{\top} + \rho (r) \partial_{\xi} e
    (\xi) \right)\\
    & + & \sqrt{\frac{\pi}{2}} \partial_{\xi} ((J_0 (r \xi) - 1) e (\xi)) +
    \frac{\xi}{\sqrt{2}} r \rho' (r) \partial_{\xi} e (\xi)^{\perp}\\
    & + & \partial_{\xi} \tmop{Err} (\xi, r)
  \end{eqnarray*}
where $(x_1,x_2)^\perp = (-x_2,x_1).$
  We compute furthermore that
  \[ \sqrt{\frac{\pi}{2}} \left( \frac{1}{\sqrt{2}} r \rho' (r) e (\xi)^{\perp}
     + \rho (r) \partial_{\xi} e (\xi) \right) = \sqrt{\frac{\pi}{4}} (r \rho'
     (r) + \rho (r)) \left(\begin{array}{c}
       1\\
       1
     \end{array}\right) + O_{\xi \rightarrow 0} (\xi) \]
  while the terms on the second line satisfy the same estimates as
  $\partial_{\xi} \tmop{Err} (\xi, r)$. As previously, we deduce that
  \[ \partial_{\xi} \widetilde{\phi} (0) = \sqrt{\frac{\pi}{4}} \langle \Xi_1,
     \sigma_3 \phi \rangle_{L^2 (r d r)} \]
  where $\Xi_1 = \left(\begin{array}{c}
    r \rho' (r) + \rho (r)\\
    r \rho' (r) + \rho (r)
  \end{array}\right)$.
  
  To complete the proof of (iii) and (iv), we will show the following estimate
  : if $\phi \in L^1 (\langle r \rangle^2 r d r)$, then for $\xi \in] 0, 1]$,
  we have
  \[ | \partial^2_{\xi} \widetilde{\phi} (\xi) | \lesssim \ln^2 (\xi) . \]
  As previously, we have
  \[ \partial^2_{\xi} \widetilde{\phi} (\xi) = \int_0^{+ \infty} \partial^2_{\xi}
     \psi (\xi, r) . \sigma_3 \phi (r) r d r. \]
  Now, by Theorem \ref{toohigh}, we check the decomposition
  \[ \psi (\xi, r) = \sqrt{\frac{\pi}{2}} b_{\flat} (\xi) J_0 (\xi r) +
     \tmop{Err}_2 (\xi, r) \]
  where for $\xi$ small and any $r \geqslant 0$,
  \[ | \partial_{\xi}^2 \tmop{Err}_2 (\xi, r) | \lesssim \ln^2 (\xi) \langle r
     \rangle^2 \]
  and for any $k \in \mathbb{N}$ $| \partial_{\xi}^k (b_{\flat} (\xi) - 1) |
  \lesssim \xi^{2 - k} \ln^2 (\xi)$.
  
  Remark that $\xi \rightarrow \sqrt{\frac{\pi}{2}} b_{\flat} (\xi) \int_0^{+
  \infty} J_0 (\xi r) . \sigma_3 \phi (r) r d r$ is the usual Fourier
  transform on the function $\sigma_3 \phi$ (multiplied by
  $\sqrt{\frac{\pi}{2}} b_{\flat} (\xi)$)for which the estimate is satisfied.
  With $\phi \in L^1 (\langle r \rangle^2 r d r)$, we then check directly that
  \[ \left| \int_0^{+ \infty} \partial_{\xi}^2 \tmop{Err}_2 (\xi, r) .
     \sigma_3 \phi (r) r d r \right| \lesssim \ln^2 (\xi) \| \langle r
     \rangle^2 \phi \|_{L^1 (r d r)}, \]
  concluding the proof.
\end{proof}

\subsection{$L^2$ bounds}
If $B$ is a Banach space of functions on the line, we denote $B_e$ and $B_o$ for the restriction of this space to even and odd functions respectively. If $w$ is a weight, then $w^{-1} B$ is defined through the norm $\| w \cdot \|_B$.

It will be natural to introduce the following space of functions on the line
$$
\widetilde{L^2} = (L^2 (| \xi | \dd \xi))_e + (| \xi |
\langle \xi \rangle^{- 1} L^2 (| \xi | \dd \xi))_o
$$
endowed with the norm
$$
\| \zeta(\xi) \|_{\widetilde{L^2}} = \| \zeta_e \|_{L^2 (| \xi | \dd \xi)} + \|| \xi |^{-1} \langle \xi \rangle \zeta_o \|_{L^2 (| \xi | \dd \xi)}.
$$

\begin{prop}[$L^2$-boundedness of the distorted Fourier transform]
\label{propL2}
The distorted Fourier transform and its inverse are bounded between the following spaces
$$
\widetilde{\mathcal{F}} : L^2(r \dd r) \longrightarrow \widetilde{L^2}, \qquad \widetilde{\mathcal{F}}^{-1} : \widetilde{L^2}\longrightarrow L^2(r \dd r). 
$$
To be more specific,
$$
\widetilde{\mathcal{F}} : \begin{array}{l}  L^2_{e}(r\dd r) \longrightarrow L^2_{e}(|\xi|\dd \xi) \\ L^2_{o}(r\dd r) \longrightarrow |\xi| \langle \xi \rangle^{-1}L^2_{o}(|\xi|\dd \xi) \end{array} \qquad \quad
\widetilde{\mathcal{F}}^{-1} : \begin{array}{l}  L^2_{e}(|\xi|\dd \xi) \longrightarrow L^2_{e}(r\dd r)  \\ |\xi| \langle \xi \rangle^{-1}L^2_{o}(|\xi|\dd \xi) \longrightarrow L^2_{o}(r\dd r) .\end{array}
$$
\end{prop}

\begin{proof}
\underline{Step 0: decomposition of the distorted Fourier transform and its inverse.} We first split these transformations between low and high frequencies and rely on \eqref{decomposition-psi-flat} and \eqref{decomposition-psi-sharp} to write
\begin{align*}
\widetilde{\mathcal F}\phi=\widetilde{\mathcal F}_\flat \phi+\widetilde{\mathcal F}_\sharp \phi \quad \mbox{and} \quad \widetilde{\mathcal F}^{-1}\zeta=\widetilde{\mathcal F}_\flat^{-1} \zeta+\widetilde{\mathcal F}_\sharp^{-1} \zeta
\end{align*}
where
\begin{align*}
& \widetilde{\mathcal{F}}_\flat (\phi)(\xi) =  \chi(\xi)\langle \phi \,, \sigma_3 \psi_\flat (\xi,r) \rangle_{L^2(r\dd r)}, \quad  \widetilde{\mathcal{F}}_\sharp (\phi)(\xi) =  (1-\chi(\xi))\langle \phi \,, \sigma_3 \psi_\sharp (\xi,r) \rangle_{L^2(r\dd r)},\\
& \widetilde{\mathcal{F}}^{-1}_\flat (\zeta)(r) =  \langle \zeta , \chi (\xi)\psi_\flat(\xi,r) \lambda' \operatorname{sign} \xi \rangle_{L^2(\frac{\dd \xi}{\pi})}, \quad \widetilde{\mathcal{F}}^{-1}_\sharp (\zeta)(r) = \langle \zeta , (1-\chi(\xi))\psi_\sharp(\xi,r) \lambda' \operatorname{sign} \xi \rangle_{L^2(\frac{\dd \xi}{\pi})}.
\end{align*}
Then, we further decompose for $a=\flat,\sharp$ between singular and regular parts by \eqref{decomposition-psi-flat} and \eqref{decomposition-psi-sharp}:
$$
\widetilde{\mathcal F}_a \phi=\widetilde{\mathcal F}_a^S \phi+\widetilde{\mathcal F}_a^R \phi \quad \mbox{and} \quad  \widetilde{\mathcal F}^{-1}_a\zeta=\widetilde{\mathcal F}_a^{-1S} \zeta+\widetilde{\mathcal F}_a^{-1R} \zeta
$$
where for $b=S,R$,
\begin{align*}
& \widetilde{\mathcal{F}}_\flat^b (\phi)(\xi) =  \chi(\xi)\langle \phi \,, \sigma_3 \psi_\flat^b (\xi,r) \rangle_{L^2(r\dd r)}, \quad  \widetilde{\mathcal{F}}_\sharp^b (\phi)(\xi) =  (1-\chi(\xi))\langle \phi \,, \sigma_3 \psi_\sharp^b (\xi,r) \rangle_{L^2(r\dd r)},\\
& \widetilde{\mathcal{F}}^{-1b}_\flat (\zeta)(r) =  \langle \zeta , \chi(\xi)\psi_\flat^b(\xi,r) \lambda' \operatorname{sign} \xi \rangle_{L^2(\frac{\dd \xi}{\pi})}, \quad \widetilde{\mathcal{F}}^{-1b}_\sharp (\zeta)(r) = \langle \zeta , (1-\chi(\xi))\psi_\sharp^b(\xi,r) \lambda' \operatorname{sign} \xi  \rangle_{L^2(\frac{\dd \xi}{\pi})}.
\end{align*}
Lastly, in order to study $\widetilde{\mathcal{F}}_\flat^S$ and $\widetilde{\mathcal{F}}_\flat^{-1S}$ we will decompose into even and odd inputs
$$
\widetilde{\mathcal{F}}_\flat^S (\phi)(\xi) =\widetilde{\mathcal{F}}_\flat^S (\phi_e)(\xi) +\widetilde{\mathcal{F}}_\flat^S (\phi_o)(\xi) \quad \mbox{and}\quad \widetilde{\mathcal{F}}_\flat^{-1S} (\zeta)(r) =\widetilde{\mathcal{F}}_\flat^{-1S} (\zeta_e)(r) +\widetilde{\mathcal{F}}_\flat^{-1S} (\zeta_o)(r)
$$
where, using that $\psi_{\flat}^b (-\xi,r)=-\sigma_1 \psi_{\flat}^b (\xi,r)$ for $b=S,R$,
\begin{align*}
& \widetilde{\mathcal{F}}_\flat^S (\phi_e)(\xi) =  \chi(\xi)\langle \phi \,, \sigma_3 \psi_{\flat o}^S (\xi,r) \rangle_{L^2(r\dd r)}, \quad \widetilde{\mathcal{F}}_\flat^S (\phi_o)(\xi) =  \chi(\xi)\langle \phi \,, \sigma_3 \psi_{\flat e}^S (\xi,r) \rangle_{L^2(r\dd r)}, \\
& \widetilde{\mathcal{F}}^{-1S}_\flat (\zeta_e)(r) =  \langle \zeta , \chi(\xi)\psi_{\flat e}^S(\xi,r) \lambda' \operatorname{sign} \xi \rangle_{L^2(\frac{\dd \xi}{\pi})}, \quad \widetilde{\mathcal{F}}^{-1S}_\flat (\zeta_o)(r) = \langle \zeta , (1-\chi(\xi))\psi_{\sharp o}^S(\xi,r) \lambda' \operatorname{sign} \xi  \rangle_{L^2(\frac{\dd \xi}{\pi})}.
\end{align*}
where $ \psi_{\flat e}^S =\frac 12 (\psi_{\flat }^S+\sigma_1 \psi_{\flat }^S)$ and $ \psi_{\flat o}^S =\frac 12 (\psi_{\flat}^S-\sigma_1 \psi_{\flat }^S)$.

\medskip

\noindent \underline{Step 1: Boundedness of $\widetilde{\mathcal F}_\sharp$ from $L^2(r\dd r)$ to $L^2(|\xi|\dd \xi)$ and of $\widetilde{\mathcal F}^{-1}_\sharp$ from $L^2(|\xi|\dd \xi)$ to $L^2(r\dd r)$}. Pick $\phi \in L^2(r\dd r)$ and $\zeta\in L^2(|\xi|\dd \xi)$ with $\| \phi\|_{L^2(r\dd r)}=1$ and $\| \zeta \|_{L^2(|\xi|\dd \xi)}=1$.

\medskip

\noindent \textit{$\bullet$ Boundedness of $\widetilde{\mathcal F}_\sharp^S$ from $L^2(r\dd r)$ to $L^2(|\xi|\dd \xi)$}. We have by \eqref{id:psi-sharp-S} that for $\xi>0$
\begin{align}
\label{formula-direct-fourier-sharp-S} \widetilde{\mathcal{F}}_\sharp^S \phi(\xi) & =\widetilde a_\sharp (\xi) (1- \chi(\xi)) \int_0^\infty e_1 \cdot \phi(r)\chi (r)  J_1(\xi r) r\dd r \\
\nonumber & \quad+\frac{b_\sharp(\xi)(1-\chi(\xi))}{\sqrt{\xi}} \int_0^\infty e(\xi) \cdot \phi(r)(1-\chi (r))\sqrt{r}  \cos(\xi r-\frac{3\pi}{4})\dd r \\
\nonumber &\qquad +\frac{c_\sharp(\xi)(1-\chi(\xi))}{\sqrt{\xi}} \int_0^\infty e(\xi) \cdot \phi(r)(1-\chi (r))\sqrt{r}  \sin(\xi r-\frac{3\pi}{4})\dd r 
\end{align}
Notice that, if a function $x\mapsto v(r)e^{i\theta}$ on $\mathbb R^2$ that is restricted to the first angular harmonic, then $\xi \mapsto \int_0^\infty J_1(\xi r ) v(r)r \dd r$ is the standard two-dimensional Fourier transform of $v$. Notice also that $\xi \mapsto \int_0^\infty (\cos (\xi r)+i\sin (\xi r)) v(r)\dd r$ is the standard one-dimensional Fourier transform of the function $x\mapsto \mathbbm 1(x>0) v(|x|)$. Therefore, the Plancherel theorem in one and two dimensions implies that
$$
\left\| \widetilde{\mathcal{F}}_\sharp^S \phi   \right\|_{L^2(|\xi| \dd \xi)} \lesssim 1.
$$

\smallskip

\noindent \textit{$\bullet$ Boundedness of $\widetilde{\mathcal F}_\sharp^{-1S}$ from $L^2(|\xi|\dd \xi)$ to $L^2(r\dd r)$}. We can assume without loss of generality that $\zeta(\xi)=0$ for $\xi<0$, since the contribution of negative $\xi$ can be treated identically to the contribution of positive $\xi$. In this case, we have by \eqref{id:psi-sharp-S} that
\begin{align}
\label{formula-inverse-fourier-sharp-S} \pi \widetilde{\mathcal{F}}_\sharp^{-1S}\zeta (r) & = e_1 \chi(r) \int_0^\infty \widetilde{a}_\sharp(\xi) \frac{\lambda'(\xi)}{\xi} (1-\chi(\xi))\zeta(\xi) J_1(\xi r) \xi \dd \xi\\
\nonumber & \quad +\frac{1-\chi(r)}{\sqrt{r}}\int_0^\infty b_\sharp(\xi) \frac{\lambda'(\xi)}{\sqrt{\xi}} (1-\chi(\xi))e(\xi)\zeta(\xi) \cos\left( \xi r-\frac{3\pi}{4}\right) \dd \xi \\
\nonumber & \qquad +\frac{1-\chi(r)}{\sqrt{r}}\int_0^\infty c_\sharp(\xi) \frac{\lambda'(\xi)}{\sqrt{\xi}} (1-\chi(\xi))e(\xi)\zeta(\xi) \sin \left( \xi r-\frac{3\pi}{4} \right) \dd \xi 
\end{align}
Similarly to the study of $\widetilde{\mathcal F}_\sharp^S$ above, we recognize the two-dimensional Fourier transform of $ \widetilde{a}_\sharp \frac{\lambda'}{\xi} (1-\chi )\zeta $ in the first term, and the one-dimensional Fourier transform of $b_\sharp \frac{\lambda'}{\sqrt{\xi}} (1-\chi)e(\xi)\zeta $ for the second term (and similarly for the third one). Therefore, since $\frac{\lambda'(\xi)}{\xi}\approx 1 $ on the support of $1-\chi(\xi)$, using once again Plancherel's theorem,
$$
\left\| \widetilde{\mathcal{F}}_\sharp^{-1S} \zeta   \right\|_{L^2(r \dd r)} \lesssim 1.
$$

\smallskip

\noindent \textit{$\bullet$ Boundedness of $\widetilde{\mathcal F}_\sharp^R$ from $L^2(r\dd r)$ to $L^2(|\xi|\dd \xi)$}. By the Cauchy-Schwarz inequality, \eqref{id:psi-sharp-R} and \eqref{bd:psi-sharp-R},
$$
|\langle  \psi_\sharp^R(\xi,\cdot), \sigma_3 \phi\rangle|\lesssim \| \psi_\sharp^R(\xi,\cdot)\|_{L^2(r\dd r)} \lesssim \| \frac{1}{|\xi|\langle r \rangle \langle r \xi \rangle^{1/2}}\|_{L^2(r\dd r)}  \lesssim \frac{1}{|\xi|^{3/2}}
$$
Using that $\xi\mapsto 1-\chi(\xi)$ is supported in $\{|\xi|\geq 1\}$, we get that $ \widetilde{\mathcal{F}}_\sharp^R (\phi)(\xi) =  (1-\chi(\xi))\langle \phi \,, \sigma_3 \psi_\sharp^b (\xi,r) \rangle_{L^2(r\dd r)}$ satisfies
\begin{equation} \label{bd:Fourier-sharp-R-L2}
\| \widetilde{\mathcal F}_\sharp^R(\phi)\|_{L^2(| \xi | \dd \xi)} \lesssim 1 .
\end{equation}

\smallskip

\noindent \textit{$\bullet$ Boundedness of $\widetilde{\mathcal F}_\sharp^{-1R}$ from $L^2(|\xi|\dd \xi)$ to $L^2(r\dd r)$}. By the Cauchy-Schwarz inequality, \eqref{id:psi-sharp-R} and \eqref{bd:psi-sharp-R},
\begin{align}
\nonumber |\widetilde{\mathcal{F}}^{-1R}_\sharp \zeta(r)| & = |\langle \zeta , (1-\chi(\xi))\psi_\sharp^R(\xi,r) \lambda' \operatorname{sign} \xi  \rangle_{L^2(\frac{\dd \xi}{\pi})}|\lesssim \| \zeta\|_{L^2(|\xi|\dd \xi)}\|(1-\chi(\xi))\psi_\sharp^R(\xi,r) \lambda' \|_{L^2(\frac{\dd \xi}{|\xi|})} \\
\label{bd:intermediaire-mathcalF-1Rsharp} &\lesssim \left\|(1-\chi(\xi))\frac{1}{\langle r \rangle \langle \xi r \rangle^{1/2}} \right\|_{L^2(\frac{\dd \xi}{|\xi|})} \lesssim \mathbb{1}(r\leq 1)\langle \ln r\rangle+ \mathbbm{1}(r>1)r^{-\frac 32}.
\end{align}
Thus,
$$
\| \widetilde{\mathcal F}_\sharp^{-1R}\zeta \|_{L^2(r \dd r)} \lesssim 1 .
$$

\medskip

\noindent \underline{Step 3: Boundedness of $\widetilde{\mathcal F}^S_\flat$ from $L^2_{o}(r\dd r)$ to $|\xi| \langle \xi\rangle^{-1}L^2_{o}(|\xi|\dd \xi)$ and from $L^2_{e}(r\dd r)$ to $L^2_{e}(|\xi|\dd \xi)$,}

\noindent \underline{and of $\widetilde{\mathcal F}^{-1S}_\flat$ from $L^2_{e}(|\xi|\dd \xi)$ to $L^2_{e}(r\dd r)$ and from $|\xi| \langle \xi\rangle^{-1}L^2_{o}(|\xi|\dd \xi)$ to $L^2_{o}(r\dd r)$.}

\smallskip

\noindent \textit{$\bullet$ Boundedness of $\widetilde{\mathcal F}_\flat^S$ from $L^2_{o}(r\dd r)$ to $|\xi| \langle \xi\rangle^{-1}L^2_{o}(|\xi|\dd \xi)$}. Pick $\phi = (\varphi, -\varphi)^\top$ with $\| \varphi\|_{L^2(r\dd r)}= 1$. We have by \eqref{id:psi-flat-S} that for $\xi>0$
\begin{align} \label{id:Fourier-flat-S-odd}
&\widetilde {\mathcal F}_\flat^{S} \phi (\xi) = \xi \langle \xi \rangle^{-1}\int_0^\infty m^S_{\flat}(\xi,r) \varphi(r)r\dd r
\end{align}
where, introducing $\widetilde b_\flat(\xi)=\sqrt{\frac{\pi}{2}}b_\flat(\xi)$,
\begin{align} \label{id:Fourier-inverse-flat-S}
m_\flat^S(\xi,r) & = \widetilde b_\flat \langle \xi \rangle \frac{e_1(\xi)+e_2(\xi)}{\xi}\chi(\xi) J_0(\xi r) +c_\flat \langle \xi \rangle \frac{e_1(\xi)+e_2(\xi)}{\xi}\chi(\xi)  \frac{\sin (\xi r-\frac \pi 4)}{\sqrt{\xi r}}(1-\chi(r)) \\
\nonumber & \qquad  +  \widetilde b_\flat \langle \xi \rangle  \frac{e_1(\xi)+e_2(\xi)}{\xi} \chi(\xi) (\rho(r)-1)\chi(\xi r)  \\
\nonumber & \qquad +c_\flat \langle \xi \rangle \frac{e_1(\xi)+e_2(\xi)}{\xi}\chi(\xi)   \sin \left(\xi r-\frac \pi 4\right) \frac{1}{\sqrt{\xi r}} \left(\chi(r)-\chi(\xi r)\right)\ \\
\nonumber & = I+II+III+IV
\end{align}
For the first and second terms, we have that the functions $\xi \mapsto \widetilde b_\flat (\xi)\langle \xi \rangle \frac{e_1(\xi)+e_2(\xi)}{\xi}\chi(\xi)$ and $\xi \mapsto c_\flat \langle \xi \rangle \frac{e_1(\xi)+e_2(\xi)}{\xi}\chi(\xi)$ are continuous at the origin. Notice that for a radial function $v$ on $\mathbb R^2$, then $\xi \mapsto \langle J_0(\xi \cdot ) ,v\rangle_{L^2(r\dd r)}$ is the standard two-dimensional Fourier transform of $v$. Notice also that $\xi \mapsto \langle \frac{\cos (\xi \cdot)+i\sin (\xi \cdot)}{\sqrt{\cdot}},v\rangle_{L^2(r\dd r)}$ is the standard one-dimensional Fourier transform of the function $x\mapsto \mathbbm 1(x>0)\sqrt{|x|}v(|x|)$. Therefore, the Plancherel theorem in one and two dimensions imply that
$$
\left\| \int_0^\infty (I(\xi,r)+II(\xi,r)) \varphi(r)r\dd r  \right\|_{L^2(|\xi| \dd \xi)}\lesssim \| \varphi\|_{L^2(r \dd r)}  \lesssim 1.
$$
Then, using that $|\rho(r)-1|\lesssim \langle r\rangle^{-2}$, by the Cauchy-Schwarz inequality,
$$
\left| \int_0^\infty III(\xi,r) \varphi(r)r\dd r \right| \lesssim  \| \langle r\rangle^{-2} \|_{L^2(r \dd r)} \lesssim 1,
$$
from which there follows that $\| \int_0^\infty III(\xi,r) \varphi(r)r\dd r  \|_{L^2(|\xi| \dd \xi)} \lesssim 1$.

Using $|c_\flat(\xi)|\lesssim |\xi|$ we have $|IV|\lesssim \frac{\sqrt{\xi}}{\sqrt{r}}(\chi(\xi r)-\chi(r))$ hence
$$
\left\| IV  \right\|_{L^2(r \dd r)} \lesssim \left[ \int \frac{\xi}{ r} |\chi(\xi r)|^2  r \dd r \right]^{1/2} \lesssim 1
$$
so that, by the Cauchy-Schwarz inequality we get $\| \int_0^\infty IV(\xi,r) \varphi(r)r\dd r  \|_{L^2(|\xi| \dd \xi)} \lesssim 1$. Combining all these estimates, we find that
$$
\| \widetilde{\mathcal F}_\flat^S \phi_o \|_{\xi \langle \xi \rangle^{-1}L^2(|\xi|\dd \xi)}=\left\| \int_0^\infty m^S_\flat(\xi,r) \varphi(r)r\dd r   \right\|_{L^2(|\xi|\dd \xi)} \lesssim 1.
$$

\smallskip

\noindent \textit{$\bullet$ Boundedness of $\widetilde{\mathcal F}_\flat^S$ from $L^2_{e} (r\dd r)$ to $L^2_{e}(|\xi|\dd \xi)$}. Pick $\phi \in L^2(r\dd r)$ with $\| \phi\|_{L^2(r\dd r)}= 1$ (which for our purpose does actually not need to be even). We have that $\widetilde{\mathcal F}_\flat^S\phi(\xi)=\chi(\xi)\langle \psi^S_\flat(\xi,\cdot),\sigma_3 \phi\rangle_{L^2(r\dd r)}$ where by \eqref{decomposition-psi-flat} that for $\xi>0$, we decompose
\begin{align}
\label{decomposition:psi-flat-S-L2-continuity} \chi(\xi)\psi_\flat^S(\xi,r) & = \widetilde b_\flat \chi(\xi) J_0(\xi r) e(\xi) +c_\flat \chi(\xi)  \frac{\sin (\xi r-\frac \pi 4)}{\sqrt{\xi r}}(1-\chi(r))e(\xi) \\
\nonumber & \qquad  + \widetilde b_\flat \chi(\xi) (\rho(r)-1) \chi(\xi r)e(\xi) +c_\flat  \chi(\xi)   \sin \left(\xi r-\frac \pi 4 \right) \frac{1}{\sqrt{\xi r}} (\chi-\chi_{\xi^{-1}}) e(\xi)\ \\
\nonumber & = I+II+III+IV .
\end{align}
The contributions of $I$, $II$, $III$ and $IV$ can be treated exactly as the contributions of the corresponding terms in \eqref{id:Fourier-inverse-flat-S}, and we obtain $\left\| \langle I+II+III+IV ,\sigma_3 \phi\rangle_{L^2(r\dd r)} \right\|_{L^2(|\xi| \dd \xi)} \lesssim 1$, implying $\| \widetilde{\mathcal F}_\flat^S\phi \|_{L^2(|\xi|\dd \xi)}\lesssim 1$ as desired.

\smallskip

\noindent \textit{$\bullet$ Boundedness of $\widetilde{\mathcal F}_\flat^{-1S}$ from $ L^2_{e}(|\xi|\dd \xi)$ to $L^2_{e}(r\dd r)$}. Pick an even function of the form $\zeta(\xi) = \varphi(|\xi|) $ with $\| \varphi\|_{L^2((0,\infty),|\xi|\dd \xi)}= 1$. We have by \eqref{decomposition-psi-flat}
\begin{align*}
&\widetilde {\mathcal F}_\flat ^{-1S} \zeta (r) =\frac 1 \pi \int_{-\infty}^\infty \zeta(\xi) \psi_{\flat,e}^S (\xi,r) \chi(\xi) \lambda'(\xi) \operatorname{sign}( \xi) \dd \xi= \int_0^\infty m^{-1S}_{\flat}(\xi,r) \varphi(\xi)\xi \dd \xi
\end{align*}
where $\psi_{\flat,e}^S=\frac12(\psi_{\flat}^S(\xi,r)+\sigma_1 \psi_\flat^S(\xi,r))$ and where, using $\psi_{\flat}(-\xi)=-\sigma_1 \psi_{\flat}(-\xi)$ and $\sigma_1e(\xi)=-e(-\xi)$,
\begin{align} 
\label{decomposition-m-flat--1S-L2} m_\flat^{-1S}(\xi,r) & = \widetilde b_\flat \lambda'(\xi)  \frac{e(\xi)-e(-\xi)}{\xi}\chi(\xi)J_0(\xi r) +c_\flat \lambda'(\xi) \frac{e(\xi)-e(-\xi)}{\xi}\chi(\xi)  \frac{\sin (\xi r-\frac \pi 4)}{\sqrt{\xi r}}(1-\chi(r)) \\
\nonumber & \qquad  +  \widetilde b_\flat \lambda'(\xi) \frac{e(\xi)-e(-\xi)}{\xi} \chi(\xi) (\rho(r)-1)\chi(\xi r)  \\
\nonumber & \qquad +c_\flat \lambda'(\xi) \frac{e(\xi)-e(-\xi)}{\xi}\chi(\xi)   \sin (\xi r-\frac \pi 4) \frac{1}{\sqrt{\xi r}} \left(\chi (r)-\chi(\xi r)\right)\ \\
\nonumber & = I+II+III+IV
\end{align}

The first and second terms can be dealt with as for the first and second terms in \eqref{id:Fourier-inverse-flat-S}, using that $\xi \mapsto \frac{e(\xi)-e(-\xi)}{\xi}$ and $\lambda'$ are smooth at the origin, appealing to the $L^2$ continuity of the inverse Fourier transform in two and one dimensions, and we obtain $\left\| \int_0^\infty (I(\xi,r)+II(\xi,r))  \varphi(\xi)\xi \dd \xi  \right\|_{L^2(r\dd r)} \lesssim 1$. Then, using $|\rho(r)-1|\lesssim \langle r \rangle^{-2}$ and the Cauchy-Schwarz inequality,
$$
\left| \int_0^\infty III(\xi,r) \varphi(\xi)\xi \dd \xi \right| \lesssim  \| \chi(\xi) (\rho-1) \|_{L^2(|\xi| \dd \xi)} \lesssim \frac{1}{\langle r\rangle^2}
$$
from which there follows that $\| \int_0^\infty III(\xi,r) \varphi(\xi)\xi \dd \xi   \|_{L^2(r\dd r)} \lesssim 1$. For $r>1$, using $|IV|\lesssim \frac{\sqrt{\xi}}{\sqrt{r}}\chi(\xi r)$ we have
$$
\left\| IV  \right\|_{L^2(|\xi| \dd \xi)} \lesssim \left[ \int_0^2 \frac{\xi}{r}\chi_{\xi^{-1}}^2(r)  \xi \dd \xi \right]^{1/2} \lesssim \frac{1}{r^{2}}
$$
so that $\| \int_0^\infty IV(\xi,r) \varphi(\xi)\xi \dd \xi  \|_{L^2(r \dd r)} \lesssim 1$ by Cauchy-Schwarz. Combining these estimates, we find that $ \| \widetilde{\mathcal F}_\flat^{-1S} \zeta_o \|_{L^2(r \dd r)} \lesssim 1$.

\smallskip

\noindent \textit{$\bullet$ Boundedness of $\widetilde{\mathcal F}^{-1S}_\flat$ from $|\xi| \langle \xi\rangle^{-1}L^2_{o}(|\xi|\dd \xi)$ to $L^2_{o}(r\dd r)$}. Pick a function $\zeta(\xi) = |\xi|\langle \xi \rangle^{-1} \varphi(\xi) $ (which for our purpose actually does not need to be odd) with $\| \varphi\|_{L^2(|\xi|\dd \xi)}= 1$. We have by \eqref{decomposition-psi-flat}
\begin{align*}
&\widetilde {\mathcal F}_\flat ^{-1S} \zeta (r) = \int_{-\infty}^\infty \psi_\flat^S(\xi,r) \langle \xi \rangle^{-1}\lambda'(\xi)\chi(\xi)\varphi(\xi)|\xi| \dd \xi .
\end{align*}
We decompose $\chi(\xi)\psi_\flat^S$ for $\xi>0$ according to \eqref{decomposition:psi-flat-S-L2-continuity}, and for $\xi<0$ we also use the same decomposition since $\psi^S_\flat(-\xi)=-\sigma_1\psi^S_\flat(\xi)$. The terms $I$, $II$, $III$ and $IV$ can be treated exactly as the terms $I$, $II$, $III$ and $IV$ in \eqref{decomposition-m-flat--1S-L2} were treated above, so that $\| \int_{-\infty}^\infty (I+II+III+IV)(\xi,r)\langle \xi \rangle^{-1} \varphi(\xi)\xi \dd \xi   \|_{L^2(r\dd r)} \lesssim 1$. We have showed $ \| \widetilde{\mathcal F}_\flat^{-1S} \zeta \|_{L^2(r \dd r)} \lesssim 1$.

\medskip

\noindent \underline{Step 4: Boundedness of $\widetilde{\mathcal F}^R_\flat$ from $L^{2}(r\dd r)$ to $L^2(|\xi|\dd \xi)$}. Recall that $$\frac{1}{\xi} \widetilde{\mathcal{F}}_\flat^R (\phi)(\xi) =  \chi(\xi)\langle \frac{\psi_\flat^R (\xi,\cdot)}{\xi}, \sigma_3 \phi \rangle_{L^2(r\dd r)}.$$

Pick $\phi \in L^2(r\dd r)$ with $\| \phi \|_{L^2(r\dd r)}\lesssim 1$. We have by \eqref{id:psi-flat-R}
\begin{align} \label{id:psiRflat-phi-expression}
\frac{1}{\xi}\langle \psi_\flat^R ,\sigma_3\phi\rangle_{L^2(r\dd r)}= \left\langle \frac{m_{\flat}^{\tmop{loc}} }{\xi}  \chi (\xi r), \sigma_3 \phi \right\rangle_{L^2(r\dd r)}+\left\langle \frac{m_{\flat, 1}^R}{\xi} \cos (\xi r),\sigma_3 \phi\right\rangle_{L^2(r\dd r)}+ \{\mbox{similar} \}.
\end{align}
Let $|\xi|\leq 2$. For the first term, we have $\| |\xi|^{-1}m_{\flat}^{\tmop{loc}} \|_{L^2(r\dd r)}\lesssim 1$ by \eqref{bd:m-flat-loc}, and for the second we have $\| |\xi|^{-1}m_{\flat,1}^{R} \|_{L^2(r\dd r)}\lesssim 1+ \ln^2 |\xi|$ by \eqref{bd:psi-flat-R}. By Cauchy-Schwarz this implies $||\xi|^{-1}\langle \psi_\flat^R ,\sigma_3\phi\rangle_{L^2(r\dd r)}|\lesssim 1+\ln^2 |\xi|$ so we deduce that $\| |\xi|^{-1} \widetilde{\mathcal{F}}_\flat^R (\phi)(\xi)\|_{L^2(|\xi|\dd\xi)}\lesssim 1$ as desired.

\medskip

\noindent \underline{Step 5: Boundedness of $\widetilde{\mathcal F}^{-1R}_\flat$ from $L^2(|\xi|\dd \xi)$ to $L^{2}(r\dd r)$}. Recall that 
$$\widetilde{\mathcal{F}}^{-1R}_\flat (\zeta)(r) =  \langle \zeta , \chi(\xi)\psi_\flat^R(\xi,r) \lambda' \operatorname{sign} \xi \rangle_{L^2(\frac{\dd \xi}{\pi})}.$$
Pick $\zeta \in L^2(|\xi|\dd \xi)$ with $\| \zeta \|_{L^2(|\xi|\dd \xi)}\lesssim 1$. We have by \eqref{id:psi-flat-R}
\begin{align} \label{id:psiRflat-phi-expression-10}
\pi \widetilde{\mathcal{F}}^{-1R}_\flat (\zeta)(r)  & = \left\langle \zeta,  \chi(\xi) m_{\flat}^{\tmop{loc}}  \chi (\xi r) \lambda' \operatorname{sign} \xi \right\rangle_{L^2(\dd \xi)}\\
\nonumber   &\qquad +\left\langle \zeta,  \chi(\xi) m_{\flat, 1}^R \cos (\xi r) \lambda' \operatorname{sign} \xi\right\rangle_{L^2(\dd \xi)}+\mbox{similar}.
\end{align}
For the first term, we have, using \eqref{bd:m-flat-loc} and $|\lambda'(\xi)|\lesssim 1$ for $|\xi|\leq 2$, that $\| |\xi|^{-1/2}  \chi(\xi) m_{\flat}^{\tmop{loc}}  \chi (\xi r) \lambda' \|_{L^2(\dd \xi)}\lesssim \langle r \rangle^{-3}$, so that we have using $\| |\xi|^{1/2}\zeta\|_{L^2(\dd \xi)}\lesssim 1$ and Cauchy-Schwarz $|\langle \zeta,  \chi(\xi) m_{\flat}^{\tmop{loc}}  \chi (\xi r) \lambda' \operatorname{sign} \xi \rangle_{L^2(\dd \xi)}|\lesssim \langle r \rangle^{-3}$. Hence
\begin{equation} \label{courmayeur2000}
\left\| \left\langle \zeta,  \chi(\xi) m_{\flat}^{\tmop{loc}}  \chi (\xi r) \lambda' \operatorname{sign} \xi \right\rangle_{L^2(\dd \xi)}\right\|_{L^{2,1}(r\dd r)}\lesssim \| \langle r \rangle^{-3}\|_{L^{2,1}(r\dd r)}\lesssim 1.
\end{equation}
For the second term, we have, using \eqref{bd:psi-flat-R} that $\| |\xi|^{-1/2} \chi(\xi) m_{\flat, 1}^R \cos (\xi r) \lambda' \|_{L^2(\dd \xi)}\lesssim \langle r \rangle^{-3/2}\langle \ln \langle r \rangle\rangle^2$, so that we have by Cauchy-Schwarz
$$
\left\| \langle \zeta,  \chi(\xi) m_{\flat, 1}^R \cos (\xi r) \lambda' \operatorname{sign} \xi\rangle_{L^2(\dd \xi)}\right\|_{L^2(r\dd r)}\lesssim \| |\xi|^{1/2} \zeta\|_{L^2(\dd \xi)} \|  \langle r \rangle^{-3/2}\langle \ln \langle r \rangle\rangle^2\|_{L^2(r\dd r)}\lesssim 1.
$$
Combining the above bounds, we find the desired estimate $\|\widetilde{\mathcal{F}}^{-1R}_\flat (\zeta) \|_{L^2(r\dd r)}\lesssim 1$.
\end{proof}

\subsection{Weighted $L^2$ and Sobolev bounds}
We define the norms of $L^2$-based weighted and Sobolev spaces as
\begin{align*}
& \| \phi\|_{L^{2,1}(r\dd r)}= \| \langle r \rangle \phi\|_{L^2(r\dd r)},\\
& \| \phi\|_{H^{1}(r\dd r)}= \| \phi \|_{L^2(r \dd r)} + \| \partial_r \phi \|_{L^2(r \dd r)}+\| \frac{1}{r}\phi \|_{L^2(r\dd r)},\\
& \| \zeta \|_{L^{2,1}(|\xi|\dd \xi)}= \| \langle \xi \rangle  \phi\|_{L^2(|\xi|\dd \xi)},\\
& \| \zeta \|_{H^{1}(|\xi|\dd \xi)}= \|  \zeta \|_{L^2(\mathbb R,|\xi| \dd \xi)}+\| \partial_\xi \zeta \|_{L^2(\mathbb R,|\xi| \dd \xi)}.
\end{align*}

\begin{prop}[Boundedness of the Fourier transform on weighted $L^2$ and Sobolev spaces]
The application $\widetilde{\mathcal F}$ is bounded between the following spaces: 
\begin{align*}
L^{2,1}_{e}(r\dd r)& \longrightarrow H^1_{e}(|\xi|\dd \xi) \\
H^{1}_{e}(r\dd r)&\longrightarrow L^{2,1}_{e}(|\xi|\dd \xi) \\
L^{2,1}_{o}(r\dd r)&\longrightarrow |\xi| \langle \xi \rangle^{-1} H^1_{o}(|\xi|\dd \xi) \\
H^{1}_{o}(r\dd r) &\longrightarrow |\xi| \langle \xi \rangle^{-1} L^{2,1}_{o}(|\xi|\dd \xi).
\end{align*}
The application $\widetilde{\mathcal F}^{-1}$ is bounded between the following spaces 
\begin{align*}
H^1_{e}(|\xi|\dd \xi) &\longrightarrow L^{2,1}_{e}(r\dd r) \\
L^{2,1}_{e}(|\xi|\dd \xi)  &\longrightarrow H^{1}_{e}(r\dd r) \\
|\xi| \langle \xi \rangle^{-1}L^2_{o}(|\xi|\dd \xi) &\longrightarrow L^2_{o}(r\dd r) \\
|\xi| \langle \xi \rangle^{-1} L^{2,1}_{o}(|\xi|\dd \xi) &\longrightarrow H^{1}_{o}(r\dd r).
\end{align*}
\end{prop}

\begin{proof} We will follow closely the proof of Proposition \ref{propL2} and we overtake its notations.

\medskip

\noindent \underline{Step 1: Boundedness of $\widetilde{\mathcal F}_\sharp$ from $H^1(r\dd r)$ to $L^{2,1}(|\xi|\dd \xi)$ and of $\widetilde{\mathcal F}^{-1}_\sharp$ from $L^{2,1}(|\xi|\dd \xi)$ to $H^1(r\dd r)$.}

Pick $\phi \in \mathcal S_1$. Since $\mathcal H \psi_\sharp(\xi,r)=\lambda(\xi)\psi_\sharp(\xi,r)$ and $\sigma_3 \mathcal{H} \sigma_3 = \mathcal{H}^*$, we have that $\lambda(\xi)\widetilde{\mathcal F}_\sharp \phi (\xi)=\widetilde{\mathcal F}_\sharp(\mathcal H \phi)(\xi)$. Since $ \| \mathcal H \phi \|_{L^2(r\dd r)}\lesssim  \|  \phi \|_{H^2(r\dd r)}$ (where $\|  \phi \|_{H^2(r\dd r)}=\|  \phi \|_{L^2(r\dd r)}+\|  \Delta_1 \phi \|_{L^2(r\dd r)}$) and we have proved that $\widetilde{\mathcal F}_\sharp$ is continuous from $L^2(r\dd r)$ to $L^{2}(|\xi|\dd \xi)$, and since $|\lambda(\xi)| \approx |\xi|^2$ for $|\xi|>1$, we deduce that $\| \langle \xi\rangle^2 \widetilde{\mathcal F}_\sharp  \phi \|_{L^2(|\xi|\dd \xi)} \lesssim  \|  \phi \|_{H^2(r\dd r)}$. Interpolating between this bound and the $L^2$ bound, this implies $\|  \widetilde{\mathcal F}_\sharp( \phi)\|_{L^{2,1}(|\xi|\dd \xi)} \lesssim  \|  \phi \|_{H^1(r\dd r)}$.

Conversely, pick $\zeta \in \mathcal S$. Then we have similarly that $\mathcal H \widetilde{\mathcal F}^{-1}_\sharp \zeta=\widetilde{\mathcal F}^{-1}_\sharp (\lambda (\xi)\zeta)$. Hence as we have proved that $\widetilde{\mathcal F}_\sharp^{-1}$ is continuous from $L^{2}(|\xi|\dd \xi)$ to $L^2(r\dd r)$, and since $|\lambda(\xi)|\approx |\xi|^2$ for $|\xi|\geq1$, we have $\|\mathcal H \widetilde{\mathcal F}^{-1}_\sharp \zeta \|_{L^2(r\dd r)}\lesssim \| \langle \xi \rangle^2 \zeta\|_{L^2(|\xi|\dd \xi)}$. Using $\| u\|_{H^2(r\dd r)}\lesssim \| \mathcal H u \|_{L^2(r\dd r)}+\| u \|_{L^2(r\dd r)}$ we infer that $\| \widetilde{\mathcal F}^{-1}_\sharp \zeta \|_{H^2(r\dd r)}\lesssim  \| \langle \xi \rangle^2 \zeta\|_{L^2(|\xi|\dd \xi)}$. Interpolating between this bound and the $L^2$ bound, this implies $\|  \widetilde{\mathcal F}_\sharp^{-1}\zeta\|_{H^{1}(r \dd r)} \lesssim  \|  \zeta \|_{L^{2,1}(|\xi| \dd \xi)}$.

\medskip

\noindent \underline{Step 2: Boundedness of $\widetilde{\mathcal F}_\sharp$ from $L^{2,1}(r\dd r)$ to $H^1(|\xi|\dd \xi)$ and of $\widetilde{\mathcal F}^{-1}_\sharp$ from $H^1(|\xi|\dd \xi)$ to $L^{2,1}(r\dd r)$}.

\smallskip

\noindent \textit{$\bullet$ Boundedness of $\widetilde{\mathcal F}_\sharp^S$ from $L^{2,1}(r\dd r)$ to $H^1(|\xi|\dd \xi)$ and of $\widetilde{\mathcal F}^{-1S}_\sharp$ from $H^1(|\xi|\dd \xi)$ to $L^{2,1}(r\dd r)$}.

Pick $\phi\in L^{2,1}(r\dd r)$. We consider the three terms in the right-hand side of \eqref{formula-direct-fourier-sharp-S}. For the first one, since $\xi \mapsto \int v(r)J_1(\xi r) r\dd r$ is the two-dimensional Fourier transform of $x\mapsto e^{i\theta}v(r)$, the $L^{2,1}\rightarrow H^1$ continuity of the Fourier transform and \eqref{bd:estimates-b-c-sharp} imply 
$$
\left\| \widetilde a_\sharp (\xi) (1- \chi(\xi)) \int_0^\infty e_1 \cdot \phi(r)\chi (r)  J_1(\xi r) r\dd r \right\|_{H^1(|\xi|\dd \xi)}\lesssim \| \phi \|_{L^{2,1}(r\dd r)}.
$$
The two other terms in \eqref{formula-direct-fourier-sharp-S} can be bounded similarly by appealing to the $L^{2,1}\rightarrow H^1$ continuity of the one-dimensional Fourier transform. Hence $\| \widetilde{\mathcal F}_\sharp \phi\|_{H^1(|\xi|\dd \xi)}\lesssim \| \phi \|_{L^{2,1}(r\dd r)}$.

The continuity of $\widetilde{\mathcal F}^{-1}_\sharp$ from $H^1(|\xi|\dd \xi)$ to $L^{2,1}(r\dd r)$ can be proved the same way, using the formula \eqref{formula-inverse-fourier-sharp-S} and appealing to the $H^1\to L^{2,1}$ continuity of the Fourier transform in one and two dimensions.

\smallskip \noindent \textit{$\bullet$ Boundedness of $\widetilde{\mathcal F}_\sharp^R$ from $L^{2,1}(r\dd r)$ to $H^1(|\xi|\dd \xi)$ and of $\widetilde{\mathcal F}^{-1R}_\sharp$ from $H^1(|\xi|\dd \xi)$ to $L^{2,1}(r\dd r)$}.

Pick $\phi\in L^{2,1}(r\dd r)$ with $\| \phi\|_{L^{2,1}(r\dd r)}\lesssim 1$. We have by \eqref{id:psi-sharp-R}:
$$
\partial_\xi \langle \psi^R_\sharp (\xi,r),\sigma_3 \phi \rangle = \left\langle -\frac{r}{\langle r\rangle}  m_{\sharp, 1}^R (\xi, r) \sin (\xi r)+\frac{1}{\langle r \rangle}\partial_\xi ( m_{\sharp, 1}^R) (\xi, r) \cos (\xi r) ,\sigma_3 \langle r \rangle \phi \right\rangle +\mbox{similar terms}
$$
Using that $| m_{\sharp, 1}^R|+|\xi||\partial_\xi m_{\sharp, 1}^R|\lesssim\frac{1}{|\xi| \langle r \rangle\langle \xi r \rangle^{1/2}}$ and $\| \langle r \rangle \phi\|_{L^2(r\dd r)}\lesssim 1$ we have by Cauchy-Schwarz
$$
\left| \langle -\frac{r}{\langle r\rangle}  m_{\sharp, 1}^R (\xi, r) \sin (\xi r)+\frac{1}{\langle r \rangle}\partial_\xi ( m_{\sharp, 1}^R) (\xi, r) \cos (\xi r) ,\sigma_3 \langle r \rangle \phi \rangle\right|\lesssim |\xi|^{-3/2}
$$
for $|\xi|\geq 1/2$. This implies $\| \widetilde{\mathcal F}_\sharp^R \phi \|_{H^1(|\xi|\dd \xi)}\lesssim 1$.

Conversely, pick $\zeta \in H^1(|\xi|d\xi)$. Then, using \eqref{id:psi-sharp-R} and $r\cos( \xi r)=\partial_\xi (\sin(\xi r))$, then integrating by parts in $\xi$ we have:
\begin{align*}
& r \widetilde{\mathcal{F}}^{-1R}_\sharp \zeta(r) = -\langle \partial_\xi \zeta , (1-\chi(\xi))m^R_{\sharp,1} \lambda' \sin (\xi r)\operatorname{sign} \xi  \rangle_{L^2(\frac{\dd \xi}{\pi})}+\langle  \zeta ,\partial_\xi[] (1-\chi(\xi))m^R_{\sharp,1} \lambda'] \sin (\xi r)\operatorname{sign} \xi  \rangle_{L^2(\frac{\dd \xi}{\pi})}\\
& +\mbox{similar}
\end{align*}
Both terms above can be bounded exactly as in \eqref{bd:intermediaire-mathcalF-1Rsharp} using  \eqref{bd:psi-sharp-R}, so that $\| r \widetilde{\mathcal F}^{-1R}_\sharp \zeta\|_{L^{2}(r\dd r)}\lesssim 1$ as desired.

\medskip

\noindent \underline{Step 3: Boundedness of $\widetilde{\mathcal F}^S_\flat$ from $L^{2,1}_{o}(r\dd r)$ to $|\xi| \langle \xi\rangle^{-1}H^1_{o}(|\xi|\dd \xi)$ and from $L^{2,1}_{e}(r\dd r)$ to $H^1_{e}(|\xi|\dd \xi)$,}

\noindent \underline{and of $\widetilde{\mathcal F}^{-1S}_\flat$ from $H^{1}_{e}(|\xi|\dd \xi)$ to $L^{2,1}_{e}(r\dd r)$ and from $|\xi| \langle \xi\rangle^{-1}H^1_{o}(|\xi|\dd \xi)$ to $L^{2,1}_{o}(r\dd r)$.}

\smallskip

\noindent \textit{$\bullet$ Boundedness of $\widetilde{\mathcal F}_\flat^S$ from $L^{2,1}_{o}(r\dd r)$ to $|\xi| \langle \xi\rangle^{-1}H^{1}_{o}(|\xi|\dd \xi)$}. Pick $\phi = (\varphi, -\varphi)^\top$ with $\| \varphi\|_{L^{2,1}(r\dd r)}= 1$. We recall that $\widetilde {\mathcal F}_\flat^{S} \phi_o (\xi)$ is given by the formulas \eqref{id:Fourier-flat-S-odd}-\eqref{id:Fourier-inverse-flat-S}.

For the contribution of $I$, we have that $\xi \mapsto \widetilde b_\flat \langle \xi \rangle \frac{e_1(\xi)+e_2(\xi)}{\xi}$ is a $C^1$ function thanks to \eqref{bd:estimates-b-c-flat}. By the $L^{2,1}\to H^1$ continuity of the two dimensional Fourier transform, we obtain that 
$$
\left\| \int_0^\infty I(\xi,r) \varphi(r)\dd r \right\|_{H^1{(|\xi|\dd \xi)}}\lesssim 1.
$$
The contribution of $II$ can be treated similarly, using the $L^{2,1}\to H^1$ continuity of the one dimensional Fourier transform and that $1-\chi (r)$ is supported on $\{r>1\}$, and we get 
$$
\left\| \int_0^\infty II(\xi,r) \varphi(r)\dd r \right\|_{H^1(|\xi|\dd \xi)}\lesssim 1.
$$
For the contribution of $III$, we have
$$
\partial_\xi III=\partial_\xi \left(  \widetilde b_\flat \langle \xi \rangle  \frac{e_1(\xi)+e_2(\xi)}{\xi} \chi(\xi)\right) (\rho(r)-1)\chi(\xi r)+ \widetilde b_\flat \langle \xi \rangle  \frac{e_1(\xi)+e_2(\xi)}{\xi} \chi(\xi) (\rho(r)-1) r \chi'(\xi r)
$$
so by \eqref{bd:estimates-b-c-flat} and $|\rho(r)-1|\lesssim \langle r \rangle^{-2}$ we get $|\partial_\xi III|\lesssim \langle r\rangle ^{-1}$. Hence
$$
\left| \partial_\xi \int_0^\infty III(\xi,r) \varphi(r)r\dd r \right| \lesssim  \|  \langle r \rangle^{-2} \|_{L^2(r \dd r)} \| \langle r \rangle \varphi \|_{L^2(r\dd r)}\lesssim 1
$$
by the Cauchy-Schwarz inequality. Hence $\| \int_0^\infty III(\xi,r) \varphi(r)r\dd r  \|_{H^1(|\xi| \dd \xi)} \lesssim 1$. For the contribution of $IV$, we compute:
\begin{align*}
\partial_\xi IV &= \partial_\xi \left( \frac{c_\flat}{\sqrt{\xi}} \langle \xi \rangle \frac{e_1(\xi)+e_2(\xi)}{\xi}\chi(\xi)\right)   \sin (\xi r-\frac \pi 4) \frac{1}{\sqrt{r}} \left(\chi-\chi_{\xi^{-1}}\right)\\
& \qquad + \frac{c_\flat}{\sqrt{\xi}} \langle \xi \rangle \frac{e_1(\xi)+e_2(\xi)}{\xi}\chi(\xi)  \cos (\xi r-\frac \pi 4) \sqrt{r} \left(\chi-\chi_{\xi^{-1}}\right)\\
& \qquad - \frac{c_\flat}{\sqrt{\xi}} \langle \xi \rangle \frac{e_1(\xi)+e_2(\xi)}{\xi}\chi(\xi)  \cos (\xi r-\frac \pi 4) \sqrt{r} \chi'_{\xi^{-1}}
\end{align*}
Using $| \frac{c_\flat}{\sqrt{\xi}} \langle \xi \rangle \frac{e_1(\xi)+e_2(\xi)}{\xi}|+ |\xi| | \partial_\xi \left( \frac{c_\flat}{\sqrt{\xi}} \langle \xi \rangle \frac{e_1(\xi)+e_2(\xi)}{\xi}\chi(\xi)\right) |\lesssim \sqrt{\xi}$ we obtain that
$$
\left| \int_0^\infty \partial_\xi IV(\xi,r) \varphi(r)r \dd r \right| \lesssim \| \varphi \|_{L^{2,1}(r\dd r)}\left\| \frac{1}{\langle r\rangle }\partial_\xi IV (\xi,r)\right\|_{L^2(r\dd r)}\lesssim |\xi|^{-1/2}
$$
by the Cauchy-Schwarz inequality. Hence $\| \int_0^\infty IV(\xi,r) \varphi(r)r\dd r  \|_{H^1(|\xi| \dd \xi)} \lesssim 1$. This shows 
$$
\| \widetilde {\mathcal F}_\flat^{S} \phi_o \|_{\xi \langle \xi\rangle^{-1}H^1(|\xi|\dd \xi)}\lesssim 1.
$$

\smallskip

\noindent \textit{$\bullet$ Boundedness of $\widetilde{\mathcal F}_\flat^S$ from $L^{2,1}_{e}(r\dd r)$ to $H^{1}_{e}(|\xi|\dd \xi)$}. Pick $\phi \in L^2(r\dd r)$ with $\| \phi\|_{L^2(r\dd r)}= 1$ (not necessarily even). We have that $\partial_\xi \widetilde{\mathcal F}_\flat^S\phi(\xi)=\langle \partial_\xi(\chi(\xi)\psi^S_\flat(\xi,\cdot)),\sigma_3 \phi\rangle_{L^2(r\dd r)}$. For $\xi>0$, $\chi(\xi) \psi^S_\flat(\xi,\cdot)$ can be decomposed as in \eqref{decomposition:psi-flat-S-L2-continuity} (and similarly for $\xi<0$ using $\psi(-\xi)=-\sigma_1 \psi(\xi)$). The contributions of $\partial_\xi I$, $\partial_\xi II$, $\partial_\xi III$ and $\partial_\xi IV$ can be treated exactly as the contributions of the corresponding terms in the above proof of the continuity of $\widetilde{\mathcal F}_\flat^S$ from $L^{2,1}_{o}(r\dd r)$ to $|\xi| \langle \xi\rangle^{-1}H^{1}_{o}(|\xi|\dd \xi)$, and we obtain $\left\| \langle I+II+III+IV ,\sigma_3 \phi\rangle_{L^2(r\dd r)} \right\|_{H^1(|\xi| \dd \xi)} \lesssim 1$. Hence $\| \widetilde{\mathcal F}_\flat^S\phi \|_{H^1(|\xi|\dd \xi)}\lesssim 1$ as desired.

\smallskip

\noindent \textit{$\bullet$ Boundedness of $\widetilde{\mathcal F}_\flat^{-1S}$ from $ H^1_{e}(|\xi|\dd \xi)$ to $L^{2,1}_{e}(r\dd r)$}. Pick an even function of the form $\zeta(\xi) =\varphi(|\xi|) $ with $\| \varphi\|_{L^2((0,\infty),|\xi|\dd \xi)}= 1$. We recall that $\widetilde {\mathcal F}_\flat ^{-1S} \zeta_e (r) = \int_0^\infty m^{-1S}_{\flat}(\xi,r) \varphi(\xi)\xi \dd \xi$ where $m_\flat^{-1S}$ is given by \eqref{decomposition-m-flat--1S-L2}.

In the terms $I$ and $II$, we have that $\xi \mapsto \frac{e(\xi)-e(-\xi)}{\xi}$, $\lambda'$, $\widetilde b_\flat$ and $c_\flat $ are $C^1$ with $c_\flat(0)=0$, thanks to \eqref{bd:estimates-b-c-flat}. Hence, appealing to the $H^1\to L^{2,1}$ continuity of the inverse Fourier transform in two and one dimensions, using that $1-\chi$ has support in $\{ r\geq 1\}$, we obtain 
$$\left\| \int_0^\infty (I(\xi,r)+II(\xi,r))  \varphi(\xi)\xi \dd \xi  \right\|_{L^{2,1}(r\dd r)} \lesssim 1.
$$

For the third term we have, using $|\rho(r)-1|\lesssim r^{-2}$, for $r>1$
$$
\left|  \int_0^\infty III(\xi,r) \varphi(\xi)\xi \dd \xi \right| \lesssim \| \varphi \|_{L^2(\xi \dd \xi)}\frac{1}{r^2} \| \chi(\xi) \chi(\xi r)\|_{L^2(\xi \dd \xi)} \lesssim \frac{1}{r^3}.
$$
Hence, $\left\| \int_0^\infty III(\xi,r)  \varphi(\xi)\xi \dd \xi  \right\|_{L^{2,1}(r\dd r)} \lesssim 1$. Finally, for the fourth term we have, using the local $H^1(\xi \dd \xi)\to L^p$ embedding for $p>2$ large and $|c_\flat|\lesssim |\xi|$, for $r>1$,
$$
\left\|  \int_0^\infty IV(\xi,r) \varphi(\xi)\xi \dd \xi \right\|_{L^2(\xi \dd \xi)}\lesssim \| \varphi \|_{L^p(\xi \dd \xi)}\frac{1}{\sqrt{r}} \| \sqrt{\xi}\chi(\xi) \chi(\xi r)\|_{L^{p'}(\xi \dd \xi)} \lesssim \frac{1}{r^{1+\frac{2}{p'}}}
$$
where $p'$ stands for the conjugate Lebesgue exponent of $p$. Hence $\left\| \int_0^\infty IV(\xi,r)  \varphi(\xi)\xi \dd \xi  \right\|_{L^{2,1}(r\dd r)} \lesssim 1$ if one chooses $p'$ close enough to $1$.

Combining these estimates, we obtain $\| \widetilde {\mathcal F}_\flat ^{-1S} \zeta_o \|_{L^{2,1}(r\dd r)}\lesssim 1$ as desired.

\smallskip

\noindent \textit{$\bullet$ Boundedness of $\widetilde{\mathcal F}^{-1S}_\flat$ from $|\xi| \langle \xi\rangle^{-1}H^1_{o}(|\xi|\dd \xi)$ to $L^{2,1}_{o}(r\dd r)$}. Pick a function $\zeta(\xi) = |\xi|\langle \xi \rangle^{-1} \varphi(\xi) $ (not necessarily odd) with $\| \varphi\|_{L^2(|\xi|\dd \xi)}= 1$. We recall that $ \widetilde {\mathcal F}_\flat ^{-1S} \zeta (r) = \int_{-\infty}^\infty \psi_\flat^S(\xi,r) \langle \xi \rangle^{-1}\chi(\xi)\varphi(\xi)\xi \dd \xi$. We decompose $\chi(\xi)\psi_\flat^S$ for $\xi>0$ according to \eqref{decomposition:psi-flat-S-L2-continuity}, and for $\xi<0$ we also use the same decomposition since $\psi^S_\flat(-\xi)=-\sigma_1\psi^S_\flat(\xi)$. The terms $I$, $II$, $III$ and $IV$ can be treated exactly as the terms $I$, $II$, $III$ and $IV$ in \eqref{decomposition-m-flat--1S-L2} were treated above in the proof of the boundedness of $\widetilde{\mathcal F}_\flat^{-1S}$ from $ H^1_{e}(|\xi|\dd \xi)$ to $L^{2,1}_{e}(r\dd r)$, so that $\| \int_{-\infty}^\infty (I+II+III+IV)(\xi,r)\langle \xi \rangle^{-1} \varphi(\xi)\xi \dd \xi   \|_{L^2(r\dd r)} \lesssim 1$. This shows $ \| \widetilde{\mathcal F}_\flat^{-1S} \zeta \|_{L^{2,1}(r \dd r)} \lesssim 1$.

\medskip

\noindent \underline{Step 3: Boundedness of $\widetilde{\mathcal F}^R_\flat$ from $L^{2,1}(r\dd r)$ to $H^1(|\xi|\dd \xi)$.} Picking $\phi \in L^{2,1}(r\dd r)$ with $\| \phi\|_{L^{2,1}(r\dd r)}=1$ and differentiating the first term in \eqref{id:psiRflat-phi-expression} gives
$$
\partial_\xi \left\langle \frac{m_{\flat}^{\tmop{loc}}}{\xi}  \chi(\xi r),\sigma_3 \phi \right\rangle_{L^2(r\dd r)}=\left\langle \partial_\xi\left(\frac{m_{\flat}^{\tmop{loc}}}{\xi} \right) \chi (\xi r),\sigma_3 \phi \right\rangle_{L^2(r\dd r)}+\left\langle \frac{m_{\flat}^{\tmop{loc}}}{\xi}  \chi'(\xi r), \sigma_3 r \phi \right\rangle_{L^2(r\dd r)}.
$$
For the first term using \eqref{bd:m-flat-loc} and $\|\phi\|_{L^{2,1}(r\dd r)}\lesssim 1$ we have $|\langle \partial_\xi(\xi^{-1} m_{\flat}^{\tmop{loc}})  \chi(\xi r), \phi \rangle_{L^2(r\dd r)}|\lesssim 1$. The second term can be bounded exactly as the first term in \eqref{id:psiRflat-phi-expression} was bounded above since $\| r \phi\|_{L^2}\lesssim\|\phi\|_{L^{2,1}}\lesssim 1$. Hence $\left\|\partial_\xi \left\langle \frac{m_{\flat}^{\tmop{loc}}}{\xi} \chi (\xi r)), \phi \right\rangle_{L^2(r\dd r)}\right\|_{L^2(|\xi|\dd \xi)}\lesssim 1$.

We differentiate the second term in \eqref{id:psiRflat-phi-expression}:
$$
\partial_\xi \left\langle \frac{m_{\flat, 1}^R }{\xi} \cos (\xi r),\sigma_3 \phi\right\rangle_{L^2(r\dd r)}=\left\langle \partial_\xi \left( \frac{m_{\flat, 1}^R }{\xi}\right) \cos (\xi r)-r\frac{m_{\flat, 1}^R }{\xi} \sin (\xi r),\sigma_3 \phi\right\rangle_{L^2(r\dd r)}
$$
We have $\| \langle r \rangle^{-1}|\xi|^{-1}m_{\flat,1}^{R} \|_{L^2(r\dd r)}\lesssim 1+ \ln^2 |\xi|$ and $\| r\frac{m_{\flat, 1}^R }{\xi} \sin (\xi r) \|_{L^2(r\dd r)}\lesssim 1+ \ln^2 |\xi|$ by \eqref{bd:psi-flat-R}. Using $\| r\phi\|_{L^2(r\dd r)}\lesssim 1$ and Cauchy-Schwarz this implies $|\partial_\xi \langle \xi^{-1}m_{\flat, 1}^R \cos (\xi r),\sigma_3 \phi \rangle_{L^2(r\dd r)}|\lesssim 1+\ln^2 |\xi|$ so we deduce that $\| \partial_\xi \langle \xi^{-1}m_{\flat, 1}^R \cos (\xi r),\sigma_3 \phi \rangle_{L^2(r\dd r)}\|_{L^2(|\xi|\dd\xi)}\lesssim 1$

Combining with the previous estimates for the two terms in \eqref{id:psiRflat-phi-expression}, we get $\|\frac{1}{\xi} \widetilde{\mathcal F}^R_\flat \phi\|_{H^1(|\xi|\dd \xi)}\lesssim 1$.

\medskip

\noindent \underline{Step 4: Boundedness of $\widetilde{\mathcal F}^{-1R}_\flat$ from $H^1(|\xi|\dd \xi)$ to $L^{2,1}(r\dd r)$.} In the decomposition \eqref{id:psiRflat-phi-expression-10}, we already know the first term defines a bounded application from $L^2(|\xi|\dd \xi)$ to $L^{2,1}(r\dd r)$ thanks to \eqref{courmayeur2000}. Hence it suffices to bound the second one.

To do so, we integrate by parts using $r\cos(\xi r)=\partial_\xi \sin (\xi r)$ and obtain
\begin{align*}
& r \left\langle \zeta,  \chi(\xi) m_{\flat, 1}^R \cos (\xi r) \lambda' \operatorname{sign} \xi\right\rangle_{L^2(\dd \xi)} \\
& = - \left\langle \partial_\xi( \chi(\xi) \lambda' \zeta), m_{\flat, 1}^R \sin (\xi r)\operatorname{sign} \xi\right\rangle_{L^2(\dd \xi)}- \left\langle \chi(\xi)\lambda'  \zeta,  \partial_\xi m_{\flat, 1}^R \sin (\xi r) \operatorname{sign} \xi\right\rangle_{L^2(\dd \xi)} \\
& = \left\langle \partial_\xi( \chi(\xi) \lambda' \zeta), [- m_{\flat, 1}^R \sin (\xi r) + \int_0^\xi (\partial_\xi m_{\flat, 1}^R(r, \eta) \sin (\eta r)d\eta)]\operatorname{sign} \xi\right\rangle_{L^2(\dd \xi)}.
\end{align*}
The term above is estimated exactly as the first term in \eqref{id:psiRflat-phi-expression-10} was estimated above, using that $\|\partial_\xi( \chi(\xi) \lambda' \zeta)\|_{L^2(|\xi|\dd\xi)}\lesssim 1$, and that $| \int_0^\xi (\partial_\xi m_{\flat, 1}^R(r, \eta) \sin (\eta r)d\eta)|\lesssim \frac{\xi^{1/2} \ln^2 (\xi)}{\langle r \rangle^{3 / 2}}$ from \eqref{bd:psi-flat-R}, producing
$$
\left\|\left\langle \partial_\xi( \chi(\xi) \lambda' \zeta), [- m_{\flat, 1}^R \sin (\xi r) + \int_0^\xi (\partial_\xi m_{\flat, 1}^R(r, \eta) \sin (\eta r)d\eta)]\operatorname{sign} \xi\right\rangle_{L^2(\dd \xi)}\right\|_{L^2(r\dd r)}\lesssim 1.
$$
Injecting this estimate and \eqref{courmayeur2000} in \eqref{id:psiRflat-phi-expression-10}, we deduce $\| r \widetilde{\mathcal{F}}^{-1R}_\flat \zeta \|_{L^2(r\dd r)}\lesssim 1$ as desired.
\end{proof}

We finish by a straightforward adaptation of the previous proposition to some higher order regularity which is a technical result that will be useful later on.

\begin{lem} \label{lem:continuity-distorted-Fourier-L2s-Hs}
For all $s\in (1,3)$, the distorted Fourier transform $\widetilde{\mathcal F}$ is bounded from $L^{2,s}(r\dd r)$ to $H^{s}(|\xi|\dd \xi)$.
\end{lem}

\begin{proof}
The proof is identical to that of Proposition \ref{propL2}, and is actually simpler since it does not require to track cancellations near the origin for $\widetilde{\phi}_o$. The limitation $s\in (1,3)$ comes to the $\xi^2\log^2 \xi$ singularity in the estimates of Theorem \ref{toohigh}, because $\xi \mapsto \xi^2\log^2 \xi \in H^{s}(|\xi|\dd \xi)$ for such $s$ but this function does not belong to $H^3$.
\end{proof}

\section{Bijectivity of the distorted Fourier transform}

\label{sectionbijectivity}

\subsection{The main result}

Recall that the distorted Fourier transform and its inverse were defined as
\begin{align*}
& \widetilde{\mathcal{F}}(\phi)(\xi) = \int_0^\infty \psi(\xi,r) \cdot \sigma_3 \phi(r) r \dd r \\
& \widetilde{\mathcal{F}}^{-1}(\zeta)(r) = \frac 1 \pi \int_{-\infty}^\infty \zeta (\xi) \psi(\xi,r) \lambda'(\xi) \operatorname{sign} \xi \dd\xi.
\end{align*}
In Section \ref{sectiondistorted} we proved various bounds enjoyed by these transformations. We now turn to our main theorem on the distorted Fourier transform: it asserts that $\widetilde{\mathcal{F}}$ and $\widetilde{\mathcal{F}}^{-1}$ are indeed inverses of each other, and that they can be used to define the evolution group.

\begin{thm}  \label{prop:invertibility-pointwise-properties-fourier}
\begin{itemize}
\item[(i)] (Right inverse)
There holds
$$
\widetilde{\mathcal{F}} \widetilde{\mathcal{F}}^{-1} = \operatorname{Id} \qquad \mbox{on $\widetilde{L^2}$}
$$

\item[(ii)] (Left inverse)
$$
\widetilde{\mathcal{F}}^{-1} \widetilde{\mathcal{F}} = \operatorname{Id} \qquad \mbox{on $L^2(r \dd r)$}.
$$

\item[(iii)] (Plancherel identity)
For $\phi_1,\phi_2 \in \mathcal{S}_1$, 
$$
\pi \int_0^\infty \sigma_3 \phi_1(r) \cdot \phi_2(r) \,r\dd r 
= \int_{-\infty}^\infty \widetilde{\phi_1}(\xi) \widetilde{\phi_2} (\xi) \lambda'(\xi) \operatorname{sign} \xi \dd \xi.
$$
\item[(iv)] (Diagonalization of $\mathcal{H}$) The distorted Fourier transform conjugates multiplication by $\lambda$ to the operator $\mathcal{H}$: if $\phi \in \mathcal{S}_1$,
$$
\mathcal{H} \phi = \widetilde{\mathcal{F}}^{-1} \lambda \widetilde{\mathcal{F}} \phi.
$$
\item[(v)] (Evolution group) If $\phi \in \mathcal{S}_1$ and $t \in \mathbb{R}$,
$$
\label{formulaeitH}
e^{it \mathcal{H}} \phi (r)= \widetilde{\mathcal{F}}^{-1} e^{it \lambda} \widetilde{\mathcal{F}} \phi (r) = \frac{1}{\pi} \int_{-\infty}^\infty e^{it \lambda(\xi)}   \widetilde{\phi} (\xi) \psi(\xi,r)\lambda'(\xi) \operatorname{sign}\xi \dd \xi.
$$
\end{itemize}
\end{thm}

\begin{proof} 

Items $(i)$ and $(ii)$ will be the object of subsections \ref{subseci} and \ref{subsecii} below. Taking these two assertions for granted, the rest of the theorem follows. 

Indeed, assuming $(i)$, we can write $\phi_2 = \widetilde{\mathcal{F}}^{-1} \widetilde{\phi_2}(\xi)$; inserting this identity in the expression $\int \sigma \phi_1(r) \cdot \phi_2(r) r \dd r$ and applying Fubini's theorem gives Plancherel's identity $(iii)$.

Similarly, we have for $\phi \in \mathcal{S}$ that 
$\mathcal H \phi=\widetilde{\mathcal{F}}^{-1} \widetilde{\mathcal{F}} (\mathcal H\phi)$.  An integration by parts using \eqref{symmetryH} shows $\widetilde{\mathcal F}(\mathcal H\phi)(\xi)=\langle \mathcal H \psi(\xi,\cdot),\sigma_3 \phi\rangle=\lambda(\xi)\widetilde{\mathcal F} \phi (\xi)$, and this gives the diagonalization formula $(iv)$.
The evolution group formula $(v)$ then follows directly from $(iv)$ upon time integration.

\begin{rem}
Our proof of (ii) bypasses the rigorous justification of the limit in which it was formally obtained from Stone's formula in Section \ref{subsec:evolution-group}. The approach we developp is general: our proof shows that this identity holds true except on a subspace spanned by eigenfunctions of $\mathcal H$ (which in the present case do not exist, so we conclude the validity of this identity). We believe this approach to be applicable to other equations, up to dealing the projection on the discrete spectrum.

We mention that Chen and Luhrmann \cite{ChenLuhrmann} faced a similar problem of justifying (v) for the linearization of the sine-Gordon equation around a moving kink, but they could use a connectedness argument specific to this equation.
\end{rem}

\subsection{Proof of the right inverse}

\label{subseci}

We prove here that $ \widetilde{\mathcal{F}} \widetilde{\mathcal{F}}^{-1} = \operatorname{Id}$ on $\widetilde{L^2}$. By the $L^2$ boundedness result in Proposition \ref{propL2}, it suffices to prove this identity on $\mathcal{S}$

Pick $\zeta \in \mathcal S$ and $\xi\in \mathbb R$. We assume without loss of generality that $\zeta$ is real-valued, and that $\xi>0$. We assume first $\xi\geq 1$ so that by \eqref{decomposition-psi-sharp} and $b_\sharp^2(\xi)+c_\sharp^2(\xi)=1$, we have
\begin{align}
\label{decomposition-psi-sharp-r>1-1} \psi(\xi,r) &=\psi_\sharp^S(\xi,r) +\psi_\sharp^R(\xi,r)  = \frac{d(\xi)e^{i\xi r}+\bar d(\xi)e^{-i\xi r}}{\sqrt{\xi r}} e(\xi)+\widetilde m_{\sharp,1}^R(r,\xi)e^{i\xi r}+\widetilde m_{\sharp,2}^R e^{-i\xi r} \\
\label{decomposition-psi-sharp-r>1-2}& =  \frac{d(\xi)e^{i\xi r}+\bar d(\xi)e^{-i\xi r}}{\sqrt{\xi r}} e(\xi)+O\left(r^{-\frac 32}\right)
\end{align}
where $d(\xi)=(\frac{b_\sharp(\xi)}{2}+\frac{c_\sharp(\xi)}{2i})e^{-i\frac{3\pi}{4}}$ satisfies $|d|^2=1/4$, and where $\widetilde m_{\sharp,1}^R=\frac 12 (m_{\sharp,1}^R-i m_{\sharp,1}^R)$ and $\widetilde m_{\sharp,1}^R=\frac 12 ( m_{\sharp,1}^R +i m_{\sharp,1}^R)$. We have $|\psi(\xi,r)|\lesssim \langle r \rangle^{-1/2}$ using \eqref{decomposition-psi-sharp-r>1-2}, and we have $\widetilde{\mathcal F}^{-1} \zeta \in L^{2,1}(r\dd r)$ by Proposition \ref{propL2}, so that $\psi(\xi,\cdot)\cdot \sigma_3 \widetilde{\mathcal F}^{-1} \zeta \in L^1(r\dd r)$. Therefore, by Fubini,
\begin{align} \label{inversion-tildeF-tildeF-1-inter1}
\widetilde{\mathcal F} \widetilde{\mathcal F}^{-1} \zeta (\xi)
= \lim_{R \to \infty} \frac{1}{\pi} \int_{-\infty}^\infty  \zeta(\eta)  \lambda'(\eta)\operatorname{sign}(\eta)  \int_0^\infty \chi \left( \frac{r}{R} \right) \psi(\xi,r) \cdot \sigma_3  \psi(\eta,r)  \,r \dd r   \dd \eta 
\end{align}
Now for fixed numbers $\xi,\eta \in \mathbb R$ with $\eta\neq \xi$ and $R>1$, by \eqref{symmetryH} we have
\begin{align*}
& (\lambda(\xi)-\lambda(\eta))\int_0^\infty \chi(\frac{r}{R})\psi(\xi,r)\cdot \sigma_3 \psi(\eta,r)r\dd r \\
&= \int_0^\infty \chi(\frac{r}{R})\mathcal H\psi(\xi,r)\cdot \sigma_3 \psi(\eta,r)r\dd r-\int_0^\infty \chi(\frac{r}{R})\psi(\xi,r)\cdot \mathcal H^*(\sigma_3 \psi(\eta,r))r\dd r \\
&= \int_0^\infty \left( \chi(\frac{r}{R})\mathcal H\psi(\xi,r)-\mathcal H\left(\chi(\frac{r}{R})\psi(\xi,r)\right)\right)\cdot \sigma_3 \psi(\eta,r)r\dd r \\
&= \int_0^\infty \left(R^{-2}\chi''(\frac{r}{R})-r^{-1}R^{-1}\chi'(\frac{r}{R})\right)  \psi(\xi,r) \cdot  \psi(\eta,r)r\dd r+2R^{-1} \int_0^\infty \chi'(\frac{r}{R}) \partial_r \psi(\xi,r) \cdot  \psi(\eta,r)r\dd r \\
&= I(\eta,R)+II(\eta,R).
\end{align*}

We have $I(\eta,R)=O(R^{-1/2})$ uniformly in $\eta \in \mathbb R$, using $|\psi(\xi,r)|\lesssim \langle r \rangle^{-1/2}$ and $|\psi(\eta,r)|\lesssim 1$ (this second inequality being uniform in $\eta\in \mathbb R$ and $r\in (0,\infty)$). For the second term, on the one hand for any $\kappa>0$ if $|\eta|>\kappa $ then using $|\psi(\eta,r)|\lesssim_\kappa \langle r \rangle^{-1/2}$ we see it is $II(\eta,R)=O_\kappa(1)$. On the other hand for any $\nu>0$, if $|\eta-\xi|>\nu$ and $|\eta+\xi|>\nu$, assuming $\eta\geq 1$ without loss of generality for the computation, we decompose according to \eqref{decomposition-psi-sharp-r>1-2} both $ \psi(\xi,r)$ and $\psi(\eta,r)$ and obtain
\begin{align*}
II(\eta,R) & =R^{-1} \int_0^\infty \chi'(\frac{r}{R})  \frac{i\xi d(\xi)e^{i\xi r}-i\xi \bar d(\xi)e^{-i\xi r}}{\sqrt{\xi r}}  \frac{d(\eta)e^{i\eta r}+\bar d(\eta)e^{-i\eta r}}{\sqrt{\eta r}} e(\xi) \cdot e(\eta) r\dd r  +O(R^{-1})\\
& = O_\nu(R^{-1})
\end{align*}
upon integrating by parts using a non-stationary phase argument since $\eta \notin \{-\xi,\xi\}$. The same estimate can be obtained for $\eta\leq -1$ using $\psi(\eta)=-\sigma_1\psi(-\eta)$ and the above computations, and for $|\eta|\leq 1$ using the decomposition \eqref{id:psi-flat-S} and \eqref{id:psi-flat-R} and a non-stationary phase argument again. Moreover, one verifies easily that the constant in the $O_\nu(R^{-1})$ does indeed solely depend on $\nu$. Combining these observations, we have for any $\nu>0$ small that
\begin{equation} \label{inversion-tildeF-tildeF-1-inter2}
\int_0^\infty \chi(\frac{r}{R})\psi(\xi,r)\cdot \sigma_3 \psi(\eta,r)r\dd r = \left\{ \begin{array}{l l} O(1) \qquad \mbox{for }\eta\in \{|\eta+\xi|\leq \nu\}, \\
O_\nu (R^{-1}) \qquad \mbox{for }\eta\in \{|\eta-\xi|>\nu\}\cap \{|\eta+\xi|>\nu\}.
\end{array} \right.
\end{equation}
In particular, on the set $\{|\eta-\xi|>\nu\}$, the function $\eta\mapsto \int_0^\infty \chi(\frac{r}{R})\psi(\xi,r)\cdot \sigma_3 \psi(\eta,r)r\dd r$ is bounded  uniformly for $R>1$, and converges to $0$ as $R\to \infty$ at everypoint except $\eta=-\xi$. Taking $\nu>0$ small enough and appealing to the dominated convergence theorem as $R\to \infty$, we obtain
\begin{align*}
& \widetilde{\mathcal F} \widetilde{\mathcal F}^{-1} \zeta (\xi)  = \lim_{R \to \infty} \frac{1}{\pi} \int_0^\infty \chi (\frac{\eta-\xi}{\nu}) \zeta(\eta)  \lambda'(\eta)  \int_0^\infty \chi \left( \frac{r}{R} \right) \psi(\xi,r) \cdot \sigma_3  \psi(\eta ,r)  \,r \dd r   \dd \eta   .
\end{align*}
We now fix $\nu>0$ small for the moment, and will take the limit $\nu\to 0$ later on. By the decomposition \eqref{decomposition-psi-sharp},
\begin{align}
\nonumber &\widetilde{\mathcal F} \widetilde{\mathcal F}^{-1} \zeta (\xi)
= \lim_{R \to \infty} \frac{1}{\pi} \int_0^\infty \chi (\frac{\eta-\xi}{\nu}) \zeta(\eta)  \lambda'(\eta)   \Bigg( \int_0^\infty \chi \left( \frac{r}{R} \right)\Bigg( \psi_\sharp^S (\xi,r) \cdot \sigma_3  \psi_\sharp^S(\eta ,r) \\
\label{decomposition-tildeF-tildeF-1-zeta} &\qquad \qquad \qquad \qquad + \psi_\sharp^S(\xi,r) \cdot \sigma_3  \psi_\sharp^R(\eta ,r) + \psi_\sharp^R(\xi,r) \cdot \sigma_3  \psi_\sharp^S(\eta,r)+ \psi_\sharp^R(\xi,r) \cdot \sigma_3  \psi_\sharp^R(\eta,r) \Bigg) \,r \dd r  \Bigg) \dd \eta .
\end{align}
Using $|\psi_\sharp^R(\xi,r)|\lesssim \langle r \rangle^{-3/2}$, the last term in the inner integral of \eqref{decomposition-tildeF-tildeF-1-zeta} is bounded by
\begin{equation} \label{decomposition-tildeF-tildeF-1-zeta-bd1} 
\int_0^\infty  \chi \left( \frac{r}{R} \right)  \psi_\sharp^R(\xi,r) \cdot \sigma_3  \psi_\sharp^R(r,\eta)  \,r \dd r =O(1) 
\end{equation}
uniformly for $R>1$. By \eqref{decomposition-psi-sharp-r>1-1}, the second term is
\begin{align*}
& \int_0^\infty \chi \left( \frac{r}{R} \right)  \psi_\sharp^S(\xi,r) \cdot \sigma_3  \psi_\sharp^R(r,\eta)  \,r \dd r \\
& = O(1)+\int_2^\infty \chi \left( \frac{r}{R} \right)  \frac{d(\xi)e^{i\xi r}+\bar d(\xi)e^{-i\xi r}}{\sqrt{\xi r}} e(\xi)\cdot \sigma_3\left(\widetilde m_{\sharp,1}^R(r,\eta)e^{i\eta r}+\widetilde m_{\sharp,2}^R(\eta,r) e^{-i\eta r} \right) \,r \dd r .
\end{align*}
The integral can be decomposed into terms of the form
\begin{align*}
 & \int_2^\infty  \chi \left( \frac{r}{R} \right) e^{i(\xi-\eta)r}e(\xi) \cdot \sigma_3 \widetilde m_{\sharp,1}^R(\eta,r) \sqrt{r} \dd r   = \int_2^\infty \chi(|\xi-\eta|r) \chi \left( \frac{r}{R} \right) e^{i(\xi-\eta)r}e(\xi) \cdot \sigma_3 \widetilde m_{\sharp,1}^R(\eta,r) \sqrt{r} \dd r\\
 & \qquad \qquad \qquad \qquad \qquad \qquad \qquad \qquad+\int_2^\infty (1-\chi(|\xi-\eta|r)) \chi \left( \frac{r}{R} \right) e^{i(\xi-\eta)r}e(\xi) \cdot \sigma_3 \widetilde m_{\sharp,1}^R(\eta,r) \sqrt{r} \dd r
\end{align*}
and of similar contributions with $e^{i(\xi+\eta)r}$, $ e^{i(\eta-\xi)r}$, $e^{-i(\xi+\eta)r}$ and $\widetilde m_{\sharp,2}^R(\eta,r)$. In order to bound the first term in the right-hand side above we simply use $|m_{\sharp,1}^R(\eta,r)|\lesssim r^{-3/2}$ and see it is $O(|\ln |\xi-\eta||)$. The second term in the right-hand side can be bounded using $e^{i(\xi-\eta)r}=-i(\xi-\eta)^{-1}\partial_r (e^{i(\xi-\eta)r})$, then integrating by parts and using $|\partial_r^j m_{\sharp,1}^R(\eta,r)|\lesssim r^{-3/2-j}$ for $j=0,1$, and we obtain that it is $O(1)$. The other contributions can be estimated similarly, and therefore, if $|\xi-\eta|\lesssim \nu$,
\begin{equation} \label{decomposition-tildeF-tildeF-1-zeta-bd2} 
\int_0^\infty \chi \left( \frac{r}{R} \right)  \psi_\sharp^S(\xi,r) \cdot \sigma_3  \psi_\sharp^R(r,\eta)  \,r \dd r =O(|\ln |\xi-\eta||) .
\end{equation}
The third term in \eqref{decomposition-tildeF-tildeF-1-zeta} can be estimated as above, which, injected in \eqref{decomposition-tildeF-tildeF-1-zeta} with \eqref{decomposition-tildeF-tildeF-1-zeta-bd1} and \eqref{decomposition-tildeF-tildeF-1-zeta-bd2} shows
\begin{align} \label{bound-tildeF-tildeF-1-inter1} 
 &\widetilde{\mathcal F} \widetilde{\mathcal F}^{-1} \zeta (\xi)
= \lim_{R \to \infty} \frac{1}{\pi} \int_{0}^\infty \chi (\frac{\eta-\xi}{\nu}) \zeta(\eta)  \lambda'(\eta) \dd \eta \int_0^\infty \chi \left( \frac{r}{R} \right) \psi_\sharp^S (\xi,r) \cdot \sigma_3  \psi_\sharp^S(\eta ,r) r\dd r+O(\nu|\ln \nu|)
\end{align}
We now inject \eqref{decomposition-psi-sharp-r>1-1} in the above formula and compute
\begin{align}  \label{bound-tildeF-tildeF-1-inter2} 
& \lim_{R \to \infty} \frac{1}{\pi} \int_{0}^\infty \chi (\frac{\eta-\xi}{\nu}) \zeta(\eta)  \lambda'(\eta)  \dd \eta \int_0^\infty \chi \left( \frac{r}{R} \right) \psi_\sharp^S (\xi,r) \cdot \sigma_3  \psi_\sharp^S(\eta ,r) r\dd r \\
\nonumber &= \mathfrak{Re} \Bigg[ \lim_{R \to \infty} \frac{2}{\pi} \int_{0}^\infty \chi (\frac{\eta-\xi}{\nu}) \zeta(\eta)  \lambda'(\eta) \frac{1}{\sqrt{\xi \eta}} e(\xi)\cdot \sigma_3 e(\eta) \\
\nonumber & \qquad \qquad \qquad \qquad  \left( \int_0^\infty \chi \left( \frac{r}{R} \right) \left( d(\xi)d(\eta)e^{i(\xi+\eta)r}+d(\xi)\overline{d(\eta)}e^{i(\xi-\eta)r} \right) \dd r \right) \dd \eta  \Bigg]+O(\nu),
\end{align}
where we used that $\zeta$ is real-valued. For the first term if $|\xi-\eta|<2\nu$ then $|\xi+\eta|>|\xi|$ for $\nu$ small enough, so that $ \int_0^\infty \chi \left( \frac{r}{R} \right) d(\xi)d(\eta)e^{i(\xi+\eta)r}\dd r=O(1)$ by a non-stationary phase argument, implying
\begin{align}  \label{bound-tildeF-tildeF-1-inter3} 
 \int_{0}^\infty \chi (\frac{\eta-\xi}{\nu}) \zeta(\eta)  \lambda'(\eta) \frac{1}{\sqrt{\xi \eta}} e(\xi)\cdot \sigma_3 e(\eta) \mathfrak{Re} \left( \int_0^\infty \chi \left( \frac{r}{R} \right) d(\xi)d(\eta)e^{i(\xi+\eta)r}\dd r\right) \dd \eta =O(\nu)
\end{align}
uniformly for $R>1$. We then explicitly compute the second term, recognizing the Fourier transform of the Heaviside function,
\begin{align}
\nonumber & \lim_{R \to \infty} \frac{2}{\pi} \int_{0}^\infty \chi (\frac{\eta-\xi}{\nu}) \zeta(\eta)  \lambda'(\eta) \frac{1}{\sqrt{\xi \eta}} e(\xi)\cdot \sigma_3 e(\eta) \int_0^\infty \chi \left( \frac{r}{R} \right)  d(\xi)\overline{d(\eta)}e^{i(\xi-\eta)r} \dd r  \dd \eta  \\
\nonumber &= 4 \int_{0}^\infty \chi (\frac{\eta-\xi}{\nu}) \zeta(\eta)  \lambda'(\eta)\frac{1}{\sqrt{\xi \eta}} e(\xi)\cdot \sigma_3 e(\eta) d(\xi)\overline{d(\eta)}\left(\frac{1}{2}\delta(\xi-\eta)-\frac{i}{\pi}p.v.\frac{1}{\eta-\xi} \right)\dd \eta  \\
\label{bound-tildeF-tildeF-1-inter4} &= \frac{2}{\xi}\lambda'(\xi)e(\xi)\cdot \sigma_3 e(\xi) |d(\xi)|^2  \zeta(\xi) +O(\nu)\ =  \zeta(\xi)+O(\nu)
\end{align}
where we used $|d(\xi)|^2=1/4$ and $\lambda'(\xi)e(\xi)\cdot \sigma_3 e(\xi)=2\xi$. Injecting \eqref{bound-tildeF-tildeF-1-inter3}, \eqref{bound-tildeF-tildeF-1-inter4} in \eqref{bound-tildeF-tildeF-1-inter2} and then in \eqref{bound-tildeF-tildeF-1-inter2} shows
$$
\widetilde{\mathcal F} \widetilde{\mathcal F}^{-1} \zeta (\xi) =\zeta(\xi)+O(\nu |\log\nu|) =\zeta(\xi) \quad \mbox{as $\nu \to 0$}
$$
This is the desired result in the case $\xi>2$.

In the case $0< \xi \leq 1/2$, we remark that between the items (i) and (ii) of Theorem \ref{toohigh}, the validity of (ii) for $\xi\geq \frac 12$ was arbitrary. It is clear from the proof that for any $\delta>0$ the decomposition \eqref{decomposition-psi-sharp} is also valid for any $|\xi|>\delta$, up to changing the value of all constants in the estimates in (ii) depending on $\delta$. Hence by taking $\delta$ small we see the identities \eqref{decomposition-psi-sharp-r>1-1}-\eqref{decomposition-psi-sharp-r>1-2} are also valid for $0<\xi\leq 1$, so that the same proof as for $\xi\geq 1$ actually applies, and we get $\widetilde{\mathcal F} \widetilde{\mathcal F}^{-1} \zeta (\xi) =\zeta(\xi)$ in that case as well. The equality for $\xi =0$ is then a consequence of Lemma \ref{lemmaveryfirst}.(i).
\end{proof}

\subsection{Proof of the right inverse} \label{subsecii} We prove here that $\widetilde{\mathcal{F}}^{- 1} \widetilde{\mathcal{F}}= \tmop{Id}$ on $L^2(r \dd r)$. We already know that $\widetilde{\mathcal{F}} \widetilde{\mathcal{F}}^{-1} = \tmop{Id}$ on $\widetilde{L^2}$. This implies that $\widetilde{\mathcal{F}}$ is surjective, hence
  \[ \textrm{Im}(\widetilde{\mathcal{F}}) = \widetilde{L^2} . \]
  
\noindent 
\tmem{Step 1. $\widetilde{\mathcal{F}}^{- 1} \widetilde{\mathcal{F}} = \tmop{Id}$ if and only if $\widetilde{\mathcal{F}}^{- 1}$ is surjective} from $\widetilde{L^2}$ onto $L^2 (r \dd r)$. Indeed, if $\widetilde{\mathcal{F}}^{- 1}$ is surjective, then any $u \in L^2(r\dd r)$ can be written as $u =\widetilde{\mathcal{F}}^{- 1} \zeta$ for some $\zeta \in \widetilde{L^2}$. Then
$$
 \widetilde{\mathcal{F}}^{- 1} \widetilde{\mathcal{F}} (u) = u \quad \Leftrightarrow \quad \widetilde{\mathcal{F}}^{- 1} \underbrace{\widetilde{\mathcal{F}} \widetilde{\mathcal{F}}^{- 1}}_{\displaystyle \mathrm{Id}}(\zeta) = \widetilde{\mathcal{F}}^{- 1}(\zeta),
$$
which proves the desired property.

\medskip 
\noindent  
  \tmem{Step 2. Closedness of $\widetilde{\mathcal{F}}^{- 1} (\widetilde{L^2})$.} Indeed, taking $u_n \in \widetilde{\mathcal{F}}^{- 1}
  (\widetilde{L^2})$ a sequence converging in $L^2(r \dd r)$, there exists
  $\varphi_n \in \widetilde{L^2} $ such that $u_n = \widetilde{\mathcal{F}}^{- 1} (\varphi_n)$ hence
  \[ \varphi_n = \widetilde{\mathcal{F}} (u_n) \]
  thanks to the identity $\widetilde{\mathcal{F}} \widetilde{\mathcal{F}}^{- 1} =
  \tmop{Id}$. By continuity of $\widetilde{\mathcal{F}}$ (see Lemma
  \ref{propL2}), we deduce that $(\varphi_n)_e$ and $(\varphi_n)_o$ converge respectively in $L^2 (| \xi | d
  \xi)$ and $| \xi | \langle \xi \rangle^{- 1} L^2 (| \xi | \dd \xi)$: let $\varphi_\infty \in \textrm{Im} \widetilde{\mathcal{F}}$ be such that the limits are $(\varphi_\infty)_e$ and $(\varphi_\infty)_o$ respectively. By continuity of $\widetilde{\mathcal{F}}^{- 1}$,
  \[ u_n \rightarrow \widetilde{\mathcal{F}}^{- 1} (\varphi_{\infty}) \in
     \widetilde{\mathcal{F}}^{- 1} (\widetilde{L^2}), \]
  concluding the proof that $\widetilde{\mathcal{F}}^{- 1} (\widetilde{L^2})$ is closed.
  
  \
  
We now define $E = (\widetilde{\mathcal{F}}^{- 1} (
\widetilde{L^2}))^{\bot}$. Since $\widetilde{\mathcal{F}}^{- 1} (\widetilde{L^2})$ is closed, we have $L^2 (r \dd r) = E + \widetilde{\mathcal{F}}^{- 1} (\widetilde{L^2})$. Our goal is to show that $E = \{ 0 \}$.

\medskip

\noindent \tmem{Step 3. Characterization of the orthogonal $E$ of $\widetilde{\mathcal{F}}^{- 1} (\widetilde{L^2})$.} We claim that \[ E = \left\{ u \in L^2 (r \dd r), \, \left\| \int_0^\infty u (r) \cdot \psi (\xi,
     r) r \dd r \right\|_{\widetilde{L^2}} = 0 \right\} . \]
Here, note that $\int_0^\infty u (r) \cdot \psi (\xi,r) r \dd r = \widetilde{\mathcal{F}}(\sigma_3 u)$ is well-defined as a function in $\widetilde{L^2}$ by boundedness of $\widetilde{F}$ (Proposition \ref{propL2}). 

In order to prove this characterization of $E$, we will rely on the Plancherel-like identity
\begin{equation}
\label{Plancherel1}
\begin{split}
& \int_0^\infty u (r) \cdot \int_{-\infty}^{+\infty} \zeta (\xi) \psi (\xi, r) \lambda' (\xi) \tmop{sign} (\xi) \dd \xi \, r \dd r \\
& \qquad \qquad = \int_{-\infty}^{+\infty} \left( \int_0^\infty u (r) \cdot \psi (\xi, r) r \dd r
     \right) \zeta (\xi) \lambda' (\xi) \tmop{sign} (\xi) \dd \xi
\end{split}
\end{equation}
or equivalently
\begin{equation}
\label{Plancherel2}
 \int_0^\infty u (r) \cdot \widetilde{\mathcal{F}}^{-1} (\zeta) (r) r \dd r = \int_{-\infty}^{+\infty} \widetilde{\mathcal{F}}(\sigma_3 u)(\xi)  \zeta (\xi) \lambda' (\xi) \tmop{sign} (\xi) \dd \xi
\end{equation}
which we claim holds for $u \in L^2 (r \dd r)$ and $\zeta \in \widetilde{L^2}$. Let us first check that the left- and right-hand sides in the above expressions make sense; this follows by the Cauchy-Schwarz inequality and Proposition \ref{propL2}:
\begin{align*}
& \left|\int_0^\infty u (r) \cdot \widetilde{\mathcal{F}}^{-1} (\zeta) (r) r \dd r \right| \leq \left\|u  \right\|_{L^2(r \dd r)} \left\|\widetilde{\mathcal{F}}^{-1} (\zeta) \right\|_{L^2(r \dd r)} \lesssim \left\|u  \right\|_{L^2(r \dd r)} \left\| \zeta \right\|_{\widetilde{L^2}}, \\
& \left| \int_{-\infty}^{+\infty} \widetilde{\mathcal{F}}(\sigma_3 u)(\xi)  \zeta (\xi) \lambda' (\xi) \tmop{sign} (\xi) \dd \xi \right| \\
& \qquad \qquad \qquad = \left| \int_{-\infty}^{+\infty} \widetilde{\mathcal{F}}(\sigma_3 u)_o(\xi)  \zeta_e (\xi) \lambda' (\xi) \tmop{sign} (\xi) \dd \xi +  \int_{-\infty}^{+\infty} \widetilde{\mathcal{F}}(\sigma_3 u)_e(\xi)  \zeta_o (\xi) \lambda' (\xi) \tmop{sign} (\xi) \dd \xi \right| \\
& \qquad \qquad \qquad \lesssim \int_{-\infty}^{+\infty} \left| \widetilde{\mathcal{F}}(\sigma_3 u)_o (\xi) \right| \left| \zeta_e (\xi)\right| \langle \xi \rangle \dd \xi + \int_{-\infty}^{+\infty} \left| \widetilde{\mathcal{F}}(\sigma_3 u)_e (\xi) \right| \left| \zeta_o(\xi) \right| \langle \xi \rangle \dd \xi \\
& \qquad \qquad \qquad \lesssim\| \widetilde{\mathcal F}(\sigma_3 u) \|_{\widetilde{L^2}} \| \zeta \|_{\widetilde{L^2}} \lesssim  \left\|u  \right\|_{L^2(r \dd r)} \left\| \zeta \right\|_{\widetilde{L^2}}.
\end{align*}
This proves that the integrals appearing in \eqref{Plancherel1} and \eqref{Plancherel2} are well-defined. Furthermore, these equalities hold by Fubini's theorem if $u$ and $\zeta$ are integrable. By a density argument and the above bounds, we conclude that \eqref{Plancherel1} and \eqref{Plancherel2} hold for all $u \in L^2 (r \dd r)$ and $\zeta \in \widetilde{L^2}$.

In light of \eqref{Plancherel1} and the above bounds, we see that $u$ is orthogonal to $\widetilde{\mathcal F}^{-1} (\widetilde{L^2})$ (i.e. $u \in E$) if and only if $\widetilde{\mathcal F}(\sigma_3 u) = 0$, which is the desired characterization of $E$.

\medskip
  
\noindent {\tmem{Step 4. $E = \{ 0 \}$ or $\dim E = + \infty$}}.
For any $u \in L^2$ and $\lambda_0 \in i\mathbb{R}^{\ast}$
\begin{align*}
\int_0^\infty ((\mathcal{H}- \lambda_0)^{- 1})^{\ast} (u) (r) \cdot 
    \psi (\xi, r) r \dd r & = \int_0^\infty u (r) \cdot (\mathcal{H}- \lambda_0)^{- 1} (\psi (\xi,
    .)) (r) r \dd r \\
    & = \frac{1}{\lambda (\xi) - \lambda_0} \int_0^\infty u (r) \cdot \psi
    (\xi, r) r \dd r
  \end{align*}
(these equalities hold in $\widetilde{L^2}$; in order to show them rigorously, we start with $u$ smooth and decaying and argue by density). In particular, if $u \in E$,
$$
\int_0^\infty ((\mathcal{H}- \lambda_0)^{- 1})^{\ast} (u) (r) \cdot 
    \psi (\xi, r) r \dd r = \frac{1}{\lambda (\xi) - \lambda_0} \int_0^\infty u (r) \cdot \psi
    (\xi, r) r \dd r = 0.
$$

This shows that $((\mathcal{H}-\lambda_0)^{- 1})^{\ast} E \subset E$. We can now argue by contradiction: if $E$ was finite-dimensional and not reduced to $\{0 \}$, then $((\mathcal{H}-\lambda_0)^{- 1})_{| E \nobracket}^{\ast}$ would have a nonzero eigenvector by the Jordan normal form theorem. In other words, this gives the existence of $\mu \in \mathbb{C}$ and $u \in L^2(r \dd r) \setminus \{ 0 \}$ such that
$$
(\mathcal{H}^\ast - \lambda_0)^{-1} u = \mu u.
$$
This implies that $u$ belongs to the domain of $\mathcal{H}^\ast-\lambda_0$. Composing the above equality by $\mathcal{H}^\ast - \lambda_0$, we obtain
$$
u = \mu (\mathcal{H}^\ast - \lambda_0) u.
$$
We learn first from this identity that $\mu \neq 0$ and second, since $\mathcal{H}^\ast = \sigma_3 \mathcal{H} \sigma_3$, that $\sigma_3 u$ is a nonzero eigenvector of $\mathcal{H}$. Since $\mathcal{H}$ does not have nonzero eigenvectors by Theorem \ref{bug}, this provides the sought contradiction.
  
\medskip

\noindent
{\tmem{Step 5. Compactness of the closed unit ball of $E$.}} We learned from step 4 that it suffices to show that $\dim E < \infty$, which, by Riesz' lemma, is equivalent to the compactness of the (closed) unit ball of $E$. We now assume that $(v_n)$ is a sequence of elements in the unit ball of $E$, and will show that it admits a convergent subsequence.

By weak compactness of the closed unit ball of $L^2$, it admits a weakly convergent subsequence which we also denote $(v_n)$ for simplicity: $v_n \overset{L^2}{\rightharpoonup} v_\infty$. Since $E$ is weakly closed, $v_\infty \in E$ and we can consider the sequence $u_n =v_n - v_\infty$. It is a sequence in $E$ which is converging weakly to zero, and we will now show that it actually converges strongly.

\medskip

\noindent {\tmem{Step 6: reformulating the orthogonality condition for $(u_n)$}.} The sequence $(u_n)$ is such that
$$
\| u_n \|_{L^2} \lesssim 1, \qquad  u_n \overset{n \to \infty}{\rightharpoonup} 0, \qquad \mbox{and} \qquad \left\|  \int_0^\infty u_n (r)\cdot \psi (\xi, r) r \dd r  \right\|_{\widetilde{L^2}} = 0.
$$
We now want to replace the functions $\psi(\xi,r)$, which have a complicated structure, with more simple model functions. This can be achieved to the cost of replacing equality to zero by convergence to zero: we claim that there exists a function $\beta \in C^1(\mathbb{R}^+,\mathbb{R})$ such that
\begin{align*}
\left\| \int_0^\infty \left( \frac{1}{\langle \beta (\xi) \rangle} J_0 (\xi r) + \frac{\beta (\xi)}{\langle \beta (\xi) \rangle} J_1 (\xi r) \right) u_{n} (r) r \dd r \right\|_{L^2(\mathbb{R}_+,\xi \dd \xi)\times L^2(\mathbb{R}_+,\xi \dd \xi)} \overset{n \to \infty}{\longrightarrow} 0.
\end{align*}
The proof of this fact (and the properties of $\beta(\xi)$) will be delegated to Lemma \ref{hypersonic} below.

\medskip

\noindent {\tmem{Step 7. Strong convergence of $u_n$ to $0$ in $L^2$ as $n \to \infty$.}} 
We now prove that $u_{n}$ converge strongly to $0$ in $L^2$ as $n\to \infty$. Recall that the two-dimensional radial Fourier transform is such that
  \[ \widehat{u} (\xi) = \int_0^\infty J_0 (\xi r) u (r) r \dd r, \qquad  u (r) = \int_0^\infty J_0 (\xi r) \widehat{u} (\xi) \xi \dd \xi \]
First, we claim that for any $u \in L^2$, if we define $\hat{v} = \frac{\beta
(\xi)}{\langle \beta (\xi) \rangle} \hat{u} \in L^2$, then we have the
identity
\begin{equation}
\int_0^\infty \frac{\beta (\xi)}{\langle \beta (\xi) \rangle} J_1
  (\xi r) u (r) r\dd r - \int_0^\infty J_1 (\xi r) v (r) r\dd r =  \int_0^\infty \left( \frac{\beta (\xi)}{\langle \beta (\xi)
  \rangle} - \frac{\beta (\eta)}{\langle \beta (\eta) \rangle} \right) K_{0,
  1} (\eta, \xi) \hat{u} (\eta) \eta \dd \eta,  \label{carvalho}
\end{equation}
where
\[ K_{0, 1} (\eta, \xi) = \lim_{\varepsilon \rightarrow 0} \int_0^\infty
   J_0 (\eta r) J_1 (\xi r) e^{- \varepsilon r} r\dd r \]
which is studied in Lemma \ref{supersonic}. To do so, we first compute for
$\varepsilon > 0$ and $v (r) = e^{- \varepsilon r} \widehat{\left( \frac{\beta
(\xi)}{\langle \beta (\xi) \rangle} \hat{u} \right)}$ that
\begin{eqnarray*}
  &  & \int_0^\infty \frac{\beta (\xi)}{\langle \beta (\xi) \rangle} J_1
  (\xi r) u (r) e^{- \varepsilon r} r\dd r - \int_0^\infty J_1 (\xi r) v
  (r) r\dd r\\
  & = & \int_0^\infty \frac{\beta (\xi)}{\langle \beta (\xi) \rangle} J_1
  (\xi r) \int_0^\infty J_0 (\eta r) \hat{u} (\eta) \eta \dd \eta e^{-
  \varepsilon r} r\dd r\\
  & - & \int_0^\infty J_1 (\xi r) e^{- \varepsilon r} \int_0^\infty
  J_0 (\eta r) \frac{\beta (\eta)}{\langle \beta (\eta) \rangle} \hat{u}
  (\eta) \eta \dd \eta r\dd r\\
  & = & \int_0^\infty \left( \frac{\beta (\xi)}{\langle \beta (\xi)
  \rangle} - \frac{\beta (\eta)}{\langle \beta (\eta) \rangle} \right)
  \int_0^\infty J_0 (\eta r) J_1 (\xi r) e^{- \varepsilon r} r\dd r \hat{u}
  (\eta) \eta \dd \eta
\end{eqnarray*}
and we conclude by taking $\varepsilon \rightarrow 0$. Applying
(\ref{carvalho}) to the sequence $u_n$, by its weak convergence to $0$ and
Lemma \ref{supersonic}, we have
\[ \int_0^\infty \left( \frac{\beta (\xi)}{\langle \beta (\xi) \rangle} -
   \frac{\beta (\eta)}{\langle \beta (\eta) \rangle} \right) K_{0, 1} (\eta,
   \xi) \widehat{u_n} (\eta) \eta \dd \eta \rightarrow 0 \]
strongly in $L^2$ when $n \rightarrow + \infty$.

Using this result and Plancherel's theorem,
\begin{equation}
\label{equationun}
\begin{split}
 0 & \longleftarrow  \int_0^\infty \left( \int_0^\infty \left(
    \frac{1}{\langle \beta (\xi) \rangle} J_0 (\xi r) u_n (r) + J_1 (\xi r)
    v_n (r) r \dd r \right) \right)^2 \xi \dd \xi\\
& \; \; = \underbrace{\left\| \frac{\widehat{u_n}}{\langle \beta (\xi) \rangle} \right\|_{L^2(\xi \dd \xi)}^2 + \| v_n \|_{L^2 (r \tmop{dr})}^2}_{\displaystyle I} \\
& \qquad \qquad \qquad \qquad +\underbrace{ 2 \int_0^\infty \left( \int_0^\infty \frac{1}{\langle \beta
    (\xi) \rangle} J_0 (\xi r) u_n (r) r \dd r \right) \left( \int_0^\infty
    J_1 (\xi r) v_n (r) r \dd r \right) \xi \dd \xi}_{\displaystyle II}
  \end{split}
\end{equation}
On the one hand, by another application of Plancherel's theorem,
\begin{equation}
\label{equationdeux}
I = \left\| \frac{\widehat{u_n}}{\langle \beta (\xi) \rangle} \right\|_{L^2(\xi \dd \xi)}^2 + \| v_n \|_{L^2 (r \tmop{dr})}^2 \sim \| u_n \|_{L^2(r \dd r)}^2.
\end{equation}
On the other hand, since
  \begin{align*}
\int_0^\infty J_1 (\xi r) v_n (r) r \dd r & = \int_0^\infty J_1 (\xi r) \int_0^\infty J_0 (\eta r) \frac{\beta (\eta)}{\langle \beta (\eta) \rangle} \widehat{u_n} (\eta) \eta \dd \eta \, r \dd r\\
& =  \int_0^\infty \frac{\beta (\eta)}{\langle \beta (\eta) \rangle}
\widehat{u_n} (\eta) \left( \int_0^\infty J_1 (\xi r) J_0 (\eta r) r \dd r \right) \eta \dd \eta
\end{align*}
we can write
\begin{align*}
& II = \int_0^\infty \left( \int_0^\infty \frac{1}{\langle \beta (\xi) \rangle} J_0 (\xi r) u_n (r) r \dd r \right) \left( \int_0^\infty J_1 (\xi r) v_n (r) r \dd r \right) \xi \dd \xi\\
& \qquad \qquad\qquad \qquad =  \int_0^\infty \int_0^\infty \widehat{u_n} (\xi) \widehat{u_n}(\eta) K (\xi, \eta) \xi \dd \xi \, \eta \dd \eta\\
& \qquad \qquad\qquad \qquad =  \frac{1}{2} \int_0^\infty \int_0^\infty \widehat{u_n} (\xi) \widehat{u_n} (\eta) (K (\xi, \eta) + K (\eta, \xi)) \xi \dd \xi \, \eta \dd \eta
  \end{align*}
  with
  \[ K (\xi, \eta) = \frac{1}{\langle \beta (\xi) \rangle} \frac{\beta
     (\eta)}{\langle \beta (\eta) \rangle} K_{0, 1} (\eta, \xi), \]
We now claim that $K(\xi,\eta) + K(\eta,\xi)$ is a Hilbert-Schmidt kernel. This will imply that
\begin{equation}
\label{equationtrois}
II \overset{n \to \infty}{\longrightarrow} 0,
\end{equation}
and gathering \eqref{equationun}, \eqref{equationdeux} and \eqref{equationtrois} this will give that $u_n\to 0$ in $L^2$ as $n\to \infty$.

In order to check that $K(\xi,\eta) + K(\eta,\xi)$ is Hilbert-Schmidt, we rely on the following decomposition
\begin{align*}
K (\xi, \eta) + K (\eta, \xi) & = \mathbf{1}_{|\xi-\eta| > 1} \left[ \frac{1}{\langle \beta (\xi) \rangle} \frac{\beta(\eta)}{\langle \beta (\eta) \rangle} K_{0, 1} (\eta, \xi) + \frac{1}{\langle \beta (\xi) \rangle} \frac{\beta (\xi)}{\langle \beta (\xi) \rangle} K_{0, 1} (\xi, \eta) \right] \\
& \qquad \qquad + \mathbf{1}_{|\xi-\eta| < 1}  \frac{1}{\langle \beta (\xi) \rangle} \frac{\beta (\eta)}{\langle \beta (\eta) \rangle} \left[K_{0, 1} (\eta, \xi) + K_{0, 1} (\xi, \eta) \right] \\
& \qquad \qquad + \mathbf{1}_{|\xi-\eta| < 1} 
   \frac{1}{\langle \beta (\eta) \rangle}  \left[ \frac{\beta (\xi)}{\langle \beta (\xi) \rangle} - \frac{\beta(\eta)}{\langle \beta (\eta) \rangle} \right] K_{0, 1} (\xi, \eta) \\
& \qquad \qquad + \mathbf{1}_{|\xi-\eta| < 1} \frac{\beta
    (\eta)}{\langle \beta (\eta) \rangle}
    \left[ \frac{1}{\langle \beta (\eta) \rangle} - \frac{1}{\langle \beta (\xi) \rangle} \right] K_{0, 1} (\xi, \eta)
  \end{align*}
The fact that this kernel is Hilbert-Schmidt now follows from Lemma \ref{supersonic}, more specifically the application of its third, fourth and second assertions to the first, second and third and fourth lines on the right-hand side, respectively.

This concludes the proof of the identity $\widetilde{\mathcal{F}}^{-1} \widetilde{\mathcal{F}} = \operatorname{Id}$; below, we will prove lemmas \ref{hypersonic} and \ref{supersonic} which were part of the argument.

\begin{lem} \label{hypersonic} There exists a function $\xi \rightarrow \beta(\xi) \in C^1(\mathbb{R}^+,\mathbb{R})$ such that
$$
\xi \rightarrow\frac{\beta(\xi)}{\xi}, \beta'(\xi) \in L^\infty (\mathbb{R}_+)
$$
with the following property: assume that the sequence $(u_n)$ is such that
$$
\| u_n \|_{L^2} \lesssim 1, \qquad  u_n \overset{n \to \infty}{\rightharpoonup} 0, \qquad \mbox{and} \qquad \left\|  \int_0^\infty u_n (r)\cdot \psi (\xi, r) r \dd r  \right\|_{\widetilde{L^2}} = 0;
$$
then
\begin{align*}
\left\| \int_0^\infty \left( \frac{1}{\langle \beta (\xi) \rangle} J_0 (\xi r) + \frac{\beta (\xi)}{\langle \beta (\xi) \rangle} J_1 (\xi r) \right) u_{n} (r) r \dd r \right\|_{L^2(\mathbb{R}_+,\xi \dd \xi)\times L^2(\mathbb{R}_+,\xi \dd \xi)} \overset{n \to \infty}{\longrightarrow} 0.
\end{align*}
\end{lem}

\begin{proof} \underline{Step 1: Even and odd parts}
By Step 3, the definition of $\widetilde{L^2}$ and Lemma \ref{lem:fourier-symmetries},
\begin{align*}
 0 & = \left\|  \int_0^\infty u_n (r)\cdot \psi (\xi, r) r \dd r  \right\|_{\widetilde{L^2}} \\
& = \left\|  \int_0^\infty u_{n,o} (r)\cdot \psi_o (\xi, r) r \dd r  \right\|_{L^2 (\mathbb{R}_+,\xi \dd \xi)} + \left\| \frac{\langle \xi \rangle}{|\xi|} \int_0^\infty u_{n,e} (r)\cdot \psi_e (\xi, r) r \dd r  \right\|_{L^2 (\mathbb{R}_+,\xi \dd \xi)},
\end{align*}
where we recall the definitions
\begin{align*}
\begin{cases}
& \psi_e(\xi,r) = \frac{\psi(\xi,r) + \sigma_1 \psi(\xi,r)}{2} \\
&  \psi_o(\xi,r) = \frac{\psi(\xi,r) - \sigma_1 \psi(\xi,r)}{2} 
\end{cases}
\qquad
\begin{cases}
& u_{n,e}(r) = \frac{u_n(r) + \sigma_1 u_n(r)}{2} \\
& u_{n,o}(r) = \frac{u_n(r) -\sigma_1 u_n(r)}{2}.
\end{cases}
\end{align*}
Setting $a_n$ and $b_n$ as follows
$$
 u_{n,e}(r) = \begin{pmatrix} a_n(r) \\ a_n(r) \end{pmatrix} \qquad u_{n,o}(r) = \begin{pmatrix} b_n(r) \\ -b_n(r) \end{pmatrix}  ,
$$
they are such that
$$
\| a_n \|_{L^2} + \| b_n \|_{L^2} \sim \| u_n \|_{L^2(r \dd r)}
$$
and the above identity can be written as
\begin{align*}
& \left\| \int_0^\infty b_n(r)  \left[ \psi^1 (\xi, r)  - \psi^2(\xi,r) \right] r \dd r \right\|_{L^2 (\mathbb{R}_+,\xi \dd \xi)} \\
& \qquad \qquad \qquad \qquad + \left\| \frac{\langle \xi \rangle}{|\xi|}\int_0^\infty a_n(r)  \left[ \psi^1 (\xi, r)  + \psi^2(\xi,r) \right] r \dd r \right\|_{L^2 (\mathbb{R}_+,\xi \dd \xi)} = 0
\end{align*}

 We fix two real numbers $0 < \xi_0 < \xi_1 < \infty$ and we will prove different estimates in each of the three ranges $\xi \in [0,\xi_0)$, $[\xi_0,\xi_1]$, $[\xi_1,\infty)$. We focus first on the case of $(b_n)$ and will explain at the end why the desired limit holds for $(a_n)$.

\medskip

\noindent 
\underline{Step 2: Medium frequencies and definition of $\beta(\xi)$} 
By Theorem \ref{toohigh} we have that for any $0<\xi_0<\xi_1$ and $\xi \in [\xi_0, \xi_1]$,
\[ \left| \psi (\xi, r) - \frac{\cos (\xi r - \alpha (\xi))}{\sqrt{\xi r}} e
   (\xi) \right| \lesssim_{\xi_0, \xi_1} \frac{1}{\langle r \rangle^{3 / 2}},
\]
where $\xi \rightarrow \alpha (\xi) \in C^1 (\mathbb{R}^+, \mathbb{R})$ with
$\alpha (\xi) = \frac{\pi}{4} + O_{\xi \rightarrow 0} (\xi)$ and $\alpha (\xi)
= \frac{3 \pi}{4} + O_{\xi \rightarrow + \infty} \left( \frac{1}{\xi} \right)$
as well as $\alpha' (\xi) = O_{\xi \rightarrow 0} (1)$ and $\alpha' (\xi) =
O_{\xi \rightarrow + \infty} \left( \frac{1}{\xi^2} \right)$.

Since $J_0 (r) = \frac{\cos \left( r - \frac{\pi}{4} \right)}{\sqrt{r}} \left(
1 + O_{r \rightarrow + \infty} \left( \frac{1}{r} \right) \right)$ and $J_1
(r) = \frac{\cos \left( r - \frac{3 \pi}{4} \right)}{\sqrt{r}} \left( 1 +
O_{r \rightarrow + \infty} \left( \frac{1}{r} \right) \right)$, the function
$\xi \rightarrow \beta (\xi) \in C^1 (\mathbb{R}^+, \mathbb{R})$ is defined as
the unique one such that
\[ \psi (\xi, r) = \frac{1}{\langle \beta (\xi) \rangle} J_0 (\xi r) +
   \frac{\beta (\xi)}{\langle \beta (\xi) \rangle} J_1 (\xi r) + O_{r
   \rightarrow + \infty} \left( \frac{1}{r^{3 / 2}} \right) . \]
Its properties when $\xi \rightarrow 0$ and $\xi \rightarrow + \infty$ follows
from the ones of $\alpha (\xi)$. We decompose
\[ \psi (\xi, r) = \left( \frac{1}{\langle \beta (\xi) \rangle} J_0 (\xi r) +
   \frac{\beta (\xi)}{\langle \beta (\xi) \rangle} J_1 (\xi r) \right) e (\xi)
   + \psi_{\heartsuit} (\xi, r) \]
and we check that $| \psi_{\heartsuit} (\xi, r) | \lesssim_{\xi_0, \xi_1}
\frac{1}{\langle r \rangle^{3 / 2}}$. Denoting
$$
d(\xi) = e_1 (\xi) - e_2(\xi), \qquad d(\xi) \sim_{\xi_0,\xi_1} 1 \;\mbox{if}\; \xi \in [\xi_0,\xi_1],
$$
we can write
\begin{align*}
& d(\xi) \int_0^\infty \left( \frac{1}{\langle \beta (\xi) \rangle} J_0 (\xi r) + \frac{\beta (\xi)}{\langle \beta (\xi) \rangle} J_1 (\xi r) \right) b_n (r) r \dd r = \underbrace{\int_0^\infty (\psi^1(\xi,r) - \psi^2(\xi,r)) b_n (r) r \dd r }_{\displaystyle 0} \\
& \qquad \qquad \qquad \qquad + \int_0^\infty (\psi_\heartsuit^1(\xi,r) - \psi_\heartsuit^2(\xi,r)) b_n(r) r \dd r.
\end{align*}
Since
\[ \int_{\xi_0}^{\xi_1} \int_0^\infty (\psi_{\heartsuit}^1 (\xi, r) -
   \psi_{\heartsuit}^2 (\xi, r))^2 r d r \xi d \xi < + \infty, \]
the operator
\[ b \rightarrow \int_0^\infty (\psi_{\heartsuit}^1 (\xi, r) -
   \psi_{\heartsuit}^2 (\xi, r)) b (r) r d r \]
is a Hilbert-Schmidt integral operator in $L^2 ([\xi_0, \xi_1])$, and by weak convergence
of $b_n$ in $L^2$, we deduce that
\[ \left\| \int_0^\infty (\psi_{\heartsuit}^1 (\xi, r) -
   \psi_{\heartsuit}^2 (\xi, r)) b_n (r) r d r \right\|_{L^2 ([\xi_0, \xi_1])}
   \rightarrow 0 \]
when $n \rightarrow + \infty$. We have therefore shown that for any $0 < \xi_0
< \xi_1$,
\[ \left\| \int_0^\infty \left( \frac{1}{\langle \beta (\xi) \rangle} J_0
   (\xi r) + \frac{\beta (\xi)}{\langle \beta (\xi) \rangle} J_1 (\xi r)
   \right) b_n (r) r d r \right\|_{L^2 ([\xi_0, \xi_1])} \rightarrow 0 \]
when $n \rightarrow + \infty.$
\medskip

\noindent \underline{Step 3: Low frequencies.} 
Before all, we need a precise description of $\psi(\xi,r)$ at low frequencies. It is provided by Theorem \ref{toohigh} which gives the following decomposition
\[ \psi (\xi, r) = c_0(\xi) J_0 (r \xi) e (\xi) + \frac{c_1 (\xi) \cos (\xi r) + c_2 (\xi) \sin
     (\xi r)}{\langle \xi r \rangle^{1 / 2}} e (\xi) + \psi_{\diamondsuit} (\xi, r) \]
where $|c_0(\xi) - 1| + | c_1 (\xi) | + | c_2 (\xi) | \lesssim | \xi |^2 \ln^2 (| \xi |)$ and
$$
\left| \psi_\diamondsuit (\xi, r) \right| \lesssim \frac{1}{(1 + r)^2} + \frac{| \xi |^2 \ln^2 (| \xi |)}{\langle \xi r \rangle^{3 / 2}}.
$$

Next, we have first that
\begin{align*}
& \left\| \int_0^\infty \left( \frac{1}{\langle \beta (\xi) \rangle} J_0 (\xi r) + \frac{\beta (\xi)}{\langle \beta (\xi) \rangle} J_1 (\xi r) \right) b_n (r) r \dd r \right\|_{L^2([0,\xi_0],\xi \dd \xi)}\\
& \qquad \qquad\qquad \qquad \qquad\qquad \qquad = \left\| \int_0^\infty J_0 (\xi r) b_n (r) r \dd r \right\|_{L^2([0,\xi_0],\xi \dd \xi)} + O(\xi_0).
\end{align*}
This follows from the bounds $|\beta(\xi)| \lesssim |\xi|$, $\| b_n \|_{L^2} = 1$ and the $L^2(r \dd r) \to L^2( \xi \dd \xi)$ boundedness of the maps $f(r) \mapsto \int f(r) J_0(r\xi) r \dd r$, $f(r) \mapsto \int f(r) J_1(r\xi) r \dd r$ (two-dimensional Plancherel theorem). By the decomposition above,
$$
\psi^1(\xi,r) - \psi^2(\xi,r) = (2 + O(|\xi|)) J_0(r\xi) + O(|\xi|) \frac{\cos(\xi r)}{\langle \xi r \rangle} + O(|\xi|) \frac{\sin(\xi r)}{\langle \xi r \rangle} + \psi_\diamondsuit^1(\xi, r) - \psi_\diamondsuit^2(\xi, r).
$$
This implies that
\begin{align*}
&\left\| \int_0^\infty J_0 (\xi r) b_n (r) r \dd r \right\|_{L^2([0,\xi_0],\xi \dd \xi)} \lesssim \underbrace{\left\| \int_0^\infty (\psi^1(\xi,r) - \psi^2(\xi,r) ) b_n (r) r \dd r \right\|_{L^2([0,\xi_0]\xi \dd \xi)}}_{\displaystyle =0} \\
&  \qquad\qquad \qquad+ |\xi_0| \left\| \int_0^\infty \frac{\cos(\xi r)}{\langle \xi_0 r \rangle^{1/2}} b_n (r) r \dd r \right\|_{L^2([0,\xi_0],\xi \dd \xi)} + |\xi_0| \left\| \int_0^\infty \frac{\sin(\xi r)}{\langle \xi_0 r \rangle^{1/2}} b_n (r) r \dd r \right\|_{L^2([0,\xi_0],\xi \dd \xi)} \\
&  \qquad \qquad \qquad + \left\| \int_0^\infty (\psi_{\diamondsuit}^1 (\xi, r) - \psi_{\diamondsuit}^2 (\xi, r)) b_n (r) r \dd r \right\|_{L^2([0,\xi_0],\xi \dd \xi)}
\end{align*}
We now claim that all the terms on the right-hand side are $O(\xi_0)$. For the last term, this simply follows from $\| \psi(\xi,r) \|_{L^2(r \dd r)} \lesssim 1$. For the second line in the above equation, this follows from the boundedness from $L^2(r \dd r)$ to $L^2(\xi \dd \xi)$ of the map $f(r) \mapsto \int f(r) \frac{\cos(\xi r)}{|\xi r|^{1/2}} r \dd r$ (one-dimensional Plancherel theorem) as well as the boundedness of the map $f(r) \mapsto \int f(r) \chi(\xi r) \cos(\xi r) r \dd r$ (which follows from the Cauchy-Schwarz inequality).

Overall, we proved that
$$
\left\| \int_0^\infty \left( \frac{1}{\langle \beta (\xi) \rangle} J_0 (\xi r) + \frac{\beta (\xi)}{\langle \beta (\xi) \rangle} J_1 (\xi r) \right) b_n (r) r \dd r \right\|_{L^2([0,\xi_0],\xi \dd \xi)} \lesssim \xi_0.
$$

\medskip

\noindent \underline{Step 4: High frequencies.} Since $|\beta(\xi)-1| \lesssim \frac{1}{\xi}$, we find by the same arguments as were used for low frequencies that
\begin{align*}
& \left\| \int_0^\infty \left( \frac{1}{\langle \beta (\xi) \rangle} J_0 (\xi r) + \frac{\beta (\xi)}{\langle \beta (\xi) \rangle} J_1 (\xi r) \right) b_n (r) r \dd r \right\|_{L^2([\xi_1,\infty],\xi \dd \xi)}\\
& \qquad \qquad\qquad \qquad \qquad\qquad \qquad = \left\| \int_0^\infty J_1 (\xi r) b_n (r) r \dd r \right\|_{L^2([\xi_1,\infty],\xi \dd \xi)} + O \left( \frac{1}{\xi_1} \right)
\end{align*}

By Theorem \ref{toohigh}, we can decompose
$$
\psi(\xi,r) = c_0(\xi) J_1(\xi,r) e(\xi) + c_1(\xi) (1 - \chi(r)) \frac{\sin(\xi r)}{\sqrt{\xi r}} e(\xi) + c_2(\xi) (1 - \chi(r)) \frac{\sin(\xi r)}{\sqrt{\xi r}} e(\xi) + \psi_{\spadesuit}(\xi,r)
$$
with $|1-c_0(\xi)| + |c_1(\xi)| + |c_2(\xi)| \lesssim \frac{1}{\xi}$ and
$$
| \psi_{\spadesuit}(\xi,r) | \lesssim \frac{1}{\xi(1+r^2)}.
$$
We can now proceed in a similar way to what was done for low frequencies: using in particular the boundedness of $f \mapsto \int J_0(r\xi) f(r) r \dd r$ and $f \mapsto \int f(r) \frac{\cos(r\xi)}{\sqrt{r \xi}} r \dd r$ from $L^2(r \dd r)$ to $L^2(\xi \dd \xi)$, we find that
$$
\left\| \int_0^\infty J_1 (\xi r) b_n (r) r \dd r \right\|_{L^2([\xi_1,\infty],\xi \dd \xi)} = \underbrace{\left\| \int_0^\infty (\psi^2(\xi,r) - \psi^1(\xi,r)) b_n (r) r \dd r \right\|_{L^2([\xi_1,\infty],\xi \dd \xi)}}_{\displaystyle 0} + O \left( \frac{1}{\xi_1} \right).
$$
Overall, this means that
$$
\left\| \int_0^\infty \left( \frac{1}{\langle \beta (\xi) \rangle} J_0 (\xi r) + \frac{\beta (\xi)}{\langle \beta (\xi) \rangle} J_1 (\xi r) \right) b_n (r) r \dd r \right\|_{L^2([\xi_1,\infty],\xi \dd \xi)} \lesssim \frac{1}{\xi_1}.
$$

\medskip

\noindent \underline{Conclusion of the argument.} Combining the contributions of small, medium and large frequencies gives a proof of the following limit
$$
\left\| \int_0^\infty \left( \frac{1}{\langle \beta (\xi) \rangle} J_0 (\xi r) + \frac{\beta (\xi)}{\langle \beta (\xi) \rangle} J_1 (\xi r) \right) b_n (r) r \dd r \right\|_{L^2(\xi \dd \xi)} \overset{n \to \infty}{\longrightarrow} 0.
$$
Following the same argument and using in addition that 
$$
\left| \frac{ \langle \xi \rangle }{|\xi|}(e^1(\xi) - e^2(\xi)) \right| \sim 1 \qquad \mbox{for any $\xi$}
$$
gives the assertion on $(a_n)$:
$$
\left\| \int_0^\infty \left( \frac{1}{\langle \beta (\xi) \rangle} J_0 (\xi r) + \frac{\beta (\xi)}{\langle \beta (\xi) \rangle} J_1 (\xi r) \right) a_n (r) r \dd r \right\|_{L^2(\mathbb{R}_+,\xi \dd \xi)} \overset{n \to \infty}{\longrightarrow} 0.
$$
Writing $u_n$ as the sum of its even and odd parts gives the desired statement.
\end{proof}

\begin{lem}
  \label{supersonic}We define the kernel for $\xi, \eta \in \mathbb{R}^+$:
  \[ K_{0, 1} (\xi, \eta) = \lim_{\varepsilon \rightarrow 0} \int_0^\infty
     J_0 (\xi r) J_1 (\eta r) e^{- \varepsilon r} r \dd r. \]
  
\begin{enumerate}
\item If a function $f \in C^1 (\mathbb{R}^+, \mathbb{R})$ satisfies
  \[ | f' (x) | + | f (x) - 1 | \lesssim \frac{1}{\langle x \rangle} \]
then the operator
  \[ T_1 (g) (\eta) = \int_0^\infty (f (\xi) - f (\eta)) K_{0, 1} (\xi,
     \eta) g (\eta) \eta \dd \eta \]
  is Hilbert-Schmidt and thus compact from $L^2 (\xi \dd \xi)$ into
  itself.

\smallskip

\item If a function $f \in C^1 (\mathbb{R}^+, \mathbb{R})$ satisfies
$$
|f'(x)| \lesssim \frac{1}{\langle x \rangle},
$$
then the operator
$$
T_2 (g) (\eta) = \int_0^\infty \mathbf{1}_{|\xi-\eta| < 1} (f (\xi) - f (\eta)) K_{0, 1} (\xi,
     \eta) g (\eta) \eta \dd \eta
$$
is Hilbert-Schmidt and thus compact from $L^2 (\xi \dd \xi)$ into itself.

\medskip

\item If $f$ is in $L^\infty(\mathbb{R}^+ \times \mathbb{R}^+,
  \mathbb{R})$ satisfies
$$
|f(\eta,\xi)| \lesssim \frac{1}{\langle \xi \rangle} \qquad \mbox{or} \qquad |f(\eta,\xi)| \lesssim \frac{1}{\langle \eta \rangle},
$$
then the operator
$$
T_3 (g) (\eta) = \int_0^\infty \mathbf{1}_{|\xi-\eta| > 1} f (\xi, \eta) K_{0, 1} (\xi, \eta)  g (\eta) \eta \dd \eta
$$
is Hilbert-Schmidt and thus compact from $L^2 (\xi \dd \xi)$ into itself.

\smallskip
\item If a function $f \in L^{\infty} (\mathbb{R}^+ \times \mathbb{R}^+,
  \mathbb{R})$ satisfies
  \[ | f (\xi, \eta) | \lesssim \frac{\eta}{\langle
     \xi \rangle \langle \eta \rangle}, \]
then the operator
  \[ T_4 (g) (\eta) = \int_0^\infty \mathbf{1}_{|\xi-\eta| < 1} f (\xi, \eta) (K_{0, 1} (\xi, \eta) +
     K_{0, 1} (\eta, \xi)) g (\eta) \eta \dd \eta \]
  is Hilbert-Schmidt and thus compact from $L^2 (\xi \dd \xi)$ into
  itself.
\end{enumerate}
\end{lem}

\begin{proof}  {\tmem{Step 1. Pointwise estimate on $K_{0,1}$.}} We claim that, for $\xi \neq \eta$ and $\xi \eta \neq 0$:
  \[ | K_{0, 1} (\xi, \eta)  | \lesssim \frac{ | \ln (\max (2, \xi^{- 1}, \eta^{- 1})) |}{| \xi^2 - \eta^2 | }. \]
By definition of the Bessel functions $J_0$ and $J_1$
  \[ (\Delta + \xi^2) (J_0 (\xi r)) = \left( \Delta + \eta^2 - \frac{1}{r^2}
     \right) (J_1 (\eta r)) = 0, \]
hence
  \begin{align*}
    0 & = \int_0^\infty (\Delta + \xi^2) (J_0 (\xi r)) J_1 (\eta r) e^{-\varepsilon r} r \dd r \\
    & =  \xi^2 K_{0, 1}^{\varepsilon} (\xi, \eta) + \int_0^\infty J_0
    (\xi r) \Delta (J_1 (\eta r) e^{- \varepsilon r}) r \dd r\\
    & = (\xi^2 - \eta^2) K_{0, 1}^{\varepsilon} (\xi, \eta) + \int_0^{+\infty} \frac{J_0 (\xi r) J_1 (\eta r)}{r^2} e^{- \varepsilon r} r \dd r\\
& \qquad \qquad + \int_0^\infty J_0 (\xi r) (J_1 (\eta r) \Delta (e^{-
    \varepsilon r}) + 2 \partial_r (J_1 (\eta r)) \partial_r (e^{- \varepsilon
    r})) r \dd r.
  \end{align*}
(the integrals above are well-defined and the boundary terms vanish since $J_1(0)=0$).  By the oscillatory nature of $J_0$ and $J_1$ we check that for $\eta \xi
  \neq 0$,
  \[ \int_0^\infty J_0 (\xi r) (J_1 (\eta r) \Delta (e^{- \varepsilon r})
     + 2 \partial_r (J_1 (\eta r)) \partial_r (e^{- \varepsilon r})) r \dd r
     \rightarrow 0 \]
  when $\varepsilon \rightarrow 0$. As $\varepsilon \to 0$, we find
  \[ K_{0, 1} (\xi, \eta) = \frac{1}{\eta^2 - \xi^2} \int_0^\infty
     \frac{J_0 (\xi r) J_1 (\eta r)}{r^2} r \dd r. \]
Using the classical decay estimate on Bessel
functions $| J_0 (r) | + | J_1 (r) | \lesssim \frac{1}{\langle r \rangle^{1 /
2}}$, we have
\begin{align*}
\left| \int_0^\infty \frac{J_0 (\xi r) J_1 (\eta r)}{r^2} r d r
  \right| & \lesssim  1 + \int_1^{+ \infty} \frac{\dd r}{r \langle \xi r \rangle^{1 /
  2} \langle \eta r \rangle^{1 / 2}}\\
  & \lesssim 1 + \left| \int_1^{\max (\xi^{- 1}, \eta^{- 1})} \frac{\dd r}{r
  \langle \xi r \rangle^{1 / 2} \langle \eta r \rangle^{1 / 2}} \right| + \int_{\max (\xi^{- 1}, \eta^{- 1})}^{+ \infty} \frac{\dd r}{r \langle \xi r
  \rangle^{1 / 2} \langle \eta r \rangle^{1 / 2}}\\
  & \lesssim 1 + \left| \int_1^{\max (\xi^{- 1}, \eta^{- 1})} \frac{\dd r}{r}
  \right| + \frac{1}{\sqrt{\xi \eta}} \int_{\max (\xi^{- 1}, \eta^{- 1})}^{+
  \infty} \frac{\dd r}{r^2}\\
  & \lesssim  1 + | \ln (\max (\xi^{- 1}, \eta^{- 1})) | +
  \frac{1}{\sqrt{\xi \eta} \max (\xi^{- 1}, \eta^{- 1})}\\
  & \lesssim  | \ln (\max (2, \xi^{- 1}, \eta^{- 1})) | .
\end {align*}
concluding the proof of
  \[ | K_{0, 1} (\xi, \eta) | \lesssim \frac{ | \ln (\max (2, \xi^{- 1}, \eta^{- 1})) |}{| \xi^2 - \eta^2 | }. \]

\medskip

\noindent  {\tmem{Step 2. Pointwise estimate on the symmetrized $K_{0,1}$.}} We focus here on $| \xi - \eta | \leqslant 1$.
By the formula derived above,
\begin{equation}
K_{0, 1} (\xi, \eta) + K_{0, 1} (\eta,\xi) = \frac{1}{\eta^2 - \xi^2}
\int_0^\infty (J_0 (\xi r) J_1 (\eta r) - J_1 (\xi r) J_0 (\eta r)) 
    \frac{\dd r}{r} . \label{partypants}
  \end{equation}
Recall from Lemma \ref{abudhabi} that
\begin{equation*}
\begin{array}{l}
J_0 (r) = \frac{\cos \left( r - \frac{\pi}{4} \right)}{\langle r \rangle^{1 / 2}} + a_0 (r) \\
J_1 (r) = \frac{\sin \left( r - \frac{\pi}{4} \right)}{\langle r
     \rangle^{1 / 2}} + a_1 (r)
\end{array}
\qquad \mbox{with} \qquad | a_j (r) | + \langle r \rangle | a_j' (r) | \lesssim \frac{1}{\langle r
     \rangle^{3 / 2}}, \qquad j=0,1.
\end{equation*}
Replacing in (\ref{partypants}) the Bessel functions by their leading
  orders, we find

\begin{align*}
& \frac{1}{\eta^2 - \xi^2} \int_0^\infty \left( \frac{\cos \left( \xi r - \frac{\pi}{4} \right)}{\langle \xi r \rangle^{1 / 2}} \frac{\sin \left( \eta r - \frac{\pi}{4} \right)}{\langle \eta r \rangle^{1 / 2}} - \frac{\cos \left( \eta r - \frac{\pi}{4} \right)}{\langle \eta r \rangle^{1 / 2}} \frac{\sin \left( \xi r - \frac{\pi}{4} \right)}{\langle \xi r \rangle^{1 / 2}} \right)  \frac{\dd r}{r}\\
& \qquad \qquad = \frac{1}{\xi + \eta} \int_0^\infty \frac{1}{\langle \xi r
    \rangle^{1 / 2} \langle \eta r \rangle^{1 / 2}} \frac{\sin ((\eta - \xi)
    r)}{(\eta - \xi)} \frac{\dd r}{r} .
  \end{align*}
  and since $| \sin (x) | \lesssim \frac{x}{\langle x \rangle}$ we have
  \[ \left| \int_0^\infty \frac{1}{\langle \xi r \rangle^{1 / 2} \langle
     \eta r \rangle^{1 / 2}} \frac{\sin ((\eta - \xi) r)}{(\eta - \xi)}
     \frac{\dd r}{r} \right| \lesssim \int_0^\infty \frac{\dd r}{\langle \xi r
     \rangle^{1 / 2} \langle \eta r \rangle^{1 / 2} \langle | \eta - \xi | r
     \rangle} . \]

First, if $\xi ,  \eta  \geqslant 2$, 
  \begin{align*}
\int_0^\infty \frac{\dd r}{\langle \xi r \rangle^{1 / 2} \langle
 \eta r \rangle^{1 / 2} \langle | \eta - \xi | r \rangle}  & \lesssim  \int_0^1 \frac{\dd r}{\langle \xi r
    \rangle^{1 / 2} \langle \eta r \rangle^{1 / 2}} + \frac{1}{\sqrt{\xi \eta} } \int_1^{1 / | \eta - \xi |} \frac{\dd r}{r} + \frac{1}{\sqrt{\xi \eta}  | \eta - \xi |}
    \int_{1 / | \eta - \xi |}^{+ \infty} \frac{\dd r}{ r^2}\\
& \lesssim  \frac{\langle \ln | \xi - \eta |  \rangle + \langle \ln  \xi  \rangle + \langle \ln  \eta  \rangle}{\sqrt{\xi \eta}},
\end{align*}

For $ \xi ,  \eta \leqslant 3$ (which cover the full range of $\xi,
  \eta \in \mathbb{R}^+$ such that $| \xi - \eta | \leqslant 1$) it follows from the inequality $\sqrt{\xi \eta} \langle r \rangle \lesssim \langle \eta r \rangle^{1/2} \langle \xi r \rangle^{1/2}$ that
\begin{align*}
\int_0^\infty \frac{\dd r}{\langle \xi r \rangle^{1 / 2} \langle \eta r \rangle^{1 / 2} \langle | \eta - \xi | r \rangle} & \lesssim \frac{1}{\sqrt{\xi \eta}} \int_0^\infty \frac{d
    r}{\langle r \rangle \langle | \eta - \xi | r \rangle}\\
    & \lesssim \frac{1}{\sqrt{\xi \eta}} \left( \int_0^{1 / | \eta - \xi |} \frac{\dd r}{\langle r \rangle } +
    \frac{1}{| \xi - \eta |} \int_{1 / | \eta - \xi |}^{+ \infty} \frac{\dd r}{r^2} \right)\\
    & \lesssim  \frac{\langle \ln (| \xi - \eta |) \rangle}{\sqrt{\xi \eta}}
  \end{align*}

There remains to deal with the remainder terms. Out of the three terms which involve these remainders, we will focus for the sake of illustration on the following term. Using the bounds on $a_1$ and $a_0$ and assuming without loss of generality that $\xi \geq \eta$,
\begin{align*}
& \left|  \int_0^\infty \left( \frac{\cos(\xi r - \frac \pi 4)}{\langle \xi r \rangle^{1/2}} a_1(\eta r) - a_1(\xi r)  \frac{\cos(\eta r - \frac \pi 4)}{\langle \eta r \rangle^{1/2}} \right)
 \frac{\dd r}{r} \right| \\
& \qquad \lesssim \int_0^\infty \frac{1}{\langle \xi r \rangle^{1/2}} |a_1(\eta r) - a_1(\xi r)| \frac{\dd r}{r} + \int_0^\infty \left| \frac{\cos(\xi r - \frac \pi 4)}{\langle \xi r \rangle^{1/2}} - \frac{\cos(\eta r- \frac \pi 4)}{\langle \eta r \rangle^{1/2}} \right| |a_1(\xi r)| \frac{\dd r}{r}  \\
& \qquad \lesssim  \int_0^\infty \frac{1}{\langle \xi r \rangle^{1/2}} \left| \frac{1}{\langle \eta r \rangle^{1/2}} - \frac{1}{\langle \xi r \rangle^{1/2}} \right|  \frac{dr}{r}  +  \int_0^\infty \left[ \frac{|\xi - \eta| r }{\langle \xi r \rangle^{1/2}} + \left| \frac{1}{\langle \eta r \rangle^{1/2}} - \frac{1}{\langle \xi r \rangle^{1/2}} \right| \right] \frac{1}{\langle \xi r \rangle^{3/2}} \frac{\dd r}{r} \\
& \qquad \lesssim \frac{|\xi-\eta|}{\xi}.
\end{align*}
Gathering the above estimates, we find 
$$
\left| K_{0, 1} (\xi, \eta) + K_{0, 1} (\eta,\xi) \right| \lesssim \frac{\langle \ln | \xi - \eta |  \rangle +  \ln \langle \xi  \rangle + \ln \langle \eta  \rangle}{(\xi+\eta) \sqrt{\xi \eta}} \qquad \mbox{if $|\xi - \eta| \leq 1$.}
$$

\medskip

\noindent  {\tmem{Step 3. The operator $T_1$ is Hilbert-Schmidt.}} It is enough to show that $(f (\xi) - f (\eta)) K_{0, 1} (\xi, \eta) \in
  L^2 (\xi \dd \xi \, \eta \dd \eta, \mathbb{R})$. By Step 1, we have
\[ | (f (\xi) - f (\eta)) K_{0, 1} (\xi, \eta) | \lesssim \left| \frac{f
     (\xi) - f (\eta)}{\xi - \eta} \right| \frac{1}{\xi + \eta} \]
  and by the estimates on $f$,
  \[ \left| \frac{f (\xi) - f (\eta)}{\xi - \eta} \right| \lesssim
     \frac{1}{\langle \xi - \eta \rangle} \left( \frac{1}{\langle \xi \rangle} + \frac{1}{\langle \eta \rangle} \right). \]
Therefore
\[ \frac{1}{\langle \xi - \eta \rangle} \left( \frac{1}{\langle \xi \rangle} + \frac{1}{\langle \eta \rangle} \right) \frac{1}{\xi + \eta} \in L^2 (\xi \dd \xi \, \eta \dd \eta) . \]
since
\[ \int_{0}^\infty \int_0^\infty \left( \frac{1}{\langle \xi -
     \eta \rangle} \frac{1}{\langle \xi \rangle} \frac{1}{\xi + \eta}
     \right)^2 \xi \eta\dd \xi \dd \eta \lesssim \int_0^\infty
     \int_0^\infty \frac{1}{\langle \xi - \eta \rangle^2}
     \frac{1}{\langle \xi \rangle^2} \dd \xi \dd \eta < \infty. \]

\medskip

\noindent  {\tmem{Step 3. The operator $T_2$ is Hilbert-Schmidt.}} Since by Step 1
\[ | (f (\xi) - f (\eta)) K_{0, 1} (\xi, \eta) | \lesssim \left| \frac{f
     (\xi) - f (\eta)}{\xi - \eta} \right| \frac{1}{\xi + \eta} \]
the estimates on $f$ give
$$
\mathbf{1}_{|\xi-\eta|<1} | (f (\xi) - f (\eta)) K_{0, 1} (\xi, \eta)| \lesssim \mathbf{1}_{|\xi-\eta|<1} \frac{1}{\langle \xi \rangle} \frac{1}{\xi+\eta}
$$
and the desired conclusion follows since
$$
\int_{0}^\infty \int_0^\infty \mathbf{1}_{|\xi-\eta|<1} \left( \frac{1}{\langle \xi \rangle} \frac{1}{\xi + \eta}
     \right)^2 \xi \eta\dd \xi \dd \eta < \infty.
$$

\medskip
  
\noindent {\tmem{Step 4: The operator $T_3$ is Hilbert-Schmidt.}} By Step 1, it suffices to show that
$$
\frac{1}{\langle \xi \rangle} \frac{\ln(\xi^{-1}+2) + \ln(\eta^{-1}+2)}{| \xi^2 - \eta^2|} \mathbf{1}_{|\xi-\eta|>1} \in L^2( \eta \dd \eta \, \xi \dd \xi)
$$
This follows from the observation that
$$
\int_{|\xi-\eta| > 1} \frac{\eta \ln(\eta^{-1}+2)^2}{(\xi+\eta)^2 (\xi - \eta)^2} \dd \eta \lesssim \frac{1}{\xi} \int_{|\xi-\eta| > 1} \frac{\ln(\eta^{-1}+2)^2}{(\xi - \eta)^2} \dd \eta \lesssim \frac{1}{\xi},
$$
since then
$$
\int_0^\infty \int_0^\infty \frac{1}{\langle \xi \rangle^2} \frac{\ln(\xi^{-1}+2)^2 + \ln(\eta^{-1}+2)^2}{| \xi^2 - \eta^2|^2} \mathbf{1}_{|\xi-\eta|>1} \eta \dd \eta \, \xi \dd \xi \lesssim \int_0^\infty \frac{\ln(\xi^{-1}+2)^2 \xi}{\langle \xi \rangle^2  \xi} \dd \xi < \infty.
$$

\medskip

\noindent {\tmem{Step 5: The operator $T_4$ is Hilbert-Schmidt.}}
By Step 2, it suffices to check that the following integral is finite:
\begin{align*}
& \int_0^\infty \int_0^\infty \mathbf{1}_{|\xi-\eta| < 1} \frac{\eta^2 }{\langle \xi \rangle^2 \langle \eta \rangle^2} \frac{\langle \ln | \xi - \eta |  \rangle^2 +  (\ln \langle \xi  \rangle)^2 + (\ln \langle \eta  \rangle)^2}{(\xi+\eta)^2 \xi \eta} \eta \dd \eta \, \xi \dd \xi  \\
& \lesssim \int_0^\infty \int_0^\infty \mathbf{1}_{|\xi-\eta| < 1} \frac{\langle \ln | \xi - \eta |  \rangle^2 +  (\ln \langle \xi  \rangle)^2}{\langle \xi \rangle^4} \dd \eta \dd \xi < \infty.
\end{align*}
\end{proof}

\section{Estimates on the evolution group}

\label{sectionestimates}

\subsection{$L^2$ estimates for localized solutions}

\begin{thm}[$L^2$ growth and boundedness] 
\label{thm-L2bound}
Let $\epsilon>0$. If $\phi \in \mathcal{S}_1$, then for $t\geq 0$ we have
$$
\| e^{it \mathcal{H}} \phi \|_{L^2} \lesssim \sqrt{ \log (2+ t ) } \ \| \widetilde{\phi}\|_{H^{1+\epsilon}}
$$
and if furthermore $\langle \phi , (\rho,\rho)^\top \rangle_{L^2(r \dd r)} =0$, then
$$
\| e^{it \mathcal{H}} \phi \|_{L^2} \lesssim  \| \widetilde{\phi}\|_{H^{1+\epsilon}}.
$$
\end{thm}

\begin{proof}
By Proposition \ref{propL2} and the semi-group evolution formula $(v)$ of Theorem \ref{prop:invertibility-pointwise-properties-fourier} we have by decomposing between even and odd parts
\begin{align*}
\| e^{it\mathcal H}\phi\|_{L^2(r\dd r)} & \sim \| (e^{it\mathcal H}\phi)_e\|_{L^2(r\dd r)}+\| (e^{it\mathcal H}\phi)_o\|_{L^2(r\dd r)} \\
 & \sim \| (e^{it\lambda}\widetilde \phi)_e\|_{L^2(|\xi|\dd \xi)}+\| (e^{it\lambda}\widetilde \phi)_o\|_{L^2(|\xi|^{-1}\langle \xi \rangle^2\dd \xi)}.
\end{align*}
The first term satisfies $\| (e^{it\lambda}\widetilde \phi)_e\|_{L^2(|\xi|\dd \xi)}\lesssim \| \widetilde{\phi}\|_{L^2(|\xi|\dd \xi)}$. As for the second term, we further decompose $(e^{it\lambda}\widetilde \phi)_o= \cos(t\lambda)\widetilde \phi_o+i\sin(t\lambda)\widetilde \phi_{e}$.

We have $\widetilde \phi_o\in L^2(|\xi|^{-1}\langle \xi \rangle \dd \xi)$ by Proposition \ref{propL2} and $\| \widetilde \phi\|_{C^{\alpha}}\lesssim \| \widetilde{\phi}\|_{H^{1+\epsilon}(|\xi|\dd \xi)}$ for some $\alpha>0$ by Sobolev embedding, so we deduce $\phi_o(0)=0$ and $|\phi_o(\xi)|\lesssim |\xi|^{\alpha} \| \widetilde{\phi}\|_{H^{1+\epsilon}(|\xi|\dd \xi)}$. Hence the first part is bounded $\|\cos(t\lambda)\widetilde \phi_o\|_{L^2(|\xi|^{-1}\langle \xi \rangle \dd \xi)} \lesssim  \| \widetilde{\phi}\|_{H^{1+\epsilon}(|\xi|\dd \xi)} $. We have for the second part that $\widetilde \phi_e(0)= \langle \phi,(\rho,\rho)^\top\rangle_{L^2(r\dd r)}$, so we deduce $\widetilde \phi_e(\xi)= \langle \phi,(\rho,\rho)^\top\rangle_{L^2(r\dd r)}+O(|\xi|^{\alpha} \| \widetilde{\phi}\|_{H^{1+\epsilon}(|\xi|\dd \xi)})$. Hence  $\| \sin(t\lambda)\widetilde \phi_e\|_{L^2(|\xi|^{-1}\langle \xi\rangle^2\dd \xi)}\lesssim \sqrt{1+ \ln \langle t \rangle}$ unless $\langle \phi, (\rho,\rho)^\top\rangle_{L^2(r\dd r)}=0$ in which case it is $\lesssim 1$. This implies the desired bounds.
\end{proof}

\subsection{Pointwise decay: a model case}

We will prove in the next subsection our main decay result (Theorem \ref{thmpointwisedecay}). Before doing so, we shall first estimate a model oscillatory integral which accounts for the contribution of low frequencies in the case where $\widetilde{\phi}$ vanishes at the origin.

\begin{lem}[Pointwise decay for a model oscillatory integral] 
\label{lempointwisedecay} Consider
$$
\mathcal I_{model}(t,r) = \int_{-\infty}^{+\infty} e^{i \Phi_{-}(t,\xi,r)} \xi \widetilde{\phi}(\xi) W (\xi,r) \dd \xi,
$$
where we are denoting
$$
\Phi_{-}(t,\xi,r) = t \lambda(\xi) - r \xi.
$$
The function $\widetilde \phi \in C^2(\mathbb R)$ is supported on $|\xi| \lesssim 1$, and for $j=0,1,2$,
\begin{equation}
\label{estimateWa2-lem}
| \partial_{\xi}^j W (\xi,r)  | \lesssim \frac{\langle r \rangle^j}{\langle \xi r \rangle^{1/2+j}}.
\end{equation}
Then, for $t>1$ we have
\begin{equation} \label{dispersive1-lem}
|\mathcal I_{model}(t,r)|\lesssim \frac{1}{t} \qquad \mbox{for }r\geq 0.
\end{equation}
\end{lem}

\begin{proof} 

We will bound the integral on $(0,\infty)$, as the one over $(-\infty,0)$ can be estimated by the same arguments, so we restrict
$$
\mathcal I_{model}(t,r) =  \int_{0}^{+\infty} e^{i \Phi_{-}(t,\xi,r)} \xi \widetilde{\phi}(\xi) W (\xi,r) \dd \xi.
$$
The phase $\Phi_-$ has a stationary point $\xi_0$ if $\frac r t> \sqrt{2}$:
\begin{equation} \label{boundphase-lem-2}
\mbox{if} \; \frac r t \geq \sqrt 2, \, |\xi| \lesssim 1,  \qquad \partial_\xi \Phi_-(t,\xi_0,r) = 0 \quad \Longrightarrow \quad  \xi_0 \sim \left[ \frac{r}{t} - \sqrt 2 \right]^{1/2}.
\end{equation}
Furthermore, if $0 \leq \xi \lesssim 1, \xi \neq 0$, we have
\begin{equation}
\begin{split}
\label{boundphase-lem}
& |\partial_\xi \Phi_-(t,\xi,r)| \sim 
\begin{cases}
t & \mbox{if $r \ll t$} \\
r & \mbox{if $r \gg t$} \\
t (\xi^2 - \xi_0^2)  & \mbox{if $r \sim t$}.
\end{cases}\\
& | \partial_\xi^2 \Phi_-(t,\xi,r)| \sim t \xi.
\end{split}
\end{equation}

\medskip

\noindent \underline{Case 1: $r \gg t$ or $t \gg r$}. We claim that in this case, under the weaker hypothesis that \eqref{estimateWa2-lem} is valid for $j=0,1$ only, then
\begin{equation} \label{estimation-model-integral-case1}
|\mathcal{I}_{model} |\lesssim \frac{1}{t}  \| \xi \widetilde{\phi} \|_{H^{1+\epsilon}(|\xi|\dd \xi)} .
\end{equation}
Indeed, after an integration by parts,
\begin{align*}
\mathcal{I}_{model} & = - \int e^{i \Phi_-(t,\xi,r)} \partial_\xi (\xi \widetilde{\phi}(\xi)) W(\xi,r)  \frac{1}{i \partial_\xi \Phi_-(t,\xi,r)} \dd \xi \\
& \qquad -  \int e^{i \Phi_-(t,\xi,r)}\xi \widetilde{\phi}(\xi)  \partial_\xi \left[ W(\xi,r)  \right] \frac{1}{i \partial_\xi \Phi_-(t,\xi,r)} \dd \xi \\
& \qquad -  \int e^{i \Phi_-(t,\xi,r)}\xi \widetilde{\phi}(\xi)  W(\xi,r) \partial_\xi \left[ \frac{1}{i \partial_\xi \Phi_-(t,\xi,r)} \right] \dd \xi \\
& = I + II + III.
\end{align*}

Each of these three terms can now be bounded with the help of the bounds \eqref{estimateWa2-lem} on $W$, and the lower bound $|\partial_\xi \Phi_- | \gtrsim t$ valid if $|\xi| \lesssim 1$ and $r \gg t$ or $t \gg r$:
\begin{align*}
& |I| \lesssim \left\| \partial_\xi (\xi \widetilde{\phi}(\xi))\right\|_{L^{2+\alpha}(|\xi|\dd \xi)} \| \frac{1}{|\xi|^{1/(2+\alpha)}}\frac 1t \frac{1}{\langle \xi r \rangle^{1/2}}\|_{L^{\frac{2+\alpha}{1+\alpha}}(|\xi|\lesssim 1, \ \dd \xi)} \lesssim \frac{1}{t} \| \xi \tilde \phi\|_{H^{1+\epsilon}(|\xi|\dd \xi)} \\
& |II| \lesssim \| \xi \widetilde{\phi} \|_{L^\infty} \int_{|\xi|\lesssim 1} \frac 1t \frac{\langle r \rangle}{\langle \xi r \rangle^{3/2}} \dd \xi \lesssim \frac 1t \| \xi \widetilde{\phi} \|_{L^\infty}  \\
& |III| \lesssim \| \xi \widetilde{\phi} \|_{L^\infty} \int_{|\xi| \lesssim 1} \frac{t\xi}{t^2} \frac{1}{\langle \xi r \rangle^{1/2}} \dd \xi \lesssim \frac{1}{t} \| \xi \widetilde{\phi} \|_{L^\infty}
\end{align*}
where for the first inequality we used H\"older with the fact that $\| \partial_\xi (\xi \tilde \phi)\|_{L^{2+\alpha}(|\xi|\dd \xi)}\leq \| \xi \tilde \phi\|_{H^{1+\epsilon}(|\xi|\dd \xi)}$ for some $\alpha>0$ by Sobolev embedding.

\medskip

\noindent \underline{Case 2: $r \sim t$ and $\xi_0 \lesssim t^{-1/3}$.} We claim that in this case, under the weaker hypothesis that \eqref{estimateWa2-lem} is valid for $j=0,1$ only, then

\begin{equation} \label{estimation-model-integral-case2}
|\mathcal{I}_{model} |\lesssim \min \left( \frac{1}{t}\left( \| \widetilde{\phi} \|_{L^\infty} + \| \xi \widetilde{\phi} \|_{W^{1,\infty}} \right) \ , \ \frac{1}{t^{2/3}}\left( \| \xi \widetilde{\phi} \|_{L^\infty} + \| \xi \widetilde{\phi} \|_{H^{1+\epsilon}} \right) \right)
\end{equation}
To treat this case, we split the integral defining $\mathcal{I}_{model}$ as follows
$$
\mathcal{I}_{model} = \int \left[ \chi_-(c_0 t^{1/3} \xi) + \chi_0(c_0 t^{1/3} \xi) + \chi_+(c_0 t^{1/3} \xi) \right] \dots \dd \xi = I + II + III,
$$
where the cut-off functions $\chi_0$, $\chi_-$ and $\chi_+$ are three smooth nonnegative functions which add up to one, and are supported on $[-2,2]$, $(-\infty,-1]$ and $[1,\infty)$ respectively, and the constant $c_0$ will be chosen to be sufficiently small for the arguments which follow to apply.

The middle term is the simplest to treat: by the bound \eqref{estimateWa2-lem} on $W$,
$$
|II| \lesssim \| \widetilde{\phi} \|_{L^\infty} \int_{-C t^{-1/3}}^{C t^{-1/3}} |\xi| \frac{\dd \xi}{\langle r \xi \rangle^{1/2}} \lesssim \frac 1t \| \widetilde{\phi} \|_{L^\infty} \quad \mbox{or} \quad |II| \lesssim \| \xi \widetilde{\phi} \|_{L^\infty} \int_{-C t^{-1/3}}^{C t^{-1/3}} \frac{\dd \xi}{\langle r \xi \rangle^{1/2}} \lesssim \frac{1}{t^{2/3}} \| \xi \widetilde{\phi} \|_{L^\infty} .
$$

Turning to $III$, it can be written after an integration by parts
\begin{align*}
III & = - \int e^{i \Phi_-(t,\xi,r)} \partial_\xi (\xi \widetilde{\phi}(\xi))  \chi_+(c_0 t^{1/3} \xi) W(\xi,r) \frac{1}{i \partial_\xi \Phi_-(t,\xi,r)} \dd \xi \\
& \qquad -  \int e^{i \Phi_-(t,\xi,r)} \xi \widetilde{\phi}(\xi)  \chi_+(c_0 t^{1/3} \xi) \partial_\xi W(\xi,r)  \frac{1}{i \partial_\xi \Phi_-(t,\xi,r)} \dd \xi \\
& \qquad -  \int e^{i \Phi_-(t,\xi,r)} \xi \widetilde{\phi}(\xi)  \chi_+(c_0 t^{1/3} \xi) W(\xi,r)  \partial_\xi \left[ \frac{1}{i \partial_\xi \Phi_-(t,\xi,r)} \right] \dd \xi \\
&\qquad - \int e^{i \Phi_-(t,\xi,r)} \xi \widetilde{\phi}(\xi)\partial_\xi [  \chi_+(c_0 t^{1/3} \xi)] W(\xi,r) \frac{1}{i \partial_\xi \Phi_-(t,\xi,r)} \dd \xi \\
& = III_1 + III_2 + III_3+III_4.
\end{align*}
We now choose $c_0$ such that $\xi > 5\xi_0$ on the support of the above integrand, which implies that $|\partial_\xi \Phi_-| \gtrsim t \xi^2$. Using this lower bound in conjunction with the estimate \eqref{estimateWa2-lem} on $W$, and keeping in mind that $r \sim t$, for $\alpha \in (0,\infty]$ we have
\begin{align*}
& |III_1| \lesssim  \|  \partial_\xi (\xi \widetilde{\phi}(\xi)) \|_{L^{2+\alpha}(|\xi|\dd \xi)} \left\| \chi_+\left(c_0t^{1/3}\xi \right) \frac{1}{\langle r \xi \rangle^{\frac 12}}\frac{1}{t\xi^2}\right\|_{L^{\frac{2+\alpha}{1+\alpha}}}\lesssim \frac{1}{t^{\frac{2}{3}+\frac{\alpha}{2+\alpha}}} \|  \partial_\xi (\xi \widetilde{\phi}(\xi)) \|_{L^{2+\alpha}(|\xi|\dd \xi)}.
\end{align*}
Hence, taking either $\alpha = +\infty$ or $\alpha >0$ small (depending on $\varepsilon$), we have
$$
 |III_1|\lesssim \min \left(\frac{1}{t} \|  \partial_\xi (\xi \widetilde{\phi}(\xi)) \|_{L^{\infty}(|\xi|\dd \xi)} \ , \frac{1}{t^{2/3}} \|  \xi \widetilde{\phi}(\xi) \|_{H^{1+\epsilon}(|\xi|\dd \xi)}\right),
$$
where we used Sobolev embedding for the second inequality. Similarly
\begin{align*}
& |III_2| \lesssim \| \widetilde{\phi} \|_{L^\infty} \int_{\xi \gtrsim t^{-1/3}} \frac 1 {t \xi^2} \xi \frac{\langle r \rangle}{\langle \xi r \rangle^{3/2}} \dd \xi \lesssim \frac{1}{t}  \|  \widetilde{\phi} \|_{L^\infty} \quad \mbox{or} \quad  |III_2|\lesssim \frac{1}{t^{2/3}}\| \xi \widetilde{\phi} \|_{L^\infty} \\
& |III_3| \lesssim  \| \widetilde{\phi} \|_{L^\infty} \int_{\xi \gtrsim t^{-1/3}} \frac{t\xi}{(t\xi^2)^2} \frac{\xi}{\langle \xi r \rangle^{1/2}} \dd \xi \lesssim \frac{1}{t}  \|  \widetilde{\phi} \|_{L^\infty}\quad \mbox{or} \quad  |III_3|\lesssim \frac{1}{t^{2/3}}\| \xi \widetilde{\phi} \|_{L^\infty}\\
& |III_4| \lesssim  \| \widetilde{\phi} \|_{L^\infty} \int_{\xi \approx t^{-1/3}} \frac{t^{1/3}}{ t\xi^2} \frac{\xi}{\langle \xi r \rangle^{1/2}} \dd \xi \lesssim \frac{1}{t}  \|  \widetilde{\phi} \|_{L^\infty}\quad \mbox{or} \quad  |III_3|\lesssim \frac{1}{t^{2/3}}\| \xi \widetilde{\phi} \|_{L^\infty}.
\end{align*}
Finally, $I$ can be estimated in a similar way to $III$.

\medskip

\noindent \underline{Case 3: $r \sim t$ and $\xi_0 \gg t^{-1/3}$.}
We introduce the new scale $R = (t \xi_0)^{-1/2}$; notice that $R \ll \xi_0$ since $\xi_0 \gg t^{-1/3}$. The integral $\mathcal{I}_{model}$ is now split as follows
$$
\mathcal{I}_{model} = \int \left[ \chi_-\left( \frac{\xi - \xi_0}{R} \right) + \chi_0\left( \frac{\xi - \xi_0}{R} \right) + \chi_+\left( \frac{\xi - \xi_0}{R} \right)\right] \dots \dd \xi = I + II + III,
$$
The middle term can be bounded immediately:
$$
|II| \lesssim  \| \widetilde{\phi} \|_{L^\infty} \frac{\xi_0 R}{\langle \xi_0 r \rangle^{1/2}} \lesssim \frac 1t  \| \widetilde{\phi} \|_{L^\infty}, 
$$
but the right term requires an integration by parts
\begin{align*}
III & = - \int e^{i \Phi_-(t,\xi,r)} \partial_\xi \left[ \xi \widetilde{\phi}(\xi)\right] \chi_+\left( \frac{\xi - \xi_0}{R} \right) W(\xi,r) \frac{1}{i \partial_\xi \Phi_-(t,\xi,r)} \dd \xi \\
&\qquad - \int e^{i \Phi_-(t,\xi,r)} \xi \widetilde{\phi}(\xi) \frac{1}{R} \chi_+'\left( \frac{\xi - \xi_0}{R} \right) W(\xi,r) \frac{1}{i \partial_\xi \Phi_-(t,\xi,r)} \dd \xi \\
& \qquad -  \int e^{i \Phi_-(t,\xi,r)} \xi \widetilde{\phi}(\xi)  \chi_+\left( \frac{\xi - \xi_0}{R} \right) \partial_\xi W(\xi,r)  \frac{1}{i \partial_\xi \Phi_-(t,\xi,r)} \dd \xi \\
& \qquad -  \int e^{i \Phi_-(t,\xi,r)} \xi \widetilde{\phi}(\xi) \chi_+\left( \frac{\xi - \xi_0}{R} \right) W(\xi,r)  \partial_\xi \left[ \frac{1}{i \partial_\xi \Phi_-(t,\xi,r)} \right] \dd \xi \\
& = III_1 + III_2+ III_3 + III_4.
\end{align*}
To estimate $III_1$, we do another integration by parts:
\begin{align*}
III_1 & = - \int e^{i \Phi_-(t,\xi,r)} \partial_\xi \left[ \partial_\xi (\xi \widetilde{\phi}(\xi)) \chi_+\left( \frac{\xi - \xi_0}{R} \right) \right] W(\xi,r) \frac{1}{(\partial_\xi \Phi_-(t,\xi,r))^2} \dd \xi \\
& \qquad -  \int e^{i \Phi_-(t,\xi,r)} \partial_\xi\left[\xi \widetilde{\phi}(\xi) \right] \chi_+\left( \frac{\xi - \xi_0}{R} \right) \partial_\xi W(\xi,r)  \frac{1}{(\partial_\xi \Phi_-(t,\xi,r))^2} \dd \xi \\
& \qquad -  \int e^{i \Phi_-(t,\xi,r)}  \partial_\xi\left[ \xi \widetilde{\phi}(\xi)\right] \chi_+\left( \frac{\xi - \xi_0}{R} \right) W(\xi,r)  \partial_\xi \left[ \frac{1}{(\partial_\xi \Phi_-(t,\xi,r))^2} \right] \dd \xi \\
& = III_{1,1} + III_{1,2} + III_{1,3}.
\end{align*}
First of all, distributing the derivatives in $III_{1,1}$, we get
\begin{align*}
|III_{1,1}| & \lesssim \| \tilde \phi \|_{W^{2,\infty}} \int \left( \chi_+\big( \frac{\xi - \xi_0}{R}\big)+\frac{1}{R}|\chi_+'\big( \frac{\xi - \xi_0}{R}\big)| \right) \frac{\dd\xi}{\langle r \xi\rangle^{\frac 12}t^2(\xi^2-\xi_0^2)^2} \\
&\lesssim  \| \tilde \phi \|_{W^{2,\infty}}\left(\frac{1}{t^2r^{\frac 12}\xi_0^{\frac 52}R}+\frac{1}{t^2r^{\frac 12}\xi_0^{\frac 52}R^2}\right) \ \lesssim \frac 1t \| \tilde \phi \|_{W^{2,\infty}}.
\end{align*}
Similarly,
\begin{align*}
|III_{1,2}| & \lesssim \| \tilde \phi \|_{W^{1,\infty}} \int  \chi_+\big( \frac{\xi - \xi_0}{R}\big)  \frac{r \dd\xi}{\langle r \xi\rangle^{\frac 32}t^2(\xi^2-\xi_0^2)^2} \lesssim \| \tilde \phi \|_{W^{1,\infty}} \frac{1}{t^2r^{\frac 12}\xi_0^{\frac 72}R} \ \lesssim \frac 1t \| \tilde \phi \|_{W^{1,\infty}}
\end{align*}
and
\begin{align*}
|III_{1,3}| & \lesssim \| \tilde \phi \|_{W^{1,\infty}} \int \chi_+\big( \frac{\xi - \xi_0}{R}\big) \frac{|\xi| \dd\xi}{\langle r \xi\rangle^{\frac 12}t^2(\xi^2-\xi_0^2)^3} \lesssim  \| \tilde \phi \|_{W^{1,\infty}} \frac{1}{t^2r^{\frac 12}\xi_0^{\frac 52}R^2} \ \lesssim \frac 1t \| \tilde \phi \|_{W^{1,\infty}}
\end{align*}
We now claim that, under the weaker hypothesis that \eqref{estimateWa2-lem} is valid for $j=0,1$ only, then
\begin{equation} \label{estimation-model-integral-case3}
|III_2|+|III_3|+|III_4|\lesssim \frac{1}{t}\| \widetilde{\phi} \|_{L^\infty}.
\end{equation}
Let us now prove \eqref{estimation-model-integral-case3}. The first term is estimated as above,
\begin{align*}
|III_{2}| & \lesssim \| \tilde \phi \|_{L^\infty} \int \frac{1}{R} \chi_+' \big( \frac{\xi - \xi_0}{R}\big) \frac{|\xi| \dd\xi}{\langle r \xi\rangle^{\frac 12}t (\xi^2-\xi_0^2)} \lesssim  \| \tilde \phi \|_{L^{\infty}} \frac{1}{t r^{\frac 12}\xi_0^{\frac 12}R} \ \lesssim \frac 1t \| \tilde \phi \|_{L^{\infty}}.
\end{align*}

The third term $III_3$ is the most delicate: it can be bounded by
$$
|III_3| \lesssim \frac{1}{t} \|  \widetilde{\phi} \|_{L^\infty} \int_{\xi > \xi_0 + CR} \frac{\xi}{\xi^2 - \xi_0^2} \frac{\langle r \rangle}{\langle \xi r \rangle^{3/2}} \dd \xi \lesssim  \frac{1}{t} \| \widetilde{\phi} \|_{L^\infty}
$$
where we estimated the above integral as follows
\begin{align*}
\int_{\xi > \xi_0 + CR} \frac{\xi}{\xi^2 - \xi_0^2} \frac{\langle r \rangle}{\langle \xi r \rangle^{3/2}} \dd \xi 
& \lesssim \int_{\xi_0 + CR}^{2 \xi_0} \frac{1}{\xi_0} \frac{1}{\xi-\xi_0} \xi_0 \frac{1}{\xi_0^{3/2} r^{1/2}} \dd \xi  + \int_{2 \xi_0}^C \frac{1}{\xi^2} \xi \frac{1}{\xi^{3/2} r^{1/2}} \dd\xi \\
& \lesssim \xi_0^{-3/2} r^{-1/2} \langle \log(\frac R {\xi_0} ) \rangle + \xi_0^{-3/2} r^{-1/2} \\
& \lesssim \xi_0^{-3/2} r^{-1/2} \langle \log(\xi_0^{-3/2} r^{-1/2} ) \rangle + \xi_0^{-3/2} r^{-1/2} \lesssim 1,
\end{align*}
where we used that $\xi_0^{-3/2} r^{-1/2} \sim \xi_0^{-3/2} t^{-1/2} \lesssim 1$. Finally,
$$
|III_4| \lesssim \|  \widetilde{\phi} \|_{L^\infty} \int_{\xi > \xi_0 + CR} \frac{\xi}{t(\xi^2 - \xi_0^2)^2} \frac{\xi}{\langle r\xi \rangle^{1/2}} \dd \xi \lesssim \frac{1}{t} \|  \widetilde{\phi} \|_{L^\infty},
$$
after estimating the integral upon splitting it into the contributions of $\xi_0 + CR < \xi < 2 \xi_0$ and $\xi > 2\xi_0$.

There remains to estimate $I$. We proceed just like for $III$ and integrate by parts (noticing that there is no boundary term since $(\xi \widetilde{\phi})(0)=0$). We obtain terms $I_1$, $I_2$, $I_3$, $I_4$ which only differ from $III_1$, $III_2$, $III_3$, $III_4$ in that $\chi_+$ is replaced by $\chi_-$. The last three terms can then be estimated in a similar way:
\begin{align*}
|I_{2}| & \lesssim \| \tilde \phi \|_{L^\infty} \int \frac{1}{R} \chi_-' \big( \frac{\xi - \xi_0}{R}\big) \frac{|\xi| \dd\xi}{\langle r \xi\rangle^{\frac 12}t (\xi^2-\xi_0^2)} \lesssim  \| \tilde \phi \|_{L^{\infty}} \frac{1}{t r^{\frac 12}\xi_0^{\frac 12}R} \ \lesssim \frac 1t \| \tilde \phi \|_{L^{\infty}}, \\
|I_3|& \lesssim \frac{1}{t} \| \widetilde{\phi} \|_{L^\infty} \int_0^{\xi_0-CR} \frac{\xi}{\xi_0^2 - \xi^2} \frac{1}{\xi^{3/2} r^{1/2}} \dd\xi \\
& \lesssim \frac{1}{t} \left[  \xi_0^{-3/2} r^{-1/2} \langle \log(\xi_0^{-3/2} r^{-1/2} ) \rangle + \xi_0^{-3/2} r^{-1/2}  \right] \|  \widetilde{\phi} \|_{L^\infty} \lesssim \frac 1t  \| \widetilde{\phi} \|_{L^\infty},\\
 |I_4| & \lesssim \|  \widetilde{\phi} \|_{L^\infty} \int_{0}^{ \xi_0 - CR} \frac{\xi}{t(\xi^2 - \xi_0^2)^2} \frac{\xi}{\langle r\xi \rangle^{1/2}} \dd \xi \lesssim \frac{1}{t} \|  \widetilde{\phi} \|_{L^\infty},
\end{align*}
and there only remains to estimate the first term that we decompose as follows in order to avoid a boundary term in an integration by parts later
\begin{align*}
I_1 & = - \int_0^\infty e^{i \Phi_-(t,\xi,r)} \partial_\xi \left[ \xi \widetilde{\phi}(\xi) \right] \chi_0\left(\frac{\xi}{c_0t^{-1/3}}\right) W(\xi,r) \frac{1}{i \partial_\xi \Phi_-(t,\xi,r)} \dd \xi \\
&\qquad - \int_0^\infty e^{i \Phi_-(t,\xi,r)} \partial_\xi \left[ \xi \widetilde{\phi}(\xi) \right] \chi_+\left(\frac{\xi}{c_0t^{-1/3}}\right) \chi_-\left( \frac{\xi - \xi_0}{R} \right) W(\xi,r) \frac{1}{i \partial_\xi \Phi_-(t,\xi,r)} \dd \xi \\
& = I_1^1+I_1^2
\end{align*}
for $c_0$ a small constant. The first term is estimated directly:
$$
|I_1^1|\lesssim \| \xi \widetilde{\phi}\|_{W^{1,\infty}} \int_0^{2c_0t^{-1/3}} \frac{1}{r^{1/2}\xi^{1/2}} \frac{1}{t(\xi_0^2-\xi^2)} \dd \xi \lesssim \| \xi \widetilde{\phi}\|_{W^{1,\infty}} \frac{1}{t^{7/6}r^{1/2}\xi_0^2} \lesssim \frac{1}{t} \| \xi \widetilde{\phi}\|_{W^{1,\infty}} 
$$
since $\xi_0\gg t^{-1/3}$. We integrate by parts for the second term and obtain terms $I_{1,1}^2$, $I_{1,2}^2$ and $I_{1,3}^2$ which are the analogues of $III_{1,1}$, $III_{1,2}$ and $III_{1,3}$ before up to replacing $ \chi_+\left( \frac{\xi - \xi_0}{R} \right) $ by $\chi_+\left(\frac{\xi}{c_0t^{-1/3}}\right) \chi_-\left( \frac{\xi - \xi_0}{R} \right) $. They are estimated exactly as before, namely:
\begin{align*}
|I_{1,1}^2| & \lesssim \| \tilde \phi \|_{W^{2,\infty}} \int \left( \chi_-\big( \frac{\xi - \xi_0}{R}\big)+\frac{1}{R}|\chi_-'\big( \frac{\xi - \xi_0}{R}\big)+t^{1/3}|\chi_0'\big( \frac{\xi}{c_0t^{1/3}}\big)| \right) \frac{\dd\xi}{\langle r \xi\rangle^{\frac 12}t^2(\xi^2-\xi_0^2)^2} \\
&\lesssim  \| \tilde \phi \|_{W^{2,\infty}}\left(\frac{1}{t^2r^{\frac 12}\xi_0^{\frac 52}R}+\frac{1}{t^2r^{\frac 12}\xi_0^{\frac 52}R^2}+\frac{1}{t^{7/3}\xi_0^4}\right) \ \lesssim \frac 1t \| \tilde \phi \|_{W^{2,\infty}} \\
|I_{1,2}^2| & \lesssim \| \tilde \phi \|_{W^{1,\infty}} \int  \chi_-\big( \frac{\xi - \xi_0}{R}\big)  \frac{r \dd\xi}{\langle r \xi\rangle^{\frac 32}t^2(\xi^2-\xi_0^2)^2} \lesssim \| \tilde \phi \|_{W^{1,\infty}} \frac{1}{t^2r^{\frac 12}\xi_0^{\frac 72}R} \ \lesssim \frac 1t \| \tilde \phi \|_{W^{1,\infty}},\\
|I_{1,3}^2| & \lesssim \| \tilde \phi \|_{W^{1,\infty}} \int \chi_-\big( \frac{\xi - \xi_0}{R}\big) \frac{|\xi| \dd\xi}{\langle r \xi\rangle^{\frac 12}t^2(\xi^2-\xi_0^2)^3} \lesssim  \| \tilde \phi \|_{W^{1,\infty}} \frac{1}{t^2r^{\frac 12}\xi_0^{\frac 52}R^2} \ \lesssim \frac 1t \| \tilde \phi \|_{W^{1,\infty}}.
\end{align*}
This concludes the proof of \eqref{dispersive1-lem}.
\end{proof}

\subsection{Pointwise decay: the main result}
\begin{thm}[Pointwise decay] 
\label{thmpointwisedecay}
Let $\epsilon>0$ and $\phi \in \mathcal S_1$. If $0<t\leqslant 1$, we have
$$
 \| e^{it \mathcal{H}} \phi \|_{L^\infty} \lesssim \frac{1}{t} \left\| \widetilde{\phi} \right\|_{H^{1+\epsilon}(|\xi|\dd \xi)}.
$$
If $t>1,$ then
\begin{equation}
\| e^{it \mathcal{H}} \phi \|_{L^\infty} \lesssim_\epsilon \frac{1}{t^{2/3}}  \left\|  \widetilde{\phi}(\xi) \right\|_{H^{1+\epsilon}(|\xi|\dd \xi)},  \label{dispersive2}
\end{equation}
and if $\langle \phi , (\rho,\rho)^\top \rangle_{L^2(r \dd r)} =0$,
\begin{equation}
\label{dispersive1}
 \| e^{it \mathcal{H}} \phi \|_{L^\infty}  \lesssim_\epsilon \frac{1}{t}  \left\|  \widetilde{\phi}(\xi) \right\|_{H^{2+\epsilon}(|\xi|\dd \xi)} .
\end{equation}
\end{thm}

\begin{proof} 
We start from the formula
$$
e^{it \mathcal{H}} \phi = \frac{1}{\pi} \int_{-\infty}^\infty e^{it \lambda(\xi)}   \widetilde{\phi} (\xi) \psi(\xi,r) \lambda'(\xi) \operatorname{sign} \xi \dd \xi,
$$
which we split into low and high frequencies (the cutoff functions 
$\chi_0$, $\chi_-$ and $\chi_+$ being three smooth nonnegative functions which add up to one, and are supported on $[-2,2]$, $(-\infty,-1]$ and $[1,\infty)$ respectively)
$$
e^{it \mathcal{H}} \phi = \frac{1}{\pi} \int_{-\infty}^\infty [ \chi_{0} (\xi) + \chi_{+}(\xi) + \chi_{-}(\xi) ] \dots \dd \xi = I_{lo} + I_{hi,+} + I_{hi,-}.
$$

\medskip

\noindent
\underline{Case 1: high frequencies $\xi \gtrsim 1$.} We claim that for $t>0$, we have
$$
|I_{hi,-}|+|I_{hi,+}|\lesssim \frac{1}{t} \| \widetilde{\phi}\|_{H^{1+\epsilon}(|\xi|\dd \xi)},
$$
and now prove this estimate. Since the terms $I_{hi,+}$ and $I_{hi,-}$ can be treated identically, we focus on the former, which corresponds to $\xi \gtrsim 1$. 
By Theorem \ref{toohigh} and Appendix \ref{truebesselclassics}, it can be written
\begin{equation}
\label{constable}
I_{hi,+} = \sum_{\pm} \int_{-\infty}^\infty e^{i \Phi_{\pm}(t,\xi,r)}   \widetilde{\phi} (\xi) W_{\pm}(\xi,r) a(\xi) \dd \xi 
\end{equation}
(note that we do not need to distinguish the remainder term in Theorem \ref{toohigh}, it can be merged with the main term). Here, we abbreviated
$$
a(\xi) = \lambda'(\xi) \operatorname{sign}(\xi) \chi_{+}(\xi), \qquad \Phi_{\pm}(t,\xi,r) = \lambda(\xi) t \pm \xi r,
$$
and the functions $W_{\pm}$ and $a$ satisfy the estimates for $|\xi| \gtrsim 1$
\begin{equation}
\label{estimateWa}
\begin{split}
& |W_{\pm}(\xi,r) | \lesssim \frac{1}{\langle r \xi \rangle^{\frac{1}{2}}}, 
\qquad |\partial_\xi W_{\pm}(\xi,r) | \lesssim \frac{1}{\xi \langle r\xi \rangle^{1/2}}  \\
& |a(\xi)| \lesssim |\xi|, \quad |a'(\xi)| \lesssim 1 .
\end{split}
\end{equation}

Since the phase $\Phi_+$ is never stationary for $t,\xi \gtrsim 1$, we shall focus on the phase $\Phi_-$. Its derivatives and its stationary point $\xi_0$ are given by
\begin{align*}
& \partial_\xi \Phi_-(t,\xi,r) = t \lambda'(\xi) - r, \qquad \partial_\xi^2 \Phi_-(t,\xi,r) = t \lambda''(\xi) \\
& \partial_\xi \Phi_-(t,\xi_0,r) = 0  \quad \Longleftrightarrow \quad \lambda'(\xi_0) = \frac rt \quad \Longrightarrow \quad \xi_0 \sim \frac rt \quad \mbox{if $|\xi_0| \gtrsim 1$}.
\end{align*}
Furthermore, if $\xi_0 \gtrsim 1$ or $\xi \gtrsim 1$ and $\xi_0 \ll 1$,
\begin{equation}
\label{boundphi}
\begin{split}
| \partial_\xi \Phi_-(t,\xi,r) | \sim t |\xi-\xi_0|, \qquad  |\partial_\xi^2 \Phi_-(t,\xi,r) | \sim t .
\end{split}
\end{equation}

We now come back to formula \eqref{constable} and, as explained above, we only keep the summand involving the $-$ phase: we denote
$$
\mathcal{I}_{hi,+} =  \int_{-\infty}^\infty e^{i \Phi_{-}(t,\xi,r)}   \widetilde{\phi} (\xi) W_{-}(\xi,r) a(\xi) \dd \xi.
$$

\medskip

\noindent \underline{Case 1.1: $\xi \gtrsim 1$ and $r \ll t$.} Integrating by parts gives the following terms
\begin{align*}
\mathcal{I}_{hi,+} & = \int_{-\infty}^\infty e^{i \Phi_{-}(t,\xi,r)}   \widetilde{\phi} (\xi) W_{-}(\xi,r) a(\xi) \dd \xi \\
&  = - \int_{-\infty}^\infty e^{i \Phi_{-}(t,\xi,r)} \partial_\xi \left( \frac{W_{-}(\xi,r) a(\xi) }{i \partial_\xi \Phi_-(t,\xi,r)} \right) \widetilde{\phi} (\xi) \dd \xi
- \int_{-\infty}^\infty e^{i \Phi_{-}(t,\xi,r)} \frac{W_{-}(\xi,r) a(\xi) }{i \partial_\xi \Phi_-(t,\xi,r)}  \partial_\xi   \widetilde{\phi} (\xi) \dd \xi.
\end{align*}
Since $\xi \gtrsim 1$ and $r \ll t$, we see that $|\partial_\xi \Phi_- |\sim t \xi$, so that, by the bounds \eqref{estimateWa} on $W_-$ and $a$,
$$
\left| \frac{W_{-}(\xi,r) a(\xi) }{\partial_\xi \Phi_{-}(t,\xi,r)} \right| \lesssim \frac{1}{ t\xi \langle r \xi \rangle^{1/2}}, \qquad \left| \partial_\xi \left( \frac{W_{-}(\xi,r) a(\xi) }{\partial_\xi \Phi_{-}(t,\xi,r)} \right) \right| \lesssim \frac{1}{ t  \xi \langle r \xi\rangle^{1/2}}.
$$

Therefore, we can bound 
$$
\left| I_{hi,+} \right| 
\lesssim_\epsilon \frac{1}{t} \left[ \| \widetilde{\phi} \|_{L^2(|\xi|\dd\xi)} + \|  \partial_\xi \widetilde{\phi} \|_{L^2(|\xi|\dd \xi)} \right]
$$

\medskip

\noindent \underline{Case 1.2: $\xi \gtrsim 1$ and $r \gtrsim t$}. In that case, we need to split the integral further in order to isolate the stationary point. Recall that $\chi_0$, $\chi_-$ and $\chi_+$ be three smooth nonnegative cutoff functions which add up to one, and are supported on $[-2,2]$, $(-\infty,-1]$ and $[1,\infty)$ respectively. We write then
\begin{align*}
\int_{-\infty}^\infty e^{i \Phi_-(t,\xi,r)}   \widetilde{\phi} (\xi) W_{-}(\xi,r) a(\xi) \dd \xi & = \int \left[ \chi_-(\sqrt t(\xi-\xi_0)) + \chi_0(\sqrt t(\xi-\xi_0)) + \chi_+(\sqrt t(\xi-\xi_0)) \right] \dots \dd \xi \\
& = I + II + III.
\end{align*}
\medskip
The middle piece $II$ is estimated in a straightforward manner: by the bounds \eqref{estimateWa} on $a$ and $W_-$ and since $\xi_0 \sim \frac rt$,
$$
|II| \lesssim \frac{1}{\sqrt{t}} \left\| \widetilde{\phi} (\xi) W_{-}(\xi,r) a(\xi) \right\|_{L^\infty} \lesssim \frac{1}{\sqrt{t}} \frac{\xi_0}{\langle r \xi_0 \rangle^{1/2}} \| \widetilde{\phi} \|_{L^\infty}
\lesssim \frac{1}{t} \| \widetilde{\phi} \|_{L^\infty}.
$$
To deal with $I$, we integrate by parts and get
\begin{align*}
I  & = - \int_{-\infty}^\infty \frac{ e^{i \Phi_-(t,\xi,r)} }{i \partial_\xi \Phi_-(t,\xi,r)} \partial_\xi \widetilde{\phi} (\xi) W_{-}(\xi,r) a(\xi) \chi_-(\sqrt t(\xi-\xi_0)) \dd \xi \\
& \qquad - \int_{-\infty}^\infty \frac{ e^{i \Phi_-(t,\xi,r)} }{i \partial_\xi \Phi_-(t,\xi,r)} \widetilde{\phi} (\xi) \partial_\xi [ W_{-}(\xi,r) a(\xi)] \chi_-(\sqrt t(\xi-\xi_0)) \dd \xi \\
& \qquad - \int_{-\infty}^\infty e^{i \Phi_-(t,\xi,r)} \widetilde{\phi} (\xi) W_{-}(\xi,r) a(\xi) \partial_\xi \left[ \frac{\chi_-(\sqrt t(\xi-\xi_0))}{i \partial_\xi \Phi_-(t,\xi,r)} \right] \dd \xi \\
& = I_1 + I_2 + I_3.
\end{align*}
Using successively the bounds \eqref{estimateWa} and \eqref{boundphi} and the fact that $\| \frac{1}{\langle r \xi\rangle^{1/2} |\xi - \xi_0|}\|_{L^2(1 \lesssim \xi \leq \xi_0 - \frac{C}{\sqrt t}, \ \xi\dd \xi)} \lesssim 1$, we have by Cauchy-Schwarz
\begin{align*}
| I_1 | \lesssim \int_{ 1 \lesssim \xi \leq \xi_0 - \frac{C}{\sqrt t}} \frac{1}{t|\xi - \xi_0|} |\partial_\xi \widetilde{\phi}| \frac{\xi}{\langle r \xi\rangle^{1/2}} \dd \xi \lesssim \frac 1 t \left\| \partial_\xi \widetilde{\phi} \right\|_{L^2(|\xi|\dd \xi)}.
\end{align*}
The term $I_2$ is lower order: indeed,
\begin{align*}
|I_2| & \lesssim \| \widetilde{\phi} \|_{L^\infty} \int_{ 1 \lesssim \xi \leq \xi_0 - \frac{C}{\sqrt t}} \frac{1}{t|\xi - \xi_0|} \frac{1}{\langle r \xi\rangle^{1/2} } \dd \xi \lesssim \frac{\| \widetilde{\phi} \|_{L^\infty}}{t \langle r \rangle^{1/2}} \int_{ 1 \lesssim \xi \leq \xi_0 - \frac{C}{\sqrt t}}\frac{1}{|\xi - \xi_0|} \dd \xi \\
& \lesssim \frac{\langle \log t \rangle + \langle \log r \rangle}{t \langle r \rangle^{1/2}} \| \widetilde{\phi} \|_{L^\infty} \lesssim \frac 1 t \| \widetilde{\phi} \|_{L^\infty}.
\end{align*}
We finally come to $I_3$, which can be bounded as follows
$$
|I_3| \lesssim  \| \widetilde{\phi} \|_{L^\infty} \int_{ 1 \lesssim \xi \leq \xi_0 - \frac{C}{\sqrt t}} \frac{\xi}{\langle r \xi \rangle^{1/2}} \frac{1}{t |\xi - \xi_0|^2} \dd \xi \lesssim \frac{\| \widetilde{\phi} \|_{L^\infty}}{t^{3/2}} \int_{ 1 \lesssim \xi \leq \xi_0 - \frac{C}{\sqrt t}} \frac{1}{ |\xi - \xi_0|^2} \dd \xi \lesssim \frac{1}{t} \|  \widetilde{\phi} \|_{L^\infty}.
$$
Similarly, $III$ can be written as $III_1 + III_2 + III_3$ (the only difference in the formulas being the replacement of $\chi_-$ by $\chi_+$); these terms can then be estimated similarly to $I_1$, $I_2$ and $I_3$. For the term $III_1$, we use that
$\| \frac{1}{\langle r \xi\rangle^{1/2} |\xi - \xi_0|}\|_{L^2( \xi \geq \xi_0 + \frac{C}{\sqrt t},\ |\xi|\dd \xi)} \lesssim 1$ to obtain that
\begin{align*}
| III_1 | \lesssim \int_{ \xi \geq \xi_0 + \frac{C}{\sqrt t}} \frac{1}{t|\xi - \xi_0|} |\partial_\xi \widetilde{\phi}| \frac{\xi}{\langle \xi r \rangle^{1/2}} \dd \xi \lesssim \frac 1 t \left\| \partial_\xi \widetilde{\phi} \right\|_{L^2(|\xi|\dd \xi)}.
\end{align*}
The term $|III_2|$ is lower order
\begin{align*}
|III_2| & \lesssim \| \widetilde{\phi} \|_{L^\infty} \int_{  \xi \geq \xi_0 + \frac{C}{\sqrt t}} \frac{1}{t|\xi - \xi_0|} \frac{1}{\langle r \xi \rangle^{1/2}} \dd \xi \lesssim \frac{\langle \log t \rangle }{t \langle r \rangle^{1/2}} \| \widetilde{\phi} \|_{L^\infty} \lesssim \frac 1 t \| \widetilde{\phi} \|_{L^\infty}
\end{align*}
and finally,
$$
|III_3| \lesssim  \| \widetilde{\phi} \|_{L^\infty} \int_{ \xi \geq \xi_0 + \frac{C}{\sqrt t}} \frac{\xi}{\langle \xi r \rangle^{1/2}} \frac{1}{t |\xi - \xi_0|^2} \dd \xi \lesssim \| \widetilde{\phi} \|_{L^\infty} \frac{t^{1/2} \xi_0^{1/2} + \xi_0^{-1/2}}{t r^{1/2}} \lesssim \frac{1}{t} \|  \widetilde{\phi} \|_{L^\infty}.
$$

\medskip

%

\noindent
\underline{Case 2: low frequencies $|\xi| \lesssim 1$ with a cancellation at zero frequency.} For high frequencies which were the object of Case 1, the vanishing of the distorted Fourier transform at frequency zero is obviously irrelevant; but low frequency estimates will be responsible for the different decay rates in \eqref{dispersive1} and \eqref{dispersive2}. For now, we will prove \eqref{dispersive1}, whose right-hand side is finite if $\widetilde{\phi}$ vanishes suitably at $0$; the proof of \eqref{dispersive2} will be the object of Case 3 below.

In order to use the decomposition \eqref{decomposition-psi-flat}, we decompose the Bessel function as
$$
J_0(r)=W_{+,1}e^{ir}+W_{-,1}e^{-ir}
$$
where $|\partial_r^j W_{\pm,1}|\lesssim r^{-j}$ (see Appendix \ref{truebesselclassics}) and further introduce
\begin{align*}
& W_{+,2}=(\rho(r)-1)e^{i\xi r}\chi_{\xi^{-1}}(r), \qquad W_{-,2}=0,\\
& W_{+,3}=\frac{e^{-i\pi/4}}{2i\sqrt{\xi r}}(1-\chi_{\xi^{-1}}(r)), \qquad W_{-,3}=- \frac{e^{i\pi/4}}{2i\sqrt{\xi r}}(1-\chi_{\xi^{-1}}(r)),\\
& W_{+,4}=\frac{1}{\xi}m_\flat^{\textup{loc}}\chi_{\xi^{-1}}(r)e^{-i\xi r}, \qquad W_{-,4}=0\\
& W_{+,5}=\frac{1}{2\xi}(m^R_{\flat,1}-im^R_{\flat,2}), \qquad W_{-,5}=\frac{1}{2\xi}( m^R_{\flat,1}+im^R_{\flat,2}),
\end{align*}
and
\begin{align} \label{definition-tildephi1}
& \tilde \phi_1=\tilde \phi_2=\frac{\tilde \phi}{\xi}b_\flat\lambda' \chi_0(\xi)e(\xi)\textup{sign}\xi, \quad \tilde \phi_{3}=\frac{\tilde \phi}{\xi}c_\flat\lambda' \chi_0(\xi)e(\xi) \textup{sign}\xi ,\quad \tilde \phi_4=\tilde \phi_5= \tilde{\phi}\lambda'\chi_0(\xi)\textup{sign}\xi,
\end{align}
so that we have
$$
I_{lo} = \sum_{\pm}\sum_{n=1}^5 \int_{-\infty}^{+\infty} e^{i \Phi_{\pm}(t,\xi,r)} \widetilde{\phi}_n(\xi) W_{\pm,n}(\xi,r) \xi \dd \xi .
$$
We recall that $\lambda'$ is $C^\infty$ near the origin, and remark that for all $n\in \{1,5\}$,
\begin{equation}
\label{estimateWa2}
\begin{split}
& | W_{\pm,n}(\xi,r)  | \lesssim \frac{1}{\langle \xi r \rangle^{1/2}} \qquad
| \partial_{\xi} W_{\pm,n}(\xi,r)  | \lesssim \frac{\langle r \rangle}{\langle \xi r \rangle^{3/2}}.
\end{split}
\end{equation}

Turning to the stationary phase analysis, we can focus without loss of generality on the $-$ case in $I_{lo}$, as in the $+$ case there is no stationary point so that estimates are much easier, and we restrict the integral to $(0,\infty)$ as the integral over $(-\infty,0)$ can be estimated similarly: we will bound the term
\begin{align} \label{id-Ilo-decomposition}
\mathcal{I}_{lo} & = \int_{0}^{+\infty} e^{i \Phi_{-}(t,\xi,r)}( \widetilde{\phi}_1(\xi) W_{-,1}(\xi,r)+\widetilde{\phi}_3(\xi) W_{-,3}(\xi,r) +\widetilde{\phi}_5(\xi) W_{-,5}(\xi,r)  )\xi \dd \xi \\
\nonumber &=\sum_{n=1,3,5} \mathcal{I}_{lo,n}.
\end{align}
We recall that the stationary point $\xi_0$ of the phase $\Phi_-$ and its derivatives satisfy \eqref{boundphase-lem-2}-\eqref{boundphase-lem}. We will bound $\mathcal{I}_{lo,n}$ for $n=1,3,5$ relying on Lemma \ref{lempointwisedecay} which estimates a model integral for $\mathcal I_{lo}$: it will either be applied directly, or in cases where it will not apply directly since $\tilde \phi_n$ is not $C^2$ (the coefficients $b_\flat$ and $c_\flat$ having a $\xi^2\ln^2\xi$ singularity at the origin) its proof will be adapted to handle the present case.

\medskip

\noindent \underline{Case 2.0: $0\leqslant t<2$}. In that case, we simply remark that $\mathcal{I}_{lo}$ is the integral of a bounded function over $\{|\xi|\lesssim 1\}$, so that
$$
|\mathcal{I}_{lo}|\lesssim \sum_{n=1,3,5} \| \xi \widetilde{\phi}_n \|_{L^\infty} \lesssim \| \widetilde{\phi}\|_{L^\infty}.
$$
Therefore, we now assume $t\geqslant 2$ in all forthcoming cases.

\medskip

\noindent \underline{Case 2.1: $r \gg t$ or $t \gg r$, or $r \sim t$ and $\xi_0 \lesssim t^{-1/3}$}. In that case, we have directly by applying  \eqref{estimation-model-integral-case1} and \eqref{estimation-model-integral-case2} to each of the three terms in \eqref{id-Ilo-decomposition} that
$$
|\mathcal{I}_{lo} |\lesssim \min \left( \frac{1}{t}\left( \| \frac{\widetilde{\phi}}{\xi} \|_{L^\infty} + \|  \widetilde{\phi} \|_{W^{1,\infty}} \right) \ , \ \frac{1}{t^{2/3}}\left( \|  \widetilde{\phi} \|_{L^\infty} + \|  \widetilde{\phi} \|_{H^{1+\epsilon}} \right) \right).
$$

\medskip

\noindent \underline{Case 2.2: $r \sim t$ and $\xi_0 \gg t^{-1/3}$.} As in the proof of Lemma \ref{lempointwisedecay}, we introduce the scale $R = (t \xi_0)^{-1/2}$ and recall that $R \ll \xi_0$ since $\xi_0 \gg t^{-1/3}$. For $n=3$ and $n=5$ the integral $\mathcal{I}_{lo,n}$ is now split as follows
\begin{equation} \label{decomposition-mathcalI-lon}
\mathcal{I}_{lo,n} = \int_0^\infty \left[ \chi_-\left( \frac{\xi - \xi_0}{R} \right) + \chi_0\left( \frac{\xi - \xi_0}{R} \right) + \chi_+\left( \frac{\xi - \xi_0}{R} \right)\right] \dots \dd \xi = I_n + II_n + III_n,
\end{equation}
The middle term can be bounded immediately, using \eqref{estimateWa2} and $|\tilde \phi_n|\lesssim |\tilde \phi|$ for $n=3,5$ (because $|c_{\flat}|\lesssim |\xi|^2|\log^2\xi|$ from \eqref{bd:estimates-b-c-flat}):
$$
|II_n| \lesssim  \| \widetilde{\phi}_n \|_{L^\infty} \frac{\xi_0 R}{\langle \xi_0 r \rangle^{1/2}} \lesssim \frac 1t  \|  \widetilde{\phi} \|_{L^\infty} .
$$
We integrate by parts the right-term
\begin{align*}
III_n & = - \int e^{i \Phi_-(t,\xi,r)} \partial_\xi (\xi \widetilde{\phi}_n)(\xi) \chi_+\left( \frac{\xi - \xi_0}{R} \right) W_{-,n}(\xi,r) \frac{1}{i \partial_\xi \Phi_-(t,\xi,r)} \dd \xi \\
&\qquad  - \int e^{i \Phi_-(t,\xi,r)} \widetilde{\phi}_n(\xi)\partial_\xi \left[  \chi_+\left( \frac{\xi - \xi_0}{R} \right) \right] W_{-,n}(\xi,r) \frac{1}{i \partial_\xi \Phi_-(t,\xi,r)} \xi \dd \xi \\
& \qquad -  \int e^{i \Phi_-(t,\xi,r)} \widetilde{\phi}_n(\xi)  \chi_+\left( \frac{\xi - \xi_0}{R} \right) \partial_\xi \left[ W_{-,n}(\xi,r) \right] \frac{1}{i \partial_\xi \Phi_-(t,\xi,r)} \xi \dd \xi \\
& \qquad -  \int e^{i \Phi_-(t,\xi,r)} \widetilde{\phi}_n(\xi) \chi_+\left( \frac{\xi - \xi_0}{R} \right) W_{-,n}(\xi,r) \partial_\xi \left[ \frac{1}{i \partial_\xi \Phi_-(t,\xi,r)} \right] \xi \dd \xi \\
& = III_{n,1} + III_{n,2} + III_{n,3}+III_{n,4}.
\end{align*}
We bound the first term using that for $n=3,5$ we have $|\partial_\xi (\xi \tilde \phi_n)|\lesssim |\xi|(\||\xi|^{-1}\tilde \phi\|_{L^\infty}+\|\partial_\xi \tilde \phi\|_{L^\infty})$ (because $|\partial_\xi^j c_{\flat}|\lesssim |\xi|^{2-j}|\log^2\xi|$ for $j=0,1$ from \eqref{bd:estimates-b-c-flat}):
\begin{align*}
|III_{n,1}| \lesssim &(\||\xi|^{-1}\tilde \phi\|_{L^\infty}+\|\partial_\xi \tilde \phi\|_{L^\infty})  \int  \chi_+\left( \frac{\xi - \xi_0}{R} \right) \frac{\xi}{\langle r \xi \rangle^{\frac 12}}\frac{1}{t(\xi^2-\xi_0^2)} \dd \xi  \\
\lesssim & \frac{ \ln R}{tr^{1/2}\xi_0^{1/2}}(\||\xi|^{-1}\tilde \phi\|_{L^\infty}+\|\partial_\xi \tilde \phi\|_{L^\infty}) \ \lesssim \frac{1}{t}(\||\xi|^{-1}\tilde \phi\|_{L^\infty}+\|\partial_\xi \tilde \phi\|_{L^\infty})
\end{align*}
because $(\xi_0 t)^{-1/2}\ln R= R\ln R\ll 1$. We apply \eqref{estimation-model-integral-case3} to the last four terms and obtain:
$$
|III_{n,2} + III_{n,3}+III_{n,4}|\lesssim \frac{1}{t}\| |\xi|^{-1}\tilde \phi_n\|_{L^\infty}\lesssim \frac{1}{t}\| \tilde \phi\|_{L^\infty}.
$$
Combining, we get 
$$
|III_n|\lesssim \frac{1}{t}(\||\xi|^{-1}\tilde \phi\|_{L^\infty}+\|\partial_\xi \tilde \phi\|_{L^\infty})
$$
The contribution of $I_n$ for $n=3,5$ can be estimated similarly to $III_n$, leading to
$$
|I_n|\lesssim \frac{1}{t}(\||\xi|^{-1}\tilde \phi\|_{L^\infty}+\|\partial_\xi \tilde \phi\|_{L^\infty}).
$$
Hence, for $n=3,5$ we have
$$
|\mathcal{I}_{lo,n} | \lesssim  \frac{1}{t}(\||\xi|^{-1}\tilde \phi\|_{L^\infty}+\|\partial_\xi \tilde \phi\|_{L^\infty}).
$$
There remains to estimate $\mathcal{I}_{lo,1}$. To do so, recall that this part bounds the contribution of leading order part of $\psi^S_\flat$ consisting of the Bessel function $J_0$, so that we need to bound
$$
\mathcal{I}_{lo,1}= \int_{0}^\infty e^{it \lambda(\xi)}   \widetilde{\phi}_1 (\xi) J_0(\xi r)\xi \dd \xi .
$$
We rely on Fourier analysis, recognizing that this is the two-dimensional inverse Fourier transform in the radial sector of $\xi \mapsto e^{it \lambda(\xi)}   \widetilde{\phi}_1 (\xi)$, and hence using that $\widetilde{\phi}_1$ is supported in $\{|\xi|\leq 2\}$ and that Fourier transform exchanges products and convolution,
$$
\mathcal{I}_{lo,1}= K(t,\cdot)* \mathcal F^{-1}( \widetilde{\phi}_1)
$$
where, up to a constant,
$$
K(t,r)=  \int_{0}^\infty e^{it \lambda(\xi)} \chi_0(\xi/2) J_0(\xi r)\xi \dd \xi  .
$$
Appealing to Lemma \ref{lempointwisedecay} we have $\| K(t,\cdot)\|_{L^\infty}\lesssim t^{-1}$ so that
$$
|\mathcal{I}_{lo,1}|\lesssim \frac{1}{t} \| \mathcal F^{-1}( \widetilde{\phi}_1)\|_{L^1}.
$$
We pick $0<\epsilon'<\epsilon$ and estimate using the two-dimensional embedding $L^{2,1+\epsilon'}\subset L^1$ and continuity $\mathcal F^{-1}:H^{1+\epsilon'}\rightarrow L^{2,1+\epsilon'}$ that $\| \mathcal F^{-1}( \widetilde{\phi}_1)\|_{L^1}\lesssim \|  \widetilde{\phi}_1 \|_{H^{1+\epsilon'}(|\xi|d\xi)}$. By the definition \eqref{definition-tildephi1} of $\widetilde{\phi}_1$ we have $\|  \widetilde{\phi}_1 \|_{H^{1+\epsilon}(|\xi|d\xi)}\approx \| \frac{1}{\xi} \widetilde{\phi} \|_{H^{1+\epsilon}(|\xi|d\xi)}$. We then recall that since $\epsilon>\epsilon'$, then for functions $f\in H^{2+\epsilon'}$ such that $f(0)=0$, we have $\| \frac{1}{\xi} f \|_{H^{1+\epsilon'}(|\xi|d\xi)}\lesssim \| f \|_{H^{2+\epsilon}(|\xi|d\xi)}$. Hence
$$
|\mathcal{I}_{lo,1}|\lesssim \frac{1}{t} \| \widetilde{\phi} \|_{H^{2+\epsilon}(|\xi |d\xi)}.
$$

\medskip

\noindent \underline{Case 3: low frequencies $|\xi| \lesssim 1$ without cancellation at zero frequency.} We now aim at proving \eqref{dispersive2}, but it turns out that the necessary modifications in Case 2 are small. Case 2.1. apply verbatim; as for case 2.3, the same scheme of proof can be followed up to small adaptations. Namely, we consider $r \sim t$ and $\xi_0 \gg t^{-1/3}$ and keep the notations of case 2.3. We decompose $\mathcal I_{lo,n}$ as in \eqref{decomposition-mathcalI-lon} for $n=1,3,5$. We estimate, using \eqref{estimateWa2},
$$
|II_n| \lesssim \left|  \int_{-\infty}^{+\infty} \chi_0\left(\frac{\xi-\xi_0}{R}\right)e^{i \Phi_{\pm}(t,\xi,r)} \widetilde{\phi}_n(\xi) W_{\pm,n}(\xi,r) \xi \dd \xi \right| \lesssim \|\xi \widetilde{\phi}_n \|_{L^\infty} \frac{ R}{\langle \xi_0 r \rangle^{1/2}} \lesssim \frac{1}{t^{2/3}}  \| \widetilde{\phi} \|_{L^\infty} .
$$
and
\begin{align*}
& |III_1| \lesssim \| \partial_\xi (\xi \widetilde{\phi}_n)\|_{L^{2+\alpha}(|\xi|\dd \xi)}\| \frac{1}{\langle r \xi \rangle^{\frac 12}t (\xi^2-\xi_0^2)|\xi|^{\frac{1}{2+\alpha}}}\|_{L^{\frac{2+\alpha}{1+\alpha}}(\dd \xi)} \\
& \qquad \qquad \lesssim \| \partial_\xi \widetilde \phi \|_{L^{2+\alpha}}\frac{1}{t^{\frac 32}\xi_0^{\frac 32+\frac{1}{2+\alpha}}R^{1-\frac{1+\alpha}{2+\alpha}}} \lesssim\frac{1}{t^{2/3+\frac{\alpha}{3(2+\alpha)}}} \| \widetilde \phi\|_{H^{1+\epsilon}}  \\
& |III_{n,2}| \lesssim \frac{1}{t} \| \xi  \widetilde{\phi}_n \|_{L^\infty} \int_{ \xi_0 + R}^{\xi_0+2R} \frac{1}{\xi^2 - \xi_0^2} \frac{1}{R \langle \xi r \rangle^{1/2}} \dd \xi \lesssim \frac 1 {t^{2/3}}  \| \widetilde{\phi} \|_{L^\infty}\\
& |III_{n,3}| \lesssim \frac{1}{t} \| \xi \widetilde{\phi}_n \|_{L^\infty} \int_{\xi > \xi_0 + CR} \frac{1}{\xi^2 - \xi_0^2} \frac{\langle r \rangle}{\langle \xi r \rangle^{3/2}} \dd \xi \lesssim \frac 1 {t^{2/3}}  \| \widetilde{\phi} \|_{L^\infty}\\
& |III_{n,4}| \lesssim \| \xi \widetilde{\phi}_n \|_{L^\infty} \int_{\xi > \xi_0 + CR} \frac{\xi}{t(\xi^2 - \xi_0^2)^2} \frac{1}{\langle r\xi \rangle^{1/2}} \dd \xi \lesssim \frac 1 {t^{2/3}}  \| \widetilde{\phi} \|_{L^\infty}
\end{align*}
where we used the H\"older inequality and the Sobolev embedding $H^{\epsilon}\subset L^{2+\alpha}$ for some $\alpha>0$ for the first inequality. This outline concludes the proof of \eqref{dispersive2}.
\end{proof}

\subsection{End the proof of Theorem \ref{thmdecayschwartz}}

\begin{proof}
If $\phi \in L^{2,s}(r\dd r)$ for some $s\in (1,3)$ then we have $\| \widetilde \phi \|_{H^{s}(|\xi|\dd \xi)}\lesssim \| \phi \|_{L^{2,s}}$ by Lemma \ref{lem:continuity-distorted-Fourier-L2s-Hs}. The desired $L^2$ bounds for $e^{it \mathcal H}\phi$ are then direct consequences of Theorem \ref{thm-L2bound}, and the desired $L^\infty$ bounds are consequences of Theorem \ref{thmpointwisedecay}.
\end{proof}

\section{Jost solutions}\label{policia}

\label{sectionjost}

The rest of this article is dedicated to the construction and description of the generalized eigenfunctions $\psi(\xi,r)$. \underline{We will henceforth only consider the case $\xi \geqslant 0$ (i.e. $\lambda \geqslant 0$)} and then define $\psi(-\xi,r) = -\sigma_1 \psi(\xi,r)$.

The aim of this section is to obtain solutions with prescribed behavior in a neighborhood of both endpoints: $0$ and $\infty$. This is easily achieved close to the origin, but at infinity, where the corresponding functions are called Jost solutions, the construction is more delicate. We construct Jost solutions which are either exponentially decaying or oscillating at infinity, and we also obtain quantitative bounds for $\lambda > \lambda_0 > 0$, for any $\lambda_0>0$ that will be used to describe the limit $\lambda \rightarrow + \infty.$.

\subsection{Equivalent formulations of $\mathcal{H}-\lambda = 0$}
We will study the ordinary differential equation
\begin{equation}
  \mathcal{H} \left(\begin{array}{c}
    u\\
    v
  \end{array}\right) = \lambda \left(\begin{array}{c}
    u\\
    v
  \end{array}\right) \label{maineq}
\end{equation}
where
\[ \mathcal{H}= \left(\begin{array}{cc}
     - \Delta + \frac{1}{r^2} + 1 & 1\\
     - 1 & \Delta - \frac{1}{r^2} - 1
   \end{array}\right) + (\rho^2 - 1) \left(\begin{array}{cc}
     2 & 1\\
     - 1 & - 2
   \end{array}\right) \]
with $\Delta = \partial_r^2 + \frac{1}{r} \partial_r$ and $\lambda \in
\mathbb{R}$. Since $\sigma_1 \mathcal{H} \sigma_1 = -\mathcal{H}$, we focus on the case $\lambda \geqslant 0$. Recalling that
\[ L_0 = \Delta - \frac{1}{r^2} + (1 - \rho^2), \]
we have
\[ \mathcal{H}= \left(\begin{array}{cc}
     - L_0 + \rho^2 & \rho^2\\
     - \rho^2 & L_0 - \rho^2
   \end{array}\right) . \]
The following lemma gives equivalent formulations of the  equation (\ref{maineq}).

\begin{lem}
  \label{pullyou}For $\lambda \geqslant 0$, the functions $(u, v)$ solve
  (\ref{maineq}) if and only if
  \begin{equation}
    \left\{\begin{array}{l}
      L_0 (u - v) = - \lambda (u + v)\\
      (L_0 - 2 \rho^2) (u + v) = - \lambda (u - v) .
    \end{array}\right. \label{maineq2}
  \end{equation}
  Furthermore, with
\begin{align*}
& a_{\pm} (\lambda) = \lambda \pm \langle \lambda \rangle \\
& \xi^2 = \langle \lambda \rangle - 1 = - \lambda - 1 + a_+ (\lambda)
\\ 
&\kappa^2 = 1 + \langle \lambda \rangle = \lambda + 1 - a_-
(\lambda) \geqslant 2
\end{align*}
this can also be written in the variables
  \[ \left(\begin{array}{c}
       \varphi\\
       \psi
     \end{array}\right) = \left(\begin{array}{c}
       \frac{u}{a_+ (\lambda)} + v\\
       - u - a_- (\lambda) v
     \end{array}\right) \]
  as
  \begin{equation}
    \left(\begin{array}{c}
      (\Delta - \kappa^2) \varphi\\
      (\Delta + \xi^2) \psi
    \end{array}\right) + V_{\lambda} \left(\begin{array}{c}
      \varphi\\
      \psi
    \end{array}\right) = 0 \label{maineq3}
  \end{equation}
  with
  \[ V_{\lambda} = - \frac{1}{r^2} \operatorname{Id} - \frac{\rho^2 (r) - 1}{\langle \lambda\rangle} \left(\begin{array}{cc}
       1 + 2 \langle \lambda \rangle & - \lambda\\
       - \lambda & - 1 + 2 \langle \lambda \rangle
     \end{array}\right) . \]
\end{lem}

The choice of the variables $\varphi, \psi$ is such that the potential
$V_{\lambda}$ converges to finite limits both when $\lambda \rightarrow 0$ and
$\lambda \rightarrow + \infty$ (remark also that this is also true for
$\frac{1}{a_+ (\lambda)}$ and $a_- (\lambda)$). The change of variables $(u,
v) \rightarrow (\varphi, \psi)$ is invertible for all $\lambda \geqslant 0$.

\begin{proof}
  If $(u, v)$ solve (\ref{maineq}), then
  \[ 0 = (\mathcal{H}- \lambda) \left(\begin{array}{c}
       u\\
       v
     \end{array}\right) = \left(\begin{array}{c}
       - L_0 (u) + \rho^2 u - \lambda u + \rho^2 v\\
       L_0 (v) - \rho^2 v - \lambda v - \rho^2 u
     \end{array}\right) . \]
  Summing and taking the difference of these two equalities yields
  (\ref{maineq2}). It can also be written as
  \[ \left( \Delta - \frac{1}{r^2} \right) \left(\begin{array}{c}
       u\\
       v
     \end{array}\right) + \left(\begin{array}{cc}
       - 1 + \lambda & - 1\\
       - 1 & - 1 - \lambda
     \end{array}\right) \left(\begin{array}{c}
       u\\
       v
     \end{array}\right) - (\rho^2 - 1) \left(\begin{array}{cc}
       2 & 1\\
       1 & 2
     \end{array}\right) \left(\begin{array}{c}
       u\\
       v
     \end{array}\right) = 0. \]
  We check that
  \[ \left(\begin{array}{cc}
       - 1 + \lambda & - 1\\
       - 1 & - 1 - \lambda
     \end{array}\right) = P \left(\begin{array}{cc}
       - \kappa^2 & 0\\
       0 & \xi^2
     \end{array}\right) P^{- 1} \]
  where
  \[ P = \left(\begin{array}{cc}
       - a_- (\lambda) & - a_+ (\lambda)\\
       1 & 1
     \end{array}\right) \]
  and
  \[ P^{- 1} = \frac{1}{2 \langle \lambda \rangle} \left(\begin{array}{cc}
       1 & a_+ (\lambda)\\
       - 1 & - a_- (\lambda)
     \end{array}\right) . \]
  Using $a_{\pm}^2 (\lambda) = 1 + 2 \lambda a_{\pm} (\lambda), a_+ (\lambda)
  + a_- (\lambda) = 2 \lambda$ and $a_+ (\lambda) a_- (\lambda) = - 1$, we
  compute that
  \[ P^{- 1} \left(\begin{array}{cc}
       2 & 1\\
       1 & 2
     \end{array}\right) P = \frac{1}{\langle \lambda \rangle}
     \left(\begin{array}{cc}
       1 + 2 \langle \lambda \rangle & - \lambda a_+ (\lambda)\\
       \lambda a_- (\lambda) & - 1 + 2 \langle \lambda \rangle
     \end{array}\right) \]

  Therefore, since $a_+ (\lambda) = \lambda + \langle \lambda \rangle \neq 0$ for
  $\lambda \geqslant 0$, for the variables
  \[ \left(\begin{array}{c}
       \varphi\\
       \psi
     \end{array}\right) = 2 \left(\begin{array}{cc}
       \frac{\langle \lambda \rangle}{a_+ (\lambda)} & 0\\
       0 & \langle \lambda \rangle
     \end{array}\right) P^{- 1} \left(\begin{array}{c}
       u\\
       v
     \end{array}\right), \]
we get the desired equation.
\end{proof}

\subsection{Solutions near $r = 0$}\label{ss22n}

\begin{lem}
  \label{zerosol}For any $\lambda \geqslant 0$, the four-dimensional
  $\mathbb{C}$-vector space $\mathcal{E}_{\lambda}$ of solutions of
  $\mathcal{H}- \lambda = 0$ on $\mathbb{R}^{+ \ast}$ can be decomposed as
  \[ \mathcal{E}_{\lambda} =\mathcal{E}_1 +\mathcal{E}_2 \]
  where $\dim (\mathcal{E}_1) = \dim (\mathcal{E}_2) = 2$, and for any $(u_1,
  v_1) \in \mathcal{E}_1$, there exists $(\kappa_1, \kappa_2) \in
  \mathbb{C}^2$ such that
  \[ \left(\begin{array}{c}
       u_1\\
       v_1
     \end{array}\right) = r \left(\begin{array}{c}
       \kappa_1\\
       \kappa_2
     \end{array}\right) + O_{r \rightarrow 0} (r^2) \]
  and for any $(u_2, v_2) \in \mathcal{E}_2$, there exists $(\kappa_1,
  \kappa_2) \in \mathbb{C}^2$ such that
  \[ \left(\begin{array}{c}
       u_2\\
       v_2
     \end{array}\right) = \frac{1}{r} \left(\begin{array}{c}
       \kappa_1\\
       \kappa_2
     \end{array}\right) + O_{r \rightarrow 0} (1) . \]
\end{lem}

\begin{proof}
  Since $L_0 = \Delta - \frac{1}{r^2} + (1 - \rho^2)$ and
  \[ (\mathcal{H}- \lambda) \left(\begin{array}{c}
       u\\
       v
     \end{array}\right) = \left(\begin{array}{cc}
       - L_0 + \rho^2 - \lambda & \rho^2\\
       - \rho^2 & L_0 - \rho^2 - \lambda
     \end{array}\right) \left(\begin{array}{c}
       u\\
       v
     \end{array}\right), \]
  we write the equation $(\mathcal{H}- \lambda) \left(\begin{array}{c}
    u\\
    v
  \end{array}\right) = 0$ as
  \[ \left( \Delta - \frac{1}{r^2} \right) \left(\begin{array}{c}
       u\\
       v
     \end{array}\right) = \left(\begin{array}{c}
       S_1 (u, v)\\
       S_2 (u, v)
     \end{array}\right) \]
  where
  \[ S_1 (u, v) = \rho^2 (2 u + v) - (1 + \lambda) u \]
  and
  \[ S_2 (u, v) = \rho^2 (2 v + u) - (1 - \lambda) v. \]
  The solutions of $\Delta - \frac{1}{r^2} = 0$ on $\mathbb{R}^{+ \ast}$ are spanned by $r$ and $\frac{1}{r}$. Furthermore, a solution of $\left( \Delta -
  \frac{1}{r^2} \right) (u) = S$ is
  \[ u (r) = r \int_{r_0}^r \frac{1}{s^3} \int_0^s t^2 S (t) \dd t \dd s \]
  for any $r_0 > 0$.

  We define the sequences of functions $(u_n, v_n)$ by
  \[ \left(\begin{array}{c}
       u_0\\
       v_0
     \end{array}\right) (r) = r \left(\begin{array}{c}
       1\\
       0
     \end{array}\right) \]
  and then
  \[ \left(\begin{array}{c}
       u_{n + 1}\\
       v_{n + 1}
     \end{array}\right) (r) = r \int_0^r \frac{1}{s^3} \int_0^s t^2
     \left(\begin{array}{c}
       S_1 (u_n, v_n)\\
       S_2 (u_n, v_n)
     \end{array}\right) (t) \dd t \dd s. \]
  Let us show that for $R_0 > 0$ small enough and $n \geqslant 0$,
  \[ \left\| \frac{| u_{n + 1} | + | v_{n + 1} |}{r^2} \right\|_{L^{\infty}
     ([0, R_0])} \lesssim_{\lambda} R_0 \left\| \frac{| u_n | + | v_n |}{r}
     \right\|_{L^{\infty} ([0, R_0])} . \]
  Indeed, we have that for $r \in [0, R_0]$,
  \[ \left| \left(\begin{array}{c}
       S_1 (u_n, v_n)\\
       S_2 (u_n, v_n)
     \end{array}\right) (t) \right| \lesssim_{\lambda} t \left\| \frac{| u_n |
     + | v_n |}{r} \right\|_{L^{\infty} ([0, R_0])} \]
  and therefore
\begin{align*}
& \left| r \int_0^r \frac{1}{s^3} \int_0^s t^2 \left(\begin{array}{c}
      S_1 (u_n, v_n)\\
      S_2 (u_n, v_n)
    \end{array}\right) (t)\dd t  \right| \lesssim_{\lambda} \left| r \int_0^r \frac{1}{s^3} \int_0^s t^3\dd t 
    \right| \left\| \frac{| u_n | + | v_n |}{r} \right\|_{L^{\infty} ([0,
    R_0])}\\
    & \qquad \qquad \qquad \qquad \lesssim_{\lambda} r^3 \left\| \frac{| u_n | + | v_n |}{r}
    \right\|_{L^{\infty} ([0, R_0])}  \lesssim_{\lambda} R_0 r^2 \left\| \frac{| u_n | + | v_n |}{r}
    \right\|_{L^{\infty} ([0, R_0])},
  \end{align*}
  leading to the result. Taking $R_0$ small enough, we deduce that $(u_n, v_n)
  \rightarrow (u_{n + 1}, v_{n + 1})$ is a contraction for the norm $\left\|
  \frac{.}{r} \right\|_{L^{\infty} ([0, R_0])}$, and thus
  \[ \left(\begin{array}{c}
       u\\
       v
     \end{array}\right) (r) = \sum_{n \in \mathbb{N}} \left(\begin{array}{c}
       u_n\\
       v_n
     \end{array}\right) (r) \]
  is well-defined on $[0, R_0]$ and solves $(\mathcal{H}- \lambda)
  \left(\begin{array}{c}
    u\\
    v
  \end{array}\right) = 0$ with
  \[ \left(\begin{array}{c}
       u\\
       v
     \end{array}\right) (r) = r \left(\begin{array}{c}
       1\\
       0
     \end{array}\right) + O_{r \rightarrow 0} (r^2) . \]
  A similar proof can be made to construct $(u, v)$ such that
  \[ \left(\begin{array}{c}
       u\\
       v
     \end{array}\right) (r) = r \left(\begin{array}{c}
       0\\
       1
     \end{array}\right) + O_{r \rightarrow 0} (r^2), \]
  and these functions form a basis of $\mathcal{E}_1$. Now, we construct
  similarly a sequence of function on $] 0, R_0 [$ by
  \[ \left(\begin{array}{c}
       u_0\\
       v_0
     \end{array}\right) (r) = \frac{1}{r} \left(\begin{array}{c}
       1\\
       0
     \end{array}\right) \]
  with
  \[ \left(\begin{array}{c}
       u_{n + 1}\\
       v_{n + 1}
     \end{array}\right) (r) = r \int_{R_0}^r \frac{1}{s^3} \int_0^s t^2
     \left(\begin{array}{c}
       S_1 (u_n, v_n)\\
       S_2 (u_n, v_n)
     \end{array}\right) (t) \dd t \dd s. \]
  We check that for $r \in] 0, R_0 [$,
  \begin{align*}
&\left| r \int_{R_0}^r \frac{1}{s^3} \int_0^s t^2
    \left(\begin{array}{c}
      S_1 (u_n, v_n)\\
      S_2 (u_n, v_n)
    \end{array}\right) (t) \dd t \dd s \right| \lesssim_{\lambda} \left( r \int_{R_0}^r \frac{1}{s^3} \int_0^s t \dd t
    \dd s \right) \| r (| u_n | + | v_n |) \|_{L^{\infty} ([0, R_0])}\\
    &\qquad \qquad \lesssim_{\lambda}  K R_0 | \ln (R_0) | \| r (| u_n | + | v_n |)
    \|_{L^{\infty} ([0, R_0])}
  \end{align*}
  and thus as previously, $(u_n, v_n) \rightarrow (u_{n + 1}, v_{n + 1})$ is a
  contraction for the norm $\| r. \|_{L^{\infty} ([0, R_0])}$ if $R_0 > 0$ is
  small enough, leading to
  \[ \left(\begin{array}{c}
       u\\
       v
     \end{array}\right) (r) = \sum_{n \in \mathbb{N}} \left(\begin{array}{c}
       u_n\\
       v_n
     \end{array}\right) (r) \]
  being well-defined on $] 0, R_0 [$ and solving $(\mathcal{H}- \lambda)
  \left(\begin{array}{c}
    u\\
    v
  \end{array}\right) = 0$ with
  \[ \left(\begin{array}{c}
       u\\
       v
     \end{array}\right) (r) = \frac{1}{r} \left(\begin{array}{c}
       1\\
       0
     \end{array}\right) + O_{r \rightarrow 0} (1) . \]
  A similar construction can be done to have a solution satisfying
  \[ \left(\begin{array}{c}
       u\\
       v
     \end{array}\right) (r) = \frac{1}{r} \left(\begin{array}{c}
       0\\
       1
     \end{array}\right) + O_{r \rightarrow 0} (1) \]
  and these two functions form a base of $\mathcal{E}_2$.
\end{proof}

\subsection{Preliminaries to the construction of Jost Solutions}\label{dropdown}

For $\lambda > 0$, recall that $\xi = \sqrt{\langle \lambda \rangle - 1} >
0$, and we define the change of variables
\[ y = \xi r. \]
Recall that
\[ \left(\begin{array}{c}
     \varphi\\
     \psi
   \end{array}\right) = \left(\begin{array}{c}
     \frac{u}{a_+ (\lambda)} + v\\
     - u - a_- (\lambda) v
   \end{array}\right) . \]
Abusing notations and denoting $(\varphi,\psi)$ for the rescaled functions through $ y = \xi r$, we see that $(u,v)$ is in the kernel of $\mathcal{H}-\lambda$ if and only if
\begin{equation}
  \left(\begin{array}{c}
    \left( \Delta - \left( \frac{\kappa}{\xi} \right)^2 -
    \frac{1}{y^2} \right) \varphi\\
    \left( \Delta + 1 - \frac{1}{y^2} \right) \psi
  \end{array}\right) (y) = \widetilde{V}_{\lambda} \left(\begin{array}{c}
    \varphi\\
    \psi
  \end{array}\right) (y) \label{lydmor}
\end{equation}
with
\begin{equation}
  \widetilde{V}_{\lambda} (y) =  \frac{\frac{1}{\xi^2} \left( \rho^2
  \left( \frac{y}{\xi} \right) - 1 \right)}{\langle \lambda \rangle}
  \left(\begin{array}{cc}
    1 + 2 \langle \lambda \rangle & - \lambda\\
    - \lambda & - 1 + 2 \langle \lambda \rangle
  \end{array}\right) . \label{aname}
\end{equation}

We start with some estimates when $\lambda \rightarrow + \infty$ for the coefficients in this equation and change of variables.

\begin{lem}
  \label{foryou} For any $\lambda_0 > 0$, we have that for any $y \geqslant 1,
  \lambda \geqslant \lambda_0, j, k \in \mathbb{N}$,
  \[ \left| \partial_{\lambda}^j \left( \frac{\kappa}{\xi} - 1
     \right) \right| + | \partial_{\lambda}^j a_- (\lambda) | + \left|
     \partial_{\lambda}^j \left( \frac{1}{a_+} \right) (\lambda) \right|
     \lesssim_{\lambda_0, j} \frac{1}{\lambda^{1 + j}} \]
  and
  \[ | \partial_{\lambda}^j \partial_y^k \widetilde{V}_{\lambda} |
     \lesssim_{\lambda_0, j, k} \frac{1}{(\xi + y)^2} \frac{1}{\lambda^j
     (1 + y)^k} . \]
\end{lem}

\begin{proof}
By expanding the formulas
\begin{align*}
& \frac{1}{a_+ (\lambda)} = \frac{1}{\lambda} \left( \frac{1}{1 + \sqrt{1 + \frac{1}{\lambda^2}}} \right) \\
& a_- (\lambda) = \lambda \left( 1 - \sqrt{1 + \frac{1}{\lambda^2}} \right) \\
& \left( \frac{\kappa}{\xi} \right)^2 = 1 + 2 \left( \frac{1 + \sqrt{1 +
     \lambda^2}}{\lambda^2} \right),
\end{align*}
the first assertion of the lemma follows easily.
We check similarly that for
  $\lambda \geqslant \lambda_0$,
  \[ \left| \partial_{\lambda}^j \left( \frac{1}{\langle \lambda \rangle}
     \left(\begin{array}{cc}
       1 + 2 \langle \lambda \rangle & - \lambda\\
       - \lambda & - 1 + 2 \langle \lambda \rangle
     \end{array}\right) \right) \right| \lesssim_{\lambda_0, j}
     \frac{1}{\lambda^j} . \]

Therefore, there only remains to show that
$$
\left| \partial_{\lambda}^j \partial_y^k \left[ \frac{1}{\xi^2} \left( \rho^2
  \left( \frac{y}{\xi} \right) - 1 \right) \right] \right| \lesssim \frac{1}{(\xi + y)^2\lambda^j y^k}
$$
Since $|\partial_\lambda^j \xi'(\lambda)| \lesssim \lambda^{-\frac 12 - j}$ for $j \in \mathbb{N}$, it suffices to show that
$$
\left| \partial_{\xi}^j \partial_y^k \left[ \frac{1}{\xi^2} \left( \rho^2
  \left( \frac{y}{\xi} \right) - 1 \right) \right] \right| \lesssim \frac{1}{(\xi^2 + y^2)\xi^j y^k}.
$$
Denoting $R_k$ for the $k$-th derivative of $\rho^2-1$, we know by Lemma \ref{rhostuff} that $|R_k(z)| \lesssim \langle z \rangle^{-2-k}$ and the above left-hand side can be written $\left| \partial_{\xi}^j \left[ \xi^{-2-k} R_k 
  \left( \frac{y}{\xi} \right)  \right] \right|$. Applying Leibniz' rule to this expression, we obtain a sum of terms of the type
$\xi^{-2-k-j} Z \left( \frac{y}{\xi} \right)$, where $Z(y) \lesssim \langle z \rangle^{-2-k}$. Since
$$
\xi^{-2-k-j} \frac{1}{1 + \left( \frac{y}{\xi} \right)^{2+k}} \lesssim \frac{1}{(\xi^2 + y^2)\xi^j y^k},
$$
we obtain the desired result.
\end{proof}

From now on, we will use the notations of Annex \ref{sectionspace} that we
recall here. In the rest of this section, given some $\lambda_0 > 0$, we will
use for $n, m, b, c \in \mathbb{N}$ the norm
\[ \mathcal{N}_{b, c}^{n, m} (f) = \sum_{j = 0}^n \sum_{k = 0}^{n + m - j} \|
   y^b (y + \xi)^c y^k \lambda^j \partial_y^k \partial_{\lambda}^j f
   \|_{L^{\infty} ([1, + \infty (\times [\lambda_0, + \infty ()} \]
which is the one studied in Annex \ref{sectionspace} for the parameters $\nu
= \xi, r_0 = 1$ (taken now in the variable $y$ instead of $r$). This norm
measure functions that decays like $\frac{1}{y^b (y + \xi)^c}$ when $y
\rightarrow + \infty$ and for which taking any combinaison of $n$ times or
less of either $y \partial_y$ or $\lambda \partial_{\lambda}$, and up to an
additional $m$ times $y \partial_y$ derivatives, does not change this decay.
Since $\xi (\lambda) \rightarrow + \infty$ when $\lambda \rightarrow +
\infty$, although for instance $\mathcal{N}_{0, 1}^{0, 0} (f) \sim
\mathcal{N}_{1, 0}^{0, 0} (f)$, the equivalence constant is not independent of
$\lambda$.

Remark that by Lemma \ref{foryou}, for any $n, m \in \mathbb{N}$,
\begin{equation}
  \mathcal{N}_{0, 2}^{n, m} (\widetilde{V}_{\lambda}) \lesssim_{n, m} 1.
  \label{dusk0}
\end{equation}
Still from Annex \ref{sectionspace}, for $b \in \mathbb{Z},x>0$ we define the
space $\mathcal{W}_b (x^+)$ as the set of functions $f : [x, + \infty
(\rightarrow \mathbb{R}$ such that $\| r^b r^k \partial_r^k f \|_{L^{\infty}
([x, + \infty ()} < + \infty$ for all $k \in \mathbb{N}$, and $\mathcal{W}_b
(1^+, \lambda_0^+)$ the set of functions $f : [1, + \infty (\times [\lambda_0,
+ \infty (\rightarrow \mathbb{R}$ such that $\| r^b r^k \lambda^j
\partial_{\lambda}^j \partial_r^k f \|_{L^{\infty} ([1, + \infty ( \times[\lambda_0 ,+\infty()} < +
\infty$ for all $j, k \in \mathbb{N}$. Remark for instance that, by Lemma \ref{foryou} we have $\frac{\kappa}{\xi} \in 1+ \mathcal{W}_1(\lambda_0^+)$ for any $\lambda_0>0$.

Finally , given a function $a$, we define the operator $D_a = \partial_y + a$,
and a choice of inverse $D_a^{- 1}$ depending on properties of $a$, given in
Definition \ref{fnpig}.

\subsection{The exponentially decaying solution near $+\infty$}

\begin{lem}
  \label{factexp}For any $\lambda > 0$, there exists a solution of
  (\ref{lydmor}) of the form
  \[ \left(\begin{array}{c}
       \varphi_2\\
       \psi_2
     \end{array}\right) =  K_1 \left( \frac{\kappa}{\xi} \cdot \right)
     \left(\begin{array}{c}
       h_1\\
       h_2
     \end{array}\right) \]
such that, for any $\lambda_0 > 0$, taking $\lambda \geqslant \lambda_0$ and
  $y \geqslant 1$, we have for any $j, k \in \mathbb{N}$,
  \[ \left| \partial_{\lambda}^j \partial_y^k \left( \left(\begin{array}{c}
       h_1\\
       h_2
     \end{array}\right) - \left(\begin{array}{c}
       1\\
       0
     \end{array}\right) \right) \right| \lesssim_{\lambda_0, j, k}
     \frac{1}{(\xi + y) \lambda^j y^k} . \]
     
  Furthermore, given $y_0 \geqslant 1$, we have
  \[ \left(\begin{array}{c}
       \varphi_2\\
       \varphi_2'\\
       \psi_2\\
       \psi_2'
     \end{array}\right) (y_0) =   
     \left(\begin{array}{c}
       K_1 \left(y_0 \right)\\
       K_1' \left(y_0 \right)\\
       0\\
       0
     \end{array}\right) + o_{\lambda \rightarrow + \infty} (1)  \]
and for any $k \in \mathbb{N}^{\ast}$,
  \[ \left| \partial_{\lambda}^k \left(\begin{array}{c}
       \varphi_2\\
       \varphi_2'\\
       \psi_2\\
       \psi_2'
     \end{array}\right) (y_0) \right| \lesssim_{\lambda_0, y_0, k}
     \frac{1}{\xi \lambda^k} . \]
\end{lem}

\begin{proof}
  We fix some $\lambda_0 > 0$ and consider $\lambda \geqslant \lambda_0$. We
  recall that equation (\ref{lydmor}) is
  \[ \left(\begin{array}{c}
       \left( \Delta - \left( \frac{\kappa}{\xi} \right)^2 - \frac{1}{y^2}
       \right) \varphi\\
       \left( \Delta + 1 - \frac{1}{y^2} \right) \psi
     \end{array}\right) = \widetilde{V}_{\lambda} \left(\begin{array}{c}
       \varphi\\
       \psi
     \end{array}\right) . \]
  We look for a solution of the form
  \[ \left(\begin{array}{c}
       \varphi\\
       \psi
     \end{array}\right) = K_1 \left( \frac{\kappa}{\xi} . \right)
     \left(\begin{array}{c}
       h_1\\
       h_2
     \end{array}\right) . \]
  Since
  \[ \left( \Delta - \left( \frac{\kappa}{\xi} \right)^2 - \frac{1}{y^2}
     \right) \left( K_1 \left( \frac{\kappa}{\xi} . \right) \right) = 0, \]
  dividing the equation by $K_1 \left( \frac{\kappa}{\xi} . \right)$ we get
  \begin{equation}
    \left(\begin{array}{c}
      h''_1 + \left( 2 \frac{\kappa}{\xi} \frac{K_1'}{K_1} \left(
      \frac{\kappa}{\xi} . \right) + \frac{1}{y} \right) h_1'\\
      h''_2 + \left( 2 \frac{\kappa}{\xi} \frac{K_1'}{K_1} \left(
      \frac{\kappa}{\xi} . \right) + \frac{1}{y} \right) h_2' + \left( 1 +
      \left( \frac{\kappa}{\xi} \right)^2 \right) h_2
    \end{array}\right) = \widetilde{V}_{\lambda} \left(\begin{array}{c}
      h_1\\
      h_2
    \end{array}\right) . \label{twoyardstone}
  \end{equation}
  Recalling the definition of the operator $D_a (h) = h' + a h$, we have with
  \[ a (y, \lambda) = 2 \frac{\kappa}{\xi} \frac{K_1'}{K_1} \left(
     \frac{\kappa}{\xi} . \right) + \frac{1}{y}, \]
  that
  \[ h''_1 + \left( 2 \frac{\kappa}{\xi} \frac{K_1'}{K_1} \left(
     \frac{\kappa}{\xi} . \right) + \frac{1}{y} \right) h_1' = D_a D_0 (h) .
  \]
  By Lemma \ref{abudhabi} and $\lambda \rightarrow \frac{\kappa}{\xi} \in
  \mathcal{W}_0 (\lambda_0^+)$ from Lemma \ref{foryou}, we have
  \[ a \in - 2 \frac{\kappa}{\xi} +\mathcal{W}_2 (1^+, \lambda_0^+) \]
  and is therefore of type $< 0$ in Definition \ref{cadmissible} (since
  $\frac{\kappa}{\xi} \rightarrow 1$ when $\lambda \rightarrow + \infty$).
  
  We now look for $b_1, b_2$ such that
  \[ D_{b_1} D_{b_2} (h_2) = h_2'' + a h_2' + \left( 1 + \left(
     \frac{\kappa}{\xi} \right)^2 \right) h_2. \]
     which leads to the system
  \[ b_1 + b_2 = a, \qquad b_2' = b_2^2 - a b_2 + 1 + \left( \frac{\kappa}{\xi}
     \right)^2. \]

We notice that then two solutions of $D_{b_1} D_{b_2}=0$ would be $h_{2\pm}=\frac{J_1\pm i Y_1}{K_1(\frac{\kappa}{\xi}\cdot)}$. We choose $b_2=-\frac{h_{2+}'}{h_{2+}}$, which is well defined as $h_{2+}$ does not vanish, and which produces $D_{b_2}h_{2+}=0$. Then $D_{b_2}h_{2-}=(h_{2-}'h_{2+}-h_{2+}'h_{2-})/h_{2+}$ which does not vanish as we recognize the Wronskian between $h_{2-}$ and $h_{2+}$. We then define $b_1=-\frac{(D_{b_2}h_{2+})'}{D_{b_2}h_{2+}}$. Notice then that $h_{2+}$ and $h_{2-}$ are solutions to $D_{b_1} D_{b_2}=0$, so that necessarily $b_1$ and $b_2$ satisfy the system above, hence $b_1=a-b_2$. We compute
\begin{equation} \label{computation-b2-exponential}
b_2=\frac{-(J_1+iY_1)'}{J_1+iY_1}+\frac{\kappa}{\xi}\frac{K_1'}{K_1}\left(\frac{\kappa}{\xi}\cdot\right)=i-\frac{\kappa}{\xi}+\mathcal W_{2,0}(1^+,\lambda_0^+)
\end{equation}
by Lemma \ref{abudhabi}.3 and $\lambda \rightarrow \frac{\kappa}{\xi} \in
  \mathcal{W}_0 (\lambda_0^+)$. This gives in turn
$$
b_1=a-b_2=-i-\frac{\kappa}{\xi}+\mathcal W_{2,0}(1^+,\lambda_0^+).
$$

Equation (\ref{twoyardstone}) is now
  \[ \left(\begin{array}{c}
       D_a (D_0 (h_1))\\
       D_{b_1} (D_{b_2} (h_2))
     \end{array}\right) = \widetilde{V}_{\lambda} \left(\begin{array}{c}
       h_1\\
       h_2
     \end{array}\right) . \]
  Since $D_0 (1) = 0$, we look for a solution of the form
  \[ \left(\begin{array}{c}
       h_1\\
       h_2
     \end{array}\right) = \left(\begin{array}{c}
       1\\
       0
     \end{array}\right) + \left(\begin{array}{cc}
       D_0^{- 1} D_a^{- 1} & 0\\
       0 & D_{b_2}^{- 1} D_{b_1}^{- 1}
     \end{array}\right) \left( \widetilde{V}_{\lambda} \left(\begin{array}{c}
       h_1\\
       h_2
     \end{array}\right) \right) . \]
  Defining the operator
  \[ \Theta = \left(\begin{array}{cc}
       D_0^{- 1} D_a^{- 1} & 0\\
       0 & D_{b_2}^{- 1} D_{b_1}^{- 1}
     \end{array}\right) (\widetilde{V}_{\lambda} .), \]
  we want to construct
  \begin{equation}
    \left(\begin{array}{c}
      h_1\\
      h_2
    \end{array}\right) = \sum_{k \in \mathbb{N}} \Theta^k \left(
    \left(\begin{array}{c}
      1\\
      0
    \end{array}\right) \right) . \label{trvaux}
  \end{equation}
  We have shown that $a, b_1, b_2$ are all of type $< 0$ in Definition
  \ref{cadmissible}. By Proposition \ref{grooving} and (\ref{dusk0}), we deduce
  that for any $n, m \in \mathbb{N}$, we have for $\lambda \geqslant
  \lambda_0$ that
  \[ \mathcal{N}_{0, 1}^{n, m} \left( D_0^{- 1} D_a^{- 1} \left(
     \widetilde{V}_{\lambda} \left(\begin{array}{c}
       h_1\\
       h_2
     \end{array}\right) \right) \right) \lesssim_{\lambda_0, n, m}
     \mathcal{N}_{0, 0}^{n, m} (h_1) +\mathcal{N}_{0, 0}^{n, m} (h_2) \]
  and
  \[ \mathcal{N}_{0, 2}^{n, m} \left( D_{b_2}^{- 1} D_{b_1}^{- 1} \left(
     \widetilde{V}_{\lambda} \left(\begin{array}{c}
       h_1\\
       h_2
     \end{array}\right) \right) \right) \lesssim_{\lambda_0, n, m}
     \mathcal{N}_{0, 0}^{n, m} (h_1) +\mathcal{N}_{0, 0}^{n, m} (h_2) . \]
In particular, for $n = m = 0$, this means that for any $R \geqslant 1$,
\begin{align*}
& \left\| \Theta \left(\begin{pmatrix} h_1\\ h_2\end{pmatrix}\right) \right\|_{L^{\infty} ([R, + \infty (\times [\lambda_0, + \infty ()} \lesssim_{\lambda_0} \frac{1}{R + \lambda_0} \left\|
     \left(\begin{pmatrix} h_1\\ h_2 \end{pmatrix}\right) \right\|_{L^{\infty} ([R, + \infty (\times [\lambda_0,
     + \infty ()}, \\
& \left\| \Theta \left(\begin{pmatrix} h_1\\ h_2\end{pmatrix}\right) \right\|_{L^{\infty} ([1, + \infty (\times [R, + \infty ()} \lesssim_{\lambda_0} \frac{1}{1 +R} \left\|
     \left(\begin{pmatrix} h_1\\ h_2 \end{pmatrix}\right) \right\|_{L^{\infty} ([1, + \infty (\times [R,
     + \infty ()},
\end{align*}
  and therefore $\Theta$ is a contraction for these norms provided that $R
  \geqslant 1$ is large enough (uniformly in $\lambda \geqslant \lambda_0$).
  This completes the construction of $(h_1, h_2)$ on $(y,\lambda_0) \in [R, + \infty )\times [\lambda_0,\infty)\cup [1, + \infty )\times [R,\infty)$ by
  (\ref{trvaux}).
  
  Similarly, for any $n, m \in \mathbb{N}$, $\Theta$ is a contraction for the
  norm $\mathcal{N}_{0, 0}^{n, m}$ if we restrict to sets of the
  form $[R_{n,m}, + \infty )\times [\lambda_0,+\infty)\cup [1, + \infty )\times [R_{n,m},+\infty)$. This gives the required estimates on
  $(h_1, h_2)$ on this set.
  
  Now, we can extend this construction to the set $(y,\lambda)\in[1,R_{n,m}\times [\lambda_0,R_{n,m}]]$ by Cauchy theory, so that $(h_{1},h_2)$ is defined on the whole set $(y,\lambda) \in [1, + \infty )\times [\lambda_0,\infty)$. Indeed, this is simply because for all $\lambda \in [\lambda_0,R_{n,m}]$, the function $(h_1(\cdot,\lambda), h_2(\cdot,\lambda))$ solves on $[R_{n, m},\infty)$ the differential
  equation (\ref{twoyardstone}), for which all coefficients are smooth, so that the regularity of $(h_1,h_2)$ extends to $[1,R_{n,m}\times [\lambda_0,R_{n,m}]]$.
  
  This completes the proof of the existence of $(h_1, h_2)$ solving
  (\ref{twoyardstone}) with
  \[ \mathcal{N}_{0, 0}^{n, m} \left( \left(\begin{array}{c}
       h_1\\
       h_2
     \end{array}\right) \right) \lesssim_{n, m} 1 \]
  for any $n, m \in \mathbb{N}$. But since
  \[ \left(\begin{array}{c}
       h_1\\
       h_2
     \end{array}\right) = \left(\begin{array}{c}
       1\\
       0
     \end{array}\right) + \Theta \left(\begin{array}{c}
       h_1\\
       h_2
     \end{array}\right), \]
  with the properties shown on $\Theta$, we have in fact that
  \begin{equation}
    \mathcal{N}_{0, 1}^{n, m} (h_1 - 1) +\mathcal{N}_{0, 2}^{n, m} (h_2)
    \lesssim_{n, m, \lambda_0} 1, \label{ejeca}
  \end{equation}
  which completes the proof of
  \[ \left| \partial_{\lambda}^j \partial_y^k \left( \left(\begin{array}{c}
       h_1\\
       h_2
     \end{array}\right) - \left(\begin{array}{c}
       1\\
       0
     \end{array}\right) \right) \right| \lesssim_{j, k, \lambda_0} \frac{1}{(y
     + \xi) y^k \lambda^j} \]
  for $y \geqslant 1, \lambda \geqslant \lambda_0$ and any $j, k \in
  \mathbb{N}$. In particular, since $\xi \rightarrow + \infty$ when $\lambda
  \rightarrow + \infty$, we have
  \[ \left(\begin{array}{c}
       h_1\\
       h_2
     \end{array}\right) = \left(\begin{array}{c}
       1\\
       0
     \end{array}\right) + o_{\lambda \rightarrow + \infty} (1) \]
  and since $\frac{\kappa}{\xi} \rightarrow 1$ when $\lambda \rightarrow +
  \infty$ by Lemma \ref{foryou}, we have for $y_0 \geqslant 1$ that
  \[ \left(\begin{array}{c}
       \varphi_2\\
       \psi_2
     \end{array}\right) (y_0) = K_1 \left( \frac{\kappa}{\xi} y_0 \right)
     \left(\begin{array}{c}
       h_1\\
       h_2
     \end{array}\right) (y_0) = K_1 (y_0) \left(\begin{array}{c}
       1\\
       0
     \end{array}\right) + o_{\lambda \rightarrow + \infty} (1) . \]
  We obtain the estimates on $\varphi_2' (y_0), \psi_2' (y_0)$ similarly. The
  estimates on their derivatives with respect to $\lambda$ are a consequence
  of $\frac{\kappa}{\xi} - 1 \in \mathcal{W}_1 (\lambda_0^+)$ and
  (\ref{ejeca}).
\end{proof}
\subsection{The oscillating solutions near $ + \infty$}

\begin{lem}
  \label{diary}For any $\lambda > 0$, there exists two real valued solutions
  of (\ref{lydmor}) of the form
  \[ \left(\begin{array}{c}
       \varphi_3\\
       \psi_3
     \end{array}\right) + i \left(\begin{array}{c}
       \varphi_4\\
       \psi_4
     \end{array}\right) = (J_1 + i Y_1) \left(\begin{array}{c}
       h_1\\
       h_2
     \end{array}\right) \]
  such that, for any $\lambda_0 > 0$, taking $\lambda \geqslant \lambda_0$ and
  $y \geqslant 1$, we have that for any $j, k \in \mathbb{N}$,
  \[ \left| \partial_{\lambda}^j \partial_y^k \left( \left(\begin{array}{c}
       h_1\\
       h_2
     \end{array}\right) - \left(\begin{array}{c}
       0\\
       1
     \end{array}\right) \right) \right| \lesssim_{\lambda_0, j, k}
     \frac{1}{(\xi + y) \lambda^j y^k} . \]
  Furthermore, given $y_0 \geqslant 1$, we have
  \[ \left(\begin{array}{c}
       \varphi_3 + i \varphi_4\\
       \varphi_3' + i \varphi_4'\\
       \psi_3 + i \psi_4\\
       \psi_3' + i \psi_4'
     \end{array}\right) (y_0) = 
     \left(\begin{array}{c}
       0\\
       0\\
       J_1+i Y_1\\
       J_1'+i Y_1'
     \end{array}\right) (y_0) + o_{\lambda \rightarrow + \infty} (1) \]
  and for any $k \in \mathbb{N}^{\ast}$,
  \[ \left| \partial_{\lambda}^k \left(\begin{array}{c}
       \varphi_3 + i \varphi_4\\
       \varphi_3' + i \varphi_4'\\
       \psi_3 + i \psi_4\\
       \psi_3' + i \psi_4'
     \end{array}\right) (y_0) \right| \lesssim_{\lambda_0, y_0, k}
     \frac{1}{\xi \lambda^k} . \]
\end{lem}

\begin{proof}
  This proof follows closely the one of Lemma \ref{factexp}. We recall that
  equation (\ref{lydmor}) is
  \[ \left(\begin{array}{c}
       \left( \Delta - \left( \frac{\kappa}{\xi} \right)^2 - \frac{1}{y^2}
       \right) \varphi\\
       \left( \Delta + 1 - \frac{1}{y^2} \right) \psi
     \end{array}\right) = \widetilde{V}_{\lambda} \left(\begin{array}{c}
       \varphi\\
       \psi
     \end{array}\right) . \]
  With $H_1 = J_1 + i Y_1$, we look for a solution of the form
  \[ \left(\begin{array}{c}
       \varphi\\
       \psi
     \end{array}\right) = H_1 \left(\begin{array}{c}
       h_1\\
       h_2
     \end{array}\right) . \]
  Since $\left( \Delta + 1 - \frac{1}{y^2} \right) H_1 = 0$,
  dividing the equation by $H_1$, we get
  \[ \left(\begin{array}{c}
       h''_1 + \left( 2 \frac{H_1'}{H_1} + \frac{1}{y} \right) h_1' - \left( 1
       + \left( \frac{\kappa}{\xi} \right)^2 \right) h_1\\
       h''_2 + \left( 2 \frac{H_1'}{H_1} + \frac{1}{y} \right) h_2'
     \end{array}\right) = \widetilde{V}_{\lambda} \left(\begin{array}{c}
       h_1\\
       h_2
     \end{array}\right) . \]
  With $a = 2 \frac{H_1'}{H_1} + \frac{1}{y} \in -2i+ \mathcal{W}^{+
  \infty}_2 (1^+,\lambda_0^+)$ by Lemma \ref{abudhabi}, we write the first equation as
  \[ h''_2 + \left( 2 \frac{H_1'}{H_1} + \frac{1}{y} \right) h_2' = D_a D_0
     (h_2), \]
  and as in the proof of Lemma \ref{factexp}, we now look for $b_1, b_2$ such
  that
  \[ D_{b_1} D_{b_2} (h_1) = h''_1 + a h_1' - \left( 1 + \left(
     \frac{\kappa}{\xi} \right)^2 \right) h_1, \]
  which leads to the system
  \[ b_1 + b_2 = a, \qquad b_2' = b_2^2 - a b_2 - \left( 1 + \left(
     \frac{\kappa}{\xi} \right)^2 \right) . \]

We solve it as in the proof of Lemma \ref{factexp}. We have that $h_{2,+}=\frac{K_1}{J_1+iY_1}$ is a solution to $D_{b_1}D_{b_2}=0$ so we choose $b_2=-\frac{h_{2,+}'}{h_{2,+}}$ and then we get $b_1=a-b_2$. Since $b_2$ is the opposite of the $b_2$ function of Lemma \ref{factexp}, we have by \eqref{computation-b2-exponential}

  \[ (y, \lambda) \rightarrow b_1 - \left( - i - \frac{\kappa}{\xi} \right), \qquad
     b_2 - \left( - i + \frac{\kappa}{\xi} \right) \in \mathcal{W}_{2} (1^+, \lambda_0^+) . \]
  Defining now
  \[ \Theta = \left(\begin{array}{cc}
       D_{b_2}^{- 1} D_{b_1}^{- 1} & 0\\
       0 & D_0^{- 1} D_a^{- 1}
     \end{array}\right) (\widetilde{V}_{\lambda} .) \]
  and we want to construct
  \[ \left(\begin{array}{c}
       h_1\\
       h_2
     \end{array}\right) = \sum_{k \in \mathbb{N}} \Theta^k \left(
     \left(\begin{array}{c}
       0\\
       1
     \end{array}\right) \right). \]
Here at first, since $a$ is of type $i$, we have only, by Proposition \ref{grooving}.3.b and (\ref{dusk0}), that for any $n, m \in \mathbb{N}$, we have for $\lambda \geqslant
  \lambda_0$ that
  \[ \mathcal{N}_{1, 0}^{n, m} \left( D_0^{- 1} D_a^{- 1} \left(
     \widetilde{V}_{\lambda} \left(\begin{array}{c}
       h_1\\
       h_2
     \end{array}\right) \right) \right) \lesssim_{\lambda_0, n, m}
     \mathcal{N}_{0, 0}^{n, m} (h_1) +\mathcal{N}_{0, 0}^{n, m} (h_2) \]
  and
  \[ \mathcal{N}_{0, 2}^{n, m} \left( D_{b_2}^{- 1} D_{b_1}^{- 1} \left(
     \widetilde{V}_{\lambda} \left(\begin{array}{c}
       h_1\\
       h_2
     \end{array}\right) \right) \right) \lesssim_{\lambda_0, n, m}
     \mathcal{N}_{0, 0}^{n, m} (h_1) +\mathcal{N}_{0, 0}^{n, m} (h_2) . \]
Because we still have the decay in $y$ (although without the factor $\xi$), the contraction still holds for the norm $\mathcal{N}_{0, 0}^{n, m}$. Then, using
 \[ \left(\begin{array}{c}
       h_1\\
       h_2
     \end{array}\right) = \left(\begin{array}{c}
       0\\
       1
     \end{array}\right) + \Theta \left(\begin{array}{c}
       h_1\\
       h_2
     \end{array}\right), \]
Using now Proposition \ref{grooving}.3.c, we have the same estimates as in the proof of Lemma \ref{factexp}. We define
  \[ \left(\begin{array}{c}
       \varphi_3\\
       \psi_3
     \end{array}\right) =\mathfrak{R}\mathfrak{e} \left( H_1
     \left(\begin{array}{c}
       h_1\\
       h_2
     \end{array}\right) \right), \left(\begin{array}{c}
       \varphi_4\\
       \psi_4
     \end{array}\right) =\mathfrak{I}\mathfrak{m} \left( H_1
     \left(\begin{array}{c}
       h_1\\
       h_2
     \end{array}\right) \right) \]
     and we can conclude similarly.
\end{proof}

\section{Scattering theory}

\label{sectionscattering}

The aim of this section is to prove the existence, uniqueness, and smooth dependence on $\lambda$ of bounded functions in the kernel of $\mathcal{H}-\lambda$ (generalized eigenfunctions) for any $\lambda \geq 0$, and to prove furthermore that these functions are never in $L^2$. 

\subsection{Properties of $L_0$}\label{ss23n}

\subsubsection{Kernel of $L_0$ and $L_0-2\rho^2$}

We recall that $L_0 = \Delta - \frac{1}{r^2} + (1 - \rho^2)$.

\begin{lem}
  \label{floor}Two linearly independent solutions of the equation $L_0 = 0$
  are $\rho$ and
  \[ P_0 (r) = \rho (r) \int_1^r \frac{d s}{s \rho^2 (s)} . \]
  The function $\rho$ is positive, and
  \[ \rho (r) \approx a r,\qquad P_0 (r) \approx - \frac{\kappa}{r}, \]
  when $r \rightarrow 0$ for some constants $\kappa, a > 0$, while
  \[ \rho (r) \approx 1,\qquad  P_0 (r) \approx \widetilde{\kappa} \ln (r), \]
  when $r \rightarrow + \infty$ for some constant $\widetilde{\kappa} > 0$.
\end{lem}

\begin{proof}
  The equation satisfied by $\rho$ directly yields $L_0 (\rho) = 0$. Now we
  look for the other solution $P$ under the form $P = \rho g$, leading to the
  equation
  \[ \rho g'' + \left( 2 \rho' + \frac{1}{r} \right) g' = 0, \]
  solved by
  \[ g' (r) = \frac{1}{r \rho^2 (r)} \]
  which gives the solution $P_0$. Its behavior when $r \rightarrow 0$ and $r
  \rightarrow + \infty$ follows from the properties of $\rho$ from Lemma
  \ref{rhostuff}.
\end{proof}

\subsubsection{Inversion results for $L_0$}

\begin{lem}
  \label{L0}Consider two functions $P, S$ solving the equation
  \[ L_0 (P) = S. \]

  a) if $P, S$ that are bounded near $r = 0$, then there exists $C_2 \in
  \mathbb{R}$ such that
  \[ P (r) = \rho (r) \left( C_2 + \int_0^r \frac{1}{t \rho^2 (t)} \left(
     \int_0^t s \rho (s) S (s)\dd s\right)\dd t  \right) . \]

  b) if \ $P$ converges to $0$ when $r \rightarrow + \infty$ and $S$ decays
  exponentially fast when $r \rightarrow + \infty$, then
  \[ P (r) = \rho (r) \left( \int_r^{+ \infty} \frac{1}{t \rho^2 (t)} \left(
     \int_t^{+ \infty} s \rho (s) S (s)\dd s\right)\dd t  \right) . \]
\end{lem}

\begin{proof}
  We look at
  \[ P'' + \frac{1}{r} P' - \frac{P}{r^2} + (1 - \rho^2) P = S \]
and change the unknown variable to $P = \rho g$, leading to
  \[ g'' + \left( \frac{2 \rho'}{\rho} + \frac{1}{r} \right) g' =
     \frac{S}{\rho} . \]
  This is equivalent to
  \[ (r \rho^2 g')' = r \rho S. \]
  We deduce that
  \[ r \rho^2 g' (r) = C_1 + \int_1^r s \rho (s) S (s)\dd s\]
  for some $C_1 \in \mathbb{R}$, leading to
  \[ g (r) = C_2 + \int_1^r \frac{1}{t \rho^2 (t)} \left( C_1 + \int_1^t s
     \rho (s) S (s)\dd s\right)\dd t  \]
  for some $C_2 \in \mathbb{R}$. In other words,
  \[ P (r) = \rho (r) \left( C_2 + \int_1^r \frac{1}{t \rho^2 (t)} \left( C_1
     + \int_1^t s \rho (s) S (s)\dd s\right)\dd t  \right) . \]
  If we suppose that $P$ and $S$ are bounded near $r = 0$, since $\frac{1}{t
  \rho^2 (t)} \approx \frac{1}{a^2 t^3}$ when $t \rightarrow 0$ for some
  constant $a > 0$ by Lemma \ref{rhostuff}, we must have that $\left( C_1 +
  \int_1^t s \rho (s) S (s)\dd s\right)$ converges to $0$ when $t \rightarrow
  0$, leading to
  \[ P (r) = \rho (r) \left( C_2 + \int_1^r \frac{1}{t \rho^2 (t)} \left(
     \int_0^t s \rho (s) S (s)\dd s\right)\dd t  \right) . \]
  Since $S$ is bounded near $r = 0$, $t \rightarrow \frac{1}{t \rho^2 (t)}
  \left( \int_0^t s \rho (s) S (s)\dd s\right)$ is bounded near $t = 0$ and
  thus, up to changing the constant $C_2 \in \mathbb{R}$ we can take
  \[ P (r) = \rho (r) \left( C_2 + \int_0^r \frac{1}{t \rho^2 (t)} \left(
     \int_0^t s \rho (s) S (s)\dd s\right)\dd t  \right) . \]

  Now, if instead $P$ converges to $0$ at infinity and $S$ decays
  exponentially fast there, since $\frac{1}{t \rho^2 (t)} \sim \frac{1}{t}$
  when $t \rightarrow + \infty$ by Lemma \ref{rhostuff}, we must have that
  $\left( C_1 + \int_1^t s \rho (s) S (s)\dd s\right)$ converges to $0$ when
  $t \rightarrow \infty$, leading to
  \[ P (r) = \rho (r) \left( C_2 - \int_1^r \frac{1}{t \rho^2 (t)} \left(
     \int_t^{+ \infty} s \rho (s) S (s)\dd s\right)\dd t  \right) . \]
  Since $\rho (r) \rightarrow 1$ when $r \rightarrow + \infty$, we must have
  that $C_2 - \int_1^r \frac{1}{t \rho^2 (t)} \left( \int_t^{+ \infty} s \rho
  (s) S (s)\dd s\right)\dd t $ converges to $0$ when $r \rightarrow + \infty$,
  leading to
  \[ P (r) = \rho (r) \left( \int_r^{+ \infty} \frac{1}{t \rho^2 (t)} \left(
     \int_t^{+ \infty} s \rho (s) S (s)\dd s\right)\dd t  \right) . \]
\end{proof}

\subsubsection{Kernel of $L_0 - 2 \rho^2$}

We recall that the equation $(L_0 - 2 \rho^2) (Q) = 0$ is
\[ \left( \partial_r^2 + \frac{1}{r} \partial_r - \frac{1}{r^2} \right) Q + (1
   - 3 \rho^2) Q = 0. \]
\begin{lem}
  \label{lethargy}There exist two linearly independent solutions of $(L_0 - 2
  \rho^2) (Q) = 0$ denoted $Q_0$ and $\widetilde{Q}_0$, and they satisfy
  \[ Q_0 (r) \approx r,\qquad \widetilde{Q}_0 (r) \approx \frac{1}{2 r} \]
  when $r \rightarrow 0$, and
  \[ Q_0 (r) \approx \frac{C_1}{\sqrt{r}} e^{\sqrt{2} r},\qquad \widetilde{Q}_0 (r)
     \approx \frac{C_2}{\sqrt{r}} e^{- \sqrt{2} r} \]
  \[ \  \]
  when $r \rightarrow + \infty$ for some constants $C_1, C_2 > 0$.
  Furthermore, both $Q_0$ and $\widetilde{Q}_0$ are strictly positive functions on
  $\mathbb{R}^{+ \ast}$.
\end{lem}

\begin{proof}
  {\tmem{Construction of $Q_0${\tmem{}}.}} We do the change of variable $Q = \rho h$, leading to the equation
  \[ \rho h'' + h' \left( 2 \rho' + \frac{\rho}{r} \right) + h \left( \rho'' +
     \frac{\rho'}{r} - \frac{\rho}{r^2} + (1 - 3 \rho^2) \rho \right) = 0. \]
  Since $\rho'' + \frac{1}{r} \rho' - \frac{\rho}{r^2} + (1 - \rho^2) \rho =
  0$, we deduce that
  \[ h'' + h' (\ln (r \rho^2))' - 2 \rho^2 h = 0, \]
  which can be written as
  \begin{equation}
    (r \rho^2 h')' = 2 r \rho^4 h. \label{freedfromdesire}
  \end{equation}
  Consider the solution of this equation with $h (0) = \frac{1}{\rho' (0)}, h'
  (0) = 0$ (we can use arguments similar to the proof of Lemma \ref{zerosol} to show
  that such a solution is well-defined). This leads to
  \[ r \rho^2 (r) h' (r) = \int_0^r 2 t \rho^4 (t) h (t)\dd t  \]
  for any $r \geqslant 0$. This implies that $h, h' \geqslant 0$ on
  $\mathbb{R}^+$. Indeeed, $h (0) = \frac{1}{\rho' (0)} > 0$ and thus if $h,
  h' \geqslant 0$ does not hold, then there exists $R > 0$ such that $h' (R) =
  0$ and $h (r) \geqslant 0$ for all $r \in [0, R]$, which is a contradiction
  with the formula above since $\rho \geqslant 0$. This leads to the fact that
  $h \geqslant \frac{1}{\rho' (0)}$ everywhere, and thus $h' (r) \geqslant
  \kappa r$ for any $r \geqslant 1$ and some $\kappa > 0$, hence $h$ is
  unbounded with $h (r) \geqslant \kappa r^2$ for $r \geqslant 1$. Writing
  $Q_0$ this solution, we have that $Q_0 \geqslant 0$, $Q_0 (r) \sim r$ when
  $r \rightarrow 0$.
  
  \

\noindent
  {\tmem{Construction of $\widetilde{Q}_0$.}}
Going back to $(L_0 - 2 \rho^2) (Q) = 0$, we look for a solution of the
  form $Q_0 g$, leading to the equation on $g$:
  \[ Q_0 g'' + \left( 2 Q_0' + \frac{1}{r} \right) g' = 0, \]
  which is solved by
  \[ g' = \frac{1}{r Q_0^2 (r)} . \]
  Since $Q_0$ is positive, does not vanish except at $r = 0$ and $Q_0 (r)
  \geqslant \kappa r^2$ for $r \geqslant 1$, we can choose
  \[ g (r) = \int_r^{+ \infty} \frac{1}{s Q_0^2 (s)}\dd s\]
  and thus
  \[ \widetilde{Q}_0 (r) = Q_0 (r) \int_r^{+ \infty} \frac{1}{s Q_0^2 (s)} d
     s \]
  satisfies $(L_0 - 2 \rho^2) (\widetilde{Q}_0) = 0$. Since $Q_0$ is strictly
  positive on $\mathbb{R}^{+ \ast}$, so is $\widetilde{Q}_0$. Finally, near $r =
  0$ since $Q_0 (r) \sim r$, we have
  \[ \int_r^{+ \infty} \frac{1}{s Q_0^2 (s)}\dd s\sim \frac{1}{2 r^2} \]
  when $r \rightarrow 0$, and thus $\widetilde{Q}_0 (r) \sim \frac{1}{2 r}$ when
  $r \rightarrow 0$.
  
  \
  
\noindent  {\tmem{Behavior of $\widetilde{Q}_0$ at infinity.}}
  We write the equation satisfied by $\widetilde{Q}_0$ in the form
  \[ \Delta \widetilde{Q}_0 - 2 \widetilde{Q}_0 = \left( \frac{1}{r^2} - 3 (1 -
     \rho^2) \right) \widetilde{Q}_0 . \]
  The solutions of $f'' + \frac{1}{r} f' - 2 f = 0$ are $K_0 \left( \sqrt{2} .
  \right)$ and $I_0 \left( \sqrt{2} . \right)$, where $K_0$ and $I_0$ are the
  modified Bessel functions. We define the operator
  \begin{align*}
T (f) (r) & = \frac{K_0 \left( \sqrt{2} r \right)}{2} \int_r^{+ \infty} t I_0
    \left( \sqrt{2} t \right) \left( \frac{1}{t^2} - 3 (1 - \rho^2) (t)
    \right) f (t)\dd t \\
& \qquad \qquad - \frac{I_0 \left( \sqrt{2} r \right)}{2} \int_r^{+ \infty} t K_0
    \left( \sqrt{2} t \right) \left( \frac{1}{t^2} - 3 (1 - \rho^2) (t)
    \right) f (t)\dd t ,
  \end{align*}
  which is such that if a function $f$ satisfies $f = T (f)$, then $(L_0 - 2
  \rho^2) (f) = 0$ by Lemma \ref{besselclassic}.a (this is the case $c_1 = c_2
  = + \infty$). We define the sequence of functions
  \[ f_0 (r) = K_0 \left( \sqrt{2} r \right) \]
  and then
  \[ f_{n + 1} (r) = T (f_n) (r) . \]
  Let us show first that this sequence is well-defined at least for $r
  \geqslant R_0 > 0$ where $R_0 > 0$ is a large constant. From Lemma
  \ref{besselclassic}.a, we have
\begin{align*}
& K_0 \left( \sqrt{2} r \right) = \sqrt{\frac{\pi}{2 \sqrt{2} r}} e^{-
     \sqrt{2} r} \left( 1 + O_{r \rightarrow + \infty} \left( \frac{1}{r}
     \right) \right) \\
& I_0 \left( \sqrt{2} r \right) = \frac{1}{\sqrt{2 \sqrt{2} \pi r}}
     e^{\sqrt{2} r} \left( 1 + O_{r \rightarrow + \infty} \left( \frac{1}{r}
     \right) \right), 
\end{align*}
  and from Lemma \ref{rhostuff},
  \[ \left| \frac{1}{t^2} - 3 (1 - \rho^2) (t) \right| \lesssim \frac{1}{(1 +
     t)^2} \]
  for $t \geqslant 1$. Therefore, if $r
  \geqslant R_0$,
  \begin{align*}
& \left| K_0 \left( \sqrt{2} r \right) \int_r^{+ \infty} t I_0 \left(
    \sqrt{2} t \right) \left( \frac{1}{t^2} - 3 (1 - \rho^2) (t) \right) f (t)
   \dd t  \right| \\
& \qquad \qquad\lesssim K_0 \left( \sqrt{2} r \right) \int_r^{+ \infty}
    \frac{1}{t^2} \left\| \frac{f}{K_0 \left( \sqrt{2} . \right)}
    \right\|_{L^{\infty} ([R_0, + \infty ))} \lesssim  \frac{K_0 \left( \sqrt{2} r \right)}{R_0} \left\|
    \frac{f}{K_0 \left( \sqrt{2} . \right)} \right\|_{L^{\infty} ([R_0, +
    \infty ))}
  \end{align*}
  and similarly
  \begin{eqnarray*}
    &  & \left| I_0 \left( \sqrt{2} r \right) \int_r^{+ \infty} t K_0 \left(
    \sqrt{2} t \right) \left( \frac{1}{t^2} - 3 (1 - \rho^2) (t) \right) f (t)
   \dd t  \right| \lesssim  \frac{K_0 \left( \sqrt{2} r \right)}{R_0} \left\|
    \frac{f}{K_0 \left( \sqrt{2} . \right)} \right\|_{L^{\infty} ([R_0, +
    \infty ))} .
  \end{eqnarray*}
  We deduce that
  \[ \left\| \frac{T (f)}{K_0 \left( \sqrt{2} . \right)} \right\|_{L^{\infty}
     ([R_0, + \infty ))} \lesssim \frac{1}{R_0} \left\| \frac{f}{K_0 \left(
     \sqrt{2} . \right)} \right\|_{L^{\infty} ([R_0, + \infty ))}, \]
  and thus for $R_0 > 0$ large enough, $f \rightarrow T (f)$ is a contraction
  for the norm $\left\| \frac{.}{K_0 \left( \sqrt{2} . \right)}
  \right\|_{L^{\infty} ([R_0, + \infty ))}$, and the sequence $f_n$ is well
  defined on $[R_0, + \infty )$, with
  \[ f = \sum_{n \in \mathbb{N}} f_n \]
  which is also well defined, solving $(L_0 - 2 \rho^2) (f) = 0$. Since
  \[ \left| \sum_{n \geqslant 1} f_n (r) \right| = \left| T \left( \sum_{n
     \geqslant 0} f_n \right) (r) \right| \lesssim \frac{K_0 \left( \sqrt{2} r
     \right)}{r} \]
  for $r \geqslant 2 R_0$ for instance, we have
  \[ f (r) \approx \sqrt{\frac{\pi}{2 \sqrt{2} r}} e^{- \sqrt{2} r} \]
  when $r \rightarrow + \infty$.
  
  Now, since $Q_0$ and $f$ are two linearly independent solutions of $(L_0 - 2
  \rho^2) = 0$ since they have different behaviours when $r \rightarrow +
  \infty$, and $\widetilde{Q}_0$ also solves this equation, we have
  \[ \widetilde{Q}_0 = \alpha f + \beta Q_0 \]
  for some $\alpha, \beta \in \mathbb{R}$, but since both $\widetilde{Q}_0$ and
  $f$ are bounded at inifnity and not $Q_0$, we deduce that $\beta = 0$, thus
  \[ \widetilde{Q}_0 (r) \approx \frac{C_2}{\sqrt{r}} e^{- \sqrt{2} r} \]
  when $r \rightarrow + \infty$ for some $C \neq 0$.

\

\noindent  {\tmem{Behavior of $Q_0$ at infinity.}} We check that
  \[ Q_0 (r) = \widetilde{Q}_0 (r) \int_0^r \frac{1}{s \widetilde{Q}_0^2 (s)}\dd s, \]
  which gives the behavior of $Q_0 (r)$ when $r \rightarrow + \infty$ from
  the one of $\widetilde{Q}_0$.
\end{proof}

\subsubsection{Inversion results for $L_0 - 2 \rho^2$}

\begin{lem}
  \label{L0-}Consider $Q, S$ two functions solving
  \[ (L_0 - 2 \rho^2) (Q) = S. \]

\begin{itemize}
\item [a)] If both are bounded near $r = 0$, then
  \[ Q (r) = Q_0 (r) \left( C_2 + \int_0^r \frac{1}{t Q_0^2 (t)} \left(
     \int_0^t s Q_0 (s) S (s)\dd s\right)\dd t  \right) \]
  for some $C_2 \in \mathbb{R}$.
\medskip

\item [b)] If $| S (r) | + | Q (r) | \lesssim e^{- a r}$ for $r > 0$ large enough
  for some $a > \sqrt{2}$, then
  \[ Q (r) = Q_0 (r) \int_r^{+ \infty} \frac{1}{t Q_0^2 (t)} \left( \int_t^{+
     \infty} s Q_0 (s) S (s)\dd s\right)\dd t . \]
\end{itemize}
\end{lem}

\begin{proof}
  a) As in the proof of Lemma \ref{L0}, we check that the solutions of $(L_0 - 2
  \rho^2) (Q) = S$ are
  \[ Q (r) = Q_0 (r) \left( C_2 + \int_1^r \frac{1}{t Q_0^2 (t)} \left( C_1 +
     \int_1^t s Q_0 (s) S (s)\dd s\right)\dd t  \right) \]
  for some constants $C_1, C_2 \in \mathbb{R}$. Again, as for the proof of Lemma
  \ref{L0}.a, if $Q$ and $S$ are smooth near $r = 0$, then up to changing $C_2
  \in \mathbb{R}$,
  \[ Q (r) = Q_0 (r) \left( C_2 + \int_0^r \frac{1}{t Q_0^2 (t)} \left(
     \int_0^t s Q_0 (s) S (s)\dd s\right)\dd t  \right) . \]
  b) Since $Q_0 (r) \sim \frac{C_1}{\sqrt{r}} e^{\sqrt{2} r}$ by Lemma
  \ref{lethargy} and $| S (r) | \lesssim e^{- a r}, a > \sqrt{2}$, we have
  that $s Q_0 (s) S (s) \in L^1 ([1, + \infty ))$, and therefore
  \[ \frac{1}{t Q_0^2 (t)} \left( C_1 + \int_1^t s Q_0 (s) S (s)\dd s\right)
     \in L^1 ([1, + \infty )) \]
  leading to
  \[ Q (r) = Q_0 (r) \left( C_2 + \int_1^{+ \infty} \frac{1}{t Q_0^2 (t)}
     \left( C_1 + \int_1^t s Q_0 (s) S (s)\dd s\right)\dd t +o(1) \right) \]
  when $r \rightarrow + \infty$, but since $Q (r) \rightarrow 0$ while $Q_0
  (r) \rightarrow + \infty$ when $r \rightarrow + \infty$, we deduce that $C_2
  + \int_1^{+ \infty} \frac{1}{t Q_0^2 (t)} \left( C_1 + \int_1^t s Q_0 (s) S
  (s)\dd s\right)\dd t  = 0$ and thus
  \[ Q (r) = Q_0 (r) \int_r^{+ \infty} \frac{1}{t Q_0^2 (t)} \left( C_1 +
     \int_1^t s Q_0 (s) S (s)\dd s\right)\dd t . \]
  We check then, using $Q_0 (r) \sim \frac{C_1}{\sqrt{r}} e^{\sqrt{2} r}$ by
  Lemma \ref{lethargy}, that
  \[ Q (r) = \widetilde{Q}_0 (r) \left( C_1 + \int_1^{+ \infty} s Q_0 (s) S (s)
    \dd s +o(1)\right) \]
  when $r \rightarrow + \infty$, but since $Q$ decays at an exponential rate
  faster than $\widetilde{Q}_0$, we must have
  \[ C_1 + \int_1^{+ \infty} s Q_0 (s) S (s)\dd s= 0, \]
  concluding the proof.
\end{proof}

\subsection{Two particular solutions}\label{ss24n}

\subsubsection{Construction of $(u_1, v_1)$}

\begin{lem}
  \label{const1}For any $\lambda > 0$, there exists a pair of functions $(u_1,
  v_1)$ solving $ \mathcal{H}(u_1,v_1)=\lambda(u_1,v_1)$ with the estimates
  \begin{align*}
& | u_1 (r) | + | v_1 (r) | \gtrsim_{\lambda}  \frac{e^{\sqrt{2}
     r}}{\sqrt{r}}  \qquad \mbox{for $|r| \geq 1$} \\
& | u_1 (r) | + | v_1 (r) | \lesssim_{\lambda} r \qquad \;\;\; \mbox{for $|r| \leq 1$}. 
  \end{align*} 
\end{lem}

\begin{proof}
  We recall from (\ref{maineq2}) that $\mathcal H (u_1,v_1)=\lambda(u_1,v_1)$ if and only if
  \[ \left\{\begin{array}{l}
       L_0 (u_1 - v_1) = - \lambda (u_1 + v_1)\\
       (L_0 - 2 \rho^2) (u_1 + v_1) = - \lambda (u_1 - v_1) .
     \end{array}\right. \]
  We consider the solution of this equation for $(u_1 - v_1) (0) = 0, (u_1 -
  v_1)' (0) = 0, (u_1 + v_1) (0) = 0, (u_1 + v_1)' (0) = - 1$. By Lemma
  \ref{L0}.a we have
  \[ (u_1 - v_1) (r) = - \lambda \rho (r) \int_0^r \frac{1}{t \rho^2 (t)}
     \left( \int_0^t s \rho (s) (u + v) (s)\dd s\right)\dd t , \]
  and by Lemma \ref{L0-}.a (since $Q_0' (0) = - 1$, we have $C_1 = - 1$),
  \[ (u_1 + v_1) (r) = - Q_0 (r) \left( 1 + \lambda \int_0^r \frac{1}{t Q_0^2
     (t)} \left( \int_0^t s Q_0 (s) (u - v) (s)\dd s\right)\dd t  \right) . \]
    We claim that $u_1+v_1<0$ and $u_1-v_1>0$. Indeed, since $(u_1 + v_1) (0) = 0, (u_1 + v_1)' (0) =
  - 1$ we have $(u_1 + v_1) (r) < 0$ for small $r > 0$ and thus $(u_1 - v_1)
  (r) > 0$ for $r > 0$ small, since $\rho$ is strictly positive on
  $\mathbb{R}^{+ \ast}$. Assume by contradiction that $r_{\ast} > 0$ is the first value where either $(u_1 -
  v_1) (r_{\ast}) = 0$ or $(u_1 + v_1) (r_{\ast}) = 0$. Since $u_1 - v_1 > 0,
  u_1 + v_1 < 0$ for $r < r_{\ast}$ and $\rho, Q_0$ are strictly positive on
  $\mathbb{R}^{+ \ast}$, we have $(u_1 -
  v_1) (r_{\ast})> 0$ and $(u_1 +
  v_1) (r_{\ast})< 0$ from the above formulas, a contradiction.
  
  In particular, $(u_1 + v_1) (r) \leqslant - Q_0 (r)$, and using Lemma
  \ref{lethargy} we check easily that $(u_1 - v_1) (r) \gtrsim_{\lambda} Q_0
  (r)$ for $r \geqslant 1$.
\end{proof}

\subsubsection{Construction of $(u_2, v_2)$}

\begin{lem}
  \label{const2}For any $\lambda > 0$, there exists a pair of functions $(u_2,
  v_2)$ solving
  \[ \mathcal{H} \left(\begin{array}{c}
       u_2\\
       v_2
     \end{array}\right) = \lambda \left(\begin{array}{c}
       u_2\\
       v_2
     \end{array}\right) \]
  that satisfy
\begin{align*}
& | u_2 (r) | + | v_2 (r) | \lesssim_{\lambda}  \frac{e^{- \kappa r}}{\sqrt{r}} \qquad \mbox{for $r \geq 1$} \\
& | u_2 (r) | + | v_2 (r) | \gtrsim_{\lambda} \frac{1}{r} \qquad \;\;\;\;\;\;  \mbox{for $r \leq 1$}.
\end{align*}
\end{lem}

\begin{proof} \underline{Behavior at $+\infty$.}
  By Lemma \ref{pullyou}, $(u_2, v_2)$ solves $\mathcal{H}- \lambda = 0$ if
  and only if
\[ \left(\begin{array}{c}
       \varphi_2\\
       \psi_2
     \end{array}\right) = \left(\begin{array}{c}
       \frac{u_2}{a_+ (\lambda)} + v_2\\
       - u_2 - a_- (\lambda) v_2
     \end{array}\right) \]
  solves
  \[ \left(\begin{array}{c}
       (\Delta - \kappa^2) \varphi_2\\
       (\Delta + \xi^2) \psi_2
     \end{array}\right) + V_{\lambda} \left(\begin{array}{c}
       \varphi_2\\
       \psi_2
     \end{array}\right) = 0. \]
We take the solution $(\varphi_2,\psi_2)$ of Lemma \ref{factexp}. In particular, $u_2, v_2$ decays exponentially fast at a rate
  $\kappa$ near $r = + \infty$
  
\medskip

\noindent {\tmem{Behavior at $0$.}}
By Lemma \ref{pullyou},
  \[ \left\{\begin{array}{l}
       L_0 (u_2 - v_2) = - \lambda (u_2 + v_2)\\
       (L_0 - 2 \rho^2) (u_2 + v_2) = - \lambda (u_2 - v_2) .
     \end{array}\right. \]
  Furthermore, since
  \[ \kappa = \sqrt{1 + \langle \lambda \rangle} > \sqrt{2} \]
  for $\lambda > 0$, by Lemma \ref{L0-}.b and the behavior of $u_2, v_2$ at
  infinity, we have
  \[ (u_2 + v_2) (r) = - \lambda Q_0 (r) \int_r^{+ \infty} \frac{1}{t Q_0^2
     (t)} \left( \int_t^{+ \infty} s Q_0 (s) (u_2 - v_2) (s)\dd s\right)\dd t .
  \]
  Similarly, by Lemma \ref{L0}.b we have
  \[ (u_2 - v_2) (r) = - \lambda \rho (r) \left( \int_r^{+ \infty} \frac{1}{t
     \rho^2 (t)} \left( \int_t^{+ \infty} s \rho (s) (u_2 + v_2) (s)\dd s
     \right)\dd t  \right) . \]
  Since $\varphi_2=K_1(\frac{\kappa}{\xi}r)(1+O(r^{-1}))$ and $\psi_2=K_1(\frac{\kappa}{\xi}r)O(r^{-1})$ as $r\to \infty$ by Lemma \ref{factexp}, and $u_2+v_2=\frac{a_+}{a_++a_-}[(a_-+1)\varphi_2+(\frac{1}{a_+}-1)\psi_2]$ with $a_-+1\neq0$ for all $\lambda>0$, we have that $u_2 + v_2 > 0$ for $r$ large enough.
  Since $\rho > 0$, we have that $u_2 - v_2 < 0$ for $r$ large enough
  
  Suppose that there exists $r_{\ast} > 0$ such that either $(u_2 + v_2)
  (r_{\ast}) = 0$ or $(u_2 - v_2) (r_{\ast}) = 0$ and $r_{\ast}$ is the
  largest value where this happens. Then, by the above two formulae since
  $\rho$ and $Q_0$ are strictly positive, we have $(u_2 + v_2) (r_{\ast}) > 0$
  and $(u_2 - v_2) (r_{\ast}) < 0$, leading to a contradiction.
  
  Now, using once again the fact that $Q_0 > 0$, since $Q_0 (r) \approx r$
  when $r$ is small, we deduce that
  \[ (u_2 + v_2) (r) = -\frac{\lambda}{2r} \int_0^{+ \infty} s Q_0 (s)
     (u_2 - v_2) (s)\dd s +O(1)\]
  when $r \rightarrow 0$, and $\int_0^{+ \infty} s Q_0 (s) (u_2 - v_2) (s)\dd s
  \neq 0$ since both $Q_0$ and $u_2 - v_2$ have constant signs and are not
  identically $0$. A similar result holds for $(u_2 - v_2) (r)$.
\end{proof}

\subsection{Description of eigenspaces and proof of Theorem \ref{bug}}\label{smn}

\begin{proof}
  {\tmem{The case $\lambda = 0$.}}
  By Lemma \ref{floor}, $P \in \{ \rho, P_0 \}$ solves $L_0 (P) = 0$, from
  which we deduce that
  \[ \mathcal{H} \left(\begin{array}{c}
       P\\
       - P
     \end{array}\right) = 0. \]
  By Lemma \ref{lethargy}, $Q \in \{ Q_0, \widetilde{Q}_0 \}$ solves $(L_0 - 2
  \rho^2) (Q) = 0$, from which we deduce that
  \[ \mathcal{H} \left(\begin{array}{c}
       Q\\
       Q
     \end{array}\right) = 0. \]
  We have therefore found 4 independent solutions of
  \[ \mathcal{H} \left(\begin{array}{c}
       u\\
       v
     \end{array}\right) = 0, \]
  which is a differential equation of order 4, thus any solution in
  $\mathcal{E}_0$ must be of the form
  \[ \left(\begin{array}{c}
       u\\
       v
     \end{array}\right) = \alpha_1 \left(\begin{array}{c}
       \rho\\
       - \rho
     \end{array}\right) + \alpha_2 \left(\begin{array}{c}
       P_0\\
       - P_0
     \end{array}\right) + \alpha_3 \left(\begin{array}{c}
       Q_0\\
       Q_0
     \end{array}\right) + \alpha_4 \left(\begin{array}{c}
       \widetilde{Q}_0\\
       \widetilde{Q}_0
     \end{array}\right) \]
  where $\alpha_1, \alpha_2, \alpha_3, \alpha_4 \in \mathbb{C}$. With the
  asymptotic behaviours of $\rho, P_0, Q_0$ and $\widetilde{Q}_0$ we check that
  the only way to make $u, v \in L^2$ is to take $\alpha_1 = \alpha_2 =
  \alpha_3 = \alpha_4 = 0,$ and the only way to make $u, v \in L^{\infty}$ is
  to take $\alpha_2 = \alpha_3 = \alpha_4 = 0$.
  
\medskip
  
\noindent  {\tmem{The case $\lambda > 0$.}} By Lemma \ref{zerosol}, the vector
  space of solutions of $\mathcal{H}- \lambda = 0$ can be written as a sum of
  two $\mathbb{C}$-vector spaces of dimensions 2, where for one of them
  solutions behaves like $\frac{\kappa}{r}$ near $r = 0$ for $\kappa \in
  \mathbb{C}^2$, and one where solutions behaves like $\kappa r$ near $r = 0$
  for $\kappa \in \mathbb{C}^2$.
  
  Recall that we have defined $(u_1, v_1)$ and $(u_2, v_2)$ in Lemmas
  \ref{const1} and \ref{const2}. Now, we define
  \[ \left(\begin{array}{c}
       u_j\\
       v_j
     \end{array}\right) (r) = \left(\begin{array}{cc}
       \frac{1}{a_+ (\lambda)} & 1\\
       - 1 & - a_- (\lambda)
     \end{array}\right)^{- 1} \left(\begin{array}{c}
       \varphi_j\\
       \psi_j
     \end{array}\right) (\xi r) \]
  for $j \in \{ 3, 4 \}$ and $\xi r \geqslant 1$ where $\varphi_j, \psi_j$
  are defined in Lemma \ref{diary}. By Lemma \ref{pullyou}, we have
  \[ \mathcal{H} \left(\begin{array}{c}
       u_j\\
       v_j
     \end{array}\right) = \lambda \left(\begin{array}{c}
       u_j\\
       v_j
     \end{array}\right) \]
  In particular, we can extend these functions on $r \in \mathbb{R}^{+ \ast}$
  as solution of this equation.
  
  Now, near $r = + \infty$, $(u_1, v_1), (u_2, v_2), (u_3, v_3)$ and
  $(u_4, v_4)$ have different behaviours when $r \rightarrow + \infty$, and
  thus they are linearly independent. Since $\mathcal{H}- \lambda = 0$ can be
  seen as an equation of order four, all solutions of the equation must be a
  linear combination of these four solutions.
  
  Among them, by Lemma \ref{const1}, $(u_1, v_1)$ grows exponentially fast
  when $r \rightarrow + \infty$ and behaves like $\kappa r$ when $r
  \rightarrow 0$ for some $\kappa \in \mathbb{C}^2$. By Lemma \ref{const2},
  $(u_2, v_2)$ decays exponentially fast when $r \rightarrow + \infty$ and
  behaves like $\frac{\kappa}{r}$ when $r \rightarrow 0$ for some $\kappa \in
  \mathbb{C}^2$.
  
  We deduce that there exists a linear combination of $(u_2, v_2), (u_3, v_3)$
  and $(u_4, v_4)$ that behaves like $\kappa r$ when $r \rightarrow 0$, and
  that any linear combination of these three solutions that behaves like
  $\kappa r$ when $r \rightarrow 0$ must be colinear to this one. Indeed, if
  no linear combination of them behaves like $\kappa r$ when $r \rightarrow
  0$, that means that $\tmop{Vect}_{\mathbb{C}} ((u_2, u_2), (u_3, v_3), (u_4,
  v_4))$ is a subspace of dimension 3 where all solutions in it behaves like
  $\frac{\kappa}{r}$ near $r = 0$, and thus is included in a space of
  dimension 2, which is a contradiction.
  
  If two such linear independent combinaison exists, denoting them $(U_1,
  V_1)$ and $(U_2, V_2)$, then $\tmop{Vect}_{\mathbb{C}} ((u_1, u_1), (U_1,
  V_1), (U_2, V_2))$ is a subspace of dimension $3$ where all solutions in it
  behaves like $\kappa r$ near $r = 0$, and this is included in a space of
  dimension 2, which is a contradiction.
  
  We denote by $(u_{\lambda}, v_{\lambda})$ a nonzero linear combination of
  $(u_2, v_2), (u_3, v_3)$ and $(u_4, v_4)$ that behaves like $\kappa r$ when
  $r \rightarrow 0$. This solution is then in $L^{\infty} (\mathbb{R}^+,
  \mathbb{C})$ as it is bounded near $r = 0$ and $(u_2, v_2), (u_3, v_3),
  (u_4, v_4)$ are all continuous on $\mathbb{R}^{+ \ast}$ and bounded near $r
  = + \infty$. Any bounded solution of $\mathcal{H}- \lambda = 0$ on
  $\mathbb{R}^{+ \ast}$ must be colinear to $(u_{\lambda}, v_{\lambda})$, as
  it can not have a component on $(u_1, v_1)$ which grows exponentially fast
  when $r \rightarrow + \infty$, and thus it must be in
  $\tmop{Vect}_{\mathbb{C}} ((u_2, v_2), (u_3, v_3), (u_4, v_4))$, where only
  a subspace of dimension 1 contains the solutions that does not blow up near
  $r = 0$.
  
  Finally, suppose that $(u, v) \in (L^2 (\mathbb{R}^+, \mathbb{C}))^2$ solves
  $\mathcal{H}- \lambda = 0$. Then
  \[ \left(\begin{array}{c}
       u\\
       v
     \end{array}\right) = \alpha_1 \left(\begin{array}{c}
       u_1\\
       v_1
     \end{array}\right) + \alpha_2 \left(\begin{array}{c}
       u_2\\
       v_2
     \end{array}\right) + \alpha_3 \left(\begin{array}{c}
       u_3\\
       v_3
     \end{array}\right) + \alpha_4 \left(\begin{array}{c}
       u_4\\
       v_4
     \end{array}\right) \]
  for some $\alpha_1, \alpha_2, \alpha_3, \alpha_4 \in \mathbb{C}$. Since
  $(u_1, v_1)$ grows exponentially fast when $r \rightarrow + \infty$ while
  the other elements of the base are bounded there, we deduce that $\alpha_1 =
  0$.
  
  Now, $(u_3, v_3), (u_4, v_4)$ are of size $\frac{1}{\sqrt{r}}$ when $r
  \rightarrow + \infty$, and are both linearly independent and not belonging
  in $L^2$. Since $(u_2, v_2)$ decays exponentially fast when $r \rightarrow +
  \infty$, we deduce that $\alpha_3 = \alpha_4 = 0$. But then, $(u_2, v_2)$
  behaves like $\frac{\kappa}{r}$ near $r = 0$ for some $\kappa \neq 0$, which
  is not in $L^2$, and thus $\alpha_2 = 0$. We deduce that $(u, v) = 0$.
\end{proof}

\subsection{Smoothness of $\lambda \rightarrow (u_{\lambda},
v_{\lambda})$}\label{ss27n}
We recall that for $\lambda \geqslant 0$,
\[ e (\lambda) = \frac{1}{\sqrt{2 \left( 1 + \lambda^2 + \lambda
   \langle \lambda \rangle \right)}} \left(\begin{array}{c}
     \lambda + \langle \lambda \rangle\\
     - 1
   \end{array}\right) \]
is such that
\[ | e (\lambda) | = 1, \]
is smooth with respect to $\lambda\in \mathbb R$ and
\[ e (0) = \frac{1}{\sqrt{2}} \left(\begin{array}{c}
     1\\
     - 1
   \end{array}\right), \qquad e (+ \infty) = \left(\begin{array}{c}
     1\\
     0
   \end{array}\right) . \]
\begin{lem}
  \label{london}There exists
  \[ \lambda \rightarrow \left(\begin{array}{c}
       u_{\lambda}\\
       v_{\lambda}
     \end{array}\right) \in \mathcal{C}^\infty (\mathbb{R}^{+ \ast},
     \mathcal{C}^\infty_{\tmop{loc}} (\mathbb{R}^+, \mathbb{R})^2) \]
  such that for all $\lambda > 0$,
  \[ \mathcal{E} \cap (L^{\infty} (\mathbb{R}^+, \mathbb{C}))^2 =
     \tmop{Span}_{\mathbb{C}} \left( \left(\begin{array}{c}
       u_{\lambda}\\
       v_{\lambda}
     \end{array}\right) \right) . \]
  The vector $\left(\begin{array}{c}
    u_{\lambda}\\
    v_{\lambda}
  \end{array}\right)$ is normalized such that
  \[ \left(\begin{array}{c}
       u_{\lambda}\\
       v_{\lambda}
     \end{array}\right) (r) = \frac{\gamma_1 (\lambda) \sin (\xi r) +
     \gamma_2 (\lambda) \cos (\xi r)}{\sqrt{\xi r}} e(\lambda) + O_{r \rightarrow + \infty} \left( \frac{1}{r} \right) . \]
  where $\lambda \rightarrow \gamma_1 (\lambda), \gamma_2 (\lambda) \in
  \mathcal{C}^\infty (\mathbb{R}^{+ \ast}, \mathbb{R})$ with $\gamma_1^2 (\lambda) +
  \gamma_2^2 (\lambda) = 1$ for all $\lambda > 0$.
\end{lem}

\begin{proof}
  With the functions $(u_2, v_2), (u_3, v_3)$ and $(u_4, v_4)$ defined in
  lemmas \ref{factexp} and \ref{diary}, denoting $(u_{\lambda}, v_{\lambda})$ a nontrivial solution
  of $\mathcal{H}- \lambda = 0$ belonging in $(L^{\infty} (\mathbb{R}^+,
  \mathbb{C}))^2$ which exists by Theorem \ref{bug}, there exists
  $\alpha_1, \alpha_2, \alpha_3 \in \mathbb{R}$ such that
  \[ \left(\begin{array}{c}
       u_{\lambda}\\
       v_{\lambda}
     \end{array}\right) = \alpha_1 \left(\begin{array}{c}
       u_2\\
       v_2
     \end{array}\right) + \alpha_2 \left(\begin{array}{c}
       u_3\\
       v_3
     \end{array}\right) + \alpha_3 \left(\begin{array}{c}
       u_4\\
       v_4
     \end{array}\right) . \]
Furthermore, since the functions $u_2, u_3, u_4, v_2, v_3, v_4$ are real-valued, so are the coefficients $\alpha_1, \alpha_2, \alpha_3$.
  
By Lemma \ref{zerosol} we denote $(U_1, V_1)$ and $(U_2, V_2)$ the
  two real valued solutions of $\mathcal{H}- \lambda = 0$ with
  \[ \left(\begin{array}{c}
       U_1\\
       V_1
     \end{array}\right) (r) = \left(\begin{array}{c}
       1\\
       0
     \end{array}\right) r + O_{r \rightarrow 0} (r^2), \qquad \left(\begin{array}{c}
       U_2\\
       V_2
     \end{array}\right) (r) = \left(\begin{array}{c}
       0\\
       1
     \end{array}\right) r + O_{r \rightarrow 0} (r^2) . \]
  These functions are uniquely determined by this condition near $r = 0$, and
  they are linearly independent. Since $(u_{\lambda}, v_{\lambda}) = \kappa r
  + O_{r \rightarrow 0} (r^2)$ near $r = 0$ for some $\kappa \in
  \mathbb{R}^2$, we have
  \[ \left(\begin{array}{c}
       u_{\lambda}\\
       v_{\lambda}
     \end{array}\right) = \beta_1 \left(\begin{array}{c}
       U_1\\
       V_1
     \end{array}\right) + \beta_2 \left(\begin{array}{c}
       U_2\\
       V_2
     \end{array}\right) \]
  for some $\beta_1, \beta_2 \in \mathbb{R}$. Matching these solutions in a
  $C^1$ way at $r = 1$, we deduce that
  \begin{eqnarray*}
\beta_1 \left(\begin{array}{c}
      U_1 (1)\\
      U_1' (1)\\
      V_2 (1)\\
      V_2' (1)
    \end{array}\right) + \beta_2 \left(\begin{array}{c}
      U_2 (1)\\
      U_2' (1)\\
      V_2 (1)\\
      V_2' (1)
    \end{array}\right) 
= \left(\begin{array}{c}
      u_{\lambda} (1)\\
      u_{\lambda}' (1)\\
      v_{\lambda} (1)\\
      v_{\lambda}' (1)
    \end{array}\right)
=  \alpha_1 \left(\begin{array}{c}
      u_2 (1)\\
      u_2' (1)\\
      v_2 (1)\\
      v_2' (1)
    \end{array}\right) + \alpha_2 \left(\begin{array}{c}
      u_3 (1)\\
      u_3' (1)\\
      v_3 (1)\\
      v_3' (1)
    \end{array}\right) + \alpha_3 \left(\begin{array}{c}
      u_4 (1)\\
      u_4' (1)\\
      v_4 (1)\\
      v_4' (1)
    \end{array}\right)
  \end{eqnarray*}
  and thus
  \begin{equation}
    A_{\lambda} \left(\begin{array}{c}
      \alpha_1\\
      \alpha_2\\
      \alpha_3\\
      \beta_1\\
      \beta_2
    \end{array}\right) = 0 \label{myrne}
  \end{equation}
  where
  \[ A_{\lambda} = \left(\begin{array}{ccccc}
       u_2 (1) & u_3 (1) & u_4 (1) & U_1 (1) & U_2 (1)\\
       u_2' (1) & u_3' (1) & u_4' (1) & U_1' (1) & U_2' (1)\\
       v_2 (1) & v_3 (1) & v_4 (1) & V_1 (1) & V_2 (1)\\
       v_2' (1) & v_3' (1) & v_4' (1) & V_1' (1) & V_2' (1)
     \end{array}\right) \in M_{4, 5} (\mathbb{R}) . \]
  By Proposition \ref{bug}, for any $\lambda \geqslant 0$, equation
  (\ref{myrne}) admits a one-dimensional space of solutions. Taking a nonzero
  element of this space, we have $\beta_1^2 + \beta_2^2 \neq 0$ as otherwise
  $(u_{\lambda}, v_{\lambda}) = 0$, and we have $\alpha_2^2 + \alpha_3^2 \neq
  0$, as otherwise $(u_{\lambda}, v_{\lambda})$ is colinear to $(u_2, v_2)$
  which does not belong to $L^{\infty} (\mathbb{R}^+, \mathbb{R})$.
  
  Now, we argue that $\lambda \rightarrow A_{\lambda} \in C^2 (\mathbb{R}^{+
  \ast}, M_{4, 5} (\mathbb{R}))$. Indeed, the functions $U_1, V_1, U_2, V_2$
  are constructed as solutions of a Cauchy problem where $\lambda$ appears as
  a parameter in a $\mathcal{C}^\infty$ way in the coefficients of the equation, thus
  there values and their derivatives at $r = 1$ are $\mathcal{C}^\infty$ functions of
  $\lambda$.
  
  By Lemmas \ref{factexp} and \ref{diary},
  \[ \lambda \rightarrow u_2, v_2, u_3, v_3, u_4, v_4 \in \mathcal{C}^\infty
     (\mathbb{R}^{+ \ast}, \mathcal{C}^\infty_{\tmop{loc}} (] 0, + \infty ),
     \mathbb{R})^2), \]
  hence their values and the values of their derivatives at $r = 1$ are
  $\mathcal{C}^\infty$ functions of $\lambda$. By Proposition \ref{bug}, for all
  $\lambda \geqslant 0$, $\dim (\tmop{Ker} (A_{\lambda})) = 1$, but since all
  of its coefficient are smooth with respect to $\lambda$, we can choose
  $(\alpha_1, \alpha_2, \alpha_3, \beta_1, \beta_2) = X_{\lambda} \in
  \tmop{Ker} (A_{\lambda}) \subset \mathbb{R}^5$ such that $X_{\lambda} \neq
  0$ for all $\lambda > 0$ and $\lambda \rightarrow X_{\lambda} \in \mathcal{C}^\infty
  (\mathbb{R}^{+ \ast}, \mathbb{R}^5)$. We then normalize it by imposing the
  value of $\alpha_2^2 + \alpha_3^2 \neq 0$ (which is always possible since
  $\alpha_2 = \alpha_3 = 0$ imposes $\tmop{Ker} (A_{\lambda})$ colinear to
  $(u_2, v_2) \nin L^{\infty}$). Writing then
  \[ \left(\begin{array}{c}
       u_{\lambda}\\
       v_{\lambda}
     \end{array}\right) = \alpha_1 \left(\begin{array}{c}
       u_2\\
       v_2
     \end{array}\right) + \alpha_2 \left(\begin{array}{c}
       u_3\\
       v_3
     \end{array}\right) + \alpha_3 \left(\begin{array}{c}
       u_4\\
       v_4
     \end{array}\right), \]
  we have that $u_2, v_2$ decay exponentially fast when $r \rightarrow +
  \infty$ while by Lemma \ref{diary},
  \begin{align*}
\left(\begin{array}{c}
      u_3\\
      v_3
    \end{array}\right) (r) & = \left(\begin{array}{cc}
      \frac{1}{a_+ (\lambda)} & 1\\
      - 1 & - a_- (\lambda)
    \end{array}\right)^{- 1} \left(\begin{array}{c}
      \varphi_3\\
      \psi_3
    \end{array}\right) (\xi r) \\
& = \frac{\cos (\xi r)}{\sqrt{\xi r}} \left(\begin{array}{cc}
      \frac{1}{a_+ (\lambda)} & 1\\
      - 1 & - a_- (\lambda)
    \end{array}\right)^{- 1} \left(\begin{array}{c}
      0\\
      1
    \end{array}\right) + o_{r \rightarrow + \infty} \left( \frac{1}{\sqrt{r}}
    \right)\\
& =  \frac{\cos (\xi r)}{\sqrt{\xi r}} \left( \frac{- 1}{a_+
    (\lambda) - a_- (\lambda)} \right) \left(\begin{array}{c}
      \lambda + \langle \lambda \rangle\\
      - 1
    \end{array}\right) + o_{r \rightarrow + \infty} \left( \frac{1}{\sqrt{r}}
    \right)
  \end{align*}
  and similarly for $u_4, v_4$, simply replacing $\cos (\xi r)$ by $\sin
  (\xi r)$. The choice of the normalization (i.e. the value of $\alpha_2^2
  + \alpha_3^2$) such that
  \[ \left(\begin{array}{c}
       u_{\lambda}\\
       v_{\lambda}
     \end{array}\right) = \frac{\gamma_1 (\lambda) \cos (\xi r) + \gamma_2
     (\lambda) \sin (\xi r)}{\sqrt{\xi r}} e(\lambda) \]
  is therefore $\mathcal{C}^\infty$ with respect to $\lambda$, concluding the proof.
\end{proof}

\section{Generalized eigenfunctions in the high frequency limit $\lambda \rightarrow \infty$}\label{ss28n}

\label{sectionhighfreq}

The aim of this subsection is to prove quantitative bounds on $(u_\lambda,v_\lambda)$ in the limit $\lambda \to \infty$. We shall rely on the variables of Lemma \ref{foryou}, that is $y =
\xi r$ and we write the equation as
\begin{equation}
  \left(\begin{array}{c}
    \left( \Delta - \left( \frac{\kappa}{\xi} \right)^2 -
    \frac{1}{y^2} \right) \varphi\\
    \left( \Delta + 1 - \frac{1}{y^2} \right) \psi
  \end{array}\right) (y) = \widetilde{V}_{\lambda} \left(\begin{array}{c}
    \varphi\\
    \psi
  \end{array}\right) \label{lydmor2}
\end{equation}
where $\widetilde{V}_{\lambda}$ is defined in (\ref{aname}).

\subsection{The solutions near $y = 0$}

\begin{lem}
  \label{sicuro}We define $\Phi_1, \Psi_1$ the solution of the Cauchy problem
  (\ref{lydmor2}) with the initial condition
  \[ \left(\begin{array}{c}
       \Phi_1\\
       \Psi_1
     \end{array}\right) (y) = \frac{1}{2}\left(\begin{array}{c}
       1\\
       0
     \end{array}\right) y + O_{y \rightarrow 0} (y^3) \]
  and $\Phi_2, \Psi_2$ the solution where
  \[ \left(\begin{array}{c}
       \Phi_2\\
       \Psi_2
     \end{array}\right) (y) =  \frac{1}{2}\left(\begin{array}{c}
       0\\
       1
     \end{array}\right) y + O_{y \rightarrow 0} (y^3) . \]
  Then, for any $y_0 > 0,$ $\lambda \rightarrow \Phi_1, \Psi_1, \Phi_2, \Psi_2
  \in \mathcal{C}^\infty (\mathbb{R}^{+ \ast}, \mathcal{C}^\infty ([0, y_0], \mathbb{R}))$
  with, for any $\lambda \geqslant \lambda_0>0, j,k \in \mathbb{N}$,
\begin{align*}
& \left\| \partial_{\lambda}^j \partial_y^k \left( \left(\begin{array}{c}
       \Phi_1\\
       \Psi_1
     \end{array}\right) - \left(\begin{array}{c}
       I_1\\
       0
     \end{array}\right) \right) \right\|_{C^1 ([0, y_0], \mathbb{R})}
     \lesssim_{j, y_0,\lambda_0,k} \frac{1}{\lambda^{j + 1}} \\
& \left\| \partial_{\lambda}^j \partial_y^k \left( \left(\begin{array}{c}
       \Phi_2\\
       \Psi_2
     \end{array}\right) - \left(\begin{array}{c}
       0\\
       J_1
     \end{array}\right) \right) \right\|_{C^1 ([0, y_0], \mathbb{R})}
     \lesssim_{j, y_0,\lambda_0,k} \frac{1}{\lambda^{j + 1}}.
\end{align*} 
\end{lem}

\begin{proof}
  Up to the change of variable, $\left(\begin{array}{c}
    \Phi_1\\
    \Psi_1
  \end{array}\right)$ and $\left(\begin{array}{c}
    \Phi_2\\
    \Psi_2
  \end{array}\right)$ are linear combinations of the functions constructed as
  a base of $\mathcal{E}_1 =\mathcal{E}_{\lambda} \cap L^{\infty} ([0, 1],
  \mathbb{R})^2$ in Lemma \ref{zerosol}.
  
  By Lemma \ref{foryou}, for $y_0 \geqslant 1$ large but independent of
  $\lambda \geqslant 1$, we have for any $j,k \in \mathbb{N}$ that
  \[ \| \partial_{\lambda}^j\partial_y^k V_{\lambda} \|_{L^{\infty} ([0, y_0])} \lesssim_{j,k,\lambda_0} \frac{1}{\lambda^{j + 1}}, \]
  We therefore construct
  \[ \left(\begin{array}{c}
       \Phi_1\\
       \Psi_1
     \end{array}\right) = \sum_{n \in \mathbb{N}} \left(\begin{array}{c}
       f_n\\
       g_n
     \end{array}\right) \]
  with
  \[ \left(\begin{array}{c}
       f_0\\
       g_0
     \end{array}\right) = \frac{\xi}{\kappa}
     \left(\begin{array}{c}
       I_1 \left( \frac{\kappa}{\xi} \cdot \right)\\
       0
     \end{array}\right) \]
  which solves
  \[ \left(\begin{array}{c}
       \left( \Delta - \left( \frac{\kappa}{\xi} \right)^2 -
       \frac{1}{y^2} \right) f_0\\
       \left( \Delta + 1 - \frac{1}{y^2} \right) g_0
     \end{array}\right) = 0, \qquad \left(\begin{array}{c}
       f_0\\
       g_0
     \end{array}\right) (y) = \frac{1}{2}\left(\begin{array}{c}
       1\\
       0
     \end{array}\right) y + O_{y \rightarrow 0} (y^3) \]
  and then
  \[ \left(\begin{array}{c}
       \left( \Delta - \left( \frac{\kappa}{\xi} \right)^2 -
       \frac{1}{y^2} \right) f_{n + 1}\\
       \left( \Delta + 1 - \frac{1}{y^2} \right) g_{n + 1}
     \end{array}\right) (y) = \widetilde{V}_{\lambda} \left(\begin{array}{c}
       f_n\\
       g_n
     \end{array}\right) \]
  using Lemma \ref{besselclassic} to invert the operators $\left( \Delta -
  \left( \frac{\kappa}{\xi} \right)^2 - \frac{1}{y^2} \right)$
  and $\left( \Delta + 1 - \frac{1}{y^2} \right)$. Using $\| V_{\lambda}
  \|_{C^1 ([0, y_0])} \lesssim_{\lambda_0} \frac{1}{\lambda}$ we show, as the
  proof of Lemma \ref{zerosol}, the convergence of the sum with the estimate
  \[ \left\| \sum_{n \geqslant 1} \left(\begin{array}{c}
       f_n\\
       g_n
     \end{array}\right) \right\|_{C^k ([0, y_0], \mathbb{R})}
     \lesssim_{\lambda_0,y_0,k} \frac{1}{\lambda} \]
  for any $y_0 \geqslant 1$ and $k\in \mathbb N$ taking then $\lambda > 0$ large enough (depending
  on $y_0$). Indeed, $\frac{\kappa}{\xi} \rightarrow 1$ when $\lambda \rightarrow + \infty$ and the potential is going to $0$ at a rate $\frac{1}{\lambda}$, with $(f_0,g_0)$ being the solution without potential.
Once $(\Phi_1,\Psi_1)$ is constructed this way, by Cauchy theory it is smooth with respect to $\lambda$ on $[0,y_0]$ as it solves a differential equation with smooth coefficient in $\lambda$ and with a boundary condition independent of $\lambda$. With Lemma \ref{foryou}, we check that $(\lambda \partial_\lambda \Phi_1,\lambda \partial_\lambda \Psi_1)$ satisfies a differential equation on $[0,y_0]$ with coefficients of size uniform in $\lambda \geqslant 1$, a source term of size $\frac{1}{\lambda}$, and with value $0$ at the boundary $y=0$. We check that this implies that it is itself of size $\frac{1}{\lambda}$. We can show the same results on higher derivatives by induction.

  To construct $\Phi_2, \Psi_2$ the reccurence relation is the same
  but we take

  \[ \left(\begin{array}{c}
       f_0\\
       g_0
     \end{array}\right) = \left(\begin{array}{c}
       0\\
       J_1
     \end{array}\right) . \]
  By Lemma \ref{foryou}, for any $j \in \mathbb{N}$,
  \[ \left| \partial_{\lambda}^j \left( \frac{\xi}{\kappa} - 1
     \right) \right| \lesssim \frac{1}{\lambda^{j + 1}} \]
  therefore $\frac{\xi}{\kappa} I_1 \left(
  \frac{\kappa}{\xi} y \right) = I_1 (y) + O_{\lambda
  \rightarrow 0} \left( \frac{1}{\lambda} \right)$ for $y \in [0, y_0]$.
\end{proof}

\subsection{Matching at $y_0$}

We recall that for $j \in \{ 2, 3, 4 \}$ we have
\[ \left(\begin{array}{c}
     \varphi_j\\
     \psi_j
   \end{array}\right) (\xi r) = \left(\begin{array}{cc}
     \frac{1}{a_+ (\lambda)} & 1\\
     - 1 & - a_- (\lambda)
   \end{array}\right) \left(\begin{array}{c}
     u_j\\
     v_j
   \end{array}\right) (r) \]
where $\varphi_j, \psi_j$ are defined in Lemmas \ref{factexp} and \ref{diary},
and for $j \in \{ 1, 2 \}$ we define $U_j, V_j$ via
\[ \left(\begin{array}{c}
     \Phi_j\\
     \Psi_j
   \end{array}\right) (\xi r) = \left(\begin{array}{cc}
     \frac{1}{a_+ (\lambda)} & 1\\
     - 1 & - a_- (\lambda)
   \end{array}\right) \left(\begin{array}{c}
     U_j\\
     V_j
   \end{array}\right) (r) \]
where $\Phi_j, \Psi_j$ are defined in \ref{sicuro}. Finally, we define
\begin{equation}
    \label{cands} \left(\begin{array}{c}
     \varphi_{\lambda}\\
     \psi_{\lambda}
   \end{array}\right) (y) := \left(\begin{array}{cc}
     \frac{1}{a_+ (\lambda)} & 1\\
     - 1 & - a_- (\lambda)
   \end{array}\right) \left(\begin{array}{c}
     u_{\lambda}\\
     v_{\lambda}
   \end{array}\right) \left( \frac{y}{\xi} \right)
\end{equation}
where $u_{\lambda}, v_{\lambda}$ is defined in Lemma \ref{london}. They are
all solutions of $\mathcal{H}- \lambda = 0$. Remark that for any $j \in
\mathbb{N}, \lambda \geqslant \lambda_0 > 0$,
\[ \left| \left(\begin{array}{cc}
     \frac{1}{a_+ (\lambda)} & 1\\
     - 1 & - a_- (\lambda)
   \end{array}\right) - \left(\begin{array}{cc}
     0 & 1\\
     - 1 & 0
   \end{array}\right) \right| \lesssim_{j, \lambda_0} \frac{1}{\lambda^{j +
   1}} . \]
We let $C(\lambda)$ be defined by
\begin{equation} \label{def:Clambda:constant}
\left(\begin{array}{cc}
       \frac{1}{a_+ (\lambda)} & 1\\
       - 1 & - a_- (\lambda)
     \end{array}\right)^{-1} \left(\begin{array}{c}
       0\\
       1
     \end{array}\right) = C (\lambda) e (\lambda) 
\end{equation}
and notice that for any $j \in \mathbb{N}$, for $\lambda\geq 1$,
\begin{equation} \label{def:Clambda:constant:bound} | \partial_{\lambda}^j (C (\lambda) - 1) | \lesssim_j \frac{1}{\xi \lambda^{j
     }}.
\end{equation}

\begin{lem}
  \label{kasablanca}There exists functions $\lambda \rightarrow
  \alpha_2, \alpha_3, \alpha_4, \beta_1, \beta_2$ smooth with respect to
  $\lambda$ such that
  \[ \left(\begin{array}{c}
       u_{\lambda}\\
       v_{\lambda}
     \end{array}\right) = \beta_1 \left(\begin{array}{c}
       U_1\\
       V_1
     \end{array}\right) + \beta_2 \left(\begin{array}{c}
       U_2\\
       V_2
     \end{array}\right) = \alpha_2 \left(\begin{array}{c}
       u_2\\
       v_2
     \end{array}\right) + \alpha_3 \left(\begin{array}{c}
       u_3\\
       v_3
     \end{array}\right) + \alpha_4 \left(\begin{array}{c}
       u_4\\
       v_4
     \end{array}\right) \]
with $\alpha_3^2 + \alpha_4^2 =C^{-2}(\lambda)\frac{\pi}{2}$ and for any $j \in \mathbb{N}$ and $\lambda \geqslant 1$,
\begin{equation} \label{estimate-coef-matching-lambda=0}
\left| \partial_{\lambda}^j \left( \left(\begin{array}{c}
       \alpha_2\\
       \alpha_3\\
       \alpha_4\\
       \beta_1\\
       \beta_2
     \end{array}\right) - \left(\begin{array}{c}
      0\\
      \sqrt{\frac{\pi}{2}}\\
       0\\
       0\\
       \sqrt{\frac{\pi}{2}}
     \end{array}\right) \right) \right| \lesssim_{j, \lambda_0}
     \frac{1}{\xi \lambda^{j}} .
\end{equation}

\end{lem}

\begin{proof}
  By Proposition \ref{bug} and Lemma \ref{sicuro}, there exists $\beta_1,
  \beta_2 \in \mathbb{R}$ such that
  \[ \left(\begin{array}{c}
       \varphi_{\lambda}\\
       \psi_{\lambda}
     \end{array}\right) = \beta_1 \left(\begin{array}{c}
       \Phi_1\\
       \Psi_1
     \end{array}\right) + \beta_2 \left(\begin{array}{c}
       \Phi_2\\
       \Psi_2
     \end{array}\right) \]
  and by Proposition \ref{bug} and Lemmas \ref{factexp}, \ref{diary}, there
  exists $\alpha_2, \alpha_3, \alpha_4 \in \mathbb{R}$ such that

  \[ \left(\begin{array}{c}
       \varphi_{\lambda}\\
       \psi_{\lambda}
     \end{array}\right) = \alpha_2 \left(\begin{array}{c}
       \varphi_2\\
       \psi_2
     \end{array}\right) + \alpha_3 \left(\begin{array}{c}
       \varphi_3\\
       \psi_3
     \end{array}\right) + \alpha_4 \left(\begin{array}{c}
       \varphi_4\\
       \psi_4
     \end{array}\right) . \]
We will now match these two solutions at $y_0 = 1$ large, taking $\lambda \geqslant 1$. Indeed, these two solutions are equal if and only if
  \[ B_{\lambda} \left(\begin{array}{c}
       \alpha_3\\
       \alpha_4\\
       \alpha_2\\
       \beta_1\\
       \beta_2
     \end{array}\right) = 0, \quad  B_{\lambda} = \left( \begin{array}{ccccc} - \varphi_3 (y_0) & - \varphi_4 (y_0) &  -\varphi_2 (y_0) &  \Phi_1 (y_0) & \Phi_2 (y_0)\\
      - \varphi_3' (y_0) & - \varphi_4' (y_0) &  -\varphi_2' (y_0) &
       \Phi_1' (y_0) & \Phi_2' (y_0)\\
      - \psi_3 (y_0) & - \psi_4 (y_0) & - \psi_2 (y_0) &
      \Psi_1 (y_0) & \Psi_2 (y_0)\\
      - \psi_3' (y_0) & - \psi_4' (y_0) &  -\psi_2' (y_0) & 
      \Psi_1' (y_0) & \Psi_2' (y_0)
    \end{array}\right). \]
By Lemmas \ref{factexp}, \ref{diary}
  and \ref{sicuro}, and ${a_+ (\lambda)}, a_- (\lambda) =O(\lambda^{-1})$ when $\lambda
  \rightarrow + \infty$ we have that
\begin{align*}
B_{\lambda}  & = \left(\begin{array}{ccccc}
      0 & 0 & -K_1(y_0) & J_1(y_0) & 0\\
      0 & 0 & -K_1'(y_0)  & J_1 '(y_0)& 0\\
      -J_1(y_0) &-Y_1 (y_0)& 0 & 0 & J_1 (y_0)\\
      -J_1 '(y_0)& -Y_1'(y_0) & 0 & 0 & J_1
'(y_0)    \end{array}\right)(= B_{+\infty}) +\widetilde B_\lambda 
\end{align*}
where
\[ | \partial_{\lambda}^j (\widetilde B_{\lambda} ) | \lesssim_{j,
     \lambda_0} \frac{1}{\xi \lambda^{j}}. \]
  From the normalization imposed on $(u_{\lambda}, v_{\lambda})$, we require the normalization $\alpha_3^2 + \alpha_4^2 = \frac{\pi}{2}$
We deduce that for any $j \in \mathbb{N}$ (note that the order in the vector is not the same as the one $B_\lambda$ is applied to),
  \[ \left| \partial_{\lambda}^j \left( \left(\begin{array}{c}
       \alpha_2\\
       \alpha_3\\
       \alpha_4\\
       \beta_1\\
       \beta_2
     \end{array}\right) - \left(\begin{array}{c}
       0\\
       \sqrt{\frac{\pi}{2}}\\
       0\\
       0\\
       \sqrt{\frac{\pi}{2}}
     \end{array}\right) \right) \right| \lesssim_{j, \lambda_0}
     \frac{1}{\xi \lambda^{j}}, \]
  concluding the proof.
\end{proof}

\subsection{End of the proof of Theorem \ref{toohigh} for high
frequencies}\label{harcourt}

\begin{proof}
  In Theorem \ref{toohigh} we used the notation
  \[ \xi = \xi \tmop{sign} (\lambda) \]
  and here we consider $\lambda \geqslant 0$, by the symmetry $\sigma_1
  \mathcal{H} \sigma_1 = -\mathcal{H}$ we get the same result for $\lambda
  \leqslant 0$. 
 
Remark also that for $\lambda > 0$ large (and thus $\xi$
  large),
  \[ \lambda \partial_{\lambda} \approx \xi \partial_{\xi} . \]
  We consider here the case $\xi \geqslant 1$. Recall that $\xi \approx
  \sqrt{\lambda}$ when $\lambda \rightarrow + \infty$.
  
By Lemma \ref{kasablanca}, we can decompose
\begin{equation} \label{gnocchi-di-patate}
\begin{pmatrix} u_\lambda \\ v_\lambda \end{pmatrix}= \chi_{\xi^{-1}}(r)\left(\beta_1\begin{pmatrix} U_1\\ V_1\end{pmatrix}+\beta_2\begin{pmatrix} U_2\\ V_2\end{pmatrix} \right)+(1-\chi_{\xi^{-1}}(r))\left(\alpha_2\begin{pmatrix} u_2\\ v_2\end{pmatrix}+\alpha_3\begin{pmatrix} u_3\\ v_3\end{pmatrix}+\alpha_4\begin{pmatrix} u_4\\ v_4\end{pmatrix} \right).
\end{equation}
In view of the leading order term $\psi^S_\sharp$ in the desired decomposition \eqref{decomposition-psi-sharp}-\eqref{id:psi-sharp-S}, we define
$$
a_\sharp=\beta_2, \quad b_\sharp=\frac{C(\lambda)}{\sqrt{\pi}}(-\alpha_3-\alpha_4)\quad \mbox{and} \quad c_\sharp=\frac{C(\lambda)}{\sqrt{\pi}}(\alpha_3-\alpha_4) .
$$
where $C(\lambda)$ is given by \eqref{def:Clambda:constant}. Then, the desired estimate \eqref{bd:estimates-b-c-sharp} for $a_\sharp$ $b_\sharp$ and $c_\sharp$ is then a direct consequence of Lemma \ref{kasablanca} and \eqref{def:Clambda:constant:bound}.

Once $\psi_\sharp^S$ is defined by \eqref{id:psi-sharp-S} with $a_\sharp,b_\sharp,c_\sharp$ as above, we compute the remainder $\psi^R_\sharp$ using \eqref{decomposition-psi-sharp} and \eqref{gnocchi-di-patate}, and it can be decomposed as
\begin{align*}
\psi^R_\sharp &= \chi_{\xi_{-1}}(r)\left(\beta_1 \begin{pmatrix} U_1 \\V_1 \end{pmatrix}+\beta_2\left( \begin{pmatrix} U_1 \\V_1 \end{pmatrix}-J_1(\xi r)\right)e(\xi) \right)+(1-\chi_{\xi_{-1}}(r)) \alpha_2 \begin{pmatrix} u_2 \\ v_2 \end{pmatrix} \\
& +(1-\chi_{\xi_{-1}}(r))\chi(r)\left( (\alpha_3-\beta_2) \begin{pmatrix} u_3 \\ v_3 \end{pmatrix}+\beta_2 \left(\begin{pmatrix} u_3 \\ v_3 \end{pmatrix}-J_1(\xi r) e(\xi)\right)+\alpha_4  \begin{pmatrix} u_4 \\ v_4 \end{pmatrix} \right) \\
&+(1-\chi(r))\left(\alpha_3\left( \begin{pmatrix} u_3 \\ v_3 \end{pmatrix}-C(\lambda)\sqrt{\frac{2}{\pi}}\frac{\cos(\xi r-\frac{3\pi}{4})}{\sqrt{\xi r}}e(\xi)\right)+\alpha_4\left( \begin{pmatrix} u_4 \\ v_4 \end{pmatrix}-C(\lambda)\sqrt{\frac{2}{\pi}}\frac{\cos(\xi r-\frac{3\pi}{4})}{\sqrt{\xi r}}e(\xi)\right) \right)\\
&=\psi^{R,1}_\sharp+\psi^{R,2}_\sharp+\psi^{R,3}_\sharp+\psi^{R,4}_\sharp
\end{align*}
where we used elementary trigonometric identities with $-\cos(\frac{3\pi}{4})=\sin(\frac{3\pi}{4})=\frac{1}{\sqrt{2}}$.

\smallskip

\noindent \underline{Estimate for $\psi_\sharp^{R,1}$}. We factorize the first term as
\begin{align*}
& \psi^{R,1}_\sharp=m^{R,1}_{\sharp,1}\cos(\xi r)+m^{R,1}_{\sharp,2}\sin(\xi r) \\
& m^{R,1}_{\sharp,1}= \cos(\xi r)\chi_{\xi_{-1}}(r)\left(\beta_1 \begin{pmatrix} U_1 \\V_1 \end{pmatrix}+\beta_2\left( \begin{pmatrix} U_1 \\V_1 \end{pmatrix}-J_1(\xi r)e(\xi)\right) \right),\\
& m^{R,1}_{\sharp,2}= \sin(\xi r)\chi_{\xi_{-1}}(r)\left(\beta_1 \begin{pmatrix} U_1 \\V_1 \end{pmatrix}+\beta_2\left( \begin{pmatrix} U_1 \\V_1 \end{pmatrix}-J_1(\xi r)e(\xi)\right) \right).
\end{align*}
Using the estimates of Lemmas \ref{sicuro} and \ref{kasablanca}, $e(\xi)=(1,0)^\top+O(\xi^{-1})$ and $\frac{1}{a_+},a_-=O(\lambda^{-1})$, and that $\chi_{\xi^{-1}}(r)$ is supported for $r<2\xi^{-1}\ll 1$, we have, for $l=1,2$,
\begin{equation} \label{bd:msharp-inter-1}
|\partial_r^k \partial_\xi^j m^{R,1}_{\sharp,l}|\lesssim \frac{1}{\xi^{1+j-k}}\mathbbm 1 (r\leq 2 \xi) \lesssim \frac{1}{\xi \langle r \rangle \langle \xi r \rangle^{1/2}} \xi^{- j} \frac{\xi^k}{\langle \xi r \rangle^k}.
\end{equation}

\noindent \underline{Estimate for $\psi_\sharp^{R,2}$}. By Lemma \ref{factexp} we have that this term is
$$
\psi_\sharp^{R,2} = (1-\chi_{\xi_{-1}}(r)) \alpha_2 K_1\left(\kappa r\right) \begin{pmatrix} h_{2,1} \\ h_{2,2}\end{pmatrix}, \quad \left|\partial_r^k \partial_\xi^j \begin{pmatrix} h_{2,1}(r,\xi) \\ h_{2,2}(r,\xi)\end{pmatrix}\right|\lesssim \frac{1}{\xi^j r^k}
$$
for $r\gtrsim \frac{1}{\xi}$. We factorize it as
\begin{align*}
& \psi_\sharp^{R,2} = m^{R,2}_{\sharp,1}\cos(\xi r)+m^{R,2}_{\sharp,2}\sin(\xi r) \\
& m^{R,2}_{\sharp,1}=\cos(\xi r)(1-\chi_{\xi_{-1}}(r))\chi(r) \alpha_2 K_1\left(\kappa r\right) \begin{pmatrix} h_{2,1} \\ h_{2,2}\end{pmatrix} \\
& m^{R,2}_{\sharp,2}=\sin(\xi r)(1-\chi_{\xi_{-1}}(r))\chi(r) \alpha_2 K_1\left(\kappa r\right) \begin{pmatrix} h_{2,1} \\ h_{2,2}\end{pmatrix}.
\end{align*}
Using that $|\partial_y^j K_1(y)|\lesssim e^{-y}$ for $|y|\geq 1$ for all $j\in \mathbb N$, and $\kappa>\xi$ with $|\partial_\xi^j(\frac{\kappa}{\xi}-1)|\lesssim \xi^{-1-j}$ for $\xi>1$ and \eqref{estimate-coef-matching-lambda=0} we get for $l=1,2$,
\begin{equation} \label{bd:msharp-inter-2}
|\partial_r^k \partial_\xi^j m^{R,2}_{\sharp,l}|\lesssim \frac{1}{\xi^{1+j-k}}e^{-\xi r} \lesssim \frac{1}{\xi \langle r \rangle \langle \xi r \rangle^{1/2}} \xi^{- j} \frac{\xi^k}{\langle \xi r \rangle^k}.
\end{equation}

\noindent \underline{Estimate for $\psi_\sharp^{R,3}$}. We only treat the first term, as the others can be treated by very similar computations. By standard properties of Bessel functions, see Lemma \ref{besselclassic}, we have $J_1(y)=\frac{\cos y}{\sqrt{y}}H_{1,1}(y)+\frac{\sin y}{\sqrt{y}}H_{1,1}(y)$ and $Y_1(y)=\frac{\cos y}{\sqrt{y}}H_{2,1}(y)+\frac{\sin y}{\sqrt{y}}H_{2,1}(y)$ where $|\partial_y^k H_{m,n}(y)|\lesssim y^{-k}$ for $m,n=1,2$, for $y\geq 1$. Combining this decomposition with the decomposition of $(u_3,v_3)$ given by Lemma \ref{diary} we have
$$
\begin{pmatrix} u_3\\ v_3\end{pmatrix}=\frac{\cos (\xi r)}{\sqrt{\xi r}}\begin{pmatrix} h_{3,1,1}\\ h_{3,1,2}\end{pmatrix}+\frac{\sin (\xi r)}{\sqrt{\xi r}}\begin{pmatrix} h_{3,2,1}\\ h_{3,2,2}\end{pmatrix}, \quad \left|\partial_\xi^j\partial_r^k h_{3,m,n}\right|\lesssim \frac{1}{\xi^{j}r^k} \mbox{ for }r\gtrsim \frac{1}{\xi},
$$
for $m,n=1,2$. We therefore factorize the first term in $\psi_\sharp^{R,3}$ as
\begin{align*}
& (1-\chi_{\xi_{-1}}(r))\chi(r) (\alpha_3-\beta_2) \begin{pmatrix} u_3 \\ v_3 \end{pmatrix} = m^{R,3,1}_{\sharp,1}\cos(\xi r)+m^{R,3,1}_{\sharp,2}\sin(\xi r), \\
& m^{R,3,1}_{\sharp,1}=(1-\chi_{\xi_{-1}}(r))\chi(r) (\alpha_3-\beta_2)\frac{1}{\sqrt{\xi r}}\begin{pmatrix} h_{3,1,1}\\ h_{3,1,2}\end{pmatrix},\\
&m^{R,3,1}_{\sharp,2}=(1-\chi_{\xi_{-1}}(r))\chi(r) (\alpha_3-\beta_2)\frac{1}{\sqrt{\xi r}}\begin{pmatrix} h_{3,2,1}\\ h_{3,2,2}\end{pmatrix}.
\end{align*}
Since $|\partial_\xi^j(\alpha_3-\beta_2)|\lesssim \xi^{-1-j}$ by \eqref{estimate-coef-matching-lambda=0}, we have for $l=1,2$, using that $(1-\chi_{\xi^{-1}}(r))\chi(r)$ has support for $\frac{1}{2\xi}\leq r \leq 2$,
$$
|\partial_\xi^j \partial_r^k m^{R,3,1}_{\sharp,l}|\lesssim \frac{1}{\xi^{1+j} \sqrt{ \xi r} r^k}\mathbbm 1(\frac{1}{2\xi}\leq r \leq 2)\lesssim \frac{1}{\xi \langle r \rangle \langle \xi r \rangle^{1/2}} \xi^{- j} \frac{\xi^k}{\langle \xi r \rangle^k}.
$$
The other terms in $\psi^{R,3}_{\sharp}$ can estimated using similarly Lemmas \ref{diary} and \ref{besselclassic}, and \eqref{def:Clambda:constant} and \eqref{estimate-coef-matching-lambda=0}, yielding eventually the following decomposition
\begin{align}
& \psi_{\sharp}^{R,3} = m^{R,3}_{\sharp,1}\cos(\xi r)+m^{R,3}_{\sharp,2}\sin(\xi r), \\
\label{bd:msharp-inter-3}& | \partial_{\xi}^j \partial_r^k m^{R,3}_{\sharp,l}|\lesssim \frac{1}{\xi^{1+j} \sqrt{ \xi r} r^k}\mathbbm 1(\frac{1}{2\xi}\leq r \leq 2) \lesssim \frac{1}{\xi \langle r \rangle \langle \xi r \rangle^{1/2}} \xi^{- j} \frac{\xi^k}{\langle \xi r \rangle^k} \quad \mbox{for }l=1,2.
\end{align}

\noindent \underline{Estimate for $\psi_\sharp^{R,4}$}. By standard properties of Bessel functions, see Lemma \ref{abudhabi} or \cite{olver2010nist}, we have
\begin{align*}
J_1(y)=\sqrt{\frac{2}{\pi y}}\cos (y-\frac{3\pi}{4}) +\frac{\cos y}{\sqrt{y}}\widetilde H_{1,1}(y)+\frac{\sin y}{\sqrt{y}}\widetilde H_{1,2}(y),
\end{align*}
with $|\partial_y^k \widetilde H_{1,n}(y)|\lesssim y^{-1-k}$ for $n=1,2$, for $y\geq 1$. Combining this decomposition with the decomposition of $(u_3,v_3)$ given by Lemma \ref{diary} we have
$$
\begin{pmatrix} u_3\\ v_3\end{pmatrix}=\sqrt{\frac{2}{\pi}}\frac{\cos (\xi r-\frac{3\pi}{4})}{\sqrt{\xi r}}C(\lambda)e(\lambda)+\frac{\cos (\xi r)}{\sqrt{\xi r}}\begin{pmatrix} \widetilde h_{3,1,1}\\ \widetilde h_{3,1,2}\end{pmatrix}+\frac{\sin (\xi r)}{\sqrt{\xi r}}\begin{pmatrix} \widetilde h_{3,2,1}\\ \widetilde h_{3,2,2}\end{pmatrix}, \quad \left|\partial_\xi^j\partial_r^k \widetilde h_{3,m,n}\right|\lesssim \frac{1}{\xi^{1+j}r^{1+k}},
$$
for $r\gtrsim 1$ and $m,n=1,2$. With this decomposition, the first term in $\psi_{\sharp}^{R,4}$ can be factorized as we did for the first term in $\psi_{\sharp}^{R,3}$ above. The second term in $\psi_{\sharp}^{R,4}$ can be dealt with similarly. This gives
\begin{align}
\nonumber & \psi_{\sharp}^{R,4} = m^{R,4}_{\sharp,1}\cos(\xi r)+m^{R,4}_{\sharp,2}\sin(\xi r), \\
\label{bd:msharp-inter-4}& | \partial_{\xi}^j \partial_r^k m^{R,l}_{\sharp,l}|\lesssim \frac{1}{\xi^{1+j} \sqrt{ \xi r} r^{1+k}}\mathbbm 1(r\geq \frac{1}{2}) \lesssim \frac{1}{\xi \langle r \rangle \langle \xi r \rangle^{1/2}} \xi^{- j} \frac{\xi^k}{\langle \xi r \rangle^k} \quad \mbox{for }l=1,2.
\end{align}

We finally set $m_{\sharp,l}^R=\sum_{n=1}^4m_{\sharp,l}^{R,n}$ for $l=1,2$. This implies the desired decomposition \eqref{id:psi-sharp-R}. Combining the estimates \eqref{bd:msharp-inter-1}, \eqref{bd:msharp-inter-2}, \eqref{bd:msharp-inter-3} and \eqref{bd:msharp-inter-4} shows the desired estimate \eqref{bd:psi-sharp-R}.

\end{proof}

\section{Generalized eigenfunctions in the zero frequency  limit $\lambda \to 0$}\label{ss29n}

\label{sectionzerofreq}

In this section, we take $0<y_0<1$ a small constant independent of $\lambda \rightarrow 0$. The matching between the solutions near $r=0$ and $r=+\infty$ will be done at two points, namely $\nu y_0 \xi^{-1}$ for $\nu \in \{\frac 12,4\}$. These values are somehow arbitrary, the idea is that we require to do the matching at two different points to get to set of estimates on the matching coefficients. Then, we will describe the solution on $r\in[y_0\xi^{-1},+\infty[$ using the matching at $y_0\xi^{-1}/2$ and the solution on $r \in [0,2y_0\xi^{-1}]$ using the matching at $4y_0\xi^{-1}$.

\subsection{Preliminaries to the construction of solutions}

We start by estimates in the limit $\lambda \rightarrow 0$ on quantities
appearing in the equation.

\begin{lem}
  \label{kast}The quantities
  \[ a_{\pm} (\lambda) = \lambda \pm \langle \lambda \rangle, \; \xi =
     \sqrt{\langle \lambda \rangle - 1}, \; \kappa = \sqrt{1 + \langle \lambda
     \rangle} \]
  satisfy the following estimates for all $\lambda \in [0, 1]$ and $j \in \mathbb{N}$:
$$
| \partial_{\lambda}^j \xi | \lesssim_j \lambda^{\max (0, 1 - j)},
\qquad  | \partial_{\lambda}^j \langle \lambda \rangle | + | \partial_{\lambda}^j a_{\pm} | + | \partial_{\lambda}^j \kappa | \lesssim_j 1
$$
as well as
  \[ \left| \partial_{\lambda}^j \left( \frac{\langle \lambda \rangle -
     1}{\lambda^2} \right) \right| + \left| \partial_{\lambda}^j \left(
     \frac{\kappa^2 - 2}{\lambda^2} \right) \right| + \left|
     \partial_{\lambda}^j \left( \frac{\xi}{\lambda} \right) \right|
     \lesssim_j 1. \]
\end{lem}

This follows from straightforward computations. We now write the equation
$\mathcal{H}- \lambda = 0$ in different forms suited to the limit $\lambda
\rightarrow 0$. We recall that it is
\[ \left(\begin{array}{c}
     (\Delta - \kappa^2) \varphi\\
     (\Delta + \xi^2) \psi
   \end{array}\right) = - V_{\lambda} \left(\begin{array}{c}
     \varphi\\
     \psi
   \end{array}\right) \]
where $\Delta=\partial_r^2+\partial_r /r$ and
\[ V_{\lambda} = - \frac{1}{r^2} \tmop{Id} - \frac{\rho^2 (r) - 1}{\langle
   \lambda \rangle} \left(\begin{array}{cc}
     1 + 2 \langle \lambda \rangle & - \lambda\\
     - \lambda & - 1 + 2 \langle \lambda \rangle
   \end{array}\right) \]
with for any $j, k \in \mathbb{N}, \lambda \in [0, 1], r \geqslant 1$,
\[ | \partial_{\lambda}^j \partial_r^k V_{\lambda} | \lesssim_{j, k}
   \frac{1}{(1 + r)^{2 + k}} . \]

\begin{lem}
  \label{eqsec9}
\begin{itemize}
\item[1)] (Perturbation of $\lambda = 0$ close to $r = + \infty$) Equation
  (\ref{maineq3}) is equivalent to
  \[ \left(\begin{array}{c}
       \left( \Delta - \kappa^2 + \frac{\kappa^2}{r^2} \right) \varphi\\
       (\Delta + \xi^2) \psi
     \end{array}\right) =\mathcal{V}_{\lambda} \left(\begin{array}{c}
       \varphi\\
       \psi
     \end{array}\right) \]
  where for all $j, k \in \mathbb{N}, r \geqslant y_0 \xi^{- 1}/2$
  \[ | \partial_{\lambda}^j \partial_r^k \mathcal{V}_{\lambda} | \lesssim_{j,
     k,y_0} \frac{\lambda^{\max (0, 1 - j)}}{(1 + r)^{2+k}} . \]
  
\item[2)] (At the scale of oscillations) In the variable $y = \xi r$, this equation
  can be written as
  \[ \left(\begin{array}{c}
       \left( \Delta_y - \left( \frac{\kappa}{\xi} \right)^2 + \frac{2}{y^2}
       \right) (\varphi)\\
       (\Delta_y + 1) (\psi)
     \end{array}\right) = \widetilde{M} \left(\begin{array}{c}
       \varphi\\
       \psi
     \end{array}\right) (y) \]
  where we have the estimate for any $j, k \in \mathbb{N}, \lambda \in [0, 1],
  y \geqslant y_0/2$,
  \[ | \partial_{\lambda}^j \partial_y^k \widetilde{M} (y, \lambda) | \lesssim_{j,
     k,y_0} \frac{\lambda^{- j}}{y^{2 + k}} \left(\begin{array}{cc}
       \lambda^2 & \lambda\\
       \lambda & \lambda^2
     \end{array}\right) \]
(where the bound is understood coordinatewise).
  
\item[3)] (Perturbation of the case $\lambda = 0$ for $r \lesssim \xi^{-1}$) The equation can also be written 
  \[ \left(\begin{array}{c}
       (L_0 - 2 \rho^2) (\varphi)\\
       L_0 (\psi)
     \end{array}\right) = \bar{V}_{\lambda} \left(\begin{array}{c}
       \varphi\\
       \psi
     \end{array}\right) \]
  where for any $j, k \in \mathbb{N}, r \geqslant 0$,
  \[ | \partial_{\lambda}^j \partial_r^k \bar{V}_{\lambda} | \lesssim_{j, k}
     \frac{\lambda^{\max (0, 2 - j)}}{(1 + r)^k} + \frac{\lambda^{\max (0, 1 -
     j)}}{(1 + r)^{2 + k}} \]
  and
  \[ \left| \partial_{\lambda}^j \partial_r^k \left(
     \frac{\bar{V}_{\lambda}}{\lambda} \right) \right| \lesssim_{j, k}
     \frac{1}{(1 + r)^k} . \]
  
\item[4)] (Combination of 1. and 3. where $L_0$ and $\xi^2$ are considered comparable) We
  also have the decomposition
  \[ \left(\begin{array}{c}
       (L_0 - 2 \rho^2) (\varphi)\\
       (L_0 + \xi^2) \psi
     \end{array}\right) = \widetilde{V}_{\lambda} \left(\begin{array}{c}
       \varphi\\
       \psi
     \end{array}\right) \]
  where
  \[ \widetilde{V}_{\lambda} = \left(\begin{array}{cc}
       \kappa^2 - 2 & - \lambda \frac{\rho^2 (r) - 1}{\langle \lambda
       \rangle}\\
       \frac{- \lambda}{\langle \lambda \rangle} (\rho^2 (r) - 1) & - 2 (1 -
       \langle \lambda \rangle) \frac{\rho^2 (r) - 1}{\langle \lambda \rangle}
     \end{array}\right) \]
  which satisfies for any $j, k \in \mathbb{N}$ that,
  \[ | \partial_{\lambda}^j \partial_r^k \widetilde{V}_{\lambda} | \lesssim_{j, k}
     \frac{1}{(1 + r)^{2 + k}} \left(\begin{array}{cc}
       0 & \lambda^{\max (0, 1 - j)}\\
       \lambda^{\max (0, 1 - j)} & \lambda^{\max (0, 2 - j)}
     \end{array}\right) + \delta_{k = 0} \left(\begin{array}{cc}
       \lambda^{\max (0, 2 - j)} & 0\\
       0 & 0
     \end{array}\right) . \]
\end{itemize}
\end{lem}

\begin{proof} 1) Writing $V = - \frac{1}{r^2} + 1 - \rho^2$ which satisfies $|
  \partial_r^k V | \lesssim_k \frac{1}{(1 + r)^{4 + k}}$ for $r \geqslant 1, k
  \in \mathbb{N}$ by Lemma \ref{rhostuff}, we decompose
  \begin{align*}
    - V_{\lambda} & = \left( \frac{1}{r^2} \tmop{Id} + \frac{\rho^2 (r) -
    1}{\langle \lambda \rangle} \left(\begin{array}{cc}
      1 + 2 \langle \lambda \rangle & - \lambda\\
      - \lambda & - 1 + 2 \langle \lambda \rangle
    \end{array}\right) \right)\\
    & = - \frac{1}{r^2} \left(\begin{array}{cc}
      \frac{1 + 2 \langle \lambda \rangle}{\langle \lambda \rangle} - 1 & -
      \frac{\lambda}{\langle \lambda \rangle}\\
      - \frac{\lambda}{\langle \lambda \rangle} & \frac{- 1 + 2 \langle
      \lambda \rangle}{\langle \lambda \rangle} - 1
    \end{array}\right) - \frac{V}{\langle \lambda \rangle} \left(\begin{array}{cc}
      1 + 2 \langle \lambda \rangle & - \lambda\\
      - \lambda & - 1 + 2 \langle \lambda \rangle
    \end{array}\right)
  \end{align*}
  hence the equation is
  \[ \left(\begin{array}{c}
       \left( \Delta - \kappa^2 + \frac{\kappa^2}{r^2} \right) \varphi\\
       (\Delta + \xi^2) \psi
     \end{array}\right) =\mathcal{V}_{\lambda} \left(\begin{array}{c}
       \varphi\\
       \psi
     \end{array}\right) \]
  where
  \[ \mathcal{V}_{\lambda} = - \frac{1}{r^2} \left(\begin{array}{cc}
       \frac{1 - \langle \lambda \rangle}{\langle \lambda \rangle} + 2 -
       \kappa^2 & - \frac{\lambda}{\langle \lambda \rangle}\\
       - \frac{\lambda}{\langle \lambda \rangle} & \frac{- 1 + \langle \lambda
       \rangle}{\langle \lambda \rangle}
     \end{array}\right) - \frac{V}{\langle \lambda \rangle}
     \left(\begin{array}{cc}
       1 + 2 \langle \lambda \rangle & - \lambda\\
       - \lambda & - 1 + 2 \langle \lambda \rangle
     \end{array}\right) \]
  satisfies that for any $r \geqslant y_0\xi^{- 1}/2, j, k \in \mathbb{N}$,
  \[ | \partial_{\lambda}^j \partial_r^k \mathcal{V}_{\lambda} | \lesssim_{j,
     k} \frac{\lambda^{\max (0, 1 - j)}}{(1 + r)^{2+k}} \]
  by Lemma \ref{kast}.
  
\medskip
\noindent  2) We decompose the potential
  \[ V_{\lambda} (r) = \frac{1}{r^2} \left(\begin{array}{cc}
       2 & 0\\
       0 & 0
     \end{array}\right) - M (r, \lambda) \]
  where 
$$M (r, \lambda) = - \frac{1}{r^2} M_1 (\lambda) + \left( \rho^2
  (r) - 1 + \frac{1}{r^2} \right) M_2 (\lambda)$$ 
with
  \[ M_1 (\lambda) = \left(\begin{array}{cc}
       \frac{1}{\langle \lambda \rangle} - 1 & - \frac{\lambda}{\langle
       \lambda \rangle}\\
       - \frac{\lambda}{\langle \lambda \rangle} & 1 - \frac{1}{\langle
       \lambda \rangle}
     \end{array}\right) \qquad  M_2 (\lambda) = \frac{1}{\langle \lambda \rangle}
     \left(\begin{array}{cc}
       1 + 2 \langle \lambda \rangle & - \lambda\\
       - \lambda & - 1 + 2 \langle \lambda \rangle
     \end{array}\right) . \]
  Thus, the equation (\ref{maineq3}) becomes
  \[ \left(\begin{array}{c}
       \left( \Delta - \kappa^2 + \frac{2}{r^2} \right) \varphi\\
       (\Delta + \xi^2) \psi
     \end{array}\right) = M (r, \lambda) \left(\begin{array}{c}
       \varphi\\
       \psi
     \end{array}\right) . \]
  Doing now the change of variable $r = \frac{y}{\xi}$, we have
  \[ \left(\begin{array}{c}
       \left( \Delta_y - \left( \frac{\kappa}{\xi} \right)^2 + \frac{2}{y^2}
       \right) \varphi\\
       (\Delta_y + 1) \psi
     \end{array}\right) = \frac{1}{\xi^2} M \left( \frac{y}{\xi}, \lambda
     \right) \left(\begin{array}{c}
       \varphi\\
       \psi
     \end{array}\right) . \]
  We define
  \[ \widetilde{M} (y, \lambda) = \frac{1}{\xi^2} M \left( \frac{y}{\xi},
     \lambda \right) = - \frac{M_1 (\lambda)}{y^2} + \frac{1}{\xi^2} \left(
     \rho^2 \left( \frac{y}{\xi} \right) - 1 + \frac{\xi^2}{y^2} \right) M_2
     (\lambda) \]
  that we have now to estimate for $y \geqslant y_0/2$ in the limit $\lambda
  \rightarrow 0$ . By Lemma \ref{kast}, we check that for any $j \in
  \mathbb{N}, \lambda \in [0, 1]$,
  \[ | \partial_{\lambda}^j M_1 (\lambda) | \lesssim_j \left(\begin{array}{cc}
       \lambda^{\max (0, 2 - j)} & \lambda^{\max (0, 1 - j)}\\
       \lambda^{\max (0, 1 - j)} & \lambda^{\max (0, 2 - j)}
     \end{array}\right) \]
  and
  \[ | \partial_{\lambda}^j M_2 (\lambda) | \lesssim_j 1. \]
  By Lemma \ref{rhostuff}, we have for $y \geqslant y_0/2, \lambda \in [0, 1]$
  that
  \[ \left| \frac{1}{\xi^2} \left( \rho^2 \left( \frac{y}{\xi} \right) - 1 +
     \frac{\xi^2}{y^2} \right) \right| \lesssim \frac{1}{\xi^2}
     \frac{1}{\left( \frac{y}{\xi} \right)^4} \lesssim \frac{\xi^2}{y^4}
     \lesssim_{y_0} \frac{\xi^2}{y^2} \lesssim_{y_0} \frac{\lambda^2}{y^2} \]
  and still with Lemma \ref{rhostuff}, we check that for any $j, k \in
  \mathbb{N}$, for $y \geqslant y_0/2, \lambda \in [0, 1]$,
  \[ \left| \partial_{\lambda}^j \partial_y^k \left( \frac{1}{\xi^2} \left(
     \rho^2 \left( \frac{y}{\xi} \right) - 1 + \frac{\xi^2}{y^2} \right)
     \right) \right| \lesssim_{j, k, y_0} \frac{\lambda^{2 - j}}{y^{2 + k}},
  \]
  concluding the estimates on $\widetilde{M}$.
  
\medskip
  
\noindent 3) From (\ref{maineq3}), the value of $V_{\lambda}$ and $L_0 = \Delta -
  \frac{1}{r^2} + (1 - \rho^2)$, we check that
  \[ \bar{V}_{\lambda} = \left(\begin{array}{cc}
       \kappa^2 - 2 & - \lambda \frac{\rho^2 (r) - 1}{\langle \lambda
       \rangle}\\
       \frac{- \lambda}{\langle \lambda \rangle} (\rho^2 (r) - 1) & - 2 (1 -
       \langle \lambda \rangle) \frac{\rho^2 (r) - 1}{\langle \lambda \rangle}
       - \xi^2
     \end{array}\right) \]
  for which the estimates follow from Lemmas \ref{kast} and \ref{rhostuff}.
  
\medskip
  
\noindent  4) The estimates on $\widetilde{V}_{\lambda}$ can be done similarly.
\end{proof}

Let us now introduce notations and norms that will be used in the rest of this
section. Our goal is to match, in the limit $\lambda \rightarrow 0$, the two
solutions which are smooth near $r = 0$ to the three bounded ones near $r = + \infty$,
and to compute precise estimates with respect to $\lambda$ on the unique
normalized function doing so. The matching will be done at $r = \nu y_0 \xi^{- 1}$
with $\nu \in \{ \frac{1}{2}, 4 \}$.

Once again we will use notations of Annex \ref{sectionspace} that we recall
here. In the rest of this section, we will use for $n, m, b \in \mathbb{N}$
the norms
\[ \widetilde{\mathcal{N}}_b^{n, m} (f) = \sum_{j = 0}^n \sum_{k = 0}^{n + m
   - j} \| r^b r^k \partial_r^k \partial_{\lambda}^j f \|_{L^{\infty} ([r_\ast, +
   \infty (\times) 0, 1])} \]
and
\[ \bar{\mathcal{N}}_b^{n, m} (f) = \sum_{j = 0}^n \sum_{k = 0}^{n + m -
   j} \| r^b r^k \lambda^j \partial_r^k \partial_{\lambda}^j f \|_{L^{\infty}
   ([y_0\xi^{- 1}/2, + \infty (\times) 0, 1])} \]
as well as
\[ \mathcal{N}_b^{n, m} (f) = \sum_{j = 0}^n \sum_{k = 0}^{n + m - j} \|
   y^b y^k \lambda^j \partial_y^k \partial_{\lambda}^j f \|_{L^{\infty} ([y_0/2, +
   \infty (\times) 0, 1])}. \]

The norm $\widetilde{\mathcal{N}}_b^{n, m}$ has a few differences with the norm $\mathcal{N}_{b, c}^{n, m}$ introduced in Section \ref{policia}. First, the norm
is for $\lambda \in (0, 1]$ rather than $\lambda \in [\lambda_0, + \infty )$
and there is no longer a $\lambda^j$ in the norm, but we can still apply the
results of Annex \ref{sectionspace} by Lemma \ref{amaru}. Secondly, we no
longer have a term of the form $(r + \nu)^c$, as we will always take $c = 0$
here. Finally, $r_\ast > 1$ is a universal constant that appear in Lemma \ref{merrymarry}.

For $\bar{\mathcal{N}}_b^{n, m}$, we have the coefficient $\lambda^j$ in the
norm, and also the norm is taken from $r_0 = y_0\xi^{- 1}/2$, but we have shown in
Annex \ref{sectionspace} that all the estimates are independent of the
choice of $r_0$. This norm will be used for functions of $r$. Finally,
$\mathcal{N}_b^{n, m}$ is the same as the one used in Section \ref{policia} for $r_0 =
1$, and will be used for functions of $y$. Since $y\partial_y = r \partial_r$, we have $\mathcal{N}_b^{n, m}(f)=\lambda^b \bar{\mathcal{N}}_b^{n, m}(f)$ if $f$ is seen as a function of $y$ on the left and of $r$ on the right.

Remark that by Lemma \ref{eqsec9}.1, for any $n, m \in \mathbb{N}$,
\begin{equation}
  \widetilde{\mathcal{N}}_2^{n, m} (\mathcal{V}_{\lambda}) \lesssim_{n, m} 1. \label{dusk}
\end{equation}
Still from Annex \ref{sectionspace} and Lemma \ref{amaru}, for $b \in
\mathbb{Z}$ we define the space $\widetilde{\mathcal{W}}_b (r_0^+)$ as the set of
functions $f : [r_0, + \infty (\rightarrow \mathbb{C}$ such that $\| r^b r^k
\partial_r^k f \|_{L^{\infty} ([r_0, + \infty ()} < + \infty$ for all $k \in
\mathbb{N}$, and $\widetilde{\mathcal{W}}_b (r_0^+, 1^-)$ the set of functions $f
: [r_0, + \infty (\times) 0, 1] \rightarrow \mathbb{R}$ such that $\| r^b r^k
\partial_{\lambda}^j \partial_r^k f \|_{L^{\infty} ([r_0, + \infty (\times) 0,
1])} < + \infty$ for all $j, k \in \mathbb{N}$.

Finally , given a function $a$, we define the operator $D_a = \partial_y + a$,
and a choice of inverse $D_a^{- 1}$ depending on properties of $a$, given in
Definition \ref{fnpig}.

\subsection{Exponentially decaying solutions on $[r_\ast, + \infty )$}

The function $K_i$ used here is defined in Lemma \ref{merrymarry}. We recall that it does not vanish on $[r_\ast,+\infty($ where $r_\ast > 0$ is a universal constant.

\begin{lem}
  \label{alegant} Given $\lambda$ small enough, the Jost solution decaying exponentially fast (that is the one constructed in Lemma \ref{factexp} and thus solution of (\ref{maineq3})) has the form
  \[ \left(\begin{array}{c}
       \varphi_2\\
       \psi_2
     \end{array}\right) (r) = K_i (\kappa r) \left(\begin{array}{c}
       h_1\\
       h_2
     \end{array}\right) (r) \]
  where for any $r \geqslant r_\ast, j, k \in \mathbb{N}$,
  \[ \left| \partial_{\lambda}^j \partial_r^k \left( \left(\begin{array}{c}
       h_1\\
       h_2
     \end{array}\right) - \left(\begin{array}{c}
       1\\
       0
     \end{array}\right) \right) \right| \lesssim_{j, k} \frac{\lambda^{\max
     (0, 1 - j)}}{(1 + r)^{1 + k}} . \]
  Furthermore, for any $j \in \mathbb{N}, \nu \in \{ \frac 12, 4 \},$
  \[ \left| \partial_{\lambda}^j \left( \frac{1}{K_i \left( y_0\nu
     \frac{\kappa}{\xi} \right)} \left(\begin{array}{c}
       \varphi_2\\
       \varphi'_2\\
       \psi_2\\
       \psi_2'
     \end{array}\right) (\nu y_0 \xi^{- 1}) - \left(\begin{array}{c}
       1\\
       - \sqrt{2}\\
       0\\
       0
     \end{array}\right) \right) \right| \lesssim_{j,y_0} \lambda^{\max (0, 2 - j)} .
  \]
\end{lem}

\begin{proof}
  We use the formulation of Lemma \ref{eqsec9}.1 for equation (\ref{maineq3}). We notice that $K_i$ does not vanish on $[r_*,\infty)$ for $r_*$ large enough. 
  We look for a solution of this equation of the form
  \[ \left(\begin{array}{c}
       \varphi_2\\
       \psi_2
     \end{array}\right) (r) = K_i (\kappa r) \left(\begin{array}{c}
       h_1\\
       h_2
     \end{array}\right) (r), \]
  and since $\left( \Delta - \kappa^2 + \frac{\kappa^2}{r^2} \right) (K_i
  (\kappa \cdot)) = 0$, dividing the equation by $K_i (\kappa \cdot)$ we get
  \[ \left(\begin{array}{c}
       D_a D_0 (h_1)\\
       h_2'' + a h_2' + (\kappa^2 + \xi^2+\frac{\kappa^2}{r^2}) h_2
     \end{array}\right) =\mathcal{V}_{\lambda} \left(\begin{array}{c}
       h_1\\
       h_2
     \end{array}\right) \]
  where $a (r, \lambda) = 2 \kappa \frac{K_i' (\kappa r)}{K_i (\kappa r)} +
  \frac{1}{r} \in - 2 \kappa + \widetilde{\mathcal{W}}_2 (r_\ast^+, 1^-)$ which is of
  type $< 0$ in Definition \ref{cadmissible}. 
  
  Similarly to the proof of Lemma \ref{factexp}, we check that
  \[ h_2'' + a h_2' + (\kappa^2 + \xi^2) h_2 = D_{b_1} D_{b_2} (h_2) \]
  where $b_1 \in - \kappa - \xi i +\mathcal{W}_2 (r_\ast^+, 1^-), b_2 \in - \kappa
  + \xi i +\mathcal{W}_2 (r_\ast^+, 1^-)$ which are both of type $< 0$ as well
  (since $\kappa \rightarrow \sqrt{2}$ when $\lambda \rightarrow 0$, hence $-
  \kappa \pm \xi i$ is uniformly far away from $0$ for $\lambda \in [0, 1]$.
  
  Still following the proof of Lemma \ref{factexp}, we define
  \[ \Theta = \left(\begin{array}{cc}
       D_0^{- 1} D_a^{- 1} & 0\\
       0 & D_{b_2}^{- 1} D_{b_1}^{- 1}
     \end{array}\right) \left(\frac{1}{\lambda}\mathcal{V}_{\lambda} \cdot \right) \]
  and we construct a solution of the form
  \[ \left(\begin{array}{c}
       h_1\\
       h_2
     \end{array}\right) = \sum_{k \in \mathbb{N}} \lambda^k \Theta^k
     \left(\begin{array}{c}
       1\\
       0
     \end{array}\right) . \]
  By Proposition \ref{cadmissible} and (\ref{dusk}), we check that for any $n,
  m \in \mathbb{N}$,
  \[ \widetilde{\mathcal{N}}_1^{n, m} \left( \Theta \left(\begin{array}{c}
       h_1\\
       h_2
     \end{array}\right) \right) \lesssim_{n, m} \widetilde{\mathcal{N}}_0^{n, m}
     (h_1) + \widetilde{\mathcal{N}}_0^{n, m} (h_2). \]
  Therefore the above series converge for $\lambda$ small enough, and we can conclude as in the proof of Lemma \ref{factexp} for the estimate
  on $(h_1, h_2)$. Concerning the value at $r = \nu y_0 \xi^{- 1}, \nu \in \{ 1/2, 4
  \}$, from Lemma \ref{merrymarry} that $\frac{K_i' (r)}{K_i (r)} = - 1 + O_{r
  \rightarrow + \infty} \left( \frac{1}{r^2} \right)$ and therefore
  \[ \frac{\kappa K_i' \left( \nu y_0 \frac{\kappa}{\xi} \right)}{K_i \left( \nu y_0
     \frac{\kappa}{\xi} \right)} = - \kappa + O_{\lambda \rightarrow 0}
     (\lambda^2) = - \sqrt{2} + O_{\lambda \rightarrow 0} (\lambda^2) \]
  by Lemma \ref{kast}. Finally, for $\nu \in \{ \frac 12, 4 \}$,
  \[ \frac{1}{K_i \left( \nu y_0\frac{\kappa}{\xi} \right)}
     \left(\begin{array}{c}
       \varphi_2\\
       \varphi'_2\\
       \psi_2\\
       \psi_2'
     \end{array}\right) (\nu y_0\xi^{- 1}) = \left(\begin{array}{c}
       h_1 (\nu y_0 \xi^{- 1})\\
       h_1' (\nu y_0\xi^{- 1}) + \frac{\kappa K_i' \left( \nu y_0\frac{\kappa}{\xi}
       \right)}{K_i \left( \nu y_0 \frac{\kappa}{\xi} \right)} h_1 (\nu y_0\xi^{-
       1})\\
       h_2 (\nu y_0 \xi^{- 1})\\
       h_2' (\nu y_0\xi^{- 1}) + \frac{\kappa K_i' \left( \nu y_0\frac{\kappa}{\xi}
       \right)}{K_i \left( \nu y_0\frac{\kappa}{\xi} \right)} h_2' (\nu y_0\xi^{-
       1})
     \end{array}\right) \]
  and with the estimates on $h_1$ and $\frac{\kappa K_i' \left( \nu y_0
  \frac{\kappa}{\xi} \right)}{K_i \left( \nu y_0\frac{\kappa}{\xi} \right)}$ we
  check the result.
\end{proof}

\subsection{Oscillating solutions on $[y_0\xi^{- 1}/2, + \infty )$}

The solutions constructed here are colinear to the ones of Lemma \ref{diary}.

\begin{lem}
  \label{mana}There exists two real valued solutions $(\varphi_3, \psi_3)$ and
  $(\varphi_4, \psi_4)$ of (\ref{maineq3}) of the form
  \[ \left(\begin{array}{c}
       \varphi_3 + i \varphi_4\\
       \psi_3 + i \psi_4
     \end{array}\right) (r) = (J_0 (\xi r) + i Y_0 (\xi r))
     \left(\begin{array}{c}
       h_1\\
       h_2
     \end{array}\right) (r) \]
  where for any $r \geqslant y_0\xi^{- 1}/2, j, k \in \mathbb{N}$,
  \[ \left| \partial_{\lambda}^j \partial_r^k \left( \left(\begin{array}{c}
       h_1\\
       h_2
     \end{array}\right) - \left(\begin{array}{c}
       0\\
       1
     \end{array}\right) \right) \right| \lesssim_{j, k,y_0} \frac{\lambda^{2 -
     j}}{(\xi r) r^k} . \]
  Furthermore, for any $j \in \mathbb{N}, \nu \in \{ \frac 12, 4 \}$,
  \[ \left| \partial_{\lambda}^j \left( \left(\begin{array}{c}
       \varphi_3 + i \varphi_4\\
       \varphi_3' + i \varphi'_4\\
       \psi_3 + i \psi_4\\
       \frac{\psi_3' + i \psi'_4}{\xi}
     \end{array}\right) (\nu y_0 \xi^{- 1}) - \left(\begin{array}{c}
       0\\
       0\\
       J_0 (\nu y_0) + i Y_0 (\nu y_0)\\
       J_0' (\nu y_0) + i Y_0' (\nu y_0)
     \end{array}\right) \right) \right| \lesssim_j \lambda^{2 - j} . \]
\end{lem}

\begin{proof}
  In this proof, we will use both variables $r$ and $y = \xi r$. Since $\xi
  \rightarrow 0$ when $\lambda \rightarrow 0$, this change of variable becomes
  degenerate and this makes this proof \ more technical. We denote $\Delta_r,
  \Delta_y$ for the Laplacian in the two coordinates.
  
  By Lemma \ref{eqsec9}.2, we use the equation in the variable $y = r \xi$
  \[ \left(\begin{array}{c}
       \left( \Delta_y - \left( \frac{\kappa}{\xi} \right)^2 + \frac{2}{y^2}
       \right) (\varphi)\\
       (\Delta_y + 1) (\psi)
     \end{array}\right) = \widetilde{M} \left(\begin{array}{c}
       \varphi\\
       \psi
     \end{array}\right) (y) . \]
  With $H_0 = J_0 + i Y_0$, we look for a solution of the form
  \[ \left(\begin{array}{c}
       \varphi\\
       \psi
     \end{array}\right) = H_0 (y) \left(\begin{array}{c}
       \xi L_1 (r)\\
       L_2 (y)
     \end{array}\right) . \]
  Remark here that $L_1, L_2$ are not functions of the same variable - this will be the standing convention. Since $(\Delta_y + 1) H_0 (y) = 0$, dividing the
  equation by $H_0$ and using $\partial_y (L_1 (r)) = \frac{1}{\xi} L_1' (r)$, the equation becomes
  \[ \left(\begin{array}{c}
       \xi \left( \frac{1}{\xi^2} L_1'' + \frac{1}{\xi} \left( 2 \frac{H_0'
       (y)}{H_0 (y)} + \frac{1}{y} \right) L_1' + \left( \frac{2}{y^2} -
       \left( 1 + \left( \frac{\kappa}{\xi} \right)^2 \right) \right) L_1
       \right)\\
       L_2'' + \left( 2 \frac{H_0' (y)}{H_0 (y)} + \frac{1}{y} \right) L_2'
     \end{array}\right) = \widetilde{M} \left(\begin{array}{c}
       \xi L_1\\
       L_2
     \end{array}\right) . \]
  Multiplying the first equation by $\xi$ we get
  \[ L_1'' + \left( 2 \xi \frac{H_0' (\xi r)}{H_0 (\xi r)} + \frac{1}{r}
     \right) L_1' + \left( \frac{2}{r^2} - (\kappa^2 + \xi^2) \right) L_1 =
     \left( \xi \widetilde{M} (\xi r, \lambda) \left(\begin{array}{c}
       \xi L_1 (r)\\
       L_2 (\xi r)
     \end{array}\right) \right) \cdot \left(\begin{array}{c}
       1\\
       0
     \end{array}\right) \]
  and the second one is
  \[ L_2'' + \left( 2 \frac{H_0' (y)}{H_0 (y)} + \frac{1}{y} \right) L_2' =
     \left( \widetilde{M} (y, \lambda) \left(\begin{array}{c}
       \xi L_1 \left( \frac{y}{\xi} \right)\\
       L_2 (y)
     \end{array}\right) \right) \cdot \left(\begin{array}{c}
       0\\
       1
     \end{array}\right) . \]
  Now, by Lemma \ref{abudhabi}.3, in the variables $y, \lambda$,
  \[ a (y) = 2 \frac{H_0' (y)}{H_0 (y)} + \frac{1}{y} \in - 2 i
     +\mathcal{W}_2 ((y_0/2)^+) \]
  is of type $i$ in Definition \ref{cadmissible} and the second equation can
  be factorized as
  \[ D_a D_0 (L_2) = \left( \widetilde{M} (y, \lambda) \left(\begin{array}{c}
       \xi L_1 \left( \frac{y}{\xi} \right)\\
       L_2 (y)
     \end{array}\right) \right) \cdot \left(\begin{array}{c}
       0\\
       1
     \end{array}\right) . \]

  Now, we want to factorize (in the variables $r, \lambda$)
  \[ h_1'' + \left( 2 \xi \frac{H_0' (\xi r)}{H_0 (\xi r)} + \frac{1}{r}
     \right) h_1' + \left( \frac{2}{r^2} - (\kappa^2 + \xi^2) \right) h_1 =
     D_{b_1} D_{b_2} (h_1) . \]
  Once again by Lemma \ref{abudhabi}.3, we have
  \[ 2 \xi \frac{H_0' (\xi r)}{H_0 (\xi r)} + \frac{1}{r} \in - 2 i \xi
     +\mathcal{W}_2 ((y_0\xi^{- 1}/2)^+, 1^-) . \]
  As in the proof of Lemma \ref{diary}, we check that we can find $b_1, b_2$
  factorizing the first equation with
  \[ (r, \lambda) \rightarrow b_1 \in - \kappa - i \xi +\mathcal{W}_2 ((y_0\xi^{-
     1}/2)^+, 1^-), \;\;b_2 \in \kappa - i \xi +\mathcal{W}_2 ((y_0\xi^{- 1}/2)^+, 1^-)
  \]
  and therefore $b_1$ is of type $< 0$ and $b_2$ is of type $> 0$ for $\lambda
  \in [0, 1]$.
  
  We now define
  \[ \Theta \left(\begin{array}{c}
       f (r)\\
       g (y)
     \end{array}\right) = \left(\begin{array}{c}
       \Theta_1 (f, g) (r)\\
       \Theta_2 (f, g) (y)
     \end{array}\right) = \left(\begin{array}{c}
       D_0^{- 1} D_a^{- 1} \left( \left( \xi \widetilde{M} (\xi r)
       \left(\begin{array}{c}
         \xi f (r)\\
         g (\xi r)
       \end{array}\right) \right) \cdot \left(\begin{array}{c}
         1\\
         0
       \end{array}\right) \right)\\
      D_{b_2}^{- 1} D_{b_1}^{- 1} \left( \left( \frac{1}{\lambda^2}\widetilde{M} (y)
       \left(\begin{array}{c}
         \xi f \left( \frac{y}{\xi} \right)\\
         g (y)
       \end{array}\right) \right) \cdot \left(\begin{array}{c}
         0\\
         1
       \end{array}\right) \right)
     \end{array}\right) \]
  where $\Theta_1 (f, g) =$ is a function of $r \in [y_0\xi^{- 1}/2, + \infty )$ and $\Theta_2 (f, g)$ a function of $y \in [y_0/2, + \infty )$. We then want to
  construct
  \begin{equation}\label{fixed-point-equation-L1-L2} \left(\begin{array}{c}
       L_1 (r)\\
       L_2 (y)
     \end{array}\right) =\left(\begin{array}{c} 0\\ 1\end{array}\right)+  \sum_{k \geq 1} \left(\begin{pmatrix} 1 & 0 \\ 0 & \lambda^2 \end{pmatrix} \Theta \right)^k
     \left(\begin{array}{c}
       0\\
       1
     \end{array}\right) .
     \end{equation}
  We denote the coefficients of $\widetilde{M}$ as
  \[ \widetilde{M} = \left(\begin{array}{cc}
       m_1 & m_2\\
       m_3 & m_4
     \end{array}\right), \]
  and we compute that
\begin{align*}
    & \left( \xi \widetilde{M} (\xi r) \left(\begin{array}{c}
       \xi f (r)\\
       g (\xi r)
     \end{array}\right) \right) \cdot \left(\begin{array}{c}
       1\\
       0
     \end{array}\right) = \xi^2 m_1 (r \xi) f (r) + \xi m_2 (r \xi) g (\xi r) \\
&   \left( \widetilde{M} (y) \left(\begin{array}{c}
       \xi f \left( \frac{y}{\xi} \right)\\
       g (y)
     \end{array}\right) \right) \cdot \left(\begin{array}{c}
       0\\
       1
     \end{array}\right) = \xi m_3 (y) f \left( \frac{y}{\xi} \right) + m_4 (y)
     g (y) . \end{align*}
  With the help of Lemma \ref{eqsec9}, we claim that for any $n, m \in \mathbb{N}$,
\begin{align*}
& \bar{\mathcal{N}}_2^{n, m} \left( \left( \xi \widetilde{M} (\xi r)
     \left(\begin{array}{c}
       \xi f (r)\\
       g (\xi r)
     \end{array}\right) \right) . \left(\begin{array}{c}
       1\\
       0
     \end{array}\right) \right) \lesssim_{n, m}  \bar{\mathcal{N}}_0^{n, m} (f)
     + \mathcal{N}_0^{n, m} (g)  \\
& \mathcal{N}_2^{n, m} \left( \frac{1}{\lambda^2} \left( \widetilde{M} (y)
     \left(\begin{array}{c}
       \xi f \left( \frac{y}{\xi} \right)\\
       g (y)
     \end{array}\right) \right) \cdot \left(\begin{array}{c}
       0\\
       1
     \end{array}\right)  \right) \lesssim_{n, m}
     \bar{\mathcal{N}}_0^{n, m} (f) + \mathcal{N}_0^{n, m} (g) . 
\end{align*}
  Recall that the norms $\bar{\mathcal{N}}$ are for the variable $r$ and the
  norm $\mathcal{N}$ for the variable $y$.
  
  Indeed, let us show the result for $n = m = 0$. By Lemma \ref{eqsec9}, we
  have for $r \geqslant y_0\xi^{- 1}/2$ that
\begin{align*}
& | \xi^2 m_1 (r \xi) f (r) | \lesssim \xi^2 \frac{\lambda^2}{r^2 \xi^2} |
     f (r) | \lesssim \frac{\| f \|_{L^{\infty} ([\xi^{- 1}, + \infty
     ()}}{r^2} \lesssim \frac{1}{r^2} \bar{\mathcal{N}}_0^{0, 0} (f), \\
&  | \xi m_2 (r \xi) g (\xi r) | \lesssim \xi \frac{\lambda}{r^2 \xi^2} | g(r \xi) | \lesssim \frac{1}{r^2} \mathcal{N}_0^{0, 0} (g), \\
& \left| \xi m_3 (y) f \left( \frac{y}{\xi} \right) \right| \lesssim
     \frac{\xi \lambda}{y^2} \left| f \left( \frac{y}{\xi} \right) \right|
     \lesssim \frac{\lambda^2}{y^2} \bar{\mathcal{N}}_0^{0, 0} (f) \\
& | m_4 (y) g (y) | \lesssim \frac{\lambda^2}{y^2} | g (y) | \lesssim
     \frac{\lambda^2}{y^2} \mathcal{N}_0^{0, 0} (g) .
\end{align*}
  Still using Lemma \ref{eqsec9}, we can extend this proof to the case $n, m
  \in \mathbb{N}$. Using now Proposition \ref{grooving}, we have for any $n, m
  \in \mathbb{N}$ that
  \[ \bar{\mathcal{N}}_1^{n, m} (\Theta_1 (f, g)) +\mathcal{N}_1^{n, m}
     (\Theta_2 (f, g)) \lesssim_{n, m} \bar{\mathcal{N}}_0^{n, m} (f)
     +\mathcal{N}_0^{n, m} (g) . \]
     This implies, using that $\bar{\mathcal{N}}_0^{n, m} (f)\lesssim \xi y_0^{-1}\bar{\mathcal{N}}_1^{n, m} (f)$ and $\xi\approx \lambda$ for $\lambda$ small, that
  \[ \bar{\mathcal{N}}_1^{n, m}\left( \Theta_1 \begin{pmatrix} 1 & 0 \\ 0 & \lambda^2 \end{pmatrix}\Theta (f, g) \right)+\mathcal{N}_1^{n, m}\left(\lambda^2 \Theta_2 \begin{pmatrix} 1 & 0 \\ 0 & \lambda^2 \end{pmatrix}\Theta (f, g) \right) \lesssim_{n, m} \xi (\bar{\mathcal{N}}_1^{n, m} (f)+\mathcal{N}_1^{n, m} (g)). \]
  We can then conclude from the two inequalities above, as in the proof of Lemma \ref{diary}, that the series in the right-hand side of \eqref{fixed-point-equation-L1-L2} converges with respect to the $\bar{\mathcal{N}}_1^{n, m}(f)+\mathcal{N}_1^{n, m} (g)$ norm, for $\lambda$ small enough depending on $y_0$, giving the existence of $L_1, L_2$ with for any $n, m \in \mathbb{N}$,
  \[ \bar{\mathcal{N}}_1^{n, m} (L_1) +\mathcal{N}_1^{n, m} \left( \frac{L_2 -
     1}{\lambda^2} \right) \lesssim_{n, m} 1. \]
  Since $\left(\begin{array}{c}
    \varphi_3 + i \varphi_4\\
    \psi_3 + i \psi_4
  \end{array}\right) = H_0 (y) \left(\begin{array}{c}
    \xi L_1 (r)\\
    L_2 (y)
  \end{array}\right)$, we define
  \[ \  \]
  \[ \left(\begin{array}{c}
       h_1 (r)\\
       h_2 (r)
     \end{array}\right) = \left(\begin{array}{c}
       \xi L_1 (r)\\
       L_2 (\xi r)
     \end{array}\right) = \left(\begin{array}{c}
       0\\
       1
     \end{array}\right) + \left(\begin{array}{c}
       \xi L_1 (r)\\
       L_2 (\xi r) - 1
     \end{array}\right) . \]
  We have shown that for any $j, k \in \mathbb{N}$,
  \[ | \partial_r^k \partial_{\lambda}^j (\xi L_1 (r)) | \lesssim_{j, k}
     \frac{\lambda^{1 - j}}{r^{1 + k}} \lesssim_{j, k} \frac{\lambda^{2 -
     j}}{(\xi r) r^k} \]
  and since $r \partial_r \sim y \partial_y$,
  \[ | r^k \partial_r^k \partial_{\lambda}^j (L_2 (\xi r) - 1) | \lesssim_{j,
     k} | y^k \partial_y^k \partial_{\lambda}^j (L_2 (\xi r) - 1) |
     \lesssim_{j, k} \frac{\lambda^{2 - j}}{y} \lesssim_{j, k}
     \frac{\lambda^{2 - j}}{\xi r} \]
  leading to
  \[ | \partial_r^k \partial_{\lambda}^j (L_2 (\xi r) - 1) | \lesssim_{j, k}
     \frac{\lambda^{2 - j}}{(\xi r) r^k} . \]
  The values of $h_1, h_2$ at $\nu y_0 \xi^{- 1}, \nu \in \{ \frac 12, 4 \}$ and their
  derivatives with respect to $\lambda$ are then direct consequences of the estimates above.
\end{proof}

\subsection{Exponentially increasing solution on $[0, 8 y_0 \xi^{- 1}]$}

We focus now on the solutions coming from $r=0$. We want to construct them for $r\in[0,8y_0\xi^{-1}]$ and estimate them for $r\in[0,4y_0\xi^{-1}]$. The matching will be done both at $r=y_0\xi^{-1}/2$ and $r=4y_0\xi^{-1}$ in order to have improved estimates on the contribution of the exponentially decreasing function constructed above and the exponentially increasing function below.

\subsubsection{Additional properties on $Q_0$}

We recall from Lemma \ref{lethargy} that $Q_0$ is a solution of $(L_0 - 2
\rho^2) (Q_0)$ smooth near $r = 0$ and that does not vanish on $\mathbb{R}^{+
\ast}$ and grows exponentially fast. We want to describe more precisely its behavior as $r \to + \infty$. The function $I_i$ is defined in Lemma \ref{merrymarry}.

\begin{lem}
  \label{heimat}There exists $c_1 > 0$ such that
  \[ \frac{Q_0 (r)}{I_i \left( \sqrt{2} r \right)} \in c_1 +\mathcal{W}_2
     (r_\ast^+) . \]
\end{lem}

\begin{proof}
  Recall that $L_0 = \Delta - \frac{1}{r^2} + (1 - \rho^2)$ and thus we write
  \[ L_0 - 2 \rho^2 = \Delta - 2 + \frac{2}{r^2} + V_{\rho} \]
  where
  \[ V_{\rho} = 3 \left( 1 - \rho^2 - \frac{1}{r^2} \right) \in
     \mathcal{W}_4 (1^+) . \]
  Writing the equation $\frac{(L_0 - 2 \rho^2) \left( I_i \left( \sqrt{2} r
  \right) h_1 \right)}{I_i \left( \sqrt{2} r \right)} = 0$ is equivalent to
  \[ h''_1 + \left( 2 \sqrt{2} \frac{I_i' \left( \sqrt{2} r \right)}{I_i
     \left( \sqrt{2} r \right)} + \frac{1}{r} \right) h'_1 = - V_{\rho} h_1 \]
  since $\left( \Delta - 2 + \frac{2}{r^2} \right) \left( I_1 \left( \sqrt{2}
  r \right) \right) = 0$. Now, by Lemma \ref{merrymarry} we have
  \[ a = 2 \sqrt{2} \frac{I_i' \left( \sqrt{2} r \right)}{I_i \left( \sqrt{2}
     r \right)} + \frac{1}{r} \in 2 \sqrt{2} +\mathcal{W}_2 (r_\ast^+) \]
  and as in the proof of Lemma \ref{abudhabi}.1 we deduce that $L_0 - 2 \rho^2
  = 0$ admits a solution on $r \in [r_\ast, + \infty )$ of the form $I_i \left(
  \sqrt{2} r \right) (1 +\mathcal{W}_2 (r_\ast^+))$. By similar arguments, it has a
  solution of the form $K_i \left( \sqrt{2} r \right) (1 +\mathcal{W}_2
  (r_\ast^+))$.
  
  They form a basis of the solutions of the second order differential equation
  $L_0 - 2 \rho^2 = 0$, of which $Q_0$ is also a solution. It is therefore a
  linear combination of them, but since $Q_0$ grows exponentially fast near $r
  = + \infty$, there exists $c_1 > 0, c_2 \in \mathbb{R}$ such that
  \[ Q_0 (r) = c_1 I_i \left( \sqrt{2} r \right) (1 +\mathcal{W}_2 (r_\ast^+)) +
     c_2 K_i \left( \sqrt{2} r \right) (1 +\mathcal{W}_2 (1^+)) . \]
  Since $\frac{K_i \left( \sqrt{2} r \right)}{I_i \left( \sqrt{2} r \right)}
  \in \mathcal{W}_2 (r_\ast^+)$ by Lemma \ref{merrymarry}, we conclude that
  \[ \frac{Q_0 (r)}{I_i \left( \sqrt{2} r \right)} \in c_1 +\mathcal{W}_2
     (r_\ast^+) . \]
\end{proof}

\subsubsection{Construction of the solution}

\begin{lem}
  \label{latecheckout}There exists $\lambda_0 > 0$ such that, for any $\lambda
  \in [0, \lambda_0]$, there exists a solution of (\ref{maineq3}) of the form
  \[ \left(\begin{array}{c}
       \Phi_1\\
       \Psi_1
     \end{array}\right) (r) = Q_0 (r) \left(\begin{array}{c}
       h_1\\
       h_2
     \end{array}\right) (r) \]
  with, for any $j, k \in \mathbb{N}, r \in [0, 8 y_0 \xi^{- 1}]$,
  \[ \left| \partial_{\lambda}^j \partial_r^k \left(
     \frac{1}{\lambda} \left[\left(\begin{array}{c}
       h_1\\
       h_2
     \end{array}\right) - \left(\begin{array}{c}
       1\\
       0
     \end{array}\right) \right] \right) \right| \lesssim_{j, k} (1 + r)^{j -
     k} . \]
  Furthermore, for any $j \in \mathbb{N} , \nu \in \{ \frac 12,4 \}$,
  \[ \left| \partial_{\lambda}^j \left( \frac{1}{I_i \left( \nu y_0
     \frac{\kappa}{\xi} \right)} \left(\begin{array}{c}
       \Phi_1\\
       \Phi_1'\\
       \Psi_1\\
       \Psi_1'
     \end{array}\right) (\nu y_0\xi^{- 1}) - \left(\begin{array}{c}
       1\\
       \sqrt{2}\\
       0\\
       0
     \end{array}\right) \right) \right| \lesssim_{j,y_0} \lambda^{2-j} .
  \]
  \[ \  \]
\end{lem}

\begin{proof}
  By Lemma \ref{eqsec9}.3, we use the equation
  \[ \left(\begin{array}{c}
       (L_0 - 2 \rho^2) (\varphi)\\
       L_0 (\psi)
     \end{array}\right) = \bar{V}_{\lambda} \left(\begin{array}{c}
       \varphi\\
       \psi
     \end{array}\right) \]
  and look for a solution of the form
  \[ \left(\begin{array}{c}
       \varphi\\
       \psi
     \end{array}\right) = \left(\begin{array}{c}
       Q_0 \widetilde{h}_1\\
       \rho \widetilde{h}_2
     \end{array}\right) . \]
  With $L_0 = \Delta - \frac{1}{r^2} + (1 - \rho^2)$ and $L_0 (\rho) = 0$, we
  have
  \[ \left(\begin{array}{c}
       Q_0 \left( \widetilde{h}_1'' + \left( 2 \frac{Q_0'}{Q_0} + \frac{1}{r}
       \right) \widetilde{h}_1' \right)\\
       \rho \left( \widetilde{h}_2'' + \left( 2 \frac{\rho'}{\rho} + \frac{1}{r}
       \right) \widetilde{h}_2' \right)
     \end{array}\right) = \bar{V}_{\lambda} \left(\begin{array}{c}
       Q_0 \widetilde{h}_1\\
       \rho \widetilde{h}_2
     \end{array}\right) . \]
  With the factorizations
\begin{align*}
& \widetilde{h}_1'' + \left( 2 \frac{Q_0'}{Q_0} + \frac{1}{r} \right)
     \widetilde{h}_1' = \frac{1}{Q_0^2 r} (Q_0^2 r \widetilde{h}_1')' \\
& \widetilde{h}_2'' + \left( 2 \frac{\rho'}{\rho} + \frac{1}{r} \right)
     \widetilde{h}_2' = \frac{1}{\rho^2 r} (\rho^2 r \widetilde{h}_2')',
\end{align*}
the first equation becomes
  \[ (Q_0^2 r \widetilde{h}_1')' = Q_0 r \bar{V}_{\lambda} \left(\begin{array}{c}
       Q_0 \widetilde{h}_1\\
       \rho \widetilde{h}_2
     \end{array}\right) \cdot \left(\begin{array}{c}
       1\\
       0
     \end{array}\right) \]
  and we choose (in order to have $\widetilde{h}_1 (0) = 1$)
  \[ \widetilde{h}_1 (r) = 1 + \int_0^r \frac{1}{Q_0^2 (t) t} \int_0^t s Q_0 (s)
     \left( \bar{V}_{\lambda} (s) \left(\begin{array}{c}
       Q_0 \widetilde{h}_1\\
       \rho \widetilde{h}_2
     \end{array}\right) (s) \right) \cdot \left(\begin{array}{c}
       1\\
       0
     \end{array}\right)\dd s\dd t. \]
  Similarly, we choose (in order to have $\widetilde{h}_2 (0) = 0$)
  \[ \widetilde{h}_2 (r) = \int_0^r \frac{1}{\rho^2 (t) t} \int_0^t s \rho (s)
     \left( \bar{V}_{\lambda} (s) \left(\begin{array}{c}
       Q_0 \widetilde{h}_1\\
       \rho \widetilde{h}_2
     \end{array}\right) (s) \right) . \left(\begin{array}{c}
       0\\
       1
     \end{array}\right) \dd s \dd t \]
  Now, we define $(h_1, h_2)$ such that
  \[ \left(\begin{array}{c}
       \varphi\\
       \psi
     \end{array}\right) = \left(\begin{array}{c}
       Q_0 \widetilde{h}_1\\
       \rho \widetilde{h}_2
     \end{array}\right) = Q_0 \left(\begin{array}{c}
       h_1\\
       h_2
     \end{array}\right), \]
  that is $h_1 = \widetilde{h}_1, h_2 = \frac{\rho}{Q_0} \widetilde{h}_2$, and we
  write the equation as the system on $h_1, h_2$:
  \[ \left(\begin{array}{c}
       h_1\\
       h_2
     \end{array}\right) = \left(\begin{array}{c}
       1\\
       0
     \end{array}\right) + \Theta \left(\begin{array}{c}
       h_1\\
       h_2
     \end{array}\right) \]
  where
  \begin{eqnarray*}
    \Theta \left(\begin{array}{c}
      h_1\\
      h_2
    \end{array}\right) & = & \left(\begin{array}{c}
      \Theta_1 (h_1, h_2)\\
      \Theta_2 (h_1, h_2)
    \end{array}\right)\\
    & = & \left(\begin{array}{c}
      \int_0^r \frac{1}{Q_0^2 (t) t} \int_0^t s Q_0^2 (s) \left(
      \bar{V}_{\lambda} (s) \left(\begin{array}{c}
        h_1\\
        h_2
      \end{array}\right) (s) \right) . \left(\begin{array}{c}
        1\\
        0
      \end{array}\right) \dd s \dd t\\
      \frac{\rho (r)}{Q_0 (r)} \int_0^r \frac{1}{\rho^2 (t) t} \int_0^t s \rho
      (s) Q_0 (s) \left( \bar{V}_{\lambda} (s) \left(\begin{array}{c}
        h_1\\
        h_2
      \end{array}\right) (s) \right) . \left(\begin{array}{c}
        0\\
        1
      \end{array}\right) \dd s \dd t
    \end{array}\right) .
  \end{eqnarray*}
  Since $Q_0 (r) \sim r$ near $r = 0$ and grows exponentially as $r \to + \infty$ while $\rho (r) \sim r$ when $r \to 0$ and
  $\rho (r) \rightarrow 1$ when $r \to + \infty$, we check that given
  $b_1, b_2 \in \mathbb{N}, b_3 \in \mathbb{Z}, c \in \{ 1, 2 \}$ we have
  \begin{equation}
    \int_0^t \frac{s^{b_1} \rho^{b_2} (s) Q_0^c (s)}{(1 + s)^{b_3}}\dd s
    \lesssim_{b_1, b_2, b_3, c} \frac{t^{b_1 + 1} \rho^{b_2} (s) Q_0^c (t)}{(1
    + t)^{b_3 + 1}} \label{swarm}
  \end{equation}
for any $t \geqslant 0$. By Lemma \ref{eqsec9}.3, we have
  \[ | \bar{V}_{\lambda} (s) | \lesssim \lambda^2 + \frac{\lambda}{(1 + s)^2},
  \]
  hence for $r \in [0, 8  y_0\xi^{- 1}]$, we have
  \begin{align*}
| \Theta_1 (h_1, h_2) (r) | & \lesssim \int_0^r \frac{1}{Q_0^2 (t) t} \int_0^t s Q_0^2 (s) \left(
    \lambda^2 + \frac{\lambda}{(1 + s)^2} \right) \dd s \dd t \| (h_1, h_2)
    \|_{L^{\infty} ([0, 8  y_0\xi^{- 1}])}\\
    & \lesssim \int_0^r \frac{1}{Q_0^2 (t) t} t Q_0^2 (t) \left( \lambda^2
    + \frac{\lambda}{(1 + t)^2} \right)\dd t \, \| (h_1, h_2) \|_{L^{\infty} ([0,
    8 y_0 \xi^{- 1}])}\\
    & \lesssim \int_0^r \left[ \lambda^2 + \frac{\lambda}{(1 + t)^2} \right] \dd t \, \| (h_1,
    h_2) \|_{L^{\infty} ([0, 8 y_0 \xi^{- 1}])}\\
    & \lesssim (r \lambda^2 + \lambda) \, \| (h_1, h_2) \|_{L^{\infty} ([0, 8 y_0
    \xi^{- 1}])}\\
    & \lesssim \lambda \| (h_1, h_2)\, \|_{L^{\infty} ([0, 8 y_0 \xi^{- 1}])}
  \end{align*}
  and we check similarly that
  \[ | (1 + r) \partial_r \Theta_1 (h_1, h_2) (r) | \lesssim \lambda\| (h_1, h_2)
     \|_{L^{\infty} ([0, 8 \xi^{- 1}])} \]
and
  \[ | (1 + r) \Theta_2 (h_1, h_2) (r) | + | (1 + r) \partial_r \Theta_2 (h_1,
     h_2) (r) | \lesssim \lambda \| (h_1, h_2) \|_{L^{\infty} ([0, 8 y_0 \xi^{-
     1}])} . \]
  With Lemma \ref{heimat}, we deduce that
  \[ \left\| \Theta \left(\begin{array}{c}
       h_1\\
       h_2
     \end{array}\right) \right\|_{C^1 ([0, 8 \xi^{- 1}])} \lesssim \lambda
     \left\| \left(\begin{array}{c}
       h_1\\
       h_2
     \end{array}\right) \right\|_{C^1 ([0, 8 y_0 \xi^{- 1}])} \]
  hence $\Theta$ is a contraction for the $C^1 ([0, 8 y_0 \xi^{- 1}])$ norm
  provided that $\lambda$ is small enough (depending on $y_0$).
  
  This completes the construction of $(h_1, h_2)$ on $[0, 8 y_0 \xi^{- 1}]$ with,
  for $k \in \{ 0, 1 \}$,
  \[ \left| \partial_r^k \left( \frac{1}{\lambda}\left( \left(\begin{array}{c}
       h_1\\
       h_2
     \end{array}\right) - \left(\begin{array}{c}
       1\\
       0
     \end{array}\right)\right)\right) \right| \lesssim_k (1 + r)^{- k} .
  \]
  Using the equation satisfied by $h_1$ and $h_2$, we check by induction that
  this still holds for any $k \in \mathbb{N}$.
  
  Now, we have
  \[  \frac{1}{\lambda}\left( \left(\begin{array}{c}
       h_1\\
       h_2
     \end{array}\right) - \left(\begin{array}{c}
       1\\
       0
     \end{array}\right) \right) = \frac{1}{\lambda } \Theta
     \left(\begin{array}{c}
       h_1\\
       h_2
     \end{array}\right) \]
  and the operator $\frac{1}{\lambda} \Theta $ is the same as $\Theta$ simply
  replacing $\bar{V}_{\lambda}$ by $\frac{\bar{V}_{\lambda}}{\lambda}$, which
  satisfies by Lemma \ref{eqsec9}.3 that for any $j \geqslant 1$,
  \[ \left| \partial^j_{\lambda} \left( \frac{\bar{V}_{\lambda}}{\lambda}
     \right) \right| \lesssim_j 1. \]
  We estimate then for $j \geqslant 1$ that
\begin{align*}
\left| \partial_{\lambda}^j \left( \frac{\Theta_1}{\lambda} \right) (h_1, h_2) (r) \right| & =  \int_0^r \frac{1}{Q_0^2 (t) t} \int_0^t s Q_0^2 (s) \left(
    \partial_{\lambda}^j \left( \frac{\bar{V}_{\lambda}}{\lambda} \right) (s)
    \left(\begin{array}{c}
      h_1\\
      h_2
    \end{array}\right) (s) \right) . \left(\begin{array}{c}
      1\\
      0
    \end{array}\right) \dd s \dd t\\
    & \lesssim  \int_0^r \frac{1}{Q_0^2 (t) t} \int_0^t s Q_0^2 (s) \dd s \dd t
    \| (h_1, h_2) \|_{L^{\infty} ([0, 8 y_0 \xi^{- 1}])}\\
    & \lesssim  \int_0^r \frac{1}{Q_0^2 (t) t} t Q_0^2 (t)\dd t \| (h_1, h_2)
    \|_{L^{\infty} ([0, 8  y_0\xi^{- 1}])}\\
    & \lesssim  (1 + r) \| (h_1, h_2) \|_{L^{\infty} ([0, 8 y_0 \xi^{- 1}])}
  \end{align*}
  and with similar estimates for $\partial_{\lambda}^j \left(
  \frac{\Theta_2}{\lambda} \right)$ and derivatives with respect to $r$, we can show the estimate on $\partial_r^k \partial_{\lambda}^j \left(
  \frac{1}{\lambda} \left[ \left(\begin{array}{c}
    h_1\\
    h_2
  \end{array}\right) - \left(\begin{array}{c}
    1\\
    0
  \end{array}\right) \right]\right) $.
  
  \
  
  We are left with the estimate of the function at $r = \nu y_0 \xi^{- 1}, \nu
  \in \{ 1/2, 4 \}$, and we need a refinement there.
  Similarly to the proof of Lemma \ref{heimat}, we check that the equation
  \[ L_0 - 2 \rho^2 - (\kappa^2 - 2) = 0 \]
  has a solution smooth near $r=0$ and belonging in $I_i (\kappa r) (1
  +\mathcal{W}_2 (r_\ast^+))$ since for $r \geqslant 1$,
  \[ L_0 - 2 \rho^2 - (\kappa^2 - 2) = \Delta - \kappa^2 +
     \frac{\kappa^2}{r^2} + O \left( \frac{1}{r^4} \right) . \]
  We denote by $\mathcal{I}_1$ this solution. Remark that since $\kappa^2-2>0$, we have $\mathcal{I}_1>0$ on $(0,\infty)$ by repeating exactly the proof of the positivity of $Q_0$ in Lemma \ref{lethargy}.
  
  Remark that by the definition of
  $\bar{V}_{\lambda}$, $\Phi_1, \Psi_1$ solves
  \[ \left(\begin{array}{c}
       (L_0 - 2 \rho^2 - (\kappa^2 - 2)) (\Phi_1)\\
       L_0 (\Psi_1)
     \end{array}\right) = \bar{\mathcal{V}}_{\lambda} \left(\begin{array}{c}
       \Phi_1\\
       \Psi_1
     \end{array}\right) . \]
  where
  \[ \bar{\mathcal{V}}_{\lambda} = \left(\begin{array}{cc}
       0 & - \lambda \frac{\rho^2 (r) - 1}{\langle \lambda \rangle}\\
       \frac{- \lambda}{\langle \lambda \rangle} (\rho^2 (r) - 1) & - 2 (1 -
       \langle \lambda \rangle) \frac{\rho^2 (r) - 1}{\langle \lambda \rangle}
       - \xi^2
     \end{array}\right) . \]
  This implies that if we write the solution as
  \[ \left(\begin{array}{c}
       \Phi_1\\
       \Psi_1
     \end{array}\right) =\mathcal{I}_1 \left(\begin{array}{c}
       \widetilde{h}_1\\
       \widetilde{h}_2
     \end{array}\right), \]
  then we have the implicit equation
  \[ \left(\begin{array}{c}
       \widetilde{h}_1\\
       \widetilde{h}_2
     \end{array}\right) = \left(\begin{array}{c}
       1\\
       0
     \end{array}\right) + \widetilde{\Theta} \left(\begin{array}{c}
       \widetilde{h}_1\\
       \widetilde{h}_2
     \end{array}\right) \]
  where
  \begin{eqnarray*}
    \widetilde{\Theta} \left(\begin{array}{c}
      \widetilde{h}_1\\
      \widetilde{h}_2
    \end{array}\right) & = & \left(\begin{array}{c}
      \int_0^r \frac{1}{\mathcal{I}_1^2 (t) t} \int_0^t s\mathcal{I}_1^2 (s)
      \left( \bar{\mathcal{V}}_{\lambda} (s) \left(\begin{array}{c}
        \widetilde{h}_1\\
        \widetilde{h}_2
      \end{array}\right) (s) \right) \cdot \left(\begin{array}{c}
        1\\
        0
      \end{array}\right) \dd s \dd t\\
      \frac{\rho (r)}{\mathcal{I}_1 (r)} \int_0^r \frac{1}{\rho^2 (t) t}
      \int_0^t s \rho (s) \mathcal{I}_1 (s) \left( \bar{\mathcal{V}}_{\lambda}
      (s) \left(\begin{array}{c}
        \widetilde{h}_1\\
        \widetilde{h}_2
      \end{array}\right) (s) \right) \cdot \left(\begin{array}{c}
        0\\
        1
      \end{array}\right) \dd s \dd t
    \end{array}\right) .
  \end{eqnarray*}
  We have
  \[ \left(\begin{array}{c}
       \widetilde{h}_1\\
       \widetilde{h}_2
     \end{array}\right) = \left(\begin{array}{c}
       1\\
       0
     \end{array}\right) + \widetilde{\Theta} \left(\begin{array}{c}
       1\\
       0
     \end{array}\right) + \widetilde{\Theta}^2 \left(\begin{array}{c}
       \widetilde{h}_1\\
       \widetilde{h}_2
     \end{array}\right) . \]
  The estimates coming from $\Theta^2 \left(\begin{array}{c}
    \widetilde{h}_1\\
    \widetilde{h}_2
  \end{array}\right)$ are easily checked. Now, we compute that
  \[ \widetilde{\Theta} \left(\begin{array}{c}
       1\\
       0
     \end{array}\right) = \left(\begin{array}{c}
       \int_0^r \frac{1}{\mathcal{I}_1^2 (t) t} \int_0^t s\mathcal{I}_1^2 (s)
       \left( \bar{\mathcal{V}}_{\lambda} (s) \left(\begin{array}{c}
         1\\
         0
       \end{array}\right) \right) \cdot \left(\begin{array}{c}
         1\\
         0
       \end{array}\right) \dd s \dd t\\
       \frac{\rho (r)}{\mathcal{I}_1 (r)} \int_0^r \frac{1}{\rho^2 (t) t}
       \int_0^t s \rho (s) \mathcal{I}_1 (s) \left(
       \bar{\mathcal{V}}_{\lambda} (s) \left(\begin{array}{c}
         1\\
         0
       \end{array}\right) \right) \cdot \left(\begin{array}{c}
         0\\
         1
       \end{array}\right) \dd s \dd t
     \end{array}\right), \]
  but now crucially,
  \[ \left( \bar{\mathcal{V}}_{\lambda} (s) \left(\begin{array}{c}
       1\\
       0
     \end{array}\right) \right) \cdot \left(\begin{array}{c}
       1\\
       0
     \end{array}\right) = 0. \]
  Furthermore,
  \[ \left( \bar{V}_{\lambda} (s) \left(\begin{array}{c}
       1\\
       0
     \end{array}\right) \right) \cdot \left(\begin{array}{c}
       0\\
       1
     \end{array}\right) = - 2 (1 - \langle \lambda \rangle) \frac{\rho^2 (r) -
     1}{\langle \lambda \rangle} - \xi^2, \]
  and we check that this implies that for all $j \in \mathbb{N}, \nu \in \{ \frac12,
  4 \}$, we have
  \[ \left| \partial_{\lambda}^j \left( \widetilde{\Theta} \left(\begin{array}{c}
       1\\
       0
     \end{array}\right) (\nu y_0 \xi^{- 1}) \right) \right| \lesssim_{j,y_0}
     \lambda^{\max (0, 2 - j)} . \]
  We can now conclude with \ Lemmas \ref{kast} and \ref{heimat}, following the
  end of the proof of Lemma \ref{alegant}.
\end{proof}

\subsection{Uniformly bounded solution on $[0, 4 y_0 \xi^{- 1}]$}

\subsubsection{Additional properties of $L_0$}

We recall that $L_0 = \Delta - \frac{1}{r^2} + (1 - \rho^2)$.

\begin{lem}
  \label{prmidv}We have
  \[ L_0 (\rho) = 0 \]
  and
  \[ (L_0 - 2 \rho^2) (r \rho') = 2 (\rho^2 (r) - 1) \rho (r) . \]
\end{lem}

\begin{proof}
  The equality $L_0 (\rho) = 0$ is exactly the equation satisfied by $\rho$.
  Taking its derivative, we have
  \[ 0 = L_0 (\rho') - \frac{\rho'}{r^2} + \frac{2 \rho}{r^3} - 2 \rho^2 \rho'
  \]
  and we compute that
  \[ L_0 (r \rho') = r L_0 (\rho') + 2 \rho'' + \frac{\rho'}{r}, \]
  leading to
  \begin{align*}
    (L_0 - 2 \rho^2) (r \rho') & =  r \left( \frac{\rho'}{r^2} - \frac{2
    \rho}{r^3} + 2 \rho^2 \rho' \right) + 2 \rho'' + \frac{\rho'}{r} - 2 r
    \rho^2 \rho'\\
    & = 2 \left( \Delta \rho - \frac{\rho}{r^2} \right) = - 2 (1 - \rho^2 (r)) \rho (r),
  \end{align*}
  concluding the proof.
\end{proof}

\begin{lem}
  \label{L0invers}(Inverse of $L_0$) Consider the operator
  \[ \mathcal{T}_1 (h) = \rho (r) \int_0^r \frac{1}{t \rho^2 (t)} \left(
     \int_0^t s \rho (s) h (s) \dd s\right)\dd t \]
  on functions $h \in C^k ([0, 8 y_0 \xi^{- 1}], \mathbb{R})$ for some $k \in
  \mathbb{N}$. Is satisfies $L_0 (\mathcal{T}_1 (h)) = h$ and we have that
  \[ \sum_{l = 0}^k \| (1 + r)^{l} \partial_r^l (\mathcal{T}_1 (h))
     \|_{L^{\infty} ([0, 8y_0 \xi^{- 1}])} \lesssim_{k,y_0} \ln^2 (\lambda) \sum_{l =
     0}^k \| (1 + r)^{2 + l } \partial_r^l h \|_{L^{\infty} ([0, 8 y_0 \xi^{-
     1}])} \]
as well as
\[ \sum_{l = 0}^k \| (1 + r)^l \partial_r^l (\mathcal{T}_1 (h)) \|_{L^{\infty}
   ([0, 3 y_0 \xi^{- 1}])} \lesssim_{k,} \lambda^{-2} y_0^2 \sum_{l = 0}^k \| (1 + r)^l
   \partial_r^l (\mathcal{T}_1 (h)) \|_{L^{\infty} ([0, 3 y_0 \xi^{- 1}])} .
\]
\end{lem}

\begin{proof}
  The equality $L_0 (\mathcal{T}_1 (h)) = h$ follows from Lemma \ref{L0}.a. Now, we compute with Lemma \ref{rhostuff} (using $\rho (r) \sim r$ when $r
  \rightarrow 0$) that if $r \in [0, 1]$,
  \begin{align*}
    | \mathcal{T}_1 (h) | (r) & \lesssim \left| \int_0^r \frac{1}{t \rho^2
    (t)} \left( \int_0^t s \rho (s) h (s) \dd s\right)\dd t \right|\\
    & \lesssim \int_0^r \frac{1}{t \rho^2 (t)} \left( \int_0^t s \rho (s)
    \dd s\right)\dd t\, \| h \|_{L^{\infty} ([0, 1])}\\
    & \lesssim \int_0^r \frac{1}{t \rho^2 (t)} \left( \int_0^t s^2\dd s
    \right)\dd t \,\| h \|_{L^{\infty} ([0, 1])}\\
    & \lesssim \int_0^r \frac{t^3}{t \rho^2 (t)}\dd t \, \| h \|_{L^{\infty}
    ([0, 1])}  \lesssim\, \| h \|_{L^{\infty} ([0, 1])}
  \end{align*}
  and we check a similar estimate for $\mathcal{T}_1 (h)' (r)$ for $r \in [0,
  1]$. Now, for $1 \leq r \leq 8 y_0 \xi^{- 1}$, using now $\rho (r)
  \sim 1$ there, we have
  \begin{eqnarray*}
\left| \int_1^r \frac{1}{t \rho^2 (t)} \left( \int_0^t s \rho (s) h
    (s) \dd s\right)\dd t \right| \lesssim \int_1^r \frac{1}{t} \int_0^t \frac{s}{(1 + s)^{2}}\dd s
   \dd t \, \| (1 + r)^{2} h \|_{L^{\infty} ([0, 8 y_0 \xi^{- 1}])} .
  \end{eqnarray*}
We then estimate
  \[ \int_1^r \frac{1}{t} \int_0^t \frac{s}{(1 + s)^2} \dd s \dd t \lesssim
     \int_1^r \frac{1}{t} \ln (t)\dd t \lesssim \ln^2 (r) \lesssim \ln^2 (\xi)
     \lesssim \ln^2 (\lambda) \]
as well as
\begin{eqnarray*}
  &  & \left| \int_1^r \frac{1}{t \rho^2 (t)} \int_0^t s \rho (s) h (s) d s d
  t \right|\\
  & \lesssim & \int_1^r \frac{1}{t \rho^2 (t)} \int_0^t s \rho (s) d s d t \|
  h \|_{L^{\infty} ([0, 8 y_0 \xi^{- 1}])}\\
  & \lesssim & \int_1^r \frac{1}{t} \int_0^t \frac{s^2}{(1 + s)} d s d t \| h
  \|_{L^{\infty} ([0, 8 y_0 \xi^{- 1}])}\\
  & \lesssim & r^2 \| h \|_{L^{\infty} ([0, 8 y_0 \xi^{- 1}])}\\
  & \lesssim & \lambda^{- 2} y_0^2 \| h \|_{L^{\infty} ([0, 8 y_0 \xi^{- 1}])},
\end{eqnarray*}
  which concludes the proof of the two estimates for $k = 0$.
  
  Now, for $r \geqslant 1$,
  \[ \mathcal{T}_1 (h)' (r) = \frac{\rho' (r)}{\rho (r)} \mathcal{T}_1 (h) (r)
     + \frac{1}{r \rho (r)} \left( \int_0^r s \rho (s) h (s) \dd s\right)\dd t
  \]
  and with $\frac{\rho' (r)}{\rho (r)} \sim \frac{1}{r^3}$, we conclude as
  previously for $k = 1$. Since $L_0 = \Delta + V$ with $| \partial_r^k V |
  \lesssim \frac{1}{(1 + r)^{4 + k}}$, we have from $L_0 (\mathcal{T}_1 (h)) =
  h$ that
  \[ \partial_r^2 (\mathcal{T}_1 (h)) = - \frac{\partial_r (\mathcal{T}_1
     (h))}{r} - V\mathcal{T}_1 (h) + h, \]
  and we can conclude for any $k \in \mathbb{N}$ by induction.
\end{proof}

\subsubsection{Candidate for the first order}

\begin{lem}
  \label{f0f}Consider the function
  \[ F_0 = \left(\begin{array}{c}
       f_1\\
       f_2
     \end{array}\right) = \left(\begin{array}{c}
       \frac{\lambda}{2 \langle \lambda \rangle}  (- r \rho' + (\rho^2 - 1)
       (J_0 (r \xi) - 1))\\
       \rho + J_0 (r \xi) - 1
     \end{array}\right) . \]
  Then, for any $r \in [0, 8 y_0 \xi^{- 1}], \xi \in [0, 1], j, k \in \mathbb{N}$
  we have
  \[ \left| \partial_r^k \partial_{\lambda}^j \left( \frac{1}{\lambda^2}
     \left( \left(\begin{array}{c}
       (L_0 - 2 \rho^2) (f_1)\\
       (L_0 + \xi^2) (f_2)
     \end{array}\right) - \widetilde{V}_{\lambda} \left(\begin{array}{c}
       f_1\\
       f_2
     \end{array}\right) \right) \right) \right| \lesssim_{j, k} \frac{(1 +
     r)^{j - k}}{(1 + r)^2} . \]
\end{lem}
We will use the formulation of Lemma \ref{eqsec9}.4 of equation (\ref{maineq3}) to construct this approximate solution. This lemma shows that $F_0$ is a solution up to an error of size $\frac{\lambda^2}{(1 + r)^2}$ on $r
\in [0, 8 y_0 \xi^{- 1}]$.

\begin{proof}
  Recall that $L_0 = \Delta + V$ with $V = - \frac{1}{r^2} + (1 - \rho^2)$. By
  Lemma \ref{rhostuff}, we have that for any $k \in \mathbb{N}, r \geqslant
  0$,
  \begin{equation}
    | \partial_r^k (\rho^2 - 1) | + | \partial_r^k (r^2 V) | + | \partial_r^k
    (r \rho') | \lesssim_k \frac{1}{(1 + r)^{2 + k}} \label{ih}
  \end{equation}
  and
  \begin{equation}
    | \partial_r^k (r^2 (\rho^2 (r) - 1)) | \lesssim_k \frac{1}{(1 + r)^k} .
    \label{at}
  \end{equation}
  We define $j_0 (y) = \frac{J_0 (y) - 1}{y^2}$ and by Lemma
  \ref{besselclassic}, we have that for any $k \in \mathbb{N}, y \in [0, 8y_0]$,
  \[ | j_0^{(k)} (y) | \lesssim_k 1. \]
  We deduce that for any $j, k \in \mathbb{N}, r \xi \leqslant 8 y_0$,
  \begin{equation}
    | \partial_r^k \partial_{\lambda}^j (j_0 (r \xi)) | \lesssim_{j, k} r^j
    \xi^k \lesssim_{j, k} (1 + r)^{j - k} . \label{emodels}
  \end{equation}
  
\medskip
\noindent \underline{Step 1. Estimate on $f_1$ and $f_2$.} Remark that $f_1 = - \frac{\lambda}{2 \langle \lambda \rangle} r \rho' +
  \frac{\lambda \xi^2}{2 \langle \lambda \rangle} (r^2 (\rho^2 (r) - 1)) j_0
  (r \xi)$, hence with (\ref{at}) and (\ref{emodels}), we have for any $j, k
  \in \mathbb{N}, r \xi \leqslant 8 y_0$ that
  \[ \left| \partial_r^k \partial_{\lambda}^j \left( \frac{f_1}{\lambda}
     \right) \right| \lesssim_{j, k} \frac{(1 + r)^{j - k}}{(1 + r)^2} . \]
  Now, we have $f_2 = \rho (r) + r^2 \xi^2 j_0 (r \xi)$, leading to (still for
  $r \xi \leqslant 8 y_0$),
  \[ | \partial_r^k \partial_{\lambda}^j (f_2) | \lesssim_{j, k} (1 + r)^{j -
     k} . \]
  
\medskip
\noindent \underline{Step 2. Estimate on $(L_0 + \xi^2) (f_2)$.}
  Since $L_0 (\rho) = 0$ and $L_0 = \Delta + V$ with $V = - \frac{1}{r^2} +
  (1 - \rho^2)$, we compute that
  \begin{eqnarray*}
    (L_0 + \xi^2) (f_2) & = & (L_0 + \xi^2) (\rho) + (\Delta + \xi^2 + V) (J_0
    (r \xi) - 1)\\
    & = & \xi^2 \rho + V J_0 (r \xi) - (\xi^2 + V)\\
    & = & \xi^2 (\rho - 1) + \xi^2 r^2 V \left( \frac{J_0 (r \xi) - 1}{r^2
    \xi^2} \right) .
  \end{eqnarray*}
  With (\ref{ih}) and (\ref{emodels}), we have therefore shown that for any
  $j, k \in \mathbb{N}, r \xi \leqslant 8 y_0$,
  \[ \left| \partial_r^k \partial_{\lambda}^j \left( \frac{(L_0 + \xi^2)
     (f_2)}{\lambda^2} \right) \right| \lesssim_{j, k} \frac{(1 + r)^{j -
     k}}{(1 + r)^2} . \]
  
\medskip
\noindent \underline{Step 3. Estimate on the second component of the equation.}
Recall that $$\widetilde{V}_{\lambda} = \left(\begin{array}{cc}
       \kappa^2 - 2 & - \lambda \frac{\rho^2 - 1}{\langle \lambda \rangle}\\
       \frac{- \lambda}{\langle \lambda \rangle} (\rho^2 - 1) & - 2 (1 -
       \langle \lambda \rangle) \frac{\rho^2 - 1}{\langle \lambda \rangle}
     \end{array}\right)$$ so that
  \begin{eqnarray*}
    &  & \left( \left(\begin{array}{c}
      (L_0 - 2 \rho^2) (f_1)\\
      (L_0 + \xi^2) (f_2)
    \end{array}\right) - \widetilde{V}_{\lambda} \left(\begin{array}{c}
      f_1\\
      f_2
    \end{array}\right) \right) \cdot \left(\begin{array}{c}
      0\\
      1
    \end{array}\right)\\
    & = & (L_0 + \xi^2) (f_2) + 2 (1 - \langle \lambda \rangle) \frac{\rho^2
    - 1}{\langle \lambda \rangle} f_2 + \frac{\lambda}{\langle \lambda
    \rangle} (\rho^2 - 1) f_1 .
  \end{eqnarray*}
  With the estimates of the previous steps, we check that we have the desired
  estimate.
  
\medskip
\noindent \underline{Step 4. Computations on the first component.}
  We have
\begin{equation}
\label{shadows}
\begin{split}
& \left( \left(\begin{array}{c}
      (L_0 - 2 \rho^2) (f_1)\\
      (L_0 + \xi^2) (f_2)
    \end{array}\right) - \widetilde{V}_{\lambda} \left(\begin{array}{c}
      f_1\\
      f_2
    \end{array}\right) \right) . \left(\begin{array}{c}
      1\\
      0
    \end{array}\right) = (L_0 - 2 \rho^2) (f_1) - (\kappa^2 - 2) f_1 + \lambda \frac{\rho^2
    - 1}{\langle \lambda \rangle} f_2 \\
& \qquad \qquad =  - \frac{\lambda}{2 \langle \lambda \rangle} (L_0 - 2 \rho^2) (r\rho') - 2 \rho^2 \frac{\lambda}{2 \langle \lambda \rangle} (\rho^2 - 1)
    (J_0 (r \xi) - 1) + \lambda \frac{\rho^2 - 1}{\langle \lambda \rangle}
    (\rho + J_0 (r \xi) - 1)  \\
& \qquad \qquad \qquad \qquad +  \frac{\lambda}{2 \langle \lambda \rangle} L_0 ((\rho^2 - 1) (J_0 (r
    \xi) - 1)) - (\kappa^2 - 2) f_1 .  
\end{split}
\end{equation}
  By Lemma \ref{prmidv}, we have $(L_0 - 2 \rho^2) (r \rho') = 2 (\rho^2 (r) -
  1) \rho (r)$, therefore
  \[ - \frac{\lambda}{2 \langle \lambda \rangle} (L_0 - 2 \rho^2) (r \rho') +
     \lambda \frac{\rho^2 - 1}{\langle \lambda \rangle} \rho = 0. \]
  Furthermore,
  \begin{eqnarray*}
- 2 \rho^2 \frac{\lambda}{2 \langle \lambda \rangle} (\rho^2 - 1)
    (J_0 (r \xi) - 1) + \lambda \frac{\rho^2 - 1}{\langle \lambda \rangle}
    (J_0 (r \xi) - 1) =  - j_0 (r \xi) \frac{\lambda}{\langle \lambda \rangle} \xi^2 r^2
    (\rho^2 - 1)^2
  \end{eqnarray*}
  which satisfies for $r \xi \leqslant 8 y_0$,
  \[ \left| \partial_r^k \partial_{\lambda}^j \left( \frac{j_0 (r \xi)
     \frac{\lambda}{\langle \lambda \rangle} \xi^2 r^2 (\rho^2 -
     1)^2}{\lambda^2} \right) \right| \lesssim_{j, k} \frac{(1 + r)^{j -
     k}}{(1 + r)^2} . \]
  This completes the estimate of the second line of the RHS of (\ref{shadows}).
  With $\kappa^2 - 2 = \xi^2$ and the estimate on $f_1$ we check that the same
  result holds for $(\kappa^2 - 2) f_1$. Finally,
  \begin{eqnarray*}
\frac{\lambda}{2 \langle \lambda \rangle} L_0 ((\rho^2 - 1) (J_0 (r
    \xi) - 1)) =  \frac{\lambda \xi^2}{2 \langle \lambda \rangle} \Delta (r^2 (\rho^2
    - 1) j_0 (r \xi)) + \frac{\lambda \xi^2}{2 \langle \lambda \rangle} V r^2
    (\rho^2 - 1) j_0 (r \xi),
  \end{eqnarray*}
  and we check the same estimates as well.
\end{proof}

\subsubsection{Construction of the solution}

\begin{lem}
  \label{aghanim}There exists $\lambda_0 > 0$ such that, for any $\lambda \in
  [0, \lambda_0]$, there exists a solution of (\ref{maineq3}) of the form
  \[ \left(\begin{array}{c}
       \widetilde{\Phi}_2\\
       \widetilde{\Psi}_2
     \end{array}\right) = F_0 + \left(\begin{array}{c}
       R_1\\
       R_2
     \end{array}\right) \]
  where $F_0$ is defined in Lemma \ref{f0f}, and for any $j, k \in \mathbb{N},
  r \in [0, 4 y_0 \xi^{- 1}]$, we have
  \[ \left| \partial_r^k \partial_{\lambda}^j \left(\frac{R_1}{\lambda^2}) \right) \right| + \left| \partial_r^k
     \partial_{\lambda}^j \left( \frac{R_2}{\lambda^2} \right) \right| \lesssim_{j, k}  \ln^2(\lambda) (1+r)^{j-k} . \]
  Finally, for any $j \in \mathbb{N}, \nu \in \{ \frac12, 4 \}$,
  \[ \left| \partial_{\lambda}^j \left( \left(\begin{array}{c}
       \widetilde{\Phi}_2\\
       \widetilde{\Phi}_2'\\
       \widetilde{\Psi}_2\\
       \frac{\widetilde{\Psi}_2'}{\xi}
     \end{array}\right) (\nu y_0 \xi^{- 1}) - \left(\begin{array}{c}
       0\\
       0\\
       J_0 (\nu y_0)\\
       J_0' (\nu y_0)
     \end{array}\right) \right) \right| \lesssim_{j,y_0} \lambda^{2 - j} \ln^2
     (\lambda) . \]
\end{lem}

A key technical detail of this proof is that we will construct the solution on
$r \in [0, 8 y_0 \xi^{- 1}]$ but only estimate the error term on $r \in [0, 4 y_0
\xi^{- 1}]$ in step 4 (before that, all computations are done on $r \in [0, 8 y_0
\xi^{- 1}]$).

\begin{proof}
  We look for a solution of the equation in formulation of Lemma
  \ref{eqsec9}.4 on $r \in [0, 8 y_0 \xi^{- 1}]$ of the form
  \[ \left(\begin{array}{c}
       \widetilde{\Phi}_2\\
       \widetilde{\Psi}_2
     \end{array}\right) = F_0 + \left(\begin{array}{c}
       R_1\\
       R_2
     \end{array}\right), \]
  where $F_0$ is defined in Lemma \ref{f0f}. With
  \[ \left(\begin{array}{c}
       P_1\\
       P_2
     \end{array}\right) = \frac{1}{\lambda^2} \left(
     \left(\begin{array}{c}
       L_0 - 2 \rho^2\\
       L_0 + \xi^2
     \end{array}\right) - \widetilde{V}_{\lambda} \right) F_0, \]
  the equation becomes
  \begin{equation}
    \left(\begin{array}{c}
      (L_0 - 2 \rho^2) (R_1)\\
      (L_0 + \xi^2) (R_2)
    \end{array}\right) = \widetilde{V}_{\lambda} \left(\begin{array}{c}
      R_1\\
      R_2
    \end{array}\right) + \lambda^2 \left(\begin{array}{c}
      P_1\\
      P_2
    \end{array}\right) . \label{heatbeatmeteora}
  \end{equation}
  
\medskip
\noindent \underline{Step 1. Formulation of a solution for $(R_1, R_2)$.}
  Inspired by \cite{KriegerSchlag}, we look for a solution $(R_1, R_2)$ of the
  previous equation on $r \in [0, 8 y_0 \xi^{- 1}]$ of the form
  \[ \left(\begin{array}{c}
       R_1\\
       R_2
     \end{array}\right) = f \left(\begin{array}{c}
       0\\
       1
     \end{array}\right) + g \left(\begin{array}{c}
       \Phi_1\\
       \Psi_1
     \end{array}\right) \]
  where $\Phi_1, \Psi_1$ are defined in Lemma \ref{latecheckout}, and they
  solve the equation
  \[ \left(\begin{array}{c}
       (L_0 - 2 \rho^2) (\Phi_1)\\
       (L_0 + \xi^2) (\Psi_1)
     \end{array}\right) = \widetilde{V}_{\lambda} \left(\begin{array}{c}
       \Phi_1\\
       \Psi_1
     \end{array}\right) . \]
  We define
  \[ \left(\begin{array}{c}
       S_1\\
       S_2
     \end{array}\right) = \widetilde{V}_{\lambda} \left(\begin{array}{c}
       0\\
       1
     \end{array}\right) = \frac{\rho^2 (r) - 1}{\langle \lambda \rangle}
     \left(\begin{array}{c}
       - \lambda\\
       - 2 (1 - \langle \lambda \rangle)
     \end{array}\right) . \]
  The first equation (divided by $\Phi_1$) of (\ref{heatbeatmeteora}) becomes
  \begin{equation}
    g'' + \left( 2 \frac{\Phi_1'}{\Phi_1} + \frac{1}{r} \right) g' = f
    \frac{S_1}{\Phi_1} + \frac{\lambda^2 P_1}{\Phi_1} \label{beatback}
  \end{equation}
  leading to
  \[ (\Phi_1^2 r g')' = \Phi_1 r (S_1 f + \lambda^2 P_1) \]
  and we choose
  \begin{equation}
    g' (r) = \frac{1}{\Phi_1^2 (r) r} \int_0^r s (\Phi_1 (S_1 f + \lambda^2
    P_1)) (s) \dd s\label{dance}
  \end{equation}
  and then
  \[ g (r) = - \int_r^{8 y_0 \xi^{- 1}} \frac{1}{\Phi_1^2 (t) t} \int_0^t s
     (\Phi_1 (S_1 f + \lambda^2 P_1)) (s) \dd s \dd t. \]
  The second equation is
  \[ L_0 (f) +\xi^2 f + \Psi_1 g'' + \left( 2 \Psi_1' + \frac{\Psi_1}{r} \right) g' = f
     S_2 + \lambda^2 P_2 . \]
  Replacing $g''$ by (\ref{beatback}) and then $g'$ by (\ref{dance}), we write
  \begin{eqnarray*}
    L_0 (f) & = & f \left( -\xi^2+S_2 - \frac{\Psi_1}{\Phi_1} S_1 \right) + \lambda^2
    \left( P_2 - \frac{\Psi_1}{\Phi_1} P_1 \right)\\
    & - & \frac{2}{\Phi_1^2 r} \left( \Psi_1' - \Psi_1 \frac{\Phi_1'}{\Phi_1}
    \right) \int_0^r s \Phi_1 (s) (S_1 f + \lambda^2 P_1) (s)\dd s.
  \end{eqnarray*}
  We therefore choose
  \[ f (r) =\mathcal{T}_1 (\mathcal{T}_2 (f) + \lambda^2 P) \]
  where
  \[ \mathcal{T}_1 (h) = \rho (r) \int_0^r \frac{1}{t \rho^2 (t)} \left(
     \int_0^t s \rho (s) h (s) \dd s\right)\dd t \]
  has been studied in Lemma \ref{L0invers},
  \[ \mathcal{T}_2 (f) = \left(-\xi^2+ S_2 - \frac{\Psi_1}{\Phi_1} S_1 \right)
     f - \frac{2}{\Phi_1^2 r} \left( \Psi_1' - \Psi_1 \frac{\Phi_1'}{\Phi_1}
     \right) \int_0^r s \Phi_1 (s) (S_1 f) (s) \dd s\]
  and
  \[ P = \left( P_2 - \frac{\Psi_1}{\Phi_1} P_1 \right) -
     \frac{2}{\Phi_1^2 r} \left( \Psi_1' - \Psi_1 \frac{\Phi_1'}{\Phi_1}
     \right) \int_0^r s \Phi_1 (s) P_1 (s)\dd s. \]
  With Lemma \ref{f0f} to estimate $P_1$ and $P_2$, Lemma \ref{latecheckout}
  to estimate $\Phi_1, \Psi_1$ (remark in particular that
  $\frac{\Psi_1}{\Phi_1} = \frac{h_2}{h_1}$ where $h_1, h_2$ are defined in
  the lemma) and (\ref{swarm}) to estimate the integral, we claim that for all
  $j, k \in \mathbb{N}, r \xi \leqslant 8 y_0$,
  \begin{equation}
    | \partial_r^k \partial_{\lambda}^j P | \lesssim \frac{(1 + r)^{j - k}}{(1
    + r)^2} . \label{nowbreath2}
  \end{equation}

\medskip
\noindent \underline{Step 2. Estimates on the operator $\mathcal{T}_2$.}
  Recall that $S_2 = - 2 (1 - \langle \lambda \rangle) \frac{\rho^2 (r) -
  1}{\langle \lambda \rangle}$ and with the notations of Lemma
  \ref{latecheckout}, for any $j, k \in \mathbb{N}$,
  \[ \left| \partial_r^k \partial_{\lambda}^j \left( \frac{\Psi_1}{\lambda
     \Phi_1} \right) \right| = \left| \partial_r^k \partial_{\lambda}^j \left(
     \frac{h_2}{\lambda h_1} \right) \right| \lesssim_{j, k} (1 + r)^{j - k} .
  \]
  With $S_1 = - \lambda \frac{\rho^2 (r) - 1}{\langle \lambda \rangle},$ this
  implies that for any $j, k \in \mathbb{N}$,
  \[ \left| \partial_r^k \partial_{\lambda}^j \left( \frac{1}{\lambda^2}
     \left( -\xi^2+S_2 - \frac{\Psi_1}{\Phi_1} S_1 \right) \right) \right|
     \lesssim_{j, k} (1 + r)^{j - k} . \]
  With similar estimates for the second term and with (\ref{swarm}), we infer
  that for any $j, k \in \mathbb{N}$
  \begin{align*}
& \sum_{l \in [0 \ldots j], m \in [0 \ldots k]} \left\| (1 + r)^{l - m
    } \partial_r^m \partial_{\lambda}^l \left(
    \frac{\mathcal{T}_2}{\lambda^2} (f) \right) \right\|_{L^{\infty} ([0, 3 y_0\xi^{- 1}])}\\
    & \qquad \qquad  \lesssim_{j, k}  \sum_{l \in [0 \ldots j], m \in [0 \ldots k]} \| (1 +
    r)^{l - m} \partial_r^m \partial_{\lambda}^l f \|_{L^{\infty} ([0, 3y_0
    \xi^{- 1}])} .
  \end{align*}
  
\medskip
\noindent \underline{Step 3. Construction and estimates on $f$.} 
  We recall that $(\tmop{Id} -\mathcal{T}_1 \mathcal{T}_2) f = \lambda^2
  \mathcal{T}_1 (P)$ hence we want to construct
  \[ f = \lambda^2 \sum_{n \in \mathbb{N}}  \left( \mathcal{T}_1
     \mathcal{T}_2 \right)^n (\mathcal{T}_1 (P)) . \]
  By Lemma \ref{L0invers} and (\ref{nowbreath2}), we have for $r \in [0, 8 y_0
  \xi^{- 1}], j, k \in \mathbb{N}$ that
  \[ | \partial_r^k \partial_{\lambda}^j \mathcal{T}_1 (P) | \lesssim_{j, k}
     \ln^2 (\lambda) (1 + r)^{j - k} . \]
  With Lemma \ref{L0invers} and the previous estimate on
  $\frac{\mathcal{T}_2}{\lambda^2}$, we check that for any $j, k \in
  \mathbb{N}$,
  \begin{align*}
    &  \sum_{l \in [0 \ldots j], m \in [0 \ldots k]} \left\| (1 + r)^{l -
    m} \partial_r^m \partial_{\lambda}^l \left( \mathcal{T}_1
    \mathcal{T}_2 (f) \right) \right\|_{L^{\infty} ([0, 8 y_0
    \xi^{- 1}])}\\
    & \qquad \qquad \lesssim_{j, k}  y_0^2 \sum_{l \in [0 \ldots j], m \in [0
    \ldots k]} \| (1 + r)^{l - m} \partial_r^m \partial_{\lambda}^l f
    \|_{L^{\infty} ([0, 8 y_0 \xi^{- 1}])} .
  \end{align*}
  We deduce that for $y_0$ small enough and then $\lambda$ small enough, $f \rightarrow  \left(
  \mathcal{T}_1 \mathcal{T}_2\right) f$ is a contraction
  for the norm
  \[ \sum_{l \in [0 \ldots j], m \in [0 \ldots k]} \| (1 + r)^{l - m}
     \partial_r^m \partial_{\lambda}^l . \|_{L^{\infty} ([0, 8 y_0 \xi^{- 1}])} \]
  and therefore we can construct $f$ for $\lambda \in [0, \lambda_0]$ (we can
  take $\lambda_0$ independent of $j, k$ by similar arguments as in the proof
  of Lemma \ref{factexp}) with the estimate
  \[ \left| \partial_r^k \partial_{\lambda}^j \left( \frac{f}{\lambda^2 } \right) \right| \lesssim_{j, k} \ln^2
     (\lambda) (1 + r)^{j - k} \]
  for any $j, k \in \mathbb{N}$ on $r \in [0, 8 y_0\xi^{- 1}]$.

\medskip
\noindent \underline{Step 4. Estimates on $g$.}
  With the notations of Lemma \ref{latecheckout}, we have
  \[ g \left(\begin{array}{c}
       \Phi_1\\
       \Psi_1
     \end{array}\right) = Q_0 g \left(\begin{array}{c}
       h_1\\
       h_2
     \end{array}\right) \]
  and
  \[ (Q_0 g) (r) = - Q_0 (r) \int_r^{8 y_0 \xi^{- 1}} \frac{1}{\Phi_1^2 (t) t}
     \int_0^t s (\Phi_1 (S_1 f + \lambda^2 P_1)) (s) \dd s \dd t. \]

  Although this function is defined for $r \in [0, 8 y_0 \xi^{- 1}]$, we are only
  going to estimate it on $r \in [0, 4 y_0 \xi^{- 1}]$. Using Lemmas \ref{heimat}
  and \ref{merrymarry}, we have that
  \[ \left(\begin{array}{c}
       \Phi_1\\
       \Psi_1
     \end{array}\right) (r) = e^{\sqrt{2} r} \left(\begin{array}{c}
       w_1\\
       w_2
     \end{array}\right) (r), Q_0 = e^{\sqrt{2} r} w_3 \]
  with $w_1, \frac{w_2}{\lambda}, w_3 \in \mathcal{W}_{\frac{1}{2}} (r_\ast^+)$ with $w_1$ and $w_3$ non vanishing on $(0,\infty)$ with $w_1,w_3\approx \frac{1}{\sqrt{1+r}}$. Therefore,
  \[ (Q_0 g) (r) = - e^{\sqrt{2} r} w_3 (r) \int_r^{8 y_0 \xi^{- 1}} \frac{e^{- 2
     \sqrt{2} t}}{w_1^2 (t) t} \int_0^t e^{\sqrt{2} s} s (w_1 (S_1 f +
     \lambda^2 P_1)) (s) \dd s \dd t. \]
  First, concerning derivatives with respect to $\lambda$, they ever falls on
  the source term $S_1 f + \lambda^2 P_1$ or on the $8 y_0 \xi^{- 1}$ upper bound in the first
  integral. But in the latter case, we can estimate
\begin{align}
\nonumber \left| \partial_{\lambda} \left( \int_r^{8 y_0 \xi^{- 1}} \right)
    \frac{e^{- 2 \sqrt{2} t}}{w_1^2 (t) t} \int_0^t e^{\sqrt{2} s} s (w_1 (S_1
    f + \lambda^2 P_1)) (s) \dd s \dd t \right| & \lesssim  \frac{1}{\xi^2} e^{- 2 \sqrt{2} (8 y_0 \xi^{- 1})} \int_0^{8 y_0
    \xi^{- 1}} e^{\sqrt{2} s}\dd s\\
\label{bd:estimate-bounded-solution-boundary-term}    & \lesssim  \frac{1}{\xi^2} e^{- \sqrt{2} (8 y_0 \xi^{- 1})},
  \end{align}
  and for $r \in [0, 4 y_0 \xi^{- 1}]$, $\left| e^{\sqrt{2} r} w_3 (r) \right|
  \lesssim e^{\sqrt{2} (4 y_0 \xi^{- 1})}$, hence
  \[ \left| e^{\sqrt{2} r} w_3 (r) \partial_{\lambda} \left( \int_r^{8 y_0 \xi^{-
     1}} \right) \frac{e^{- 2 \sqrt{2} t}}{w_1^2 (t) t} \int_0^t e^{\sqrt{2}
     s} s (w_1 (S_1 f + \lambda^2 P_1)) (s) \dd s \dd t \right| \lesssim
     \frac{1}{\xi^2} e^{- 4\sqrt{2} y_0\xi^{- 1}} \lesssim e^{-\frac{y_0}{\xi}}. \]
  We claim that we can have similar estimates for any amount of derivatives
  with respect to $\lambda$ and $r$ if at least one of the derivatives with
  respect to $\lambda$ falls on the $8 y_0 \xi^{- 1}$ in the integral (changing
  the $\frac{1}{\xi^2}$ by another negative power of $\xi$ and then absorbing it by the exponentially decaying term $e^{-4\sqrt{2} \frac{y_0}{\xi}}$). We therefore
  obtain
  \begin{align} \label{id-partiallambda-bounded-solution}
   \partial_{\lambda}^j \left( \frac{(Q_0 g) (r)}{\lambda^2} \right) = -
     e^{\sqrt{2} r} w_3 (r) \int_r^{8 y_0 \xi^{- 1}} \frac{e^{- 2 \sqrt{2}
     t}}{w_1^2 (t) t} \int_0^t e^{\sqrt{2} s} s \left( w_1
     \partial_{\lambda}^j \left( \frac{S_1 f}{\lambda^2} + P_1 \right) \right)
     (s) \dd s \dd t + O(e^{-\frac{y_0}{\xi}}).
     \end{align}
  Denoting $\mathcal{S}_j = \partial_{\lambda}^j \left( \frac{S_1
  f}{\lambda^2} + P_1 \right)$, we have shown previously that for any $j, k
  \in \mathbb{N}$ and $r \in [0, 4 y_0 \xi^{- 1}]$,
  \[ | \partial_r^k \mathcal{S}_j | \lesssim_{j, k} \frac{(1 + r)^{j - k}}{(1
     + r)^2} . \]
  Furthermore, we define for $j, k \in \mathbb{N}$ the quantity
  \[ Z_{j, k} (t) = \int_0^t e^{\sqrt{2} (s - t)} \partial_s^k (s w_1
     (s) \mathcal{S}_j (s))\dd s. \]
  By integration by parts, we have
  \begin{align*}
Z_{j, k} (t) & =  \int_0^t \frac{1}{\sqrt{2}} \partial_s \left( e^{\sqrt{2} (s - t)}
    \right) \partial_s^k (s w_1 (s) \mathcal{S}_j (s))\dd s\\
    & =  \frac{1}{\sqrt{2}} \left( \partial_t^k (t w_1 (t) \mathcal{S}_j
    (t)) - e^{- \sqrt{2} t} \partial_s^k (s w_1 (s) \mathcal{S}_j (s))_{| s =
    0 \nobracket} \right) - \frac{1}{\sqrt{2}} Z_{j, k + 1} (t)
  \end{align*}
  and for any $j, k \in \mathbb{N}$,
  \[ \left| \frac{1}{\sqrt{2}} \left( \partial_t^k (t w_1 (t) \mathcal{S}_j
     (t)) - e^{- \sqrt{2} t} \partial_s^k (s w_1 (s) \mathcal{S}_j (s))_{| s =
     0 \nobracket} \right) \right| \lesssim_{j, k} (1 + t)^{-\frac{3}{2} + j -
     k} \]
  and
  \[ | Z_{j, k} (t) | \lesssim_{j, k} (1 + t)^{-\frac{3}{2} + j - k}, \]
  hence by induction we have for any $j, k \in \mathbb{N}$,
  \[ | \partial_r^k Z_{j, 0} (t) | \lesssim_{j, k} (1 + t)^{-\frac{3}{2} + j -
     k} . \]
  Now, remark that by \eqref{id-partiallambda-bounded-solution},
  \begin{eqnarray*}
  \partial_{\lambda}^j \left( \frac{(Q_0 g) (r)}{\lambda^2} \right) =  w_3 (r) \int_r^{8 y_0 \xi^{- 1}} e^{\sqrt{2} (r - t)} \frac{Z_{j, 0}
    (t)}{w_1^2 (t) t}\dd t+O(e^{-\frac{y_0}{\xi}})
  \end{eqnarray*}
  and for any $j, k \in \mathbb{N}$,
  \[ \left| \partial_t^k \left( \frac{Z_{j, 0} (t)}{w_1^2 (t) t} \right)
     \right| \lesssim_{j, k} (1 + t)^{-\frac{3}{2} + j - k} . \]
  Now, with $\mathcal{Z}_{j, k} (t) = \partial_t^k \left( \frac{Z_{j, 0}
  (t)}{w_1^2 (t) t} \right)$, by integration by parts, we have
\begin{align*}
  &  \int_r^{8 y_0 \xi^{- 1}} e^{\sqrt{2} (r - t)} \mathcal{Z}_{j, k} (t) d t  = \frac{1}{- \sqrt{2}} \int_r^{8 y_0 \xi^{- 1}} \partial_t \left( e^{\sqrt{2} (r - t)} \right) \mathcal{Z}_{j, k} (t)\dd t\\
    & \qquad \qquad \qquad  = \frac{1}{- \sqrt{2}} \left( e^{\sqrt{2} (r - 8 y_0 \xi^{- 1})}  \mathcal{Z}_{j, k} (3 y_0 \xi^{- 1}) -\mathcal{Z}_{j, k} (r) \right) +  \frac{1}{\sqrt{2}} \int_r^{8 y_0\xi^{- 1}} e^{\sqrt{2} (r - t)} \mathcal{Z}_{j, k + 1} (t)\dd t.
\end{align*}
  We claim that the term with $e^{- \sqrt{2} (8 y_0 \xi^{- 1})}$ can be estimated as similar terms before were estimate in \eqref{bd:estimate-bounded-solution-boundary-term} and that their contribution is $O(e^{-\frac{y_0}{\xi}})$, and by induction we have that for any $j, k \in
  \mathbb{N}$ and $r \in [0, 4 y_0 \xi^{- 1}]$,
  \[ \left| \partial_r^k \partial_{\lambda}^j \left( \frac{(Q_0 g)
     (r)}{\lambda^2} \right) \right| \lesssim_{j, k} (1 + r)^{-2+j - k} . \]
  Finally, recall that
  \[ \left(\begin{array}{c}
       \widetilde{\Phi}_2\\
       \widetilde{\Psi}_2
     \end{array}\right) = F_0 + f \left(\begin{array}{c}
       0\\
       1
     \end{array}\right) + g \left(\begin{array}{c}
       \Phi_1\\
       \Psi_1
     \end{array}\right) \]
  and using
  \[ g \left(\begin{array}{c}
       \Phi_1\\
       \Psi_1
     \end{array}\right) = Q_0 g \left(\begin{array}{c}
       h_1\\
       h_2
     \end{array}\right) \]
  from Lemma \ref{latecheckout} and its estimates on $h_1, h_2$, we conclude
  the estimates on $\widetilde{\Phi}_2, \widetilde{\Psi}_2$.
  
  Finally, for the values at $\nu y_0 \xi^{- 1}, \nu \in \{ \frac12, 4 \}$, from its
  explicit formulation we compute that for any $j \in \mathbb{N}$,
  \[ \left| \partial_{\lambda}^j \left( F_0 (\nu y_0 \xi^{- 1}) -
     \left(\begin{array}{c}
       0\\
       J_0 (\nu y_0)
     \end{array}\right) \right) \right| \lesssim_j \lambda^{2 - j} \]
  and
  \[ \left| \partial_{\lambda}^j \left( \frac{1}{\xi} F_0' (\nu y_0 \xi^{- 1}) -
     \left(\begin{array}{c}
       0\\
       J_0' (\nu y_0)
     \end{array}\right) \right) \right| \lesssim_j \lambda^{2 - j}, \]
  concluding the proof.
\end{proof}

\subsection{Matching at $r = \nu y_0 \xi^{- 1}, \nu \in \{ \frac 12, 4 \}$}
We recall that $C(\lambda)$ is defined by \eqref{def:Clambda:constant}. We notice that for any $j \in \mathbb{N}$, for $\lambda\leq 1$,
\begin{equation} \label{def:Clambda:constant:bound-lambda-to 0} | \partial_{\lambda}^j (C (\lambda) +\frac{1}{\sqrt{2}}) | \lesssim_j \lambda^{\max(0,1-j)} .
\end{equation}

\begin{lem}
  \label{secondmatching}There exists $\lambda_0 > 0$ and functions
  \[ \lambda \rightarrow \beta_1, \beta_2, \alpha_2,
     \alpha_3, \alpha_4 \in C^1 ([0, \lambda_0], \mathbb{R}) \cap
     C^{\infty} (( 0, \lambda_0], \mathbb{R}) \]
  such that the function $(\varphi_{\lambda}, \psi_{\lambda})$ defined in
  (\ref{cands}) is
\begin{equation} \label{id:matching-functions-low}\left(\begin{array}{c}
       \varphi_{\lambda}\\
       \psi_{\lambda}
     \end{array}\right) =\beta_1 \left(\begin{array}{c}
       \Phi_1\\
       \Psi_1
     \end{array}\right) + \beta_2 \left(\begin{array}{c}
       \widetilde{\Phi}_2\\
       \widetilde{\Psi}_2
     \end{array}\right) = \alpha_2 \left(\begin{array}{c}
       \varphi_2\\
       \psi_2
     \end{array}\right) + \alpha_3 \left(\begin{array}{c}
       \varphi_3\\
       \psi_3
     \end{array}\right) + \alpha_4 \left(\begin{array}{c}
       \varphi_4\\
       \psi_4
     \end{array}\right) 
\end{equation}
  with $(\alpha_3)^2 + (\alpha_4)^2 = \frac{\pi}{2 C^{2}(\lambda)}$. They satisfy
  the estimate for any $j \in \mathbb{N}$,
  \begin{align}
\label{estimate-parameters-matching-low-1}&     \left| \partial_{\lambda}^j \left( \left(\begin{array}{c}
       \alpha_3\\
       \alpha_4\\
       \beta_2
     \end{array}\right) (\lambda) - \sqrt{\pi}
     \left(\begin{array}{c}
       1\\
       0\\
       1
     \end{array}\right) \right) \right| \lesssim_j \lambda^{2 - j} \ln^2
     (\lambda),\\
\label{estimate-parameters-matching-low-2}     & |\partial_\lambda^j  \beta_1|\lesssim e^{-\frac{3 y_0\kappa}{\xi}}, \quad |\partial_\lambda^j  \alpha_2|\lesssim e^{\frac{2 y_0\kappa}{3\xi}}.
     \end{align}
\end{lem}

\begin{proof}
We pick $\nu\in \{\frac 12,4\}$ and look for coefficients of the form $(\beta_1,\alpha_2)=(\frac{\widetilde \beta^{\nu}_1}{I_i \left(y_0 \nu \frac{\kappa}{\xi} \right)}, \frac{\widetilde \alpha^{\nu}_2}{K_i \left(y_0 \nu \frac{\kappa}{\xi} \right)})$ and $(\beta_2,\alpha_3,\alpha_4)=(\beta_2^\nu,\alpha_3^\nu,\alpha_4^\nu)$, so that \eqref{id:matching-functions-low} becomes

\begin{equation} \label{id:matching-functions-low-2}\left(\begin{array}{c}
       \varphi_{\lambda}\\
       \psi_{\lambda}
     \end{array}\right) = \frac{\widetilde \beta^{\nu}_1}{I_i \left(y_0 \nu
     \frac{\kappa}{\xi} \right)} \left(\begin{array}{c}
       \Phi_1\\
       \Psi_1
     \end{array}\right) + \beta^{\nu}_2 \left(\begin{array}{c}
       \widetilde{\Phi}_2\\
       \widetilde{\Psi}_2
     \end{array}\right) = \frac{\widetilde \alpha^{\nu}_2}{K_i \left(y_0 \nu
     \frac{\kappa}{\xi} \right)} \left(\begin{array}{c}
       \varphi_2\\
       \psi_2
     \end{array}\right) + \alpha^{\nu}_3 \left(\begin{array}{c}
       \varphi_3\\
       \psi_3
     \end{array}\right) + \alpha^{\nu}_4 \left(\begin{array}{c}
       \varphi_4\\
       \psi_4
     \end{array}\right) 
\end{equation}

  Let us start by recalling the values of the various functions appearing above at $r = \nu y_0 \xi^{- 1}, \nu \in \{
  \frac 12, 4 \}$. For the ones coming from $r = 0$, we have by Lemma
  \ref{latecheckout} (remark the additional factor $\frac{1}{\xi}$ in the last
  component) that for any $j \in \mathbb{N}$,
    \begin{equation} \label{somebodyscream-10}  \left| \partial_{\lambda}^j \left( \frac{1}{I_i \left( \nu y_0
     \frac{\kappa}{\xi} \right)} \left(\begin{array}{c}
       \Phi_1\\
       \Phi_1'\\
       \Psi_1\\
       \frac{\Psi_1'}{\xi}
     \end{array}\right) (\nu y_0 \xi^{- 1}) - \left(\begin{array}{c}
       1\\
       \sqrt{2}\\
       0\\
       0
     \end{array}\right) \right) \right| \lesssim_j \left(\begin{array}{c}
       \lambda^{\max (0, 2 - j)}\\
       \lambda^{\max (0, 2 - j)}\\
       \lambda^{\max (0, 2 - j)}\\
       \lambda^{\max (0, 1 - j)}
     \end{array}\right) . 
     \end{equation}
  By Lemma \ref{aghanim},
  \begin{equation}
    \left| \partial_{\lambda}^j \left( \left(\begin{array}{c}
      \widetilde{\Phi}_2\\
      \widetilde{\Phi}_2'\\
      \widetilde{\Psi}_2\\
      \frac{\widetilde{\Psi}_2'}{\xi}
    \end{array}\right) (\nu y_0 \xi^{- 1}) - \left(\begin{array}{c}
      0\\
      0\\
      J_0 (\nu y_0)\\
      J_0' (\nu y_0)
    \end{array}\right) \right) \right| \lesssim_j \lambda^{2 - j} \ln^2
    (\lambda) . \label{somebodyscream}
  \end{equation}
  For the functions coming from $r = + \infty$, by Lemma \ref{alegant} we have
  \begin{equation} \label{somebodyscream-11}  \left| \partial_{\lambda}^j \left( \frac{1}{K_i \left( \nu y_0
     \frac{\kappa}{\xi} \right)} \left(\begin{array}{c}
       \varphi_2\\
       \varphi'_2\\
       \psi_2\\
       \frac{\psi_2'}{\xi}
     \end{array}\right) (\nu y_0 \xi^{- 1}) - \left(\begin{array}{c}
       1\\
       - \sqrt{2}\\
       0\\
       0
     \end{array}\right) \right) \right| \lesssim_j \left(\begin{array}{c}
       \lambda^{\max (0, 2 - j)}\\
       \lambda^{\max (0, 2 - j)}\\
       \lambda^{\max (0, 2 - j)}\\
       \lambda^{\max (0, 1 - j)}
     \end{array}\right) 
     \end{equation}
  and finally by Lemma \ref{mana},
  \begin{equation} \label{somebodyscream-12} \left| \partial_{\lambda}^j \left( \left(\begin{array}{c}
       \varphi_3 + i \varphi_4\\
       \varphi_3' + i \varphi'_4\\
       \psi_3 + i \psi_4\\
       \frac{\psi_3' + i \psi'_4}{\xi}
     \end{array}\right) (\nu y_0 \xi^{- 1}) - \left(\begin{array}{c}
       0\\
       0\\
       J_0 (\nu y_0) + i Y_0 (\nu y_0)\\
       J_0' (\nu y_0) + i Y_0' (\nu y_0)
     \end{array}\right) \right) \right| \lesssim_j \lambda^{2 - j} .
     \end{equation}
  We know the existence of coefficients $\beta^{\nu}_1, \beta^{\nu}_2, \alpha^{\nu}_2,
  \alpha^{\nu}_3, \alpha^{\nu}_4$ such that \eqref{id:matching-functions-low-2} holds true from Lemma \ref{london}, with $(\alpha^{\nu}_3,
  \alpha^{\nu}_4)\neq (0,0)$. Since two solutions of a second order system of ordinary differential equations are equal if their values and the values of their derivatives at a point agree, we must have
  \begin{align*}
& \frac{\widetilde \beta_1^{\nu}}{I_i \left( \nu y_0 \frac{\kappa}{\xi} \right)}
    \left(\begin{array}{c}
      \Phi_1\\
      \Phi_1'\\
      \Psi_1\\
      \frac{\Psi_1'}{\xi}
    \end{array}\right) (\nu y_0 \xi^{- 1}) + \beta^{\nu}_2 \left(\begin{array}{c}
      \widetilde{\Phi}_2\\
      \widetilde{\Phi}_2'\\
      \widetilde{\Psi}_2\\
      \frac{\widetilde{\Psi}_2'}{\xi}
    \end{array}\right) (\nu y_0 \xi^{- 1})\\
    & \qquad \qquad \qquad =  \frac{\widetilde \alpha_2^{\nu}}{K_i \left( \nu y_0\frac{\kappa}{\xi} \right)}
    \left(\begin{array}{c}
      \varphi_2\\
      \varphi'_2\\
      \psi_2\\
      \frac{\psi_2'}{\xi}
    \end{array}\right) (\nu y_0 \xi^{- 1}) + \alpha^{\nu}_3 \left(\begin{array}{c}
      \varphi_3\\
      \varphi_3'\\
      \psi_3\\
      \frac{\psi_3'}{\xi}
    \end{array}\right) (\nu y_0 \xi^{- 1}) + \alpha^{\nu}_4 \left(\begin{array}{c}
      \varphi_4\\
      \varphi'_4\\
      \psi_4\\
      \frac{\psi'_4}{\xi}
    \end{array}\right) (\nu y_0 \xi^{- 1}).
  \end{align*}
 In this form, the equation becomes
\begin{align*}
    0 & =  \left(\begin{array}{cccc}
      - \frac{\Phi_1}{I_i (\kappa .)} & \frac{\varphi_2}{K_i (\kappa .)} &
      \varphi_3 & \varphi_4\\
      - \frac{\Phi_1'}{I_i (\kappa .)} & \frac{\varphi'_2}{K_i (\kappa .)} &
      \varphi_3' & \varphi'_4\\
      - \frac{\Psi_1}{I_i (\kappa .)} & \frac{\psi_2}{K_i (\kappa .)} & \psi_3
      & \psi_4\\
      - \frac{\Psi_1' / \xi}{I_1 (\kappa .)} & \frac{\psi_2' / \xi}{K_1
      (\kappa .)} & \frac{\psi_3'}{\xi} & \frac{\psi'_4}{\xi}
    \end{array}\right) (\nu y_0 \xi^{- 1}) \left(\begin{array}{c}
      \widetilde \beta^{\nu}_1\\
      \widetilde \alpha_2^{\nu}\\
      \alpha_3^{\nu}\\
      \alpha_4^{\nu}
    \end{array}\right) - \beta^{\nu}_2 \left(\begin{array}{c}
      \widetilde{\Phi}_2\\
      \widetilde{\Phi}_2'\\
      \widetilde{\Psi}_2\\
      \frac{\widetilde{\Psi}_2'}{\xi}
    \end{array}\right) (\nu y_0 \xi^{- 1})\\
    & = \left(\begin{array}{cc}
      \mathcal{A}_{\nu} & \mathcal{B}_{\nu}\\
      \mathcal{C}_{\nu} & \mathcal{D}_{\nu}
    \end{array}\right) \left(\begin{array}{c}
     \widetilde \beta^{\nu}_1\\
      \widetilde \alpha_2^{\nu}\\
      \alpha_3^{\nu}\\
      \alpha_4^{\nu}
    \end{array}\right) - \beta^{\nu}_2 \left(\begin{array}{c}
      \widetilde{\Phi}_2\\
      \widetilde{\Phi}_2'\\
      \widetilde{\Psi}_2\\
      \frac{\widetilde{\Psi}_2'}{\xi}
    \end{array}\right) (\nu y_0 \xi^{- 1}).
  \end{align*}
The matrix above will be shown to be invertible below, so we must have $\beta^2_\nu\neq 0$. Up to renormalizing, by linearity we assume $\beta^2_\nu=1$, so that
\begin{align*}
    & \left(\begin{array}{c}
     \widetilde \beta^{\nu}_1\\
     \widetilde \alpha_2^{\nu}\\
      \alpha_3^{\nu}\\
      \alpha_4^{\nu}
    \end{array}\right) = \left(\begin{array}{cc}
      \mathcal{A}_{\nu} & \mathcal{B}_{\nu}\\
      \mathcal{C}_{\nu} & \mathcal{D}_{\nu}
    \end{array}\right)^{-1}  \left(\begin{array}{c}
      \widetilde{\Phi}_2\\
      \widetilde{\Phi}_2'\\
      \widetilde{\Psi}_2\\
      \frac{\widetilde{\Psi}_2'}{\xi}
    \end{array}\right) (\nu y_0 \xi^{- 1})= \left(\begin{array}{cc}
      \mathcal{A}_{\nu} & \mathcal{B}_{\nu}\\
      \mathcal{C}_{\nu} & \mathcal{D}_{\nu}
    \end{array}\right)^{-1}  \left( \left(\begin{array}{c}
      0\\
      0\\
      J_0 (\nu y_0)\\
      J_0' (\nu y_0)
    \end{array}\right) +O(\lambda^2 \ln^2 \lambda)\right).
  \end{align*}
where we used \eqref{somebodyscream}. We have by the estimates \eqref{somebodyscream-10}, \eqref{somebodyscream-11} and \eqref{somebodyscream-12} that for any $j \in
  \mathbb{N}$ 
\begin{align*}
& \left| \partial_{\lambda}^j \left( \mathcal{A}_{\nu} -
     \left(\begin{array}{cc}
       - 1 & 1\\
       - \sqrt{2} & - \sqrt{2}
     \end{array}\right) \right) \right| + | \partial_{\lambda}^j
     \mathcal{B}_{\nu} | \lesssim_j \lambda^{\max (0, 2 - j)}, \\
& | \partial_{\lambda}^j \mathcal{C}_{\nu} | \lesssim_j \left(\begin{array}{cc}
     \lambda^{\max (0, 2 - j)} & \lambda^{\max (0, 2 - j)}\\
     \lambda^{\max (0, 1 - j)} & \lambda^{\max (0, 1 - j)}
   \end{array}\right) \\
& \left| \partial_{\lambda}^j \left( \mathcal{D}_{\nu} -
     \left(\begin{array}{cc}
       J_0 (\nu y_0) & Y_0 (\nu y_0)\\
       J_0' (\nu y_0) & Y_0' (\nu y_0)
     \end{array}\right) \right) \right| {\lesssim_j}  \lambda^{\max (0, 2 -
     j)} . 
\end{align*}
This implies, appealing to \eqref{somebodyscream} to estimate derivatives, that
  \begin{equation} \label{estimate-parameters-matching-low-inter} \left| \partial_\lambda^j \left(\left(\begin{array}{c}
      \widetilde \beta^{\nu}_1\\
       \widetilde \alpha^{\nu}_2\\
       \alpha^{\nu}_3\\
       \alpha^{\nu}_4\\
       \beta^{\nu}_2
     \end{array}\right) - \left(\begin{array}{c}
       0\\
       0\\
       1\\
       0\\
       1
     \end{array}\right)\right)\right|\lesssim \lambda^{2-j} \ln^2\lambda . 
     \end{equation}
Since the two functions appearing in the left-hand side of \eqref{id:matching-functions-low-2} are linearly independent, and that the three functions in the right-hand side are also linearly independent, this decomposition is actually independent of $\nu$, so that $(\beta_1,\beta_2,\alpha_2,\alpha_3,\alpha_4)=(\frac{\widetilde \beta^{\nu}_1}{I_i \left(y_0 \nu \frac{\kappa}{\xi} \right)},\beta^\nu_2, \frac{\widetilde \alpha^{\nu}_2}{K_i \left(y_0 \nu \frac{\kappa}{\xi} \right)},\alpha^{\nu}_3,\alpha^{\nu}_4)$ for $\nu=\frac{1}{2},4$. The first desired inequality \eqref{estimate-parameters-matching-low-1} then follows directly from \eqref{estimate-parameters-matching-low-inter}.

We now take $\nu=4$ and write $\beta_1=\widetilde \beta^{4}_1/I_i \left( \frac{4y_0\kappa}{\xi} \right)$. As $|\partial_\lambda^j\widetilde \beta_1^{4}|\lesssim \lambda^2\ln^2\lambda$, and $|\partial_\lambda^j (1/I_i \left(\frac{4y_0\kappa}{\xi} \right))|\lesssim \sqrt{\frac{y_0}{\xi}}e^{-4y_0\kappa/\xi}$, we get $|\partial_\lambda^j \beta_1|\lesssim  e^{- 3y_0\kappa/\xi}$ for $\xi $ small enough. This is the first desired estimate in \eqref{estimate-parameters-matching-low-2}.

In turn, we take $\nu=\frac 12$ and write $\alpha_2=\widetilde \alpha^{1/2}_2/K_i \left( \frac{y_0\kappa}{2\xi} \right)$. As $|\partial_\lambda^j \widetilde \alpha_2^{1/2}|\lesssim \lambda^2\ln^2\lambda$, and $|\partial_\lambda^j (1/K_i \left(\frac{y_0\kappa}{2\xi} \right))|\lesssim \sqrt{\frac{y_0}{\xi}}e^{\frac{y_0\kappa}{2\xi}}$, we get $|\partial_\lambda^j \beta_1|\lesssim  e^{\frac{2y_0\kappa}{3\xi}}$ for $\xi $ small enough. This is the second desired estimate in \eqref{estimate-parameters-matching-low-2}.

Finally, by the linearity of the equation, in order to obtain the desired normalization at $r=\infty$ in Theorem \ref{toohigh} we eventually multiply $(\beta_1, \alpha_2,\alpha_3,\alpha_4,\beta_2)$ by a positive scalar in order to have $(\alpha_3)^2+(\alpha_4)^2=\frac{\pi}{2C^2(\lambda)}$. By \eqref{def:Clambda:constant:bound-lambda-to 0}, after this renormalization the estimates \eqref{estimate-parameters-matching-low-1} and \eqref{estimate-parameters-matching-low-2} are still valid.

\end{proof}

\subsection{End of the proof of Theorem \ref{toohigh} for small frequencies}

\begin{proof}
We will prove the desired decomposition and estimates for $0\leq \lambda \leq 2$, which, by the symmetry $\psi(\xi)=-\sigma_1\psi(\xi)$, will imply the desired result for $-2\leq \lambda \leq 0$ as well. Remark that for $\lambda$ small we have $\lambda \approx \xi$, and hence $\lambda \partial_{\lambda} \approx \xi \partial_{\xi} $.

We will prove the suitable decomposition for 
$$
\Upsilon=\begin{pmatrix}  \frac{1}{a_+ (\lambda)} & 1\\ - 1 & - a_- (\lambda)  \end{pmatrix} \psi, \quad \Upsilon_\flat^S=\begin{pmatrix}  \frac{1}{a_+ (\lambda)} & 1\\ - 1 & - a_- (\lambda)  \end{pmatrix} \psi_\flat^S, \quad \Upsilon_\flat^R=\begin{pmatrix}  \frac{1}{a_+ (\lambda)} & 1\\ - 1 & - a_- (\lambda)  \end{pmatrix} \psi_\flat^R,
$$
what will imply the desired results for $\psi$. By Lemma \ref{secondmatching}, we can decompose
\begin{equation} \label{id:proof:decomposition-varphi} \Upsilon = \chi_{y_0\xi^{-1}}\left(\beta_1 \left(\begin{array}{c}
       \Phi_1\\
       \Psi_1
     \end{array}\right) + \beta_2 \left(\begin{array}{c}
       \widetilde{\Phi}_2\\
       \widetilde{\Psi}_2
     \end{array}\right)\right) +(1-\chi_{y_0\xi^{-1}})\left(\alpha_2 \left(\begin{array}{c}
       \varphi_2\\
       \psi_2
     \end{array}\right) + \alpha_3 \left(\begin{array}{c}
       \varphi_3\\
       \psi_3
     \end{array}\right) + \alpha_4 \left(\begin{array}{c}
       \varphi_4\\
       \psi_4
     \end{array}\right) \right)
\end{equation}

\noindent \textbf{Step 1}. \emph{The coefficients $b_\flat$ and $c_\flat$ and proof of \eqref{bd:estimates-b-c-flat}}. In view of the leading order term $\psi^S_\flat$ in the desired decomposition \eqref{decomposition-psi-flat}-\eqref{id:psi-flat-S}, we define
$$
b_\flat= C(\lambda) \sqrt{\frac{2}{\pi}}\alpha_3 \quad \mbox{and} \quad c_\flat= C(\lambda)\sqrt{\frac{2}{\pi}}\alpha_4 .
$$
Then we have indeed $b_\flat^2+c_\flat^2=1$ since $(\alpha_3)^2 + (\alpha_4)^2 = \frac{\pi}{2 C^{2}(\lambda)}$ from Lemma \ref{secondmatching}. The desired estimates \eqref{bd:estimates-b-c-flat} then follow directly from \eqref{estimate-parameters-matching-low-1}.

\medskip

\noindent \textbf{Step 2}. \emph{The remainder $\Upsilon^R_\flat$ and proof of \eqref{id:psi-flat-R}, \eqref{bd:m-flat-loc} and \eqref{bd:psi-flat-R}}. Once $\psi_\flat^S$ is defined by \eqref{id:psi-flat-S} with $b_\flat,c_\flat$ as above, we compute the remainder $\Upsilon^R_\flat$ using \eqref{decomposition-psi-flat} and \eqref{id:proof:decomposition-varphi}, and it can be decomposed as
\begin{align*}
\Upsilon^R_\flat &=  \chi_{y_0\xi^{-1}}(r)\left(\beta_1\begin{pmatrix}\Phi_1\\ \Psi_1\end{pmatrix} +\beta_2\begin{pmatrix} \widetilde \Phi_2\\ \widetilde \Psi_2\end{pmatrix} \right)+(1-\chi_{y_0\xi^{-1}})\left(\alpha_2\begin{pmatrix}\varphi_2 \\ \psi_2 \end{pmatrix} +\alpha_3\begin{pmatrix}\varphi_3 \\ \psi_3 \end{pmatrix}+\alpha_4\begin{pmatrix}\varphi_4 \\ \psi_4 \end{pmatrix}\right) \\
&\quad - \alpha_3 \left[(\rho -1)\chi_{\xi^{-1}}+J_0(\xi r) \right]\begin{pmatrix} 0\\ 1 \end{pmatrix}-\frac{\alpha_4}{\sqrt{\xi r}}\sin(\xi r -\frac{\pi }{4}) (1-\chi_{\xi^{-1}})\begin{pmatrix} 0\\ 1 \end{pmatrix} \\
& = \beta_1 \chi_{y_0\xi^{-1}}\begin{pmatrix}\Phi_1\\ \Psi_1\end{pmatrix}+(\beta_2-\alpha_3) \chi_{y_0\xi^{-1}}\begin{pmatrix} \widetilde \Phi_2\\ \widetilde \Psi_2\end{pmatrix}+\alpha_3 (\chi_{y_0\xi^{-1}}-\chi_{\xi^{-1}}) (\rho-1) \\
& \quad +\alpha_3 \chi_{y_0\xi^{-1}}\left( \begin{pmatrix} \widetilde \Phi_2\\ \widetilde \Psi_2\end{pmatrix}-[\rho-1+J_0(\xi r)]\begin{pmatrix} 0 \\ 1\end{pmatrix}\right) + \alpha_2 (1-\chi_{y_0\xi^{-1}})\begin{pmatrix} \varphi_2 \\ \psi_2 \end{pmatrix}\\
& \quad+ \alpha_3 (1-\chi_{y_0 \xi^{-1}})\left( \begin{pmatrix} \varphi_3 \\ \psi_3 \end{pmatrix}-J_0(\xi r)\begin{pmatrix} 0 \\ 1 \end{pmatrix}\right) +\alpha_4 (1-\chi_{y_0\xi^{-1}})\left( \begin{pmatrix} \varphi_4 \\ \psi_4 \end{pmatrix} -\frac{\sin (\xi r-\frac \pi 4)}{\sqrt{\xi r}}\begin{pmatrix} 0 \\ 1 \end{pmatrix} \right)\\
& =\Upsilon^{R,1}_\flat +\Upsilon^{R,2}_\flat +\Upsilon^{R,3}_\flat +\Upsilon^{R,4}_\flat +\Upsilon^{R,5}_\flat +\Upsilon^{R,6}_\flat +\Upsilon^{R,7}_\flat .
\end{align*}

\smallskip

\noindent \underline{Estimate for $\Upsilon_\flat^{R,1}$}. We have $|\partial_\lambda^j \beta_1 |\lesssim e^{- 3y_0\kappa /\xi}$ by \eqref{estimate-parameters-matching-low-2}, and $|\partial_\lambda^j \partial_r^k (\Phi_1,\Psi_1)|\lesssim \langle r \rangle^{-1/2}e^{\kappa r}$. Hence, since $\chi_{y_0\xi^{-1}}$ localises for $r\leq 2y_0/\xi$, and that $e^{-3y_0\kappa/\xi}e^{r\kappa}\lesssim e^{-y_0/\xi}$ in this zone as $\kappa\geq 1$, we infer $ \left|\partial_\lambda^j \partial_r^k \Upsilon^{R,1}\right|\lesssim e^{-\frac{y_0}{\xi}} $ which will give an irrelevant contribution as $\xi\to 0$. Indeed, factorizing
\begin{align*}
& \Upsilon^{R,1}_\flat=M^{R,1}_{\flat,1}\cos(\xi r)+M^{R,1}_{\flat,2}\sin(\xi r) , \quad M^{R,1}_{\flat,1}= \cos(\xi r)\Upsilon^{R,1}, \quad M^{R,1}_{\flat,2}= \sin(\xi r)\Upsilon^{R,1}, 
\end{align*}
we have
\begin{equation} \label{bd:mflat-inter-1}
|\partial_r^k \partial_\xi^j M^{R,1}_{\flat,1}|+|\partial_r^k \partial_\xi^j M^{R,1}_{\flat,2}|\lesssim e^{-\frac{y_0}{2\xi}}\mathbbm (r\leq 2y_0/\xi).
\end{equation}

\smallskip
     
\noindent \underline{Estimates for $\Upsilon_\flat^{R,2}$ and $\Upsilon_\flat^{R,3}$}. We have $|\partial_\lambda^j (\beta_2-\alpha_3) |\lesssim \lambda^{2-j} \ln^2 \lambda$ by \eqref{estimate-parameters-matching-low-1}, and we have $|\partial_\lambda^j \partial_r^k (\widetilde \Phi_2,\widetilde \Phi_3)|\lesssim \langle r\rangle^{j-k}$ using Lemma \ref{aghanim}, hence $|\partial_\lambda^j \partial_r^k \Upsilon^{R,2}_\flat|\lesssim \lambda^{2-j}\ln^2 \lambda \langle r\rangle^{-k}$. We have $|\partial_r^k(\rho-1)|\lesssim \langle r\rangle^{-2}$, so that using \eqref{estimate-parameters-matching-low-1} and that $\chi_{y_0\xi^{-1}}-\chi_{\xi^{-1}}$ has support inside $\{ y_0/(2\xi)\leq r \leq 2/\xi$ we infer $|\partial_\lambda^j \partial_r^k \Upsilon^{R,3}_\flat|\lesssim \lambda^{2-j} \langle r\rangle^{-k}$. Factorizing
\begin{align*}
& \Upsilon^{R,l}_\flat=M^{R,l}_{\flat,1}\cos(\xi r)+M^{R,l}_{\flat,2}\sin(\xi r), \quad M^{R,l}_{\flat,1}= \cos(\xi r)\Upsilon^{R,2}, \quad M^{R,2}_{\flat,2}= \sin(\xi r)\Upsilon^{R,2}, 
\end{align*}
we have, since $\Upsilon^{R,2}_\flat$ is supported in $\{r\leq 2y_0/\xi\}$ and $\Upsilon^{R,3}_\flat$ is supported in $\{y_0/(2\xi)\leq r\leq 2/\xi\}$,
\begin{equation} \label{bd:mflat-inter-2}
|\partial_r^k \partial_\xi^j M^{R,l}_{\flat,1}|+|\partial_r^k \partial_\xi^j M^{R,l}_{\flat,2}|\lesssim \lambda^{2-j} \ln^2 \lambda \langle r\rangle^{-k}\mathbbm 1 (r\leq 2/\xi).
\end{equation}
for $l=2,3$.

\smallskip

\noindent \underline{Estimate for $\Upsilon_\flat^{R,4}$}. Combining Lemma \ref{aghanim} and Lemma \ref{f0f}, we have
\begin{align*}
\Upsilon^{R,4}_\flat &=\alpha_3 \chi_{y_0\xi^{-1}}\left(\frac{\lambda}{2 \langle \lambda \rangle}  (- r \rho' + (\rho^2 - 1) (J_0 (r \xi) - 1)) \begin{pmatrix} 1\\ 0 \end{pmatrix}+ \begin{pmatrix} R_1\\ R_2\end{pmatrix} \right) \\
& = \chi_{\xi^{-1}}M_{\flat}^{\tmop{loc}} (\xi, r)+M_{\flat,1}^{R,4} (\xi, r)\cos (\xi r)+M_{\flat,2}^{R,4} (\xi, r)\sin (\xi r)
\end{align*}
where we introduced
\begin{align*}
& M_{\flat}^{\tmop{loc}} (\xi, r) =\frac{\lambda}{2\langle \lambda \rangle} \sqrt{\pi} \chi_{y_0\xi^{-1}} (- r \rho' + (\rho^2 - 1) (J_0 (r \xi) - 1)) \begin{pmatrix} 1\\ 0 \end{pmatrix} ,\\
& M_{\flat,1}^{R,4} =\cos(\xi r)\widetilde \Upsilon^{R,4}_\flat, \qquad M_{\flat,2}^{R,4} =\sin(\xi r)\widetilde \Upsilon^{R,4}_\flat,\\
& \widetilde \Upsilon^{R,4}_\flat= \chi_{y_0\xi^{-1}}\left(( \alpha_3-\sqrt{\pi})\frac{\lambda}{2 \langle \lambda \rangle}  (- r \rho' + (\rho^2 - 1) (J_0 (r \xi) - 1)) \begin{pmatrix} 1\\ 0 \end{pmatrix}+ \alpha_3 \begin{pmatrix} R_1\\ R_2\end{pmatrix} \right).
\end{align*}
Since $\rho=1+O(r^{-2})$ as $r\to \infty$ we have
\begin{equation} \label{bd:mloc-inter-2}
\left| \partial_{\xi}^k \partial_r^j \left( \frac{M_{\flat}^{\tmop{loc}}(\xi, r)}{\xi} \right) \right| \lesssim_{j, k} (1 + r)^{k - j-2} .
\end{equation}
By Lemma \ref{aghanim} and \eqref{estimate-parameters-matching-low-1}, as $\chi_{y_0\xi^{-1}}$ has support inside $\{r\leq 2y_0/\xi\}$ we have $|\partial_\lambda^j \partial_r^k \widetilde \Upsilon^{R,4}_\flat|\lesssim \lambda^{2-j} \ln^2 \lambda \langle r\rangle^{-k}$. Hence
\begin{equation} \label{bd:mflat-inter-3}
|\partial_r^k \partial_\xi^j M^{R,4}_{\flat,1}|+|\partial_r^k \partial_\xi^j M^{R,4}_{\flat,2}|\lesssim \lambda^{2-j} \ln^2 \lambda \langle r\rangle^{-k}\mathbbm 1 (r\leq 2/\xi).
\end{equation}

\smallskip

\noindent \underline{Estimate for $\Upsilon_\flat^{R,5}$}. We have $|\partial_\lambda^j \alpha_1 |\lesssim e^{2y_0\kappa /(3\xi)}$ by \eqref{estimate-parameters-matching-low-2}, and $|\partial_\lambda^j \partial_r^k (\varphi_2 , \psi_2)|\lesssim \langle r \rangle^{-1/2}e^{-\kappa r}$ by Lemma \ref{alegant}. Hence, since $1-\chi_{y_0\xi^{-1}}$ has support in $\{r\geq y_0/\xi\}$ where $e^{2y_0\kappa /(3\xi)} e^{-\kappa r}\lesssim e^{-\frac{y_0}{5\xi}}e^{-\frac{r}{5}}$ we infer $|\partial_\lambda^j \partial_r^k \Upsilon_\flat^{R,5}|\lesssim e^{-\frac{y_0}{6\xi}}e^{-\frac{r}{6}}$. This term can therefore be factorized similarly as $\Upsilon^{R,1}_\flat $, namely
\begin{align}
\nonumber & \Upsilon^{R,5}_\flat=M^{R,5}_{\flat,1}\cos(\xi r)+M^{R,5}_{\flat,2}\sin(\xi r) ,\\
& \label{bd:mflat-inter-5} |\partial_r^k \partial_\xi^j M^{R,5}_{\flat,1}|+|\partial_r^k \partial_\xi^j M^{R,5}_{\flat,2}|\lesssim  e^{-\frac{y_0}{6\xi}}e^{-\frac{r}{7}}.
\end{align}

\smallskip

\noindent \underline{Estimates for $\Upsilon_\flat^{R,6}$ and $\Upsilon_\flat^{R,7}$}. We have by Lemma \ref{mana}
$$
\Upsilon^{R,6}_\flat =  \alpha_3 (1-\chi_{y_0 \xi^{-1}})\left( J_0(\xi r)\begin{pmatrix} \textup{Re}(h_1)  \\ \textup{Re}(h_2-1)  \end{pmatrix}-Y_0(\xi r)\begin{pmatrix} \textup{Im}(h_1)  \\ \textup{Im}(h_2)  \end{pmatrix}\right) .
$$
We decompose $J_0(y)= \cos (y) H_{1,1}(y)+ \sin (y) H_{1,2}(y)$ and $Y_0(y)= \cos (y) H_{2,1}(y)+ \sin (y) H_{2,2}(y)$ where $|\partial_y^k H_{l,m}(y)|\lesssim \langle y \rangle^{-k-1/2}$ by Lemma \ref{abudhabi}. Therefore:
\begin{align*}
\Upsilon^{R,6}_\flat & = \cos (\xi r)  \alpha_3 (1-\chi_{y_0 \xi^{-1}}) \left( H_{1,1}(\xi r)\begin{pmatrix} \textup{Re}(h_1)  \\ \textup{Re}(h_2-1)  \end{pmatrix}-H_{2,1}(\xi r)\begin{pmatrix} \textup{Im}(h_1)  \\ \textup{Im}(h_2)  \end{pmatrix}\right)\\
&+\sin (\xi r)  \alpha_3 (1-\chi_{y_0 \xi^{-1}}) \left( H_{1,2}(\xi r)\begin{pmatrix} \textup{Re}(h_1)  \\ \textup{Re}(h_2-1)  \end{pmatrix}-H_{2,2}(\xi r)\begin{pmatrix} \textup{Im}(h_1)  \\ \textup{Im}(h_2)  \end{pmatrix}\right)  .
\end{align*}
We have using \eqref{estimate-parameters-matching-low-2}, that $1-\chi_{y_0 \xi^{-1}}$ has support in $\{r\geq y_0/(2\xi)\}$ and Lemma \ref{mana},
$$
\left| \partial_\lambda^j\partial_r^k\left( \alpha_3 (1-\chi_{y_0 \xi^{-1}}) \left( H_{1,1}(\xi r)\begin{pmatrix} \textup{Re}(h_1)  \\ \textup{Re}(h_2-1)  \end{pmatrix}-H_{2,1}(\xi r)\begin{pmatrix} \textup{Im}(h_1)  \\ \textup{Im}(h_2)  \end{pmatrix}\right)\right) \right|\lesssim \frac{\lambda^{2-j}}{(\xi r)^{\frac 32}r^k}\mathbbm 1(r\geq \frac{y_0}{2\xi}).
$$
The second term in the decomposition of $\Upsilon^{R,6}_\flat$ can be estimated similarly. In addition, $\Upsilon_\flat^{R,7}$ can be decomposed similarly using Lemma \ref{mana} and \eqref{estimate-parameters-matching-low-2}. Therefore, we have
\begin{align}
\nonumber & \Upsilon^{R,l}_\flat=M^{R,l}_{\flat,1}\cos(\xi r)+M^{R,l}_{\flat,2}\sin(\xi r) ,\\
& \label{bd:mflat-inter-6} |\partial_r^k \partial_\xi^j M^{R,l}_{\flat,1}|+|\partial_r^k \partial_\xi^j M^{R,l}_{\flat,2}|\lesssim  \frac{\lambda^{2-j}}{(\xi r)^{\frac 32}r^k}\mathbbm 1(r\geq \frac{y_0}{2\xi})
\end{align}
for $l=6,7$.

\smallskip

\noindent \underline{End of the proof}. We finally set $M_{\flat,l}^R=\sum_{n=1}^7 M_{\flat,l}^{R,n}$ for $l=1,2$. This implies the decomposition
$$
\Upsilon^R_\flat =M^{\textup{loc}}_\flat \chi_{\xi^{-1}}(r)+M^{R}_{\flat,1}\cos(\xi r)+M^{R}_{\flat,2}\sin(\xi r)
$$
where, by \eqref{bd:mflat-inter-1}, \eqref{bd:mflat-inter-2}, \eqref{bd:mflat-inter-3}, \eqref{bd:mflat-inter-5} and \eqref{bd:mflat-inter-6},
$$
 |\partial_r^k \partial_\xi^j M^{R}_{\flat,1}|+|\partial_r^k \partial_\xi^j M^{R}_{\flat,2}|\lesssim_{j,k} \frac{\xi^{2 - j} \ln^2 (\xi)}{\langle \xi r \rangle^{3 / 2} (1 + r)^k} 
$$
and $M^{\textup{loc}}_\flat$ satisfies \eqref{bd:mloc-inter-2}. Going back in original variables, we define
$$
m^{R}_{\flat,l} = \begin{pmatrix}  \frac{1}{a_+ (\lambda)} & 1\\ - 1 & - a_- (\lambda)  \end{pmatrix}^{-1} M^{R}_{\flat,l} \quad \mbox{and} \quad M^{\textup{loc}}_{\flat} = \begin{pmatrix}  \frac{1}{a_+ (\lambda)} & 1\\ - 1 & - a_- (\lambda)  \end{pmatrix}^{-1} M^{\textup{loc}}_{\flat},
$$
for $l=1,2$. This implies the desired identity \eqref{id:psi-flat-R}. Since the above matrix in the right-hand side is smooth at $\lambda=0$, we readily obtain the desired estimates \eqref{bd:m-flat-loc} and \eqref{bd:psi-flat-R} from the estimates on $M^{R}_{\flat,1}$, $M^{R}_{\flat,2}$ and $M^{\textup{loc}}_{\flat}$.

\end{proof}

\appendix

\section{Spaces of symbols and inverse differential operators}
\label{sectionspace}

This section is devoted to the key technical result used in the construction
of Jost solutions. It focuses on functions $f \in \mathcal{C}^{\infty} ([r_0, + \infty
[, \mathbb{C})$ which enjoy symbol-type estimates: they decay at infinity at a polynomial rate, $| f (r) | \leqslant
\frac{1}{r^b}$ for some $b \in \mathbb{R}$, and each derivative decays faster by a factor $\frac{1}{r}$: $| \partial_r^l f (r) | \lesssim \frac{1}{r^{b + l}}$. If this inequality is satisfied  for all $l \in [0 \ldots j]$, then the function $f$ belongs to the space
$\mathcal{W}_b^j (r_0^+)$ defined below.

The main result of this section, Proposition \ref{grooving}, is to show that
if $a$ and $f$ are two functions in such spaces, then we can find a solution
$g$ of the equation $\partial_r g + a g = f$ which also belongs to these spaces. In addition, we will let $a$ and $f$ depend on a second parameter $\lambda$, with similar
decay properties, and show that $g$ inherits them.

\subsection{Definitions of the $\mathcal{W}$ spaces}

\begin{definition}[The $\mathcal{W}$ spaces]
  Given $b \in \mathbb{R}, j  \in \mathbb{N}, r_0 > 0$, we define the space
  \[ \mathcal{W}_b^j (r_0^+) = \{ f \in C^j ([r_0, + \infty ),
     \mathbb{C}), \quad \| r^{b + l} f^{(l)} (r)
     \|_{L^{\infty} ([r_0, + \infty ))} < + \infty \quad \mbox{for all }0\leq j \leq l\} . \]
  We define similarly
  \[ \mathcal{W}_b^j (r_0^-) = \{ f \in C^j ((0,r_0]),
     \mathbb{C}), \quad \| r^{b + l} f^{(l)} (r)
     \|_{L^{\infty} ((0, r_0])} < + \infty  \quad \mbox{for all }0\leq j \leq l . \}\]
  Given $b \in \mathbb{R}, j, k \in \mathbb{N}, r_0, \lambda_0 > 0$, we
  define the space $\mathcal{W}^{j, k}_{b} (r_0^+, \lambda_0^+)$ as the
  subspace of
  \[ (r, \lambda) \rightarrow f (r, \lambda) \in \mathcal{C}^{j,k} ([r_0, + \infty
     [\times [\lambda_0, + \infty ), \mathbb{C}) \]
  such that for all $l \leqslant j, m \leqslant k$,
  \[ \| r^{b + l} \lambda^{ m} \partial_r^l \partial_{\lambda}^m f  (r,
     \lambda) \|_{L^{\infty} ([r_0, + \infty )\times [\lambda_0, + \infty ))}
     < + \infty . \]
  We extend naturally this definition to $\mathcal{W}^{j, k}_{b}
  (r_0^{\pm}, \lambda_0^{\pm})$. Finally, we define
  \[ \mathcal{W}_{b} (r_0^{\pm}, \lambda_0^{\pm}) = \bigcap_{j, k \in
     \mathbb{N}} \mathcal{W}^{j, k}_{b} (r_0^{\pm}, \lambda_0^{\pm}) . \]
\end{definition}

Typically, $\frac{1}{r^b} \in \mathcal{W}_b (1^+), \frac{2+\lambda}{r^b(1+\lambda)} \in
\mathcal{W}_{b, c} (1^+, 1^+)$ but $\frac{\sin (r)}{r^b} \nin \mathcal{W}_c
(r_0^+)$ for any $c \in \mathbb{R}, r_0 > 0$.  We define the set $\Lambda_0^{\pm}$ by $\Lambda_0^+ = [\lambda_0, + \infty )$ and $\Lambda_0^- = ( 0, \lambda_0]$.

\begin{definition}[Norms on the $\mathcal{W}$ spaces]\label{phynee}
The spaces above will be endowed with the following norms
\begin{itemize}
\item[1)] Norms on $\mathcal{W}^j_b (r_0^+)$. Given $j \in \mathbb{N}, b \in
  \mathbb{Z}$, we define
  \[ \| f \|_{b, j} = \sum_{l = 0}^j \| r^b r^l \partial_r^l f
     \|_{L^{\infty} ([r_0, + \infty ))} \]
  on functions $f \in \mathcal{W}^j_b (r_0^+)$.
  
\item[2)] Norms on $\mathcal{W}^{j, k}_{b} (r_0^+, \lambda_0^{\pm})$. Fix some function
  $\nu:\lambda\mapsto \nu(\lambda) \geqslant 0$. For $a, c, n, m \in \mathbb{N}$ we define the norm
  \[ \mathcal{N}_{b, c}^{n, m} (f) = \sum_{j = 0}^n \sum_{k = 0}^{n + m
     - j} \| r^b (r + \nu)^c r^k \lambda^j \partial_r^k \partial_{\lambda}^j f
     \|_{L^{\infty} ([r_0, + \infty )\times \Lambda_0^{\pm})}. \]
  Note that the associated Banach space is a subset of $\mathcal{W}^{j, k}_{b+c} (r_0^+, \lambda_0^{\pm})$, and is equal to this set if $\sup_{\Lambda_0^\pm} \nu(\lambda)<\infty$.
\end{itemize}
\end{definition}

Remark that $\| f \|_{a, j} =\mathcal{N}_{a, 0}^{j, 0} (f)$. Let us furthermore explain
the reason for the derivative count in the definition of $\mathcal{N}_{b, c}^{n, m}$. With $D_a (g) = \partial_r g + a g$, we will aim at estimating $D_a^{- 1} (f)$ in terms of $f$. It will become clear in the following that in order to bound $k$ $r$-derivatives of $D_a^{-1} f$, $k$ $r$-derivatives on $f$ are needed, but in order to estimate $j$ $\lambda$-derivatives of $D_a^{-1} f$, $j$ derivatives on $f$ are needed, which might be chosen as derivatives with respect to $r$ or $\lambda$. Thus the norms $\mathcal{N}_{b, c}^{n, m}$ are choosen to be stable by $D_a^{- 1}$

Regarding the factor $\nu \geqslant 0$, remark that for $\nu=\nu_0$ constant for instance the norms
$\mathcal{N}_{1, 0}^{0, 0}$ and $\mathcal{N}_{0, 1}^{0, 0}$ are equivalent, but the factor between them depends on $\nu_0$. For $\nu$ non-constant this norm allows to propagate finer decay properties. We will need to track precisely
the dependence on $\nu$, which explains why we introduced this variation.

The following lemma is easily checked.

\begin{lem}[Sum and product laws]
  \label{hotshot}
\begin{itemize} 
\item[1)] Given $f \in \mathcal{W}^{n + m, n}_{b + c} (r_0^+,
  \lambda_0^{\pm}), g \in \mathcal{W}^{n + m, n}_{b + c} (r_0^+,
  \lambda_0^{\pm})$, we have
  \[ \mathcal{N}_{b, c}^{n, m} (f + g) \lesssim
     \mathcal{N}_{b,c}^{n, m} (f) +\mathcal{N}_{b, c}^{n, m} (g) . \]
\item[2)] Under the same notations, we have
  \[ \mathcal{N}_{b_1 + b_2, c_1 + c_2}^{n, m} (f g) \lesssim_{n, m}
     \mathcal{N}_{b_1, c_1}^{n, m} (f) \mathcal{N}_{b_2, c_2}^{n, m} (g) . \]
  The implicit constants are independent of $\nu$ and $r_0$.
\end{itemize}
\end{lem}

We recall that our goal is to study solutions of the equation $\partial_r g +
a g = f$ given functions $a, f$. Here, we introduce properties required on $a$
to do so.

\begin{definition}[Types for the coefficient $a(r,\lambda)$]
  \label{cadmissible}Given $\lambda_0 > 0$ and a choice of $\pm$, Taking now $r_0 > 0$, we say that a
  function
  \[ (r, \lambda) \rightarrow a (r, \lambda) \in C^{\infty, \infty} ([r_0, +
     \infty )\times \Lambda_0^{\pm}, \mathbb{C}) \]
  that can be decomposed as
  \[ a (r, \lambda) = \widetilde{a} (\lambda) + \bar{a} (r, \lambda) \]
  where $\widetilde{a} \in \mathcal{W}_0 (\lambda_0^{\pm}), \bar{a} \in
  \mathcal{W}_{2} (r_0^+, \lambda_0^{\pm})$ is:

  \begin{itemize}
\item of type $0$ if $\widetilde{a} = 0$.
  
\item of type $> 0$ if $a_0 = \inf_{\lambda \in \Lambda_0^{\pm}}
  \mathfrak{R}\mathfrak{e} (\widetilde{a} (\lambda)) > 0$.
  
\item of type $< 0$ if $- a_0 = \sup_{\lambda \in \Lambda_0^{\pm}}
  \mathfrak{R}\mathfrak{e} (\widetilde{a} (\lambda)) < 0$.
  
\item of type $i$ if for all $\lambda \in \Lambda_0^{\pm}$,
  $\mathfrak{R}\mathfrak{e} (\widetilde{a} (\lambda)) = 0$ and $a_0 =
  \inf_{\lambda \in \Lambda_0^{\pm}} | \widetilde{a} (\lambda) | > 0$.
  \end{itemize}
\end{definition}

Finally, we come to the inverse differential operators whose boundedness properties will be of interest.
\begin{definition}[Direct and inverse differential operators]\label{fnpig}
Given a function $a$, we define
\[ D_a (f) = \partial_r f + a f. \]
Depending on the type of $a$, we define next
\begin{itemize}  
\item[1)] If it is of type $0, < 0$ or $i$, we define for $r \geqslant r_0$
  \[ D_a^{- 1} (f) = - \int_r^{+ \infty} e^{\int_r^t a (s, \lambda)\dd s}
     f (t, \lambda)\dd t. \]
  
\item[2)] If it is of type $> 0$, we define for $r \geqslant r_0$
  \[ D_a^{- 1} (f) = \int_1^r e^{\int_r^t a (s, \lambda)\dd s} f (t,
     \lambda)\dd t. \]
\end{itemize}

In both cases, $D_a D_a^{- 1} = \operatorname{Id}$ on functions $(r, \lambda) \rightarrow
  f (r, \lambda) \in \mathcal{W}_{2} (r_0^+, \lambda_0^{\pm})$.
\end{definition}

In the rest of this section, we will prove boundedness of the operator $D_a^{- 1}$ for various combinations of the norms $\mathcal{N}_{b, c}^{n, m}$ provided $a$ is of one of the types defined above. Implicit constants will be allowed to depend on $n, m,
b, c, \mathcal{N}_{b, c}^{n, m} (a)$ and $a_0$ from Definition
\ref{cadmissible}, but not on $r_0, \lambda_0$ and $\nu$.

\subsection{Formulae on $D_a^{- 1}$}

\begin{lem}[Integration by parts in $D_a^{-1}$] 
  \label{khen}Consider $a = \widetilde{a} + \bar{a}$ of one of the types of Definition \ref{cadmissible} for some $r_0, \lambda_0 > 0$.
\begin{itemize}
\item[1)] If $a$ is of type $< 0$ or $i$, we let
  \[ P_a (f) (r, \lambda) = \frac{f (r, \lambda)}{\widetilde{a}} . \]
  
\item[2)] If $a$ is of type $> 0$, we let
  \[ P_a (f) (r, \lambda) = \frac{f (r, \lambda) - e^{\int_r^1 a (s, \lambda)
    \dd s} f (1, \lambda)}{\widetilde{a}} . \]
\end{itemize}

  In both cases, we have the identity
  \[ D_a^{- 1} = P_a - \frac{1}{\widetilde{a}} D_a^{- 1} D_{\bar{a}} \]
on the space on $\mathcal{W}^{n + m,
  n}_{b + c} (r_0^+, \lambda_0^{\pm}) $ with $b + c \geqslant 2$.
Furthermore, for any $b, c, n, m, \nu \geqslant 0$ with $b + c \geqslant 2$, we have
  \[ \mathcal{N}_{b, c}^{n, m} (P_a (f)) \lesssim \mathcal{N}_{b, c}^{n, m}
     (f), \]
  where the implicit constant is independent of $\nu, r_0$ and $\lambda_0$.
\end{lem}

\begin{proof}
  First, remark that if $a$ is of type $< 0$ or $i$, we have for any $t
  \geqslant r \geqslant r_0$ that
  \begin{equation}
    \left| e^{\int_r^t a (s, \lambda)\dd s} \right| \lesssim 1, \label{ytbh}
  \end{equation}
  since $a = \widetilde{a} + \bar{a}$ with $\mathfrak{R}\mathfrak{e} (\widetilde{a})
  \leqslant 0$ and $\bar{a} \in \mathcal{W}_{2,
  0} (r_0^+, \Lambda_0^{\pm}) \subset L^1 ([r_0, + \infty )) $.

Now, by integration by parts, defining $d = 1$ if $a = \widetilde{a} + \bar{a}$
  is of type $> 0$ and $d = + \infty$ otherwise, we have
  \begin{align*}
D_a^{- 1} (f) & = \int_d^r e^{(t - r) \widetilde{a} (\lambda) + \int_r^t
    \bar{a} (s, \lambda)\dd s} f (t, \lambda)\dd t\\
    & =  \frac{1}{\widetilde{a} (\lambda)} \int_d^r \partial_t (e^{(t - r)
    \widetilde{a} (\lambda)}) e^{\int_r^t \bar{a} (s, \lambda)\dd s} f (t, \lambda)\dd t\\
    & =  \frac{1}{\widetilde{a} (\lambda)} \left( f (r, \lambda) - e^{\int_r^d a
    (s, \lambda)\dd s} f (d, \lambda) \right) -  \frac{1}{\widetilde{a} (\lambda)} \int_d^r e^{(t - r) \widetilde{a}
    (\lambda) + \int_r^t \bar{a} (s, \lambda)\dd s} (\partial_t f + \bar{a} f)
    (t, \lambda)\dd t\\
    & =  P_a (f) - D_a^{- 1} (D_{\bar{a} } f)
  \end{align*}
  since $D_{\bar{a} } f = \partial_r f + \bar{a} f$ and in the case $d = +
  \infty$, using (\ref{ytbh}) and $f (+ \infty, \lambda) = 0$.
  
  Now, remark that for any $n, m \in \mathbb{N}$,
  \[ \mathcal{N}_{0, 0}^{n, m} \left( \frac{1}{\widetilde{a}} \right) \lesssim 1.
  \]
  By Lemma \ref{hotshot}.2, this implies that
  \[ \mathcal{N}_{b, c}^{n, m} \left( \frac{f}{\widetilde{a}} \right) \lesssim
     \mathcal{N}_{b, c}^{n, m} (f) . \]
  This completes the proof in the case 1. If now $a=\widetilde{a}+\bar{a}$ is of type $> 0$, we have $a \geq a_0 + \bar{a}$ and with $\bar{a} \in L^1 ([r_0, + \infty )) $, we deduce that
  \[ \left| e^{\int_r^1 a (s, \lambda)\dd s} \right| \lesssim e^{(1 -
     r) a_0} \]
and we check that for any $j, k \in \mathbb{N}$,
  \[ \left| r^k \partial_r^k \lambda^j \partial_{\lambda}^j \left( e^{\int_r^1
     a (s, \lambda)\dd s} \right) \right| \leqslant r^{j + k} e^{(1 -
     r) a_0} . \]
  This implies that for any $m, n, b \in \mathbb{N}$,
  \[ \mathcal{N}_{b, 0}^{n, m} \left( e^{\int_r^1 a (s, \lambda)\dd s} \right)
     \lesssim_{b,n,m} 1. \]
  Finally, remark that for any $n, m \in \mathbb{N}$,
  \[ \mathcal{N}_{0, 0}^{n, m} (f (1, \lambda)) \lesssim \frac{\mathcal{N}_{b,
     c}^{n, m} (f)}{(1 + \nu)^c} \]
  which, combined with previous estimates and
  \[ \mathcal{N}_{b, c}^{n, m} \left( \frac{g}{(1 + \nu)^c} \right) \lesssim
     \mathcal{N}_{b + c, 0}^{n, m} (g), \]
  leads to
  \begin{align*}
\mathcal{N}_{b, c}^{n, m} \left( \frac{e^{\int_r^1 a (s, \lambda) d
    s} f (1, \lambda)}{\widetilde{a}} \right) & \lesssim  \mathcal{N}_{b, c}^{n, m} (f) \mathcal{N}_{b, c}^{n, m}
    \left( \frac{e^{\int_r^1 a (s, \lambda)\dd s}}{(1 + \nu)^c} \right) \\
& \lesssim  \mathcal{N}_{b, c}^{n, m} (f) \mathcal{N}_{b + c, 0}^{n, m}
    \left( e^{\int_r^1 a (s, \lambda)\dd s} \right) \lesssim  \mathcal{N}_{b, c}^{n, m} (f),
  \end{align*}
  concluding the proof in the case 2.
\end{proof}

\begin{lem}[Commuting derivatives with $D_a^{-1}$] \label{unknown}Consider $a = \widetilde{a} + \bar{a}$ of one of the types of
  Definition \ref{cadmissible}.
\begin{itemize}
\item[1)] We have
  \[ r \partial_r (D_a^{- 1}) = r \tmop{Id} - r a D_a^{- 1} . \]

\item[2)] We have
  \[ \lambda \partial_{\lambda} (D_a^{- 1}) = D_a^{- 1} \left( \left( \int_1^\cdot
     \lambda \partial_{\lambda} a + \lambda \partial_{\lambda} \right)
     \tmop{Id} \right) - \left( \int_1^\cdot \lambda \partial_{\lambda} a \right)
     D_a^{- 1} . \]
\end{itemize}
\end{lem}

These results will be used to derive estimates on $D_a^{- 1} (f)$ involving derivatives from pointwise estimates on $D_a^{-1}$.

\begin{proof}
  Recall that for all types of $a$, we have $D_a D_a^{- 1} = \tmop{Id}$, and
  since $D_a = \partial_r + a$, we deduce that
  \[ r \partial_r (D_a^{- 1}) = r (D_a - a) D_a^{- 1} = r \tmop{Id} - r a
     D_a^{- 1} . \]
  Now, recall also that there exists $d \in \{ 1, + \infty \}$ such that
  \[ D_a^{- 1} (f) = \int_d^r e^{\int_r^t a (s, \lambda)\dd s} f (t,
     \lambda)\dd t. \]
  We compute then that
  \[ \lambda \partial_{\lambda} (D_a^{- 1} (f)) = D_a^{- 1} (\lambda
     \partial_{\lambda} f) + \int_d^r \left( \int_r^t \lambda
     \partial_{\lambda} a (s, \lambda) \dd s\right) e^{\int_r^t a (s, \lambda)
    \dd s} f (t, \lambda)\dd t \]
  and since
  \[ \int_r^t \lambda \partial_{\lambda} a (s, \lambda)\dd s = - \int_1^r
     \lambda \partial_{\lambda} a (s, \lambda)\dd s + \int_1^t \lambda
     \partial_{\lambda} a (s, \lambda)\dd s, \]
  we can conclude.
\end{proof}

\subsection{Boundedness of $D_a^{- 1}$}

\begin{prop}[Estimates for $D_a^{- 1}$ in $\mathcal{W}$ spaces]
  \label{grooving}Consider $a = \widetilde{a} + \bar{a}$ of one of the types of
  Definition \ref{cadmissible} and $n, m, b, c \in \mathbb{N}$ with $c
  \geqslant 2$.
\begin{itemize}
\item[1)] If $a$ is of type $0$, then for any $f \in \mathcal{W}_{b + c, 0}^{n, m}
  (r_0^+, \Lambda_0^{\pm})$,
  \[ \mathcal{N}_{b, c - 1}^{n, m + 1} (D_a^{- 1} (f)) \lesssim
     \mathcal{N}_{b, c}^{n, m} (f) . \]

\item[2)] If $a$ is of type $< 0$ or $>0$, then for any $f \in \mathcal{W}_{b + c,
  0}^{n, m} (r_0^+, \Lambda_0^{\pm})$,
  \[ \mathcal{N}_{b, c}^{n, m} (D_a^{- 1} (f)) \lesssim \mathcal{N}_{b, c}^{n,
     m} (f) . \]

\item[3)] If $a$ is of type $i$, then for any $f \in \mathcal{W}_{b + c, 0}^{n, m}
  (r_0^+, \Lambda_0^{\pm})$,
  \[ a) \qquad \mathcal{N}_{b, c - 1}^{n, m + 1} (D_a^{- 1} (f)) \lesssim
     \mathcal{N}_{b, c}^{n, m} (f) . \]
  and
  \[ b) \qquad \mathcal{N}_{b + 1, c - 1}^{n, m} (D_a^{- 1} (f)) \lesssim
     \mathcal{N}_{b, c}^{n, m + 1} (f) \]
as well as (for $c=2$)
\[ c) \qquad \mathcal{N}_{b, 2 }^{n, m} (D_a^{- 1} (f)) \lesssim
     \mathcal{N}_{b, 2}^{n, m + 2} (f)   \]
\end{itemize}
\end{prop}

To understand this result, let us estimate $D_a^{- 1} (f)$ for $a$ constant
and $f = \frac{1}{(r + \nu)^c}, c \geqslant 2, \nu \geqslant 0$. For $a = 0$,
which is of type $0$, we have
\[ D_0^{- 1} \left( \frac{1}{(r + \nu)^c} \right) = - \int_r^{+ \infty}
   \frac{1}{(t + \nu)^c} = \frac{1}{(c - 1) (r + \nu)^{c - 1}}, \]
and we see a loss of a factor $(r + \nu)$. The first case of this proposition
claims that this is still true if $a$ is more generally of type $0$ and $f$
has a decay like $\frac{1}{(r + \nu)^c}$. It also claims that we can ask one
more derivative on $D_a^{- 1} (f)$ than on $f$, and for $a = 0$ this is clear
since $\partial_r (D_0^{- 1}) = \tmop{Id}$, and therefore
\[ \left| r \partial_r \left( D_0^{- 1} \left( \frac{1}{(r + \nu)^c} \right)
   \right) \right| = \frac{r}{(r + \nu)^c} \lesssim \frac{1}{(r + \nu)^{c -
   1}} \]
since $\nu \geqslant 0$. In other words, we have shown that
\[ \mathcal{N}_{0, c - 1}^{0, 1} \left( D_0^{- 1} \left( \frac{1}{(r + \nu)^c}
   \right) \right) \lesssim \mathcal{N}_{0, c}^{0, 0} \left( \frac{1}{(r +
   \nu)^c} \right) . \]
To complete in general the estimate of $\mathcal{N}_{0, c - 1}^{n, m+1} \left( D_a^{- 1}\left( f
   \right) \right)$, we have to understand how derivatives commute with $D_a^{- 1}$, but this is exactly the purpose of Lemma \ref{unknown}.
For $a = - 1$, which is of type $< 0$, we have
\[ \left| D_{- 1}^{- 1} \left( \frac{1}{(r + \nu)^c} \right) \right| \lesssim
   \left| \int_r^{+ \infty} \frac{e^{r - t}}{(t + \nu)^c}\dd t \right| \lesssim
   \frac{1}{(r + \nu)^c} \int_r^{+ \infty} e^{r - t}\dd t \lesssim \frac{1}{(r
   + \nu)^c} . \]
We see here no loss of decay. For $a = + 1$, which is of type $> 0$, we have
\[ D_1^{- 1} \left( \frac{1}{(r + \nu)^c} \right) = \int_1^r \frac{e^{t -
   r}}{(t + \nu)^c}\dd t. \]
Decomposing the integral on $[1, r / 2]$ and $[r / 2, r]$ and using that the
exponential decays faster than any algebraic decay, we can prove also that we
have no loss of decay.

Finally, for $a = i$, which is of type $i$, we have
\[ D_i^{- 1} \left( \frac{1}{(r + \nu)^c} \right) = \int_r^{+ \infty}
   \frac{e^{i (t - r)}}{(t + \nu)^c}\dd t. \]
If we simply estimate $| e^{i (r - t)} | \leqslant 1$, we have the same loss of decay
as for $a = 0$. But an integration by parts using the identity by 
$e^{i (t - r)} = \frac{1}{i} \partial_t (e^{i (t - r)})$ recovers the original decay, but this comes at the price of controlling a derivative of $f$.

\begin{proof}
1) We start with the case $n = m = 0$. Then, by (\ref{ytbh}), since $c
  \geqslant 2$,
  \begin{eqnarray*}
    | D_a^{- 1} (f) (r, \lambda) |  \lesssim  \int_r^{+ \infty} | f (t,
    \lambda) |\dd t  \lesssim  \int_r^{+ \infty} \frac{1}{t^b (t + \nu)^c} d
    t\mathcal{N}_{b, c}^{0, 0} (f) \lesssim  \frac{1}{r^b (r + \nu)^{c - 1}}
  \end{eqnarray*}
  which implies that $\mathcal{N}_{b, c - 1}^{0, 0} (D_a^{- 1} (f)) \lesssim
  \mathcal{N}_{b, c}^{0, 0} (f)$. By Lemma \ref{unknown}.1, we have
  \[ r \partial_r (D_a^{- 1} f) = r f - r a D_a^{- 1} (f) \]
  but since here $a$ is of type $0$, we have $| a (r, \lambda) | \lesssim
  \frac{1}{r^2}$ hence $| r a | \lesssim 1$ and
  \[ | r f | \lesssim \frac{1}{r^b (r + \nu)^{c - 1}} \mathcal{N}_{b, c}^{0,
     0} (f), \]
  concluding the proof of
  \[ \mathcal{N}_{b, c - 1}^{0, 1} (D_a^{- 1} (f)) \lesssim \mathcal{N}_{b,
     c}^{0, 0} (f) . \]

 We will first show the result in the case $n = 0$ by induction on $m$. We have
  by Lemma \ref{unknown}.1 that
  \[ (r \partial_r)^{m + 2} (D_a^{- 1} f) = (r \partial_r)^{m + 1} (r f) - (r
     \partial_r)^{m + 1} (r a D_a^{- 1} (f)) \]
  and by Lemma \ref{hotshot},
  \begin{eqnarray*}
    \mathcal{N}_{b, c - 1}^{0, 0} ((r \partial_r)^{m + 1} (r f)) \lesssim 
    \mathcal{N}_{b - 1, c}^{0, 0} ((r \partial_r)^{m + 1} (r f)) \lesssim  \mathcal{N}_{b, c}^{0, m + 1} (f) \mathcal{N}_{- 1, 0}^{0, m + 1} (r) \lesssim \mathcal{N}_{b, c}^{0, m + 1} (f) .
  \end{eqnarray*}
  By induction, we have now
  \begin{align*}
    \mathcal{N}_{b, c - 1}^{0, 0} ((r \partial_r)^{m + 1} (r a D_a^{- 1} (f)))
    & \lesssim \mathcal{N}_{b, c - 1}^{0, m + 1} (D_a^{- 1} (f))
    \mathcal{N}_{0, 0}^{0, m + 1} (r a)\\
    & \lesssim \mathcal{N}_{b, c - 1}^{0, m + 1} (D_a^{- 1} (f)) \lesssim  \mathcal{N}_{b, c}^{0, m} (f) .
  \end{align*}
Putting together the two estimates above, we have shown that
  \begin{equation}
    \mathcal{N}_{b, c - 1}^{0, m + 2} (D_a^{- 1} f) \lesssim \mathcal{N}_{b,
    c}^{0, m + 1} (f),
  \end{equation}
  completing the proof by induction. 
  
Let us now show the estimate for $n>0$ by induction on $n$. That is, we assume the result has been shown for all values up to $n \in \mathbb{N}$ included and for any $m \in \mathbb{N}$. We want to show the estimate for $n+1$ and any $m \in \mathbb{N}$. By the definition of the $\mathcal N$ norms we have $\mathcal N^{n,m}_{b,c}(u)\approx \mathcal N^{n-1,m+1}_{b,c}(u)+\mathcal N^{n-1,m}_{b,c}(\lambda \partial_\lambda u)$. Hence
\begin{align*}
& \mathcal N^{n+1,m+1}_{b,c-1}(D_a^{-1}f)\lesssim \mathcal N^{n,m+2}_{b,c-1}(D_a^{-1}f)+\mathcal N^{n,m+1}_{b,c-1}(\lambda\partial_{\lambda}D_a^{-1}f)\\
&\qquad  \lesssim \mathcal N^{n,m+1}_{b,c}(f)+\mathcal N^{n,m+1}_{b,c-1}\left\{D_a^{- 1} \left( \left( \int_1^\cdot \lambda \partial_{\lambda} a +
    \lambda \partial_{\lambda} \right) f \right)-\left( \int_1^\cdot \lambda \partial_{\lambda} a \right) D_a^{- 1} (f) \right\}\\
    &\qquad  \lesssim \mathcal N^{n+1,m}_{b,c}(f)+\mathcal N^{n,m+1}_{b,c-1}\left\{D_a^{- 1} \left( \left( \int_1^\cdot \lambda \partial_{\lambda} a +
    \lambda \partial_{\lambda} \right) f \right)\right\}+\mathcal N^{n,m+1}_{b,c-1}\left\{\left( \int_1^\cdot \lambda \partial_{\lambda} a \right) D_a^{- 1} (f) \right\}
\end{align*}
where we used the induction hypothesis for the first term, and Lemma \ref{unknown} for the second.  Since $a$ is of type $0$, we check that $\mathcal{N}_{0, 0}^{n, m} \left(
  \int_1^\cdot \lambda \partial_{\lambda} a \right) \lesssim 1$ for any $n, m \in
  \mathbb{N}$, which implies by Lemma \ref{hotshot} that $ \mathcal{N}_{b, c}^{n, m} \left( \left( \int_1^\cdot \lambda
     \partial_{\lambda} a + \lambda \partial_{\lambda} \right) f \right)
     \lesssim \mathcal{N}_{b, c}^{n+1, m} (f) $ and by the induction hypothesis,
$$
\mathcal N^{n,m+1}_{b,c-1}\left\{D_a^{- 1} \left( \left( \int_1^\cdot \lambda \partial_{\lambda} a +
    \lambda \partial_{\lambda} \right) f \right)\right\} \lesssim \mathcal N^{n,m}_{b,c}\left\{ \left( \int_1^\cdot \lambda \partial_{\lambda} a +
    \lambda \partial_{\lambda} \right) f \right\}\lesssim \mathcal{N}_{b, c}^{n+1, m} (f).
$$
The last term can be estimated similarly, so that $\mathcal N^{n+1,m+1}_{b,c-1}(D_a^{-1}f)\lesssim \mathcal{N}_{b, c}^{n+1, m}$. This shows the desired estimate is true for $n+1$, and hence for all integers by induction.

\medskip
  
\noindent  2) We now turn to the case of type $> 0$ or $<0$. Then for $n = m = 0$, if $a$ is of type $>0$,
\begin{align*}
| D_a^{- 1} (f) | & \lesssim \int_1^r e^{(t - r) a_0} | f (t, \lambda) |
   \dd t \lesssim  \int_1^{r / 2} e^{(t - r) a_0} | f (t, \lambda) |\dd t +
    \int_{r / 2}^r e^{(t - r) a_0} | f (t, \lambda) |\dd t\\
    & \lesssim  \left( \frac{e^{- \frac{a_0 r}{2}}}{(1 + \nu)^c} +
    \frac{1}{r^b (r + \nu)^c} \int_{r / 2}^r e^{(t - r) a_0}\dd t \right)
    \mathcal{N}_{b, c}^{0, 0} (f) \lesssim  \frac{\mathcal{N}_{b, c}^{0, 0} (f)}{r^b (r + \nu)^c} .
\end{align*}
  This implies that $\mathcal{N}_{b, c}^{0, 0} (D_a^{- 1} (f)) \lesssim
  \mathcal{N}_{b, c}^{0, 0} (f)$. If $a$ is of type $< 0$, we have instead
  \begin{eqnarray*}
    | D_a^{- 1} (f) | \lesssim \int_r^{+ \infty} e^{(r - t) a_0} | f (t,
    \lambda) |\dd t \lesssim  \frac{1}{r^b (r + \nu)^c} \int_r^{+ \infty} e^{(r - t) a_0}
   \dd t\mathcal{N}_{b, c}^{0, 0} (f) \lesssim \frac{\mathcal{N}_{b, c}^{0, 0} (f)}{r^b (r + \nu)^c}
  \end{eqnarray*}
  leading also to $\mathcal{N}_{b, c}^{0, 0} (D_a^{- 1} (f)) \lesssim
  \mathcal{N}_{b, c}^{0, 0} (f)$.
  
In both cases, let us now show the results for $n = 0$ by induction on
  $m$. By Lemmas \ref{khen} and \ref{unknown}.1, we have
  \begin{eqnarray*}
    (r \partial_r)^{m + 1} (D_a^{- 1} f) & = & (r \partial_r)^{m + 1} (P_a
    (f)) - (r \partial_r)^{m + 1} \left( \frac{1}{\widetilde{a}} D_a^{- 1}
    (D_{\bar{a}} f) \right)\\
    & = & (r \partial_r)^{m + 1} (P_a (f)) - \frac{1}{\widetilde{a}} (r
    \partial_r)^m (r D_{\bar{a}} f - r a D_a^{- 1} D_{\bar{a}} f) .
  \end{eqnarray*}
  The term with $P_a$ is estimated in Lemma \ref{khen}, and we check with
  \[ \mathcal{N}_{- 1, 0}^{n, m} (r) +\mathcal{N}_{0, 0}^{n, m} (a)
     +\mathcal{N}_{2, 0}^{n, m} (\bar{a}) \lesssim 1 \]
  that
  \[ \mathcal{N}_{b, c}^{0, 0} ((r \partial_r)^m (r D_{\bar{a}} f)) \lesssim
     \mathcal{N}_{b, c}^{0, m + 1} (f) \]
  and by the induction hypothesis,
  \[ \mathcal{N}_{b, c}^{0, 0} ((r \partial_r)^m (r a D_a^{- 1} D_{\bar{a}}
     f)) \lesssim \mathcal{N}_{b, c}^{0, m + 1} (f), \]
  concluding the proof of $\mathcal{N}_{b, c}^{0, m} (D_a^{- 1} (f)) \lesssim
  \mathcal{N}_{b, c}^{0, m} (f)$.

  Let us now estimate $\lambda \partial_\lambda D_a^{-1}(f).$ We denote $g = \int_1^. \lambda \partial_{\lambda} a$ that now satisfies
\[ \mathcal{N}_{- 1, 0}^{n, m} (g) \lesssim 1. \]
By Lemma \ref{unknown}.2, we have that
\[ \lambda \partial_{\lambda} (D_a^{- 1} (f)) = D_a^{- 1} (\lambda
   \partial_{\lambda} f) + D_a^{- 1} (f g) - g D_a^{- 1} (f) . \]
By Lemma \ref{khen}, we compute that
\[ D_a^{- 1} (f g) - g D_a^{- 1} (f) = P_a (f g) - g P_a (f) -
   \frac{1}{\widetilde{a}} (D_a^{- 1} (D_{\bar{a}} (f g)) - g D_a^{- 1}
   D_{\bar{a}} (f)) . \]
If $a$ is of type $< 0$ (or type $i$), then $P_a (f g) - g P_a (f) = 0$, and
if $a$ is of type $> 0$, since $g (1, \lambda) = 0$, we have
\[ \widetilde{P}_a (f) := P_a (f g) - g P_a (f) = \frac{g (r, \lambda)
   e^{\int_r^1 a} f (1, \lambda)}{\widetilde{a}} . \]
Since $\left| e^{\int_r^1 a} \right| \leqslant e^{(1 - r) a_0}$ and $a_0 > 0$,
we check that for any $n, m, b, c \in \mathbb{N}$,
\[ \mathcal{N}_{b, c}^{n, m} (\widetilde{P}_a (f)) \lesssim \mathcal{N}_{b, c}^{n,
   m} (f) . \]
Now, since $D_{\bar{a}} (f g) = g D_{\bar{a}} (f) + D_0 (g) f$, we have
\begin{eqnarray*}
D_a^{- 1} (D_{\bar{a}} (f g)) - g D_a^{- 1} D_{\bar{a}} (f) =  D_a^{- 1} (D_0 (g) f) + D_a^{- 1} (g D_{\bar{a}} (f)) - g D_a^{- 1}  D_{\bar{a}} (f).
\end{eqnarray*}
We have therefore shown the identity
\[ \lambda \partial_{\lambda} D_a^{- 1} = D_a^{- 1} \lambda \partial_{\lambda}
   + \widetilde{P}_a - \frac{1}{\widetilde{a}} (D_a^{- 1} (D_0 (g) \operatorname{Id} + g
   D_{\bar{a}}) - g D_a^{- 1} D_{\bar{a}}) . \]
For any $n, m, b, c \in \mathbb{N}$, we have
\[ \mathcal{N}_{0, 0}^{n, m} (D_0 (g)) \lesssim 1, \quad \mathcal{N}_{- 1, 0}^{n, m} (g) \lesssim 1, \quad 
\mathcal{N}_{b, c}^{n, m} (D_{\bar{a}} (f)) \lesssim
   \mathcal{N}_{b + 1, c}^{n + 1, m} (f) \]
which implies in particular that
\[ \mathcal{N}_{b, c}^{n, m} (D_0 (g) f + g D_{\bar{a}} (f)) \lesssim
   \mathcal{N}_{b, c}^{n + 1, m} (f) . \]
We can now conclude by induction as in the first case.

\medskip
\noindent 3) Since $a$ is of type $i$, using (\ref{ytbh}) we check as in the first
  case that
  \[ \mathcal{N}_{b, c - 1}^{0, 1} (D_a^{- 1} (f)) \lesssim \mathcal{N}_{b,
     c}^{0, 0} (f) . \]
  Now, by Lemma \ref{khen}, we also have
  \[ D_a^{- 1} f = P_a - \frac{1}{\widetilde{a}} D_a^{- 1} (D_{\bar{a}} f), \]
  and since $\mathcal{N}_{b + 1, c}^{0, 0} (D_{\bar{a}} f) \lesssim \mathcal{N}_{b,
  c}^{0, 1} (f)$ because $\bar{a} \in
  \mathcal{W}_{2} (r_0^+, \lambda_0^{\pm})$, with the previous estimate we have
  \[ \mathcal{N}_{b + 1, c - 1}^{0, 0} (D_a^{- 1} (D_{\bar{a}} f)) \lesssim
     \mathcal{N}_{b, c}^{0, 1} (f) . \]
  With Lemma \ref{khen}, we conclude that
  \[ \mathcal{N}_{b + 1, c - 1}^{0, 0} (D_a^{- 1} (f)) \lesssim
     \mathcal{N}_{b, c}^{0, 1} (f) . \]
Applying once more Lemma \ref{khen}, we have
\[ D_a^{- 1} f = P_a - \frac{1}{\widetilde{a}}(P_aD_{\bar{a}} f -\frac{1}{\widetilde{a}}  D_a^{- 1} D_{\bar{a}}^2 f  ),  \]
and we estimate
\[ r^b(r+\nu)^2 | D_a^{- 1}D_{\bar{a}}^2 f | \lesssim \int_r^{+\infty} t^b(t+\nu)^2 D_{\bar{a}}^2 f(t,\lambda)\dd t \lesssim 
 \mathcal{N}_{b, 2}^{0, 2}(f).\]
With similar estimates for the other terms, we conclude the proof of
\[ \mathcal{N}_{b, 2}^{0, 0}(f) \lesssim \mathcal{N}_{b, 2}^{0, 2}(f). \]
  The extension to $n, m \in \mathbb{N}$ derivatives can be done as in the
  previous case.
\end{proof}

\subsection{Variations of Proposition \ref{grooving}}

\begin{lem} \label{amaru}
  We define the space $\widetilde{\mathcal{W}}_b^{j, k} (r_0^+, \lambda_0^+)$ as
  the space of functions such that for all $l \leqslant j, m \leqslant k$,
  \[ \| r^{b + l} \partial_r^l \partial_{\lambda}^m f (r, \lambda)
     \|_{L^{\infty} ([r_0, + \infty) \times [\lambda_0, + \infty),
     \mathbb{C})} < + \infty \]
  (notice the abscence of the factor $\lambda^m$ compared to the norm for
  $\mathcal{W}_b^{j, k} (r_0^+, \lambda_0^+)$). We extend this definition
  naturally to $\widetilde{\mathcal{W}}_b^{j, k} (r_0^{\pm}, \lambda_0^{\pm}),
  \widetilde{\mathcal{W}}_b (r_0^{\pm}, \lambda_0^{\pm})$. Then, if $a = \widetilde{a}
  + \bar{a}$ with $\widetilde{a} \in \widetilde{\mathcal{W}}_0 (\lambda_0^{\pm}),
  \bar{a} \in \widetilde{\mathcal{W}}_2 (r_0^+, \lambda_0^{\pm})$ and otherwise
  satisfies the other hypothesis of Definition \ref{cadmissible} to have a
  type, and if we define the norms $\widetilde{\mathcal{N}}_{b, c}^{n, m}$ as in
  Definition \ref{phynee}.2 but removing the factor $\lambda^j$, Proposition
  \ref{grooving} still holds after adding tildes to all spaces and norms.
\end{lem}

This result is easily checked by noticing that on $(\lambda
\partial_{\lambda})$ we only used operator laws satisfied by derivatives, and
replacing it by $\partial_{\lambda}$ does not change anything else in the
proofs.

\section{Properties of Bessel functions}

\label{truebesselclassics}

\subsection{Classical properties of Bessel functions}

\begin{lem}[\cite{olver2010nist}]
  \label{besselclassic}
We have the following properties.

\begin{itemize}
\item[a)] Given $j \in \mathbb{N}$, the solutions of $\Delta_r - 1 - \frac{j^2}{r^2}
= 0$ are the modified Bessel functions $I_j$ and $K_j$. They satisfy that for
any $j \in \mathbb{N}^{\ast}, k \in \mathbb{N}$ and $r \in] 0, 1]$,
\[ \left| \partial_r^k \left( \frac{I_j (r)}{r^j} \right) \right| + |
   \partial_r^k (r^j K_j (r)) | \lesssim_{j, k} 1 \]
and $I'_j \geqslant 0, K_j' < 0$ on $\mathbb{R}^{+ \ast}$, with
\[ I_j (r) = \frac{e^r}{\sqrt{2 \pi r}} \left( 1 + O_{r \rightarrow + \infty}
   \left( \frac{1}{r} \right) \right), K_j (r) = \frac{e^{-
   r}}{\sqrt{\frac{2}{\pi} r}} \left( 1 + O_{r \rightarrow + \infty} \left(
   \frac{1}{r} \right) \right) . \]
The Wronskian of $I_j$ and $K_j$ is
\[ W (K_j, I_j) = I_j K_j' - I_j' K_j = - \frac{1}{r}, \]
and finally the solutions of $\left( \Delta_r - 1 - \frac{j^2}{r^2} \right) f
= g$ are given by
\[ f (r) = I_j (r) \left[ A + \frac{1}{2} \int_{c_1}^r t K_j (t) g (t)\dd t
   \right] - K_j (r) \left[ B + \frac{1}{2} \int_{c_2}^r t I_j (t) g (t)\dd t
   \right] \]
where $c_1$, $c_2$, $A$ and $B$ are constants.

\medskip

\item[b)] Given $j \in \mathbb{N}$, the real solutions of $\Delta_r + 1 - \frac{j^2}{r^2}
= 0$ are spanned by the Bessel functions $J_j, Y_j$; combining them gives the Hankel function $ H_j (r) = J_j+ i Y_j$. The functions $J_j$ are
given by the integral formulas
\begin{equation}
J_j (r) = \frac{1}{2 \pi} \int_0^{2 \pi} \cos (jt - r \sin (t)) \dd t
  \label{formulaintJ0}.
\end{equation}
Near $r = 0$, we have that for any $k \in \mathbb{N}$ and $r \in [0, 1]$,
\[ \left| \partial_r^k \left( \frac{J_0 (r) - 1}{r^2} \right) \right|
   \lesssim_k 1, \qquad \left| \partial_r^k \left( \frac{J_1 (r) - \frac{r}{2}}{r^2} \right) \right| \]
and near $r = + \infty$, the Hankel function of the first kind is given by
\[ H_j (r) = J_j (r) + i Y_j (r) = \sqrt{\frac{2}{\pi r}} e^{i \left( r
   - \frac{j \pi}{2} - \frac{\pi}{4} \right)} \left( 1 + O_{r \rightarrow +
   \infty} \left( \frac{1}{r} \right) \right) . \]
The Wronskian of $J_j$ and $Y_j$ is
\[ W (J_j, Y_j) = Y_j J_j' - Y_j' J_j = \frac{2}{\pi r}, \]
and finally, the solutions of $\left( \Delta_r + 1 - \frac{j^2}{r^2} \right) f
= g$ are given by
\[ f (r) = Y_j (r) \left[ A + \frac{\pi}{2} \int_{c_1}^r t J_j (t) g (t)\dd t
   \right] - J_j (r) \left[ B + \frac{\pi}{2} \int_{c_2}^r t Y_j (t) g (t)\dd t
   \right] , \]
  where $c_1, c_2,A,B$ are constants.

\end{itemize}
\end{lem}

\begin{proof}
To obtain the formulas for the resolvents, observe that $\underline{f} =\sqrt{r} f(r)$ solves an equation of the type
$$
\underline{f}'' + V \underline{f} = 0,
$$
for a potential $V(r)$. If $F_1$ and $F_2$ are two independent solutions of this equation with Wronskian
$$
W(F_1,F_2) = w_0.
$$
Then the general solution of 
$$
\underline{f}'' + V \underline{f} = g
$$
is provided by
$$
\underline{f} (r) = \frac{1}{w_0} \left[A + \int_c^r g(s) F_1(s) \dd s \right] F_2(r) + \frac{1}{w_0} \left[B + \int_c^r g(s) F_2(s) \dd s \right] F_1(r),
$$
for constants $c,A,B$.
\end{proof}

\subsection{Properties related to the $\mathcal{W}$ spaces}

\begin{lem}
  \label{abudhabi}

  1) For any $j \in \mathbb{N}$, the modified Bessel function $K_j$ does not
  cancel and we have
  \[ r \rightarrow \sqrt{\frac{2 r}{\pi}} e^r K_j (r) \in 1 +\mathcal{W}_1
     (1^+) . \]
  Furthermore,
  \[ r \rightarrow 2 \frac{K_j' (r)}{K_j (r)} + \frac{1}{r} \in - 2
     +\mathcal{W}_2 (1^+) . \]

  2) For any $j \in \mathbb{N}$, the modified Bessel function $I_j$ does not
  cancel on $\mathbb{R}^{+ \ast}$ and we have
  \[ r \rightarrow \sqrt{2 \pi r} e^{- r} I_j (r) \in 1 +\mathcal{W}_1 (1^+) .
  \]
  Furthermore,
  \[ r \rightarrow 2 \frac{I_j' (r)}{I_j (r)} + \frac{1}{r} \in 2
     +\mathcal{W}_2 (1^+), \]
  and finally
  \[ r \rightarrow 2 \frac{I_j' (r)}{I_j (r)} + \frac{1}{r} \in \frac{2 j +
     1}{r} +\mathcal{W}_0 (1^-) . \]

  3) For any $j \in \mathbb{N}$, we define $H_j = J_j + i Y_j$ where $J_j,
  Y_j$ are the Bessel functions. $H_j$ does not vanish and we have
  \[ r \rightarrow \sqrt{\frac{2 r}{\pi}} e^{i \left( \frac{j \pi}{2} +
     \frac{\pi}{4} \right)} e^{- i r} H_j (r) \in 1 +\mathcal{W}_1 (1^+) . \]
  Furthermore,
  \[ r \rightarrow 2 \frac{H_j'}{H_j} + \frac{1}{r} \in - 2 i +\mathcal{W}_2
     (1^+) . \]
\end{lem}

\begin{proof}
  \
  
  1) $K_j$ solves the equation
  \[ \Delta K_j - K_j - \frac{j^2}{r^2} K_j = 0. \]
  Since $\Delta u = u'' + \frac{u'}{r}$, we have
  \[ \sqrt{r} \Delta \left( \frac{u}{\sqrt{r}} \right) = u'' + \frac{u}{4 r^2}
  \]
  and thus
  \[ e^r \sqrt{r} (\Delta - 1) \left( \frac{v e^{- r}}{\sqrt{r}} \right) = v''
     - 2 v' + \frac{v}{4 r^2} . \]
  We therefore look for a solution of $\Delta - 1 - \frac{j^2}{r^2} = 0$ of
  the form $\frac{v e^{- r}}{\sqrt{r}}$, leading to the equation (where $D_a =
  \partial_r + a$)
  \[ D_{- 2} D_0 (v) = \frac{j^2 - \frac{1}{4}}{r^2} v. \]
  Formally, we write it as
  \[ v = D_0^{- 1} D_{- 2}^{- 1} \left( \frac{j^2 - \frac{1}{4}}{r^2} v
     \right) + 1 \]
  since $D_0 (1) = 0$, and thus with
  \[ \Theta = D_0^{- 1} D_{- 2}^{- 1} \left( \frac{j^2 -
     \frac{1}{4}}{r^2} . \right) \]
  we have $(\tmop{Id} - \Theta) (v) = 1$, hence
  \[ v = \sum_{k = 0}^{+ \infty} \Theta^k (1) . \]
  We recall from Definition \ref{phynee} the norm
  \[ \| f \|_{b, k} = \sum_{l = 0}^k \| r^b r^l \partial_r^l f \|_{L^{\infty}
     ([r_0, + \infty (, \mathbb{C})} \]
  for $k, b \in \mathbb{N}$ and $r_0 \geqslant 1$. Since $\frac{j^2 -
  \frac{1}{4}}{r^2} \in \mathcal{W}_2 (1^+)$, we have for any $k \in
  \mathbb{N}$ that
  \[ \left\| \frac{j^2 - \frac{1}{4}}{r^2} \right\|_{2, k} \lesssim_{j, k} 1.
  \]
  By Proposition \ref{grooving}, since $- 2$ is of type $< 0$, we have
  \[ \left\| D_{- 2}^{- 1} \left( \frac{j^2 - \frac{1}{4}}{r^2} f \right)
     \right\|_{2, k} \lesssim_{j, k} \left\| \frac{j^2 - \frac{1}{4}}{r^2} f
     \right\|_{2, k} \lesssim_{j, k} \| f \|_{0, k} \]
  and since $0$ is of type $0$, we have
  \[ \| \Theta (f) \|_{1, k} = \left\| D_0^{- 1} D_{- 2}^{- 1} \left(
     \frac{j^2 - \frac{1}{4}}{r^2} f \right) \right\|_{1, k} \lesssim_{j, k}
     \left\| D_{- 2}^{- 1} \left( \frac{j^2 - \frac{1}{4}}{r^2} f \right)
     \right\|_{2, k} \lesssim_{j, k} \| f \|_{0, k} . \]
  In particular, this means that
  \[ \| \Theta (f) \|_{0, k} \lesssim_{j, k} \frac{1}{r_0} \| f \|_{0, k} \]
  and thus if we take $r_k \geqslant 1$ large enough (depending on $j$ and
  $k$), $\mathcal{O}$ is a contraction for the norm $\| . \|_{0, k}$ for this
  $r_k$. We can therefore define $v$ on $[r_k, + \infty ($, and extend it on
  $[1, r_k]$ as a solution of a Cauchy problem, since $\frac{v e^{-
  r}}{\sqrt{r}}$ solves the equation $\Delta - 1 - \frac{j^2}{r^2} = 0$, where
  all coefficients are uniformly bounded on $[1, r_k]$, as well as the
  boundaries conditions at $r_k$. Since we can do that for all $k \in
  \mathbb{N}$, we conclude that
  \[ v = 1 + \Theta (v) \in 1 +\mathcal{W}_1 (1^+) . \]
  Now, $\frac{v e^{- r}}{\sqrt{r}}$ solves the equation $\Delta - 1 -
  \frac{j^2}{r^2} = 0$ which has exactly two linearly independent solutions
  $I_j, K_j$, but $I_j$ grows exponentially while $\frac{v e^{- r}}{\sqrt{r}}$
  and $K_j$ decays exponentially, they are therefore colinear. Since by Lemma
  \ref{besselclassic} we have
  \[ K_j (r) = \sqrt{\frac{2}{\pi r}} e^{- r} (1 + o_{r \rightarrow + \infty}
     (1)), \]
  by identification we have
  \[ \frac{v e^{- r}}{\sqrt{r}} = \sqrt{\frac{\pi}{2}} K_j (r) \]
  therefore
  \[ \sqrt{\frac{\pi r}{2}} e^r K_j (r) = v (r) \in 1 +\mathcal{W}_1 (1^+) .
  \]
  By definition of the space $\mathcal{W}_1 (1^+)$, we have
  $\sqrt{\frac{\pi}{2}} v' \in \mathcal{W}_2 (1^+)$, therefore
  \[ \frac{1}{2 \sqrt{r} } e^r K_j (r) + \sqrt{r} e^r K_j (r) + \sqrt{r} e^r
     K_j' (r) \in \mathcal{W}_2 (1^+), \]
  dividing by $\sqrt{r} e^r K_j (r) \in \mathcal{W}_0 (1^+)$ which does not
  cancel, we deduce that
  \[ \frac{1}{2 r} + 1 + \frac{K_j'}{K_j} \in \mathcal{W}_2 (1^+) \]
  hence
  \[ 2 \frac{K_j' (r)}{K_j (r)} + \frac{1}{r} \in - 2 +\mathcal{W}_2 (1^+) .
  \]

  The cases 2 and 3 can be done with the same arguments as in the first case,
  looking now in the case 2 for a solution of $\Delta - 1 - \frac{j^2}{r^2} =
  0$ of the form $\frac{v e^r}{\sqrt{r}}$, and in case 3 for a solution of
  $\Delta + 1 - \frac{j^2}{r^2} = 0$ of the form $\frac{v e^{i r}}{\sqrt{r}}$.
\end{proof}

\begin{lem}
  \label{merrymarry}The solutions of $\Delta_r - 1 +
  \frac{1}{r^2} = 0$ are denoted as the Bessel-type functions $I_i$ and $K_i$.
  They are real valued, do not cancel on $r \in [r_\ast,+\infty($ for some universal constant $r_\ast > 0$, and they satisfy the same properties as the Bessel functions in Lemma \ref{abudhabi} items 1 and 2.
\end{lem}

The proof is identical to the one of Lemma \ref{abudhabi}.1, replacing $j^2$
by $i^2 = - 1$.

\section{Comparison with the flat case} \label{section flat}

It is instructive to compare the formulas and estimates which were obtained in this paper to the flat case, where both become much easier to derive.

We will first consider spectral projectors for the Laplacian for functions restricted to the first angular harmonic, and then the evolution problem obtained by linearizing the Gross-Pitaevskii equation around 1, which is also a stationary solution. Both cases can be treated by resorting to the (standard) Fourier transform.

\subsection{Fourier analysis of functions restricted to the first angular harmonic}
Consider such a function:
\begin{equation} \label{fourier-spherical:identity-f}
u(x)=e^{i\theta}\mathsf u (r), \qquad x = re^{i \theta}.
\end{equation}
Then the Laplace operator becomes
$$
\Delta u(x) =(\Delta_1 \mathsf u)(r)e^{i\theta}, \qquad \Delta_1=\partial_{rr}+\frac{1}{r}\partial_r -\frac{1}{r^2}.
$$
Bessel functions of the first kind are eigenfunctions of $\Delta_1$:
$$
\Delta_1 J_1(\sqrt{\lambda}\cdot)=-\lambda J_1(\sqrt{\lambda}\cdot).
$$
We want to derive the Fourier transform for functions of the form \eqref{fourier-spherical:identity-f}. We start with the Fourier inversion formula in two dimensions
$$
u(x)=\frac{1}{(2\pi)^2}\int e^{i\eta \cdot x}\int e^{-i\eta \cdot y}u(y) \dd y \dd \eta
$$
which gives, with the polar decompositions $|x|=r'$, $\eta=\xi e^{i\theta'}$ and $y=re^{i\theta}$,
$$
\mathsf u(r')=\frac{1}{(2\pi)^2}\int_0^\infty \int_0^{2\pi} e^{ir' \xi \cos\theta'}\int_0^\infty \int_0^{2\pi} e^{-ir\xi \cos(\theta-\theta')}e^{i\theta}\mathsf u(r) \dd\theta \, r \dd r \dd \theta' \xi \dd\xi .
$$
Since
$$
\int_0^{2\pi}e^{ir \xi \cos(\theta)}e^{-i\theta}d\theta = 
2\pi i J_1(\xi r)
$$
by \eqref{formulaintJ0}, this gives
$$
\mathsf u(r')=\int_0^\infty J_1(\xi r')\int_0^\infty J_1(\xi r) \mathsf u(r) \,r \dd r \, \xi \dd \xi
$$

\subsection{Formula for the group stemming from the linearization around $1$}

Besides the vortex which is at the heart of the present article, another stationary solution of the Gross-Pitaevskii equation \eqref{GP} is the function identically equal to $1$. 
Linearizing \eqref{GP} around this solution and complexifying the equation as in Section \ref{subsectionlinearized} by setting
$$
\mathcal{V} = \begin{pmatrix} u \\ v \end{pmatrix} = \begin{pmatrix} w \\ \overline{w} \end{pmatrix}.
$$
results in the equation
$$
i \partial_t \mathcal{V} = \mathcal{G} \mathcal{V}, \qquad 
\mathcal{G} = \begin{pmatrix} - \Delta + 1 & 1 \\ -1 & \Delta - 1 \end{pmatrix},
$$
or after taking the Fourier transform
$$
i \partial_t \widehat{\mathcal{V}} = \widehat{\mathcal{G}} \widehat{\mathcal{V}}, \qquad 
\widehat{\mathcal{G}} = \begin{pmatrix} 1 + |\eta|^2 & 1 \\ -1 & -|\eta|^2 - 1 \end{pmatrix}.
$$
(where $\eta \in \mathbb{R}^2$). In the coordinates
$$
\begin{cases} \alpha = \frac{1}{2}(v_1 + v_2) \\ \beta= \frac{1}{2i}(v_1 - v_2),
\end{cases}
$$
the equation becomes
\begin{align*}
\begin{cases}
\partial_t {\alpha} = - \Delta {\beta} \\
\partial_t {\beta} = (\Delta - 2) {\alpha}.
\end{cases}
\end{align*}
This can be solved to give
\begin{equation}
\label{formulaalphabeta}
\begin{cases}
\alpha = \cos(tH) \alpha_0 + \sin(tH) U \beta_0 \\
\beta = -\sin(tH) U^{-1} \alpha_0 + \cos(tH) \beta_0, 
\end{cases}
\end{equation}
where
\begin{equation}
\label{defHU}
H = \sqrt{-\Delta(2-\Delta)}, \qquad U = \frac{\sqrt{-\Delta}}{\sqrt{2-\Delta}}.
\end{equation}

Coming back to $\mathcal{V} = (u,v)$, this means
\begin{align*}
\mathcal{V}  & = \frac 14 e^{itH}
\begin{pmatrix} 2 - U-U^{-1} & U-U^{-1} \\
U^{-1}-U & 2+ U+U^{-1} \end{pmatrix} \mathcal{V}_0 + \frac 14 e^{-itH}
\begin{pmatrix} 2 + U+U^{-1} & U^{-1}-U \\
U-U^{-1} & 2- U-U^{-1} \end{pmatrix} \mathcal{V}_0 \\
& = \frac 1 2 \sum_{\pm} e^{\pm it H} M_{\pm} \mathcal{V}_0.
\end{align*}

We now observe that the symbols $m_{\pm}$ of $M_{\pm}$ can be written
\begin{align*}
& m_{-}(\eta) = \frac{1}{|\eta| \sqrt{|\eta|^2 + 2}}
\begin{pmatrix} |\eta|^2 + 1 + |\eta| \sqrt{|\eta|^2+2} & 1 \\ -1 & - |\eta|^2 - 1 + |\eta| \sqrt{|\eta|^2+2}\end{pmatrix} \\
& m_+(\eta) =  \frac{1}{|\eta| \sqrt{|\eta|^2 + 2}}
\begin{pmatrix} -  |\eta|^2 -1 + |\eta| \sqrt{ |\eta|^2 + 2} & -1 \\ 1 & |\eta|^2 + 1 + |\eta| \sqrt{|\eta|^2+2} \end{pmatrix}.
\end{align*}

For $\rho>0$, this can be written
$$
m_{\pm}(\eta) = \mp \frac{1+\eta^2}{\eta \sqrt{2 + \eta^2}} e(\mp \eta) [\sigma_3 e(\mp \eta)]^\top =  \mp \frac{\lambda'(\eta)}{2\eta}  e(\mp \eta) [\sigma_3 e(\mp \eta)]^\top.
$$
Therefore, the formula
$$
\mathcal{V}(t,r) = e^{-it \mathcal{G}} \mathcal{V}_0 = \sum_{\pm} \frac{1}{2(2\pi)^2} \int_{\mathbb{R}^2}  \int_{\mathbb{R}^2} e^{\mp it |\eta| \sqrt{|\eta|^2 +2}} m_{\pm}(\eta) e^{i \eta\cdot (x-y)} \mathcal{V}_0(y) \dd y \dd \eta
$$
becomes for radial functions (with the help of \eqref{formulaintJ0})
$$
\mathcal{V}(t,r) = e^{-it \mathcal{G}} \mathcal{V}_0 = \frac 1 \pi \int_{-\infty}^\infty e^{-it \lambda(\xi)}\sqrt{\frac{\pi}2} J_0(\xi r) e(\xi) \int_0^\infty \left[ \sqrt{\frac{\pi}2}J_0(\xi s)\sigma_3 e(\xi) \cdot \mathcal{V}_0(s) \right]  s \dd s \, {\lambda'(\xi)} \operatorname{sign}(\xi)  \dd \xi.
$$
This is the analog of \eqref{formulagroupxi}. Note that the role of the generalized eigenfunctions $\psi(\xi,r)$ of $\mathcal{H}$ is played by the generalized eigenfunctions $\sqrt{\frac{\pi}2} J_0(\xi r) e(\xi)$ of $\mathcal{G}$. Here, the prefactor $\sqrt{\frac{\pi}2}$ is meant to ensure that the normalizations at $r = \infty$ of these generalized eigenfunctions agree; Then the prefactors $\frac{1}{\pi}$ above and in \eqref{formulagroupxi} agree.

To develop further the analogy, we will denote the generalized eigenfunctions of $\mathcal{G}$ as
$$
\overset{\circ}{\psi}(\xi,r) = \sqrt{\frac{\pi}2} J_0(\xi r) e(\xi)
$$
and the associated distorted Fourier transform with its inverse
\begin{align*}
& \overset{\circ}{\mathcal{F}}(\phi)(\xi) = \int_0^\infty \overset{\circ}{\psi}(\xi,r) \cdot \sigma_3 \phi(s)  s \dd s \\
& \overset{\circ}{\mathcal{F}}^{-1}(\zeta)(r) = \frac 1 \pi \int_{-\infty}^\infty \zeta(\xi) \overset{\circ}{\psi}(\xi,r) {\lambda'(\xi)} \operatorname{sign}(\xi)  \dd \xi 
\end{align*}
so that the above formula becomes
$$
e^{-it \mathcal{G}} \mathcal{V}_0 = \frac 1 \pi \int_{-\infty}^\infty e^{-it \lambda(\xi)}\sqrt{\frac{\pi}2} J_0(\xi r) e(\xi) \overset{\circ}{\mathcal{F}}(\mathcal{V}_0)(\xi) \, {\lambda'(\xi)} \operatorname{sign}(\xi)  \dd \xi.
$$

Of particular importance is the generalized eigenspace of $\mathcal{G}$ at 0: straightforward computations reveal that
$$
\mathcal{G} \overset{\circ}{\psi}(0,r) = 0, \qquad \mathcal{G} \partial_\xi \overset{\circ}{\psi}(0,r) = 2 \overset{\circ}{\psi}(0,r)
$$
where
$$
\overset{\circ}{\psi}(0,r) = \frac{\sqrt{\pi}}{2} \begin{pmatrix}  1 \\ -1 \end{pmatrix}, \qquad \partial_\xi \overset{\circ}{\psi}(0,r) = \frac{\sqrt{\pi}}{2\sqrt 2} \begin{pmatrix}  1 \\ 1 \end{pmatrix}.
$$

\subsection{Estimates for the group stemming from the linearization around $1$} 

\subsubsection{$L^2$ estimates} It follows from the formula \eqref{formulaalphabeta} that
\begin{align*}
& \| (\alpha(t),\beta(t) \|_{L^2} \lesssim \langle t \rangle \| (\alpha_0,\beta_0) \|_{L^2} \\
& \| (\alpha(t), U \beta(t) \|_{L^2} \lesssim \| (\alpha_0,\beta_0 ) \|_{L^2},
\end{align*}
and the factor $\langle t \rangle$ on the right-hand side is optimal.
Turning to the case where $\alpha,\beta \in \mathcal{C}_0^\infty$, we see that
\begin{align*}
& \| (\alpha(t),\beta(t) \|_{L^2} \sim \langle \log t \rangle \qquad \mbox{if $\widehat{\alpha}(0) \neq 0$} \\
& \| (\alpha(t),\beta(t) \|_{L^2} \sim 1  \quad \qquad \mbox{if $\widehat{\alpha}(0) = 0$}.
\end{align*}

\subsubsection{$L^\infty$ estimates} We start with the following estimate, which appears in \cite{GustafsonNakanishiTsai3}
\begin{equation}
\label{estimateGNT}
\| e^{itH} \phi \|_{\dot{B}^0_{\infty,2}} \lesssim t^{-1 + \frac{\theta}{2}} \| U^{\frac{3}{2}\theta} \phi \|_{\dot{B}_{1,2}^0} \qquad \mbox{if $0\leq \theta \leq 1$ and $\operatorname{Supp} \widehat{\phi} \subset B(0,1)$} 
\end{equation}
(here, $\dot{B}^s_{p,q}$ refers to the standard homogeneous Besov space while $H$ and $U$ are defined in \eqref{defHU}; we stated this estimate for functions localized on low frequencies to avoid further technical details).

Applying this estimate in conjunction with the formula \eqref{formulaalphabeta} gives in particular (for $\theta=0$ and $\theta = 2/3$ respectively)
\begin{align*}
& \| (U^{-1} \alpha (t), \beta(t) ) \|_{\dot{B}^0_{\infty,2}} 
\lesssim  t^{-1} \| (U^{-1} \alpha_0, \beta_0 ) \|_{\dot{B}^0_{1,2}}\\
& \| (U^{-1} \alpha (t), \beta(t) ) \|_{\dot{B}^0_{\infty,2}} 
\lesssim  t^{-\frac 23} \| ( \alpha_0, U \beta_0 ) \|_{\dot{B}^0_{1,2}}.
\end{align*}

To understand better what the estimate \eqref{estimateGNT} really implies, it is helpful to state the pointwise decay which it entails for Schwartz class functions: 
\begin{align*}
& \|  (U^{-1} \alpha (t), \beta(t) ) \|_{L^\infty} \lesssim t^{-1} \qquad \mbox{if $\alpha_0, \beta_0 \in \mathcal{S}$, $\widehat{\alpha_0}(0) = \widehat{\beta_0}(0) = 0$}\\
& \| (U^{-1} \alpha (t), \beta(t) ) \|_{L^\infty} \lesssim t^{-\frac 23 +} \qquad \mbox{if $\alpha_0, \beta_0 \in \mathcal{S}$}.
\end{align*}

\subsubsection{The vanishing conditions}
How are we to understand the vanishing conditions on $\widehat{\alpha_0}(0)$, $\widehat{\beta_0}(0)$, $\partial_\xi \widehat{\alpha_0}(0)$ in terms of the distorted Fourier transform? We claim that
\begin{itemize}
\item $\widehat{\alpha_0}(0)=0$ if and only if $\overset{\circ}{\mathcal{F}}(\mathcal{V}_0)(0) = 0$.
\item $\widehat{\beta_0}(0) =0$ if and only if $\partial_\xi \overset{\circ}{\mathcal{F}}(\mathcal{V}_0)(0) = 0$.
\end{itemize}

Indeed,
\begin{align*}
& \overset{\circ}{\mathcal{F}}(\mathcal{V}_0)(0) = 
 \int_0^\infty 
\sigma_3 \overset{\circ}{\psi}(0,r) \cdot  \mathcal{V}_0(s) s \dd s 
= 
\sqrt{\frac{\pi}{2}} \int_0^\infty \begin{pmatrix}1 \\ 1 \end{pmatrix} \cdot  \mathcal{V}_0(s) s \dd s = {\frac{2\sqrt{\pi}}{\sqrt 2}} \widehat{\alpha}(0) \\
& \partial_\xi \overset{\circ}{\mathcal{F}}(\mathcal{V}_0)(0) =  \int_0^\infty 
\sigma_3 \partial_\xi \overset{\circ}{\psi}(0,r) \cdot  \mathcal{V}_0(s) s \dd s =\frac{\sqrt{\pi}}{2 \sqrt 2} \int_0^\infty \begin{pmatrix}1 \\ -1 \end{pmatrix} \cdot  \mathcal{V}_0(s) s \dd s = i \sqrt{\frac{\pi}{2}} \widehat{\beta}(0).
\end{align*}

\bibliographystyle{abbrv}

\bibliography{references}

\end{document}